%% file: main.tex
\documentclass[letterpaper, 11pt]{article}

\usepackage[utf8]{inputenc}
\usepackage[T1]{fontenc}
\usepackage{lmodern}
\usepackage{microtype}
\usepackage[margin=1.2in]{geometry}

\usepackage{amsmath, amssymb, amsfonts, amsthm}
\usepackage{mathtools}                 
\usepackage{mathrsfs}                  
\usepackage{dsfont}                    
\usepackage{bm}                        
\usepackage{stmaryrd}                  
\usepackage{latexsym}                  
\usepackage{wasysym}                   
\usepackage{harmony}                   

\usepackage{graphicx}
\graphicspath{{Figures/}}
\usepackage[abs]{overpic}
\usepackage{float}
\usepackage{caption}
\usepackage{subcaption}
\usepackage{tikz-cd}
\usepackage{quiver}
\usetikzlibrary{shapes.geometric, arrows.meta, positioning}
\usepackage{placeins}
\usepackage{morefloats}
\newcommand{\diagcell}[2]{%
  \begin{tikzpicture}[baseline={(0,-1.5em)}]
    \draw[line width=0.4pt] (-0.5em,0.35em) -- (3.9em,-3.0em);
    \node[anchor=north east, inner sep=2.5pt] at (3.9em,0.35em) {#2};
    \node[anchor=south west, inner sep=2.5pt] at (-0.5em,-3.0em) {#1};
  \end{tikzpicture}}

\usepackage{array}
\usepackage{tabularray}
\usepackage{colortbl}
\usepackage{makecell}
\usepackage{enumitem}
\usepackage{mdframed}
\usepackage{etoolbox}
\usepackage{epigraph}
\usepackage[size=tiny]{todonotes}      
\usepackage{comment}
\newcommand{\todojz}[1]{\todo[color=green!30,linecolor=green!80!black,size=\footnotesize]{#1}}
\newcommand{\jz}[1]{\todo[inline,color=green!30,linecolor=green!80!black,size=\footnotesize]{#1}}
\newcommand{\tm}[1]{\todo[color=blue!30,linecolor=blue!80!black,size=\footnotesize]{#1}}
\newcommand{\unneeded}[1]{}

\usepackage[backend=biber, maxnames=4, style=alphabetic, sorting=nyt]{biblatex}
\AtEveryBibitem{%
  \clearfield{urlyear}\clearfield{urlmonth}\clearfield{urlday}%
  \clearlist{urldate}\clearlist{note}%
  \ifboolexpr{ test {\ifentrytype{online}}
            or test {\ifentrytype{misc}}
            or test {\ifentrytype{software}} }
    {}{\clearfield{url}}%
}

\usepackage[pdfencoding=auto, psdextra, linktocpage=true]{hyperref}
\hypersetup{colorlinks, citecolor=teal, linkcolor=purple, urlcolor=teal}
\usepackage{zref-clever}
\zcsetup{cap, noabbrev}
\zcsetup{countertype={thm=theorem, thmintro=theorem}}

\NewDocumentCommand{\SetThmType}{mmm}{
  \zcsetup{countertype={#1=#3}}%
  \AddToHook{env/#1/begin}{\zcsetup{countertype={#2=#3}}}}
\SetThmType{prop}{thm}{proposition}
\SetThmType{lem}{thm}{lemma}
\SetThmType{cor}{thm}{corollary}
\SetThmType{conj}{thm}{conjecture}
\SetThmType{question}{thm}{question}
\SetThmType{defn}{thm}{definition}
\SetThmType{conv}{thm}{convention}
\SetThmType{ex}{thm}{example}
\SetThmType{constr}{thm}{construction}
\SetThmType{rem}{thm}{remark}
\SetThmType{longrem}{thm}{remark}
\SetThmType{assuminner}{thm}{assumption}
\SetThmType{propintro}{thmintro}{proposition}
\SetThmType{corintro}{thmintro}{corollary}

\zcRefTypeSetup{conjecture}{Name-sg=Conjecture, name-sg=conjecture, Name-pl=Conjectures, name-pl=conjectures}
\zcRefTypeSetup{question}{Name-sg=Question, name-sg=question, Name-pl=Questions, name-pl=questions}
\zcRefTypeSetup{convention}{Name-sg=Convention, name-sg=convention, Name-pl=Conventions, name-pl=conventions}
\zcRefTypeSetup{construction}{Name-sg=Construction, name-sg=construction, Name-pl=Constructions, name-pl=constructions}
\zcRefTypeSetup{assumption}{Name-sg={Standing Assumption}, name-sg={standing assumption}, Name-pl={Standing Assumptions}, name-pl={standing assumptions}}

\numberwithin{equation}{section}

\theoremstyle{plain}
\newtheorem{thm}{Theorem}[section]
\newtheorem*{thm*}{Theorem}
\newtheorem{prop}[thm]{Proposition}
\newtheorem{lem}[thm]{Lemma}
\newtheorem{cor}[thm]{Corollary}
\newtheorem{conj}[thm]{Conjecture}

\newtheorem{question}[thm]{Question}

\theoremstyle{definition}
\newtheorem{defn}[thm]{Definition}

\newtheorem{assuminner}[thm]{Standing Assumption}
\newenvironment{assum}
  {\begin{mdframed}[linewidth=1pt,skipabove=10pt,skipbelow=10pt,
                    innertopmargin=6pt,innerbottommargin=6pt]%
   \vspace{-\topsep}\begin{assuminner}}
  {\end{assuminner}\end{mdframed}}
  
\newtheorem{ex}[thm]{Example}
\newtheorem{constr}[thm]{Construction}
\newtheorem{rem}[thm]{Remark}
\newtheorem{longrem}[thm]{Remark}

\theoremstyle{plain}
\newtheorem{thmintro}{Theorem}

\newtheorem{propintro}[thmintro]{Proposition}
\newtheorem{corintro}[thmintro]{Corollary}

\newcommand{\qeddiamond}{\hfill$\Diamond$}
\AtEndEnvironment{ex}{\qeddiamond}
\AtEndEnvironment{constr}{\qeddiamond}
\AtEndEnvironment{longrem}{\qeddiamond}

\newcommand{\Z}{\mathbb{Z}}
\newcommand{\Q}{\mathbb{Q}}
\newcommand{\R}{\mathbb{R}}

\newcommand{\T}{\mathbb{T}}
\newcommand{\F}{\mathcal{F}}
\newcommand{\G}{\mathcal{G}}
\newcommand{\Zg}{\mathcal{Z}}

\DeclareMathOperator{\fs}{fs}
\DeclareMathOperator{\fineslope}{fs}   
\DeclareMathOperator{\slope}{s}

\DeclareMathOperator{\interior}{int}
\DeclareMathOperator{\guts}{guts}

\DeclareMathOperator{\denominator}{denominator}
\DeclareMathOperator{\Mod}{Mod}
\DeclareMathOperator{\Homeo}{Homeo}
\DeclareMathOperator{\Hom}{Hom}
\DeclareMathOperator{\twig}{tw}
\DeclareMathOperator{\rot}{rot}
\DeclareMathOperator{\mirror}{mirror}

\newcommand{\ssq}{{\mathord{\scalebox{0.45}{$\square$}}}}
\newcommand{\xsquigright}[1]{%
  \begin{tikzcd}[ampersand replacement=\&]
    {}\&{}
    \arrow["{#1}", squiggly, from=1-1, to=1-2]
  \end{tikzcd}}

\title{Ziggurats, taut foliations, and contact structures}

\author{Thomas Massoni\thanks{Department of Mathematics, Stanford University, Stanford, USA. Email: \href{mailto:tmassoni@stanford.edu}{tmassoni@stanford.edu}.} \and
Jonathan Zung\thanks{Department of Mathematics, Georgia Institute of Technology, Atlanta, USA. Email: \href{mailto:jzung3@gatech.edu}{jzung3@gatech.edu}.}}

\date{\today}

\begin{document}


\maketitle

\begin{NoHyper}
{\let\thefootnote\relax
 }
\end{NoHyper}

\begin{abstract}
We study the geography of taut foliations on compact $3$-manifolds with toroidal boundary components. The central object of our work is the set of boundary multislopes realized by foliations transverse to a fixed flow on such a manifold. We prove that these sets exhibit remarkable structural properties---rationality, rigidity, and convexity---which motivate the name \emph{ziggurats}. Our main tool, of independent interest, is a two-way correspondence between foliations and contact structures on $3$-manifolds with boundary: we generalize the Eliashberg--Thurston theorem, which produces pairs of positive and negative contact structures from foliations, and a construction of the first author, which builds foliations from such contact pairs. Using both directions of this correspondence, we bring contact-geometric methods to bear on the architecture of ziggurats.
\end{abstract}

\hfill

\begin{figure}[ht]
    \centering
    \includegraphics[width=0.7\linewidth]{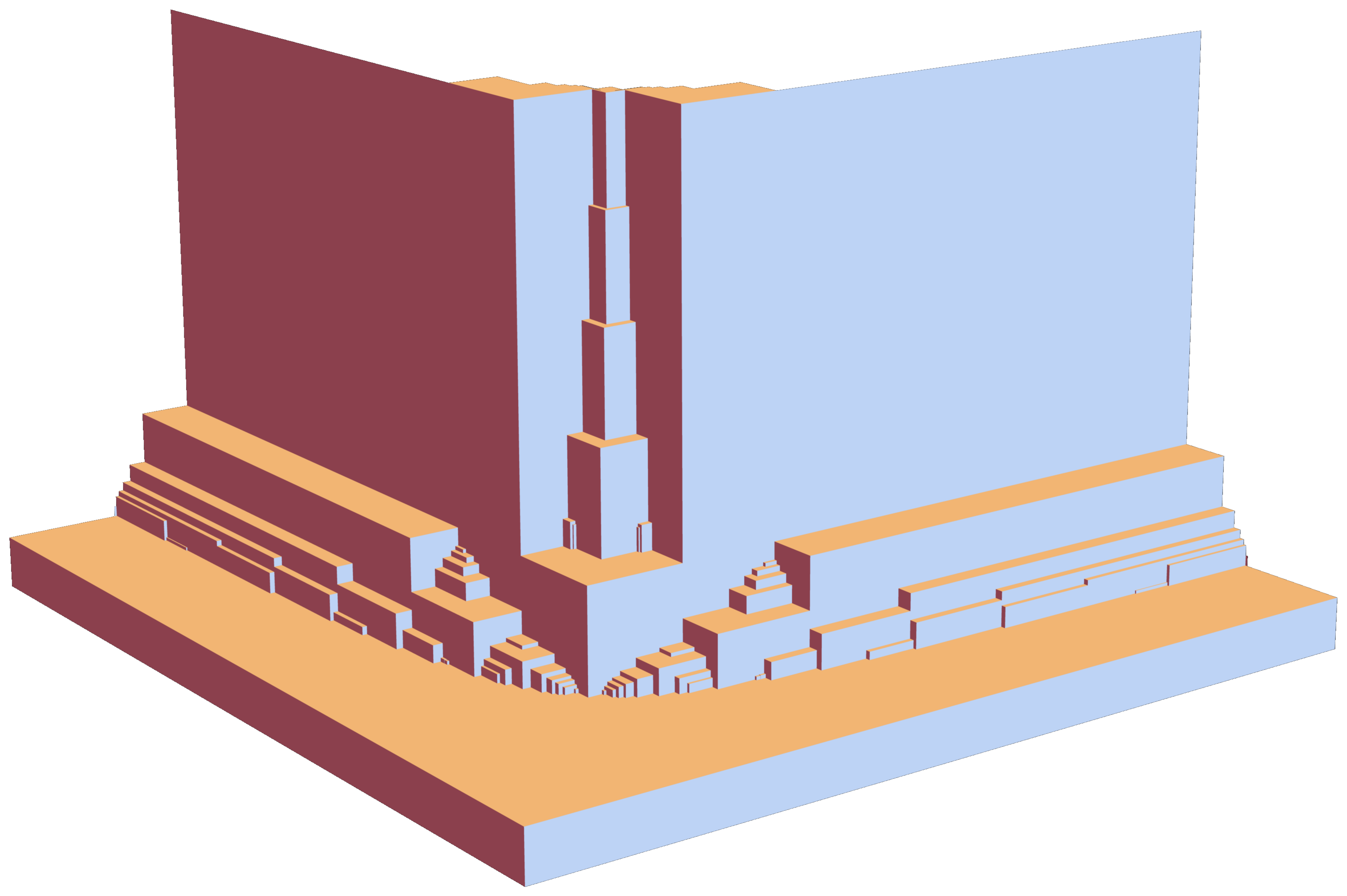}
\end{figure}

\hfill

\clearpage

{\small
\tableofcontents
}

\clearpage

\epigraph{In the middle of the precinct there was a tower of solid masonry, a furlong in length and breadth, upon which was raised a second tower, and on that a third, and so on up to eight. The ascent to the top is on the outside, by a path which winds round all the towers. When one is about half-way up, one finds a resting-place and seats, where persons are wont to sit some time on their way to the summit. On the topmost tower there is a spacious temple, and inside the temple stands a couch of unusual size, richly adorned, with a golden table by its side.}%
{---\textup{Herodotus}, \emph{Histories} I.181, trans.\ G.~Rawlinson}

    \section{Introduction}

        \subsection{Motivation}

A \textbf{foliation} $\F$ on a $3$-manifold $M$ is a decomposition of $M$ into surfaces locally modeled on $\R^2\times \R$. A foliation is \textbf{taut} if it admits a transverse loop through every point, or equivalently, if it admits a transverse \emph{nonwandering flow} (see \zcref{def:NW}). An outstanding problem in $3$-manifold theory is to determine which $3$-manifolds admit taut foliations.

\begin{question}\label{ques:1}
    Given a nonwandering flow $\Phi$ on $M$, does there exist a (necessarily taut) foliation transverse to $\Phi$?
\end{question}

A basic operation on $3$-manifolds is that of \emph{Dehn surgery} along knots or links. It allows one not only to create new manifolds from old ones, but also to study them \emph{in families}; concretely, one may fix a link in a given $3$-manifold and study all of its Dehn surgeries at the same time by varying the surgery slopes. This Dehn surgery perspective also permits us to focus on simpler flows on manifolds with boundary. For example, in the important cases where $\Phi$ is a periodic flow or a transitive pseudo-Anosov flow, any $(M,\Phi)$ arises from surgery along a collection of orbits of a \emph{suspension flow} (see~\cite{fried.TransitiveAnosovFlows,brunella.SurfacesSectionExpansive} for the pseudo-Anosov case). This motivates the following definition:

\begin{defn}
    Let $M$ be a compact oriented $3$-manifold with $n$ framed toroidal boundary components, and $\Phi$ be a nonwandering flow on $M$ tangent to $\partial M$. We define
    $$\Zg(\Phi) \coloneqq \big\{\bm{s} \mid \text{$\exists$ foliation $\F$ transverse to $\Phi$ with well-defined boundary multislope $\bm{s}$}\big\} \subset \R^n.$$
    We refer to $\Zg(\Phi)$ as the \textbf{(foliation) ziggurat} of $\Phi$, borrowing the evocative terminology coined by Calegari--Walker for related sets arising in circle dynamics (see \zcref{sec:circledynamics}).
\end{defn}

\begin{figure}[ht]
    \centering
        \begin{subfigure}{0.45\linewidth}
    \centering
    \includegraphics[width=\linewidth]{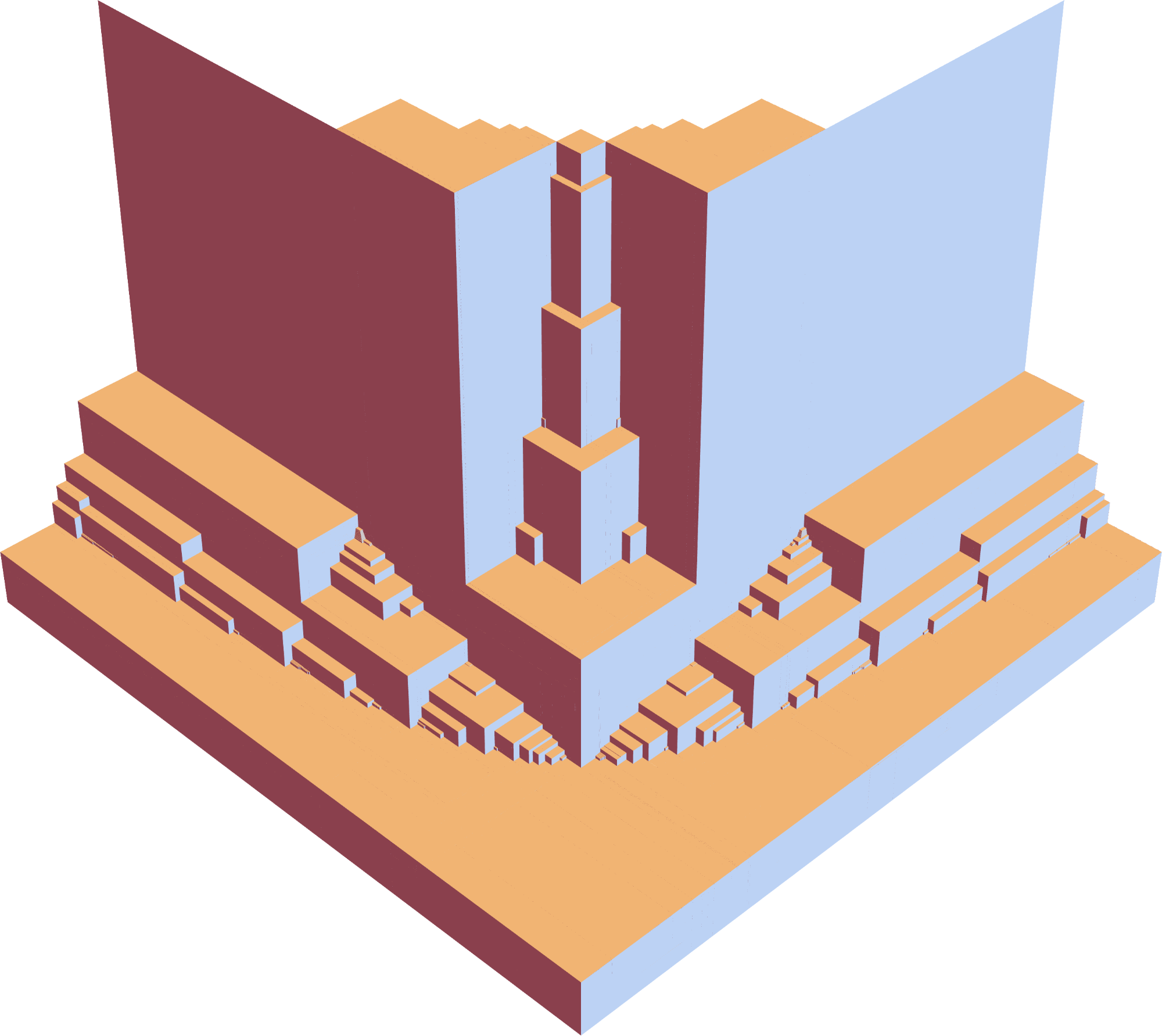}
    \caption{Front.}
    \end{subfigure}
    \hfill
    \begin{subfigure}{0.49\linewidth}
    \centering
    \includegraphics[width=\linewidth]{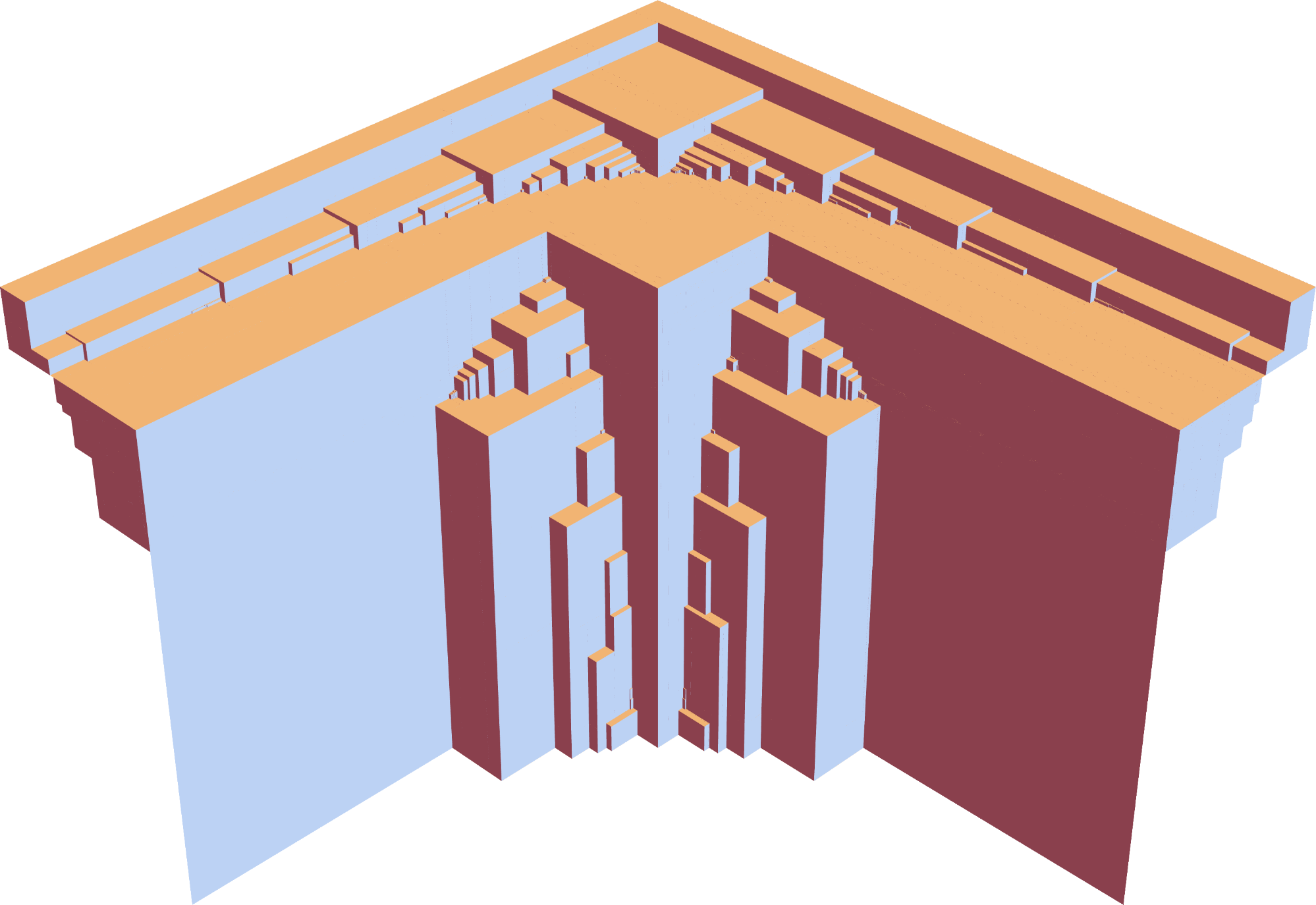}
    \caption{Back, upside down.}
    \end{subfigure}
    
    \caption{$\Zg(\Phi)$ for $\Phi$ a pseudo-Anosov flow on the \emph{magic manifold} $s776$. It is the complement of the minimally twisted $3$-chain link ($6^3_1$ in Rolfsen's table~\cite{Rolfsen}, L6a5 in Thistlethwaite's notation; see \href{https://katlas.org/wiki/L6a5}{\texttt{https://katlas.org/wiki/L6a5}}).}
    \label{fig:magicziggurat}
\end{figure}

A complete understanding of $\Zg(\Phi)$ would essentially answer \zcref{ques:1} on all Dehn fillings of $(M,\Phi)$. The goal of this article is to prove a variety of structural results for $\Zg(\Phi)$ which will explain the moniker ``ziggurat''. Before stating our results, let us survey some examples in which $\Zg(\Phi)$ is well understood. 

    \subsection{Ziggurats for periodic and pseudo-Anosov flows}

A privileged role in the theory is played by periodic and pseudo-Anosov flows. A flow $\Phi$ without fixed points is \textbf{periodic} if there exists $T>0$ such that $\Phi^T$ is the identity. A flow is \textbf{pseudo-Anosov} if it is a hyperbolic flow with finitely many singular orbits modeled on semi-branched covers of standard hyperbolic orbits (\emph{$p$-prongs}); see~\cite{mosher.SurfacesBranchedSurfaces,mosher.DynamicalSystemsHomologyNormI}. 

\begin{conj}[Calegari, Thurston \cite{calegari.PromotingEssentialLaminations, thurston.ThreemanifoldsFoliationsCircles}] \label{conj:almosttransverse}
    If $\F$ is a taut foliation on a closed atoroidal $3$-manifold, then there exists a nonwandering flow $\Phi$ which is almost\footnote{In the sense of Mosher~\cite{mosher.SurfacesBranchedSurfaces}.} transverse to $\F$ and is either periodic or pseudo-Anosov.
\end{conj}

This is a far-reaching generalization of Nielsen--Thurston theory for fibrations over $S^1$. The flow should record how the conformal structures of leaves of $\F$ change in the transverse direction. \zcref{conj:almosttransverse} is known in the Seifert fibered case by work of Thurston and Brittenham \cite{thurston.FoliationsThreeManifoldsWhich,brittenham.EssentialLaminationsSeifertfibered}. In the hyperbolic case, the theory of universal circles pioneered by Thurston in  \cite{thurston.ThreemanifoldsFoliationsCircles} and extensively developed by Calegari, Dunfield, Fenley, and others has proven \zcref{conj:almosttransverse} in many cases~\cite{calegari.dunfield.LaminationsGroupsHomeomorphisms,calegari.GeometryRcoveredFoliations,fenley.RegulatingFlowsTopology,calegari.PromotingEssentialLaminations}. In another line of work, Gabai--Mosher~\cite{mosher.LaminationsFlowsFiniteDepth, Gabai1997problems} and Landry--Tsang~\cite{landry.tsang.EndperiodicMapsSplitting, LandryTsangSutured} have resolved \zcref{conj:almosttransverse} for finite depth foliations. Upcoming work of Landry--Taylor~\cite{landry.taylor.TransverseFiniteness} also proves uniqueness and finiteness results for pseudo-Anosov flows (almost-)transverse to a given finite depth taut foliation.

In view of \zcref{conj:almosttransverse}, an ambitious yet principled route towards a classification of taut foliations on atoroidal $3$-manifolds is first to classify the periodic and pseudo-Anosov flows on $M$, and then to answer \zcref{ques:1} for each of them.

\medskip

\emph{Periodic flows} in dimension $3$ are completely understood via the classification of Seifert fibered spaces. The foliation ziggurat $\Zg(\Phi)$ is also completely understood thanks to work of Eisenbud--Hirsch--Neumann, Jankins--Neumann, Naimi, and others~\cite{eisenbud.hirsch.ea.TransverseFoliationsSeifert, JankinsNeumann1985, Naimi1994}. After cutting along vertical tori, the computation of $\Zg(\Phi)$ reduces to the following two theorems:

\begin{thm}[Milnor--Wood \cite{milnor.ExistenceConnectionCurvature,wood.BundlesTotallyDisconnected}]
    Let $M=\Sigma_g^\circ \times S^1$, where $\Sigma_g^\circ$ is a genus $g \geq 1$ surface with one boundary component, and $\Phi$ be the periodic flow tangent to the $S^1$-direction. The boundary of $M$ is framed so that $\partial \Sigma^\circ_g \times \{\mathrm{pt}\}$ has slope $0$. Then $$\Zg(\Phi)=[-2g+1,2g-1].$$
\end{thm}

\begin{rem}
    The rational points in $(-2g+1,2g-1)$ are attained by foliations with \emph{linear} holonomy at the boundary, so these extend to the corresponding Dehn filling.
\end{rem}

\begin{thm}[Jankins--Neumann~\cite{JankinsNeumann1985}, Naimi~\cite{Naimi1994}]
Suppose $M=(\text{Pair of pants})\times S^1$, with $\Phi$ the periodic flow in the $S^1$ direction. Then the Jankins--Neumann ziggurat $\Zg(\Phi)$ is the subset of $\R^3$ cut out by the following inequality (combined with its sibling obtained by negating all three coordinates):
for $(r_1,r_2,r_3)\in \Zg(\Phi)$,
$$-r_3 \leq \sup_{q\geq 1}  \frac{p_1 + p_2 + 1}q,$$
where $p_1/q \leq r_1$ and $p_2/q \leq r_2$ are the best rational subapproximants of $r_1$ and $r_2$ with denominator $q$.
\end{thm}

\begin{figure}[ht]
    \centering
    \includegraphics[width=0.5\linewidth]{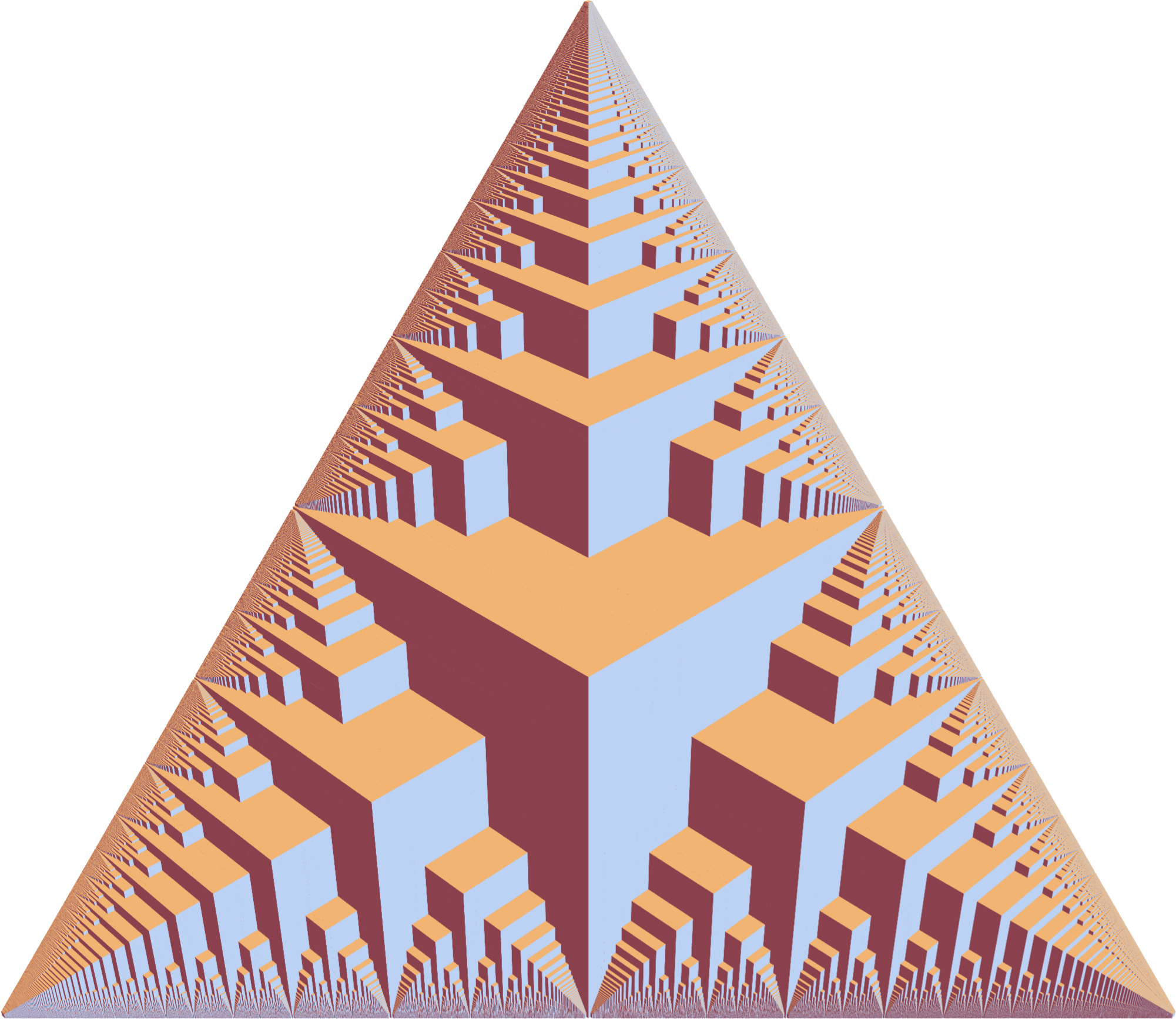}
    \caption{The Jankins--Neumann ziggurat, restricted to $(0,1)\times (0,1) \times (-1,0)$.}
    \label{fig:jankinsneumann}
\end{figure}

\medskip

\emph{Pseudo-Anosov flows}, on the other hand, remain beyond the reach of a complete classification. Despite this, much progress has been made towards the following conjecture:

\begin{conj}\label{conj:finiteness}
    On any closed $3$-manifold, there exist at most finitely many pseudo-Anosov flows up to orbit equivalence.
\end{conj}

Gabai--Li~\cite{GabaiLi2026finiteness} have recently announced a proof of \zcref{conj:finiteness} for hyperbolic $3$-manifolds, and Bowden--Colin have independently announced the Anosov case using different techniques~\cite{bowden.colin.FinitenessAnosovContact}. Still, only a little is known about $\Zg(\Phi)$ in the pseudo-Anosov case. The first existence result was proven by Roberts:

\begin{thm}[\cite{roberts.TautFoliationsPunctured}]
		Suppose $\Phi$ is a pseudo-Anosov suspension flow on a fibration $\Sigma^\circ_{g}\to M \to S^1$, where $\Sigma_g^\circ$ is a genus $g \geq 1$ surface with one boundary component. Then $\Zg(\Phi)$ contains an open interval around the slope of the fiber (the longitude).
\end{thm}

Further progress has been made by Li--Roberts \cite{li.roberts.TautFoliationsKnot} who extended \cite{roberts.TautFoliationsPunctured} to arbitrary (nontrivial) knot complements in $S^3$, Krishna \cite{krishna.TautFoliationsPositive}, Nie \cite{nie.PositiveBraidClosures}, and Zhao \cite{zhao.CoorientableTautFoliations} who studied the maximal interval of slopes contained in $\Zg(\Phi)$ for some families of knots, and Santoro \cite{santoro.LspacesTautFoliations} who extended the techniques to $2$-bridge links. All of these results rely on \cite{li.LaminarBranchedSurfaces} to produce taut foliations from a branched surface obtained from small modifications of a sutured hierarchy.

Using the technology of veering triangulations \cite{agol.IdealTriangulationsPseudoAnosov, gueritaud.VeeringTriangulationsCannonThurston}, we began a systematic experimental exploration of $\Zg(\Phi)$ for many different pseudo-Anosov flows in~\cite{Z24}. Our software is available online \cite{zung.JonathanzungTransversefol}. Several examples of (approximations to) $\Zg(\Phi)$ are shown in \zcref{fig:ziggurat_examples}. These figures display intricate staircasing phenomena, especially near topologically significant multislopes. These experiments were our first hint that features of the Jankins--Neumann ziggurat might reappear in the setting of hyperbolic manifolds. We will explain many of these phenomena in this article.

\begin{figure}[p!]
\centering
    \begin{subfigure}{0.45\linewidth}
    \centering
    \includegraphics[width=\linewidth]{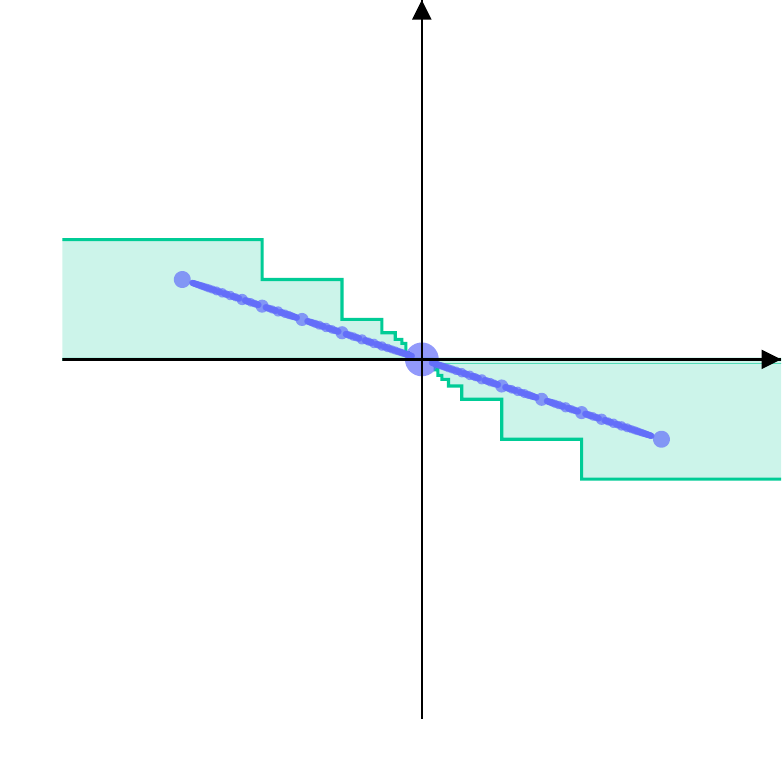}
    \caption{\texttt{eLMkbcddddedde\_2100}}
    \end{subfigure}
    \hfill
    \begin{subfigure}{0.45\linewidth}
    \centering
    \includegraphics[width=\linewidth]{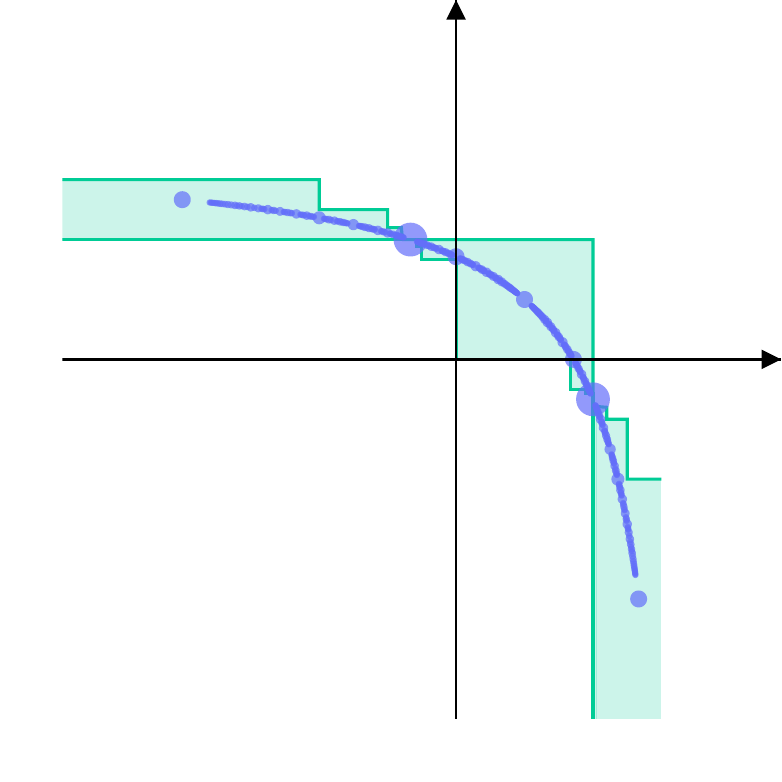}
    \caption{\texttt{fLLQcbeddeehhbghh\_01110}}
    \end{subfigure}
    
    \par\vspace{2em}
    
    \begin{subfigure}{0.45\linewidth}
    \centering
    \includegraphics[width=\linewidth]{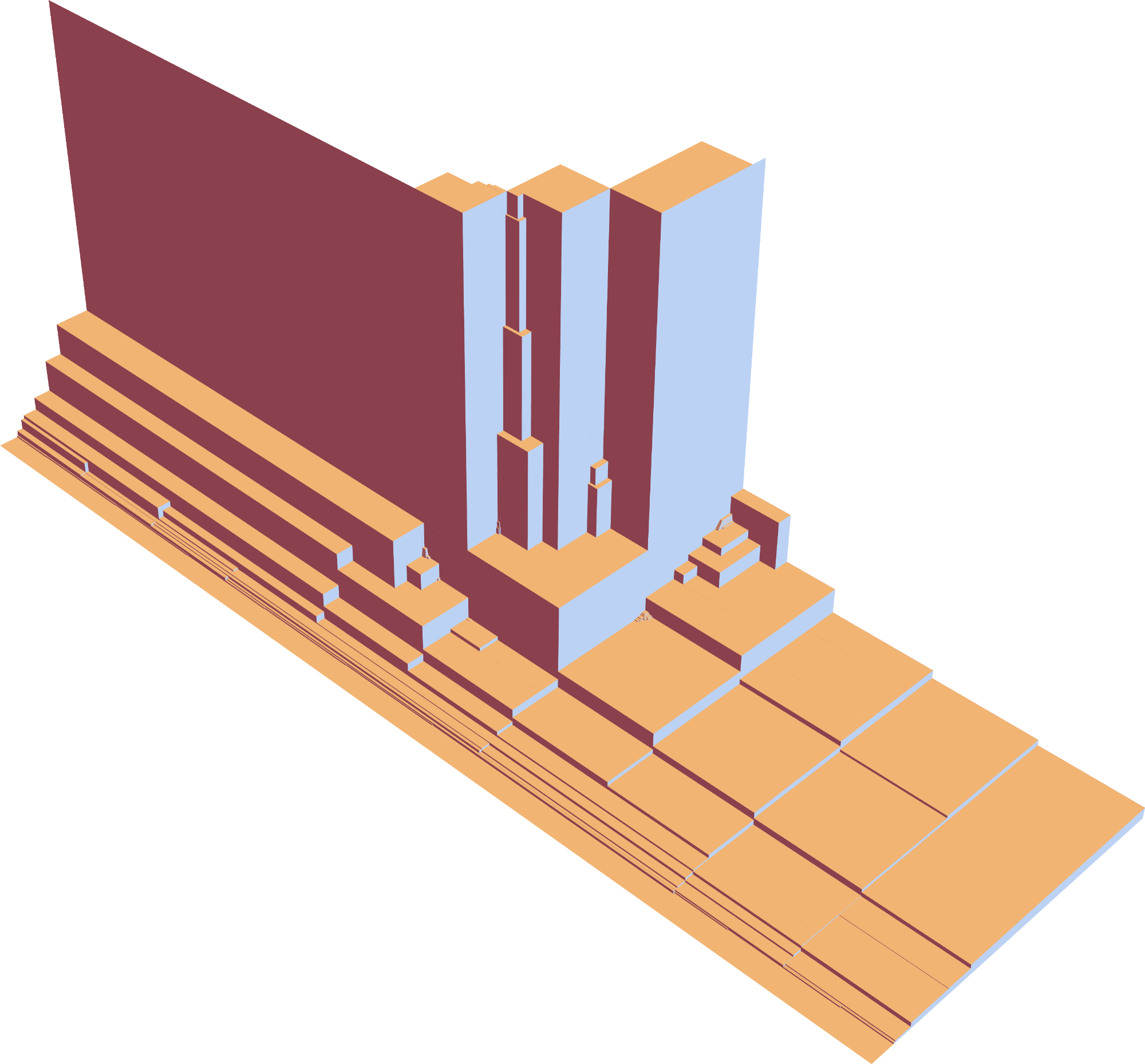}
    \caption{\texttt{ivLLQQccdgffhghhvvaaaaavv\_01111220}}
    \end{subfigure}
    \hfill
    \begin{subfigure}{0.5\linewidth}
    \centering
    \includegraphics[width=0.8\linewidth]{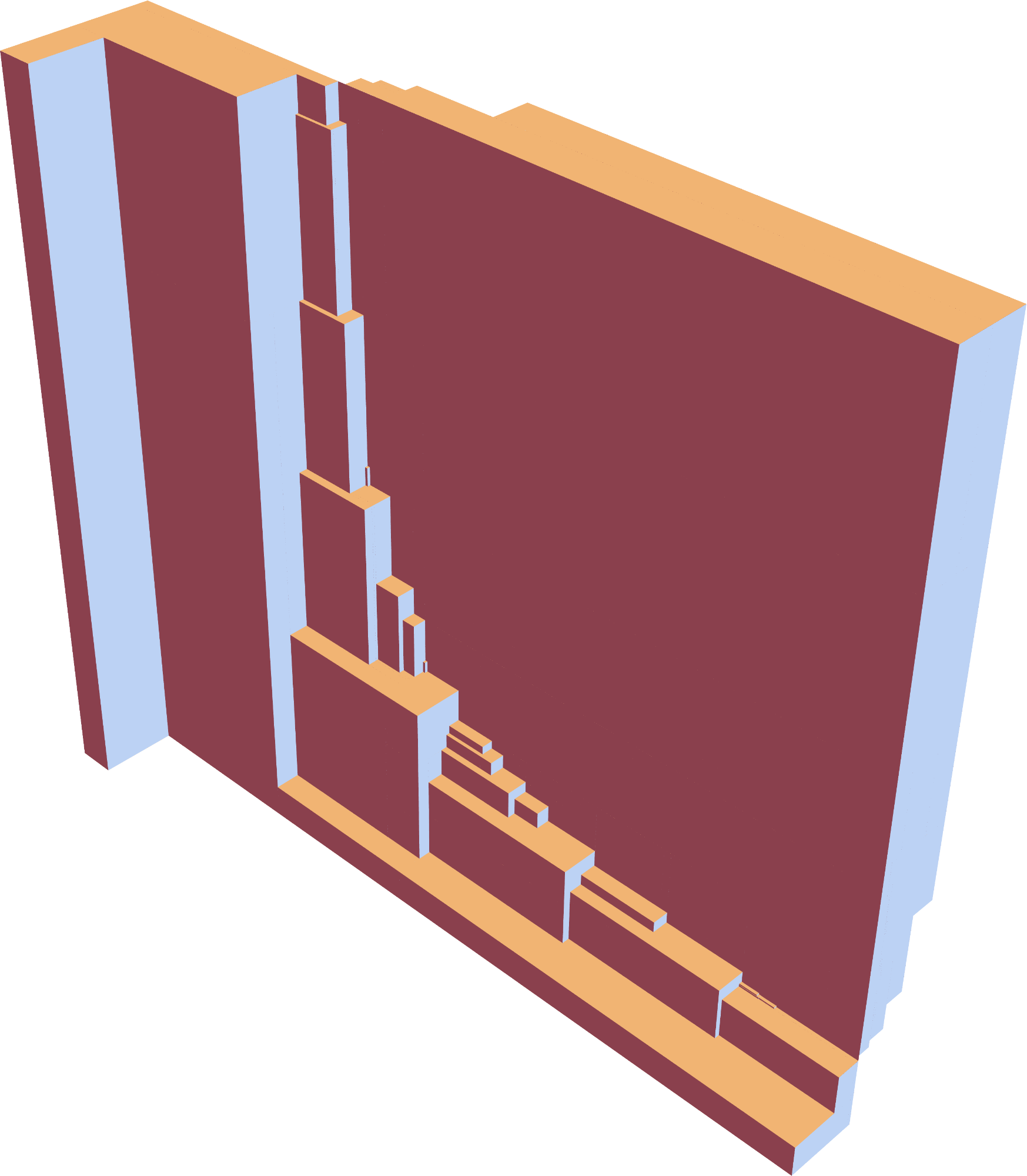}
    \caption{\texttt{jLLAvQQbcdeihhiihtsfxedxhdt\_201021201}}
    \end{subfigure}
    \caption{Examples of ziggurats for pseudo-Anosov flows on $3$-manifolds with $2$ or $3$ boundary components. The flows are specified by their veering isosigs as recorded in the veering census \cite{giannopoulos.schleimer.ea.CensusVeeringStructures}. In the 2D examples, we also show the boundary multislopes of \emph{fibrations} transverse to $\Phi$ with blue dots, with fibers of more negative Euler characteristic shown with smaller dots. This comparison highlights our philosophy that $\Zg(\Phi)$ is a \emph{non-abelian} generalization of a fibered face of the Thurston norm ball.
    }\label{fig:ziggurat_examples}
\end{figure}

        \subsection{Ziggurats in circle dynamics} \label{sec:circledynamics}

Let $\widetilde{\rot} : \widetilde{\Homeo^+}(S^1) \to \R$ be the (lifted) rotation number of a (lifted) circle homeomorphism. In \cite{calegari.walker.ZigguratsRotationNumbers}, Calegari--Walker studied the extremal rotation number problem: Given a word $w$ in the generators $a,b$, determine the set
$$\Zg(w) \coloneqq \left\{(r,s,t) \ \big\vert \  \exists a,b \in 
\widetilde{\Homeo^+}(S^1) \text{ with } \widetilde{\rot}(a)=r, \ \widetilde{\rot}(b)=s, \ \widetilde{\rot}(w)=t \right\} \subset \R^3.$$
When $w=ab$, $\Zg(w)$ is the Jankins--Neumann ziggurat up to a change of coordinates. They show that when $w$ is an arbitrary positive word (i.e., it does not contain $a^{-1}$ or $b^{-1}$), many of the structural properties of the Jankins--Neumann ziggurat persist. In particular, they show that $\Zg(w)$ is assembled from blocks with rational corners.

Except for the Jankins--Neumann ziggurat, $\Zg(w)$ does not appear to be a foliation ziggurat for any choice of $\Phi$.  Our methods are also quite different from \cite{calegari.walker.ZigguratsRotationNumbers}: Calegari--Walker rely on a combinatorial analysis of the interleaving pattern of the fixed point sets of $a$ and $b$, while we rely heavily on $3$-manifold technology including contact geometry. We think that there ought to be a more general framework which simultaneously generalizes \cite{calegari.walker.ZigguratsRotationNumbers} and the present work.

\begin{question}
    What is the broadest dynamical context in which the ziggurat phenomenon occurs?
\end{question}

        \subsection{Foliations and contact structures}

Eliashberg and Thurston showed in~\cite{ET} that a cooriented $C^2$-foliation without spherical leaves on a closed $3$-manifold may be approximated by positive and negative contact structures. Bowden and Kazez--Roberts generalized this to $C^0$-foliations \cite{B16, KR17}. More recently, the first author proved a converse, reconstructing a foliation from a suitable pair of contact structures \cite{M24}. With these tools in hand, we can pass back and forth between the foliated and contact worlds.

\begin{center}
    \begin{tikzcd}[column sep=8em]
    \textcolor{violet}{\F} \arrow[r, bend left=30, "\text{Eliashberg--Thurston~\cite{ET}}"{pos=0.5, anchor=mid, yshift=0.8em}] & (\textcolor{blue}{\xi_-}, \textcolor{red}{\xi_+}) \arrow[l, bend left=30, "\text{M.~\cite{M24}}"{pos=0.5, anchor=mid, yshift=-0.8em}]
    \end{tikzcd}
\end{center}

In this article, we refine both directions of this correspondence in the case when $M$ has boundary: we upgrade the forward direction to give control over $\partial \xi_-$ and $\partial \xi_+$, and we upgrade the reverse direction to give control over $\partial \F$.

We need the tools of both foliation theory and contact geometry to prove our strongest results. Roughly speaking, the condition for a $2$-plane field to be a contact structure is \emph{open}. Thus, the contact perspective offers the flexibility needed to prove that $\Zg(\Phi)$ is open near certain points. On the other hand, the condition for a $2$-plane field to be (tangent to) a foliation is \emph{closed}. Therefore, the foliation perspective is useful for proving that $\Zg(\Phi)$ is closed near certain points.

        \subsection{Notation}

For the reader's convenience, we include a short list of notation that will be used throughout the article.

\begin{itemize}

    \item $M$: compact, oriented, connected $3$-manifold, with nonempty boundary unless otherwise stated. The boundary components are framed tori and are numbered $\partial_1 M, \dots, \partial_n M$. 
    
    \item $\Phi$: smooth, nonsingular, nonwandering flow on $M$ tangent to $\partial M$ having boundary slope $+\infty$ on each component of $\partial M$. Note that $\Phi$ need not be a linear flow on $\partial M$, but must have a well-defined boundary slope.
    
    \item $\mathcal{Z}(\Phi)$: real multislopes realized by a (necessarily taut) $C^0$-foliation positively transverse to $\Phi$. In this article, multislopes are oriented.
    
    \item $\mathcal{Z}^\pm(\Phi)$: real multislopes realized by positive/negative contact structures positively transverse to $\Phi$.
    
    \item $\mathcal Z_\aleph(\Phi)$: the set of generalized multislopes realized by $C^0$-foliations positively transverse to $\Phi$, where $\aleph$ is treated as a wildcard that can assume any real number.

    \item $I \coloneqq [0,1]$.
        
    \item For multislopes $\bm r$ and $\bm s$, we say $\bm r < \bm s$ when $r_i < s_i$ for all $1\leq i \leq n$. Similarly, we write $\bm{r} \leq \bm{s}$ when $r_i \leq s_i$ for all $1\leq i \leq n$. We define the closed box $$[\bm{r}, \bm{t}]=\{ \bm{s}\in \R^n \mid r_i \leq s_i \leq t_i \text{ for all } 1\leq i \leq n\}.$$ We similarly define the open box $(\bm{r}, \bm{t})$.\footnote{This notation is a bit unfortunate since it conflicts with the one for pairs of multislopes. We apologize to the reader and hope that it will not create any confusion.}

    \item Given a foliation $\F$, lamination $\Lambda$, contact structure $\xi_\pm$, we denote the restricted foliations/laminations on $\partial M$ by $\partial \F, \partial \Lambda$, and $\partial \xi_\pm$, respectively.

    \item We call a foliation positively transverse to $\Phi$ a \textbf{$\Phi$-horizontal foliation}, or simply a \textbf{horizontal foliation} when $\Phi$ is clear from context. We use the same terminology for surfaces, laminations, and minimal sets, and also foliations on $\partial M$ positively transverse to $\Phi |_{\partial M}$.
\end{itemize}

\subsection{Orientation conventions}

The following assumptions will hold throughout the article.

\begin{itemize}
    \item All surfaces, laminations, foliations, and plane fields in $M$ are oriented and cooriented, and transverse to $\partial M$. The laminations and foliations have $C^1$ immersed leaves and continuous tangent plane fields.

    \item Pictures of $\partial M$ are drawn looking from outside the manifold. In this convention, adding twisting to a positive contact structure \emph{decreases} the slope of the boundary foliation for free. Adding twisting to a negative contact structure \emph{increases} the slope of the boundary foliation. See \zcref{fig:contactplanes}.
\end{itemize}

\begin{figure}[htbp]
    \centering
    \includegraphics[width=0.55\linewidth]{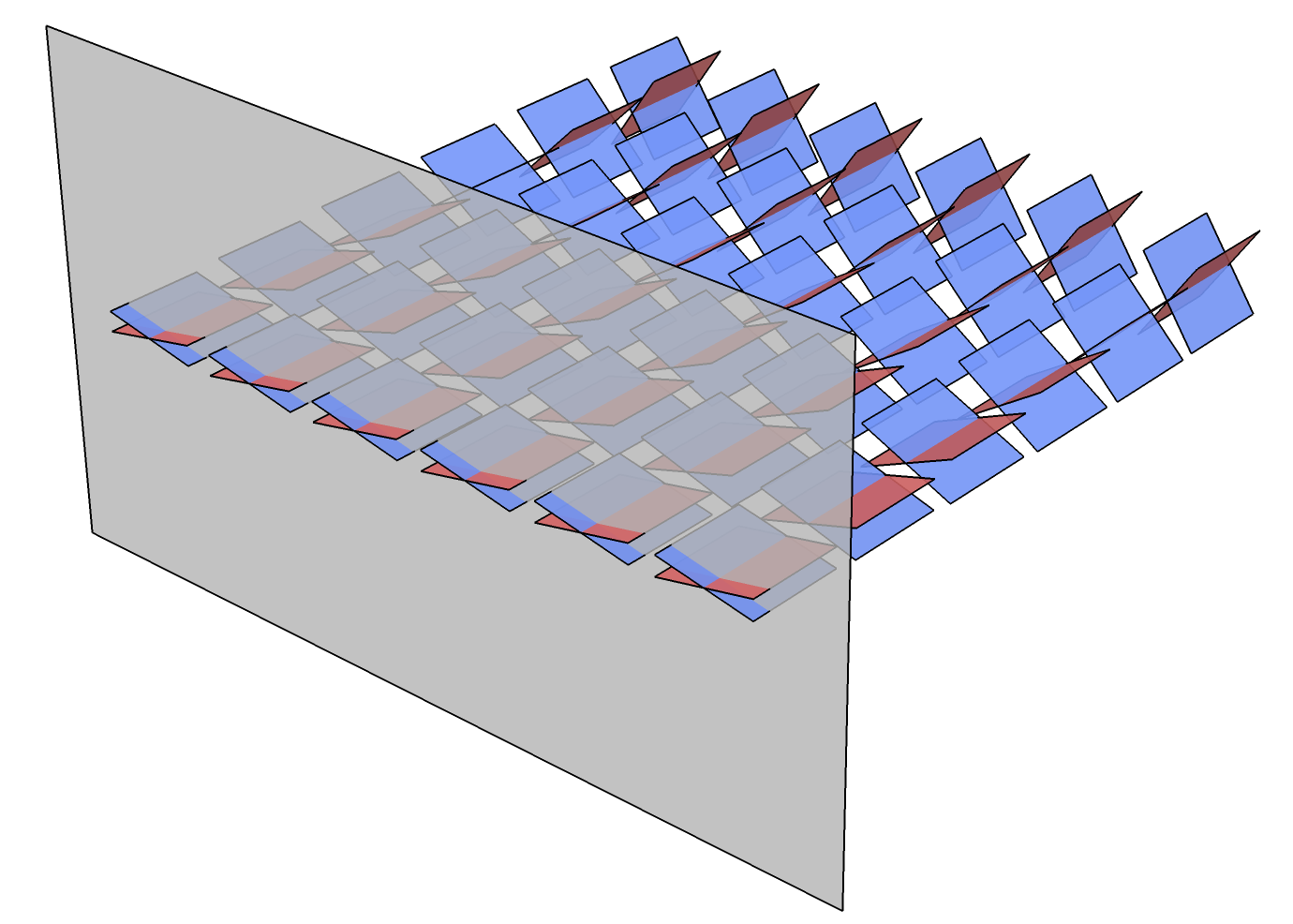}
    \caption{A positive contact structure (red) and a negative contact structure (blue), in a neighborhood of $\partial M$ (gray).}
    \label{fig:contactplanes}
\end{figure}

        \subsection{Results}

We establish two collections of results. First, we extend the approximation techniques of Eliashberg--Thurston and the construction of foliations from contact structures by the first author to the setting of manifolds with boundary. Second, we apply those results to the study of ziggurats. Our results relating foliations and contact structures might be of independent interest.

            \subsubsection{Foliations and contact structures on manifolds with boundary}

We assume $\partial M \neq \varnothing$. We say that the boundary foliation of a positive (resp.~negative) contact structure \textbf{dominates} another foliation $\G$ on $\partial M$ if it pointwise has higher (resp.~lower) slope than $\G$. See \zcref{fig:domination}.

\begin{figure}[htbp]
\centering
    \begin{subfigure}{0.45\linewidth}
    \centering
    \includegraphics[width=\linewidth]{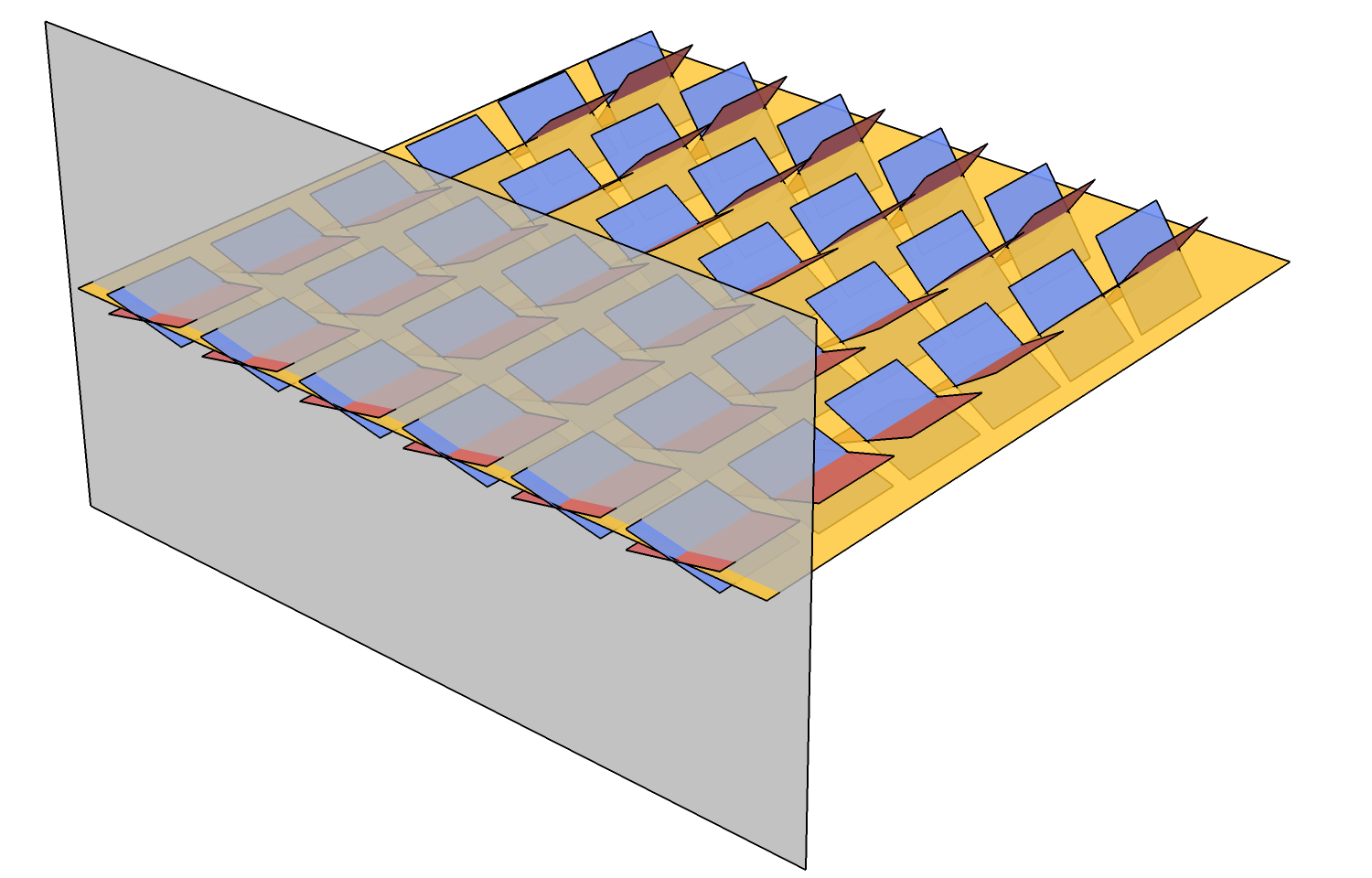}
    \caption{}
    \label{fig:dominating}
    \end{subfigure}%
    \hspace{1em}%
    \begin{subfigure}{0.45\linewidth}
    \centering
    \includegraphics[width=\linewidth]{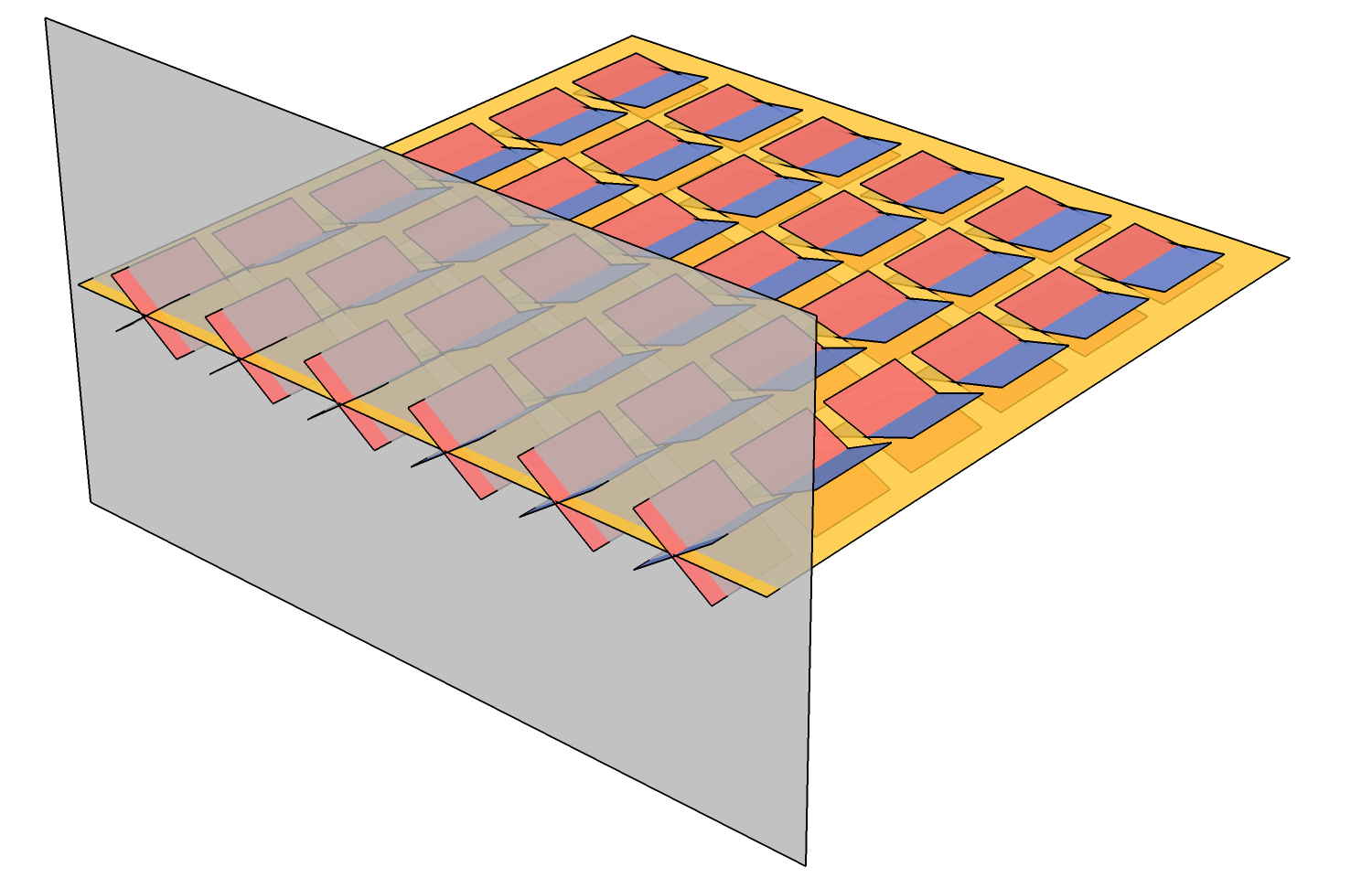}
    \caption{}
    \label{fig:nondominating}
    \end{subfigure}
    \caption{In \ref{fig:dominating}, the positive contact structure (red) and negative contact structure (blue) dominate a foliation (yellow) along $\partial M$ (gray). In \ref{fig:nondominating}, the contact structures do not dominate the foliation.}\label{fig:domination}
\end{figure}

\begin{thmintro}[Eliashberg--Thurston with boundary]\label{thmintro:ET}
    Let $\mathcal{F}$ be a $C^0$-foliation on $M$ transverse to $\partial M$.
    \begin{enumerate}
        \item (Coarse version) Assume that $(M,\F)$ is not homeomorphic to either of the following:
        \begin{itemize}
            \item $S^1 \times D^2$ with the standard foliation by disks,
            \item $\T^2\times I$ with a foliation semiconjugate to an irrational foliation.
        \end{itemize}
        Then for every $\varepsilon > 0$, there exists a positive (resp.~negative) contact structure $\xi$  which is $\varepsilon$-close to $T\F$ in the $C^0$ topology, and whose boundary foliation dominates a foliation $\G = \G_\varepsilon$ on $\partial M$ which is monotone equivalent to $\partial \mathcal{F}$.
        \item (Fine version) If moreover $\mathcal{F}$ has neither compact planar leaves nor $\T^2$-type planar minimal sets (\zcref{def:pems}), then $\G$ can be chosen isotopic to $\partial \F$ along $\partial M$.
    \end{enumerate}
\end{thmintro}

Here, we say that two foliations are \textbf{monotone equivalent} if they can be connected by a finite sequence (zigzag) of semiconjugacies. We will prove a more precise version in the presence of compact planar leaves and $\T^2$-type planar minimal sets; see \zcref{thm:fineET} below. Note that this theorem does not involve any transverse flow, as opposed to its converse.

\begin{thmintro}[Converse to Eliashberg--Thurston with boundary]\label{thmintro:ETconverse}
    Let $(\xi_-, \xi_+)$ be a contact pair on $M$ positively transverse to $\Phi$. Furthermore, let $\mathcal{G}_\partial$ be a smooth foliation on $\partial M$ which is transverse to $\Phi\vert_{\partial M}$ and which is dominated by $\partial \xi_-$ and $\partial \xi_+$ along $\partial M$. 
    
    Then, there exists a $C^0$-foliation $\mathcal{F}$ on $M$ transverse to $\Phi$ (and in particular to $\partial M$) such that $\partial \mathcal{F}$ is semiconjugate to $\mathcal{G}_\partial$, via a map $\partial \mathcal{F} \to \mathcal{G}_\partial$.
\end{thmintro}

Here, a contact pair $(\xi_-, \xi_+)$ is a pair of contact structures with opposite signs with respect to the orientation on $M$. The reader should imagine that the foliation $\G_\partial$ in the theorem is ``sandwiched'' between $\xi_-$ and $\xi_+$ along $\partial M$, as in \zcref{fig:dominating}. Unfortunately, it might not be possible to make $\partial \F$ \emph{isotopic} to $\G_\partial$ in general, but this result will be powerful enough for all our applications.

            \subsubsection{Applications to ziggurats}

In the rest of the introduction and also in \zcref{sec:mainresults}, we make the following standing assumptions unless otherwise stated.

\begin{assum}\label{assum:1} \leavevmode
\begin{itemize}[leftmargin=*]
    \item $M$ is a compact, connected, oriented $3$-manifold with toroidal boundary components. The boundary components of $M$ are framed. We also assume that $M \ncong S^1 \times D^2$ and $M \ncong \mathbb{T}^2 \times I$, so that the hypotheses of the coarse Eliashberg--Thurston theorem (item 1 of \zcref{thmintro:ET}) are satisfied. 
    
    \item $\Phi$ is a smooth, nonsingular and nonwandering flow on $M$ whose restriction to each boundary component of $\partial M$ induces a Reebless foliation of slope $+\infty$ (the \textbf{degeneracy slope}).
\end{itemize}
\end{assum}
Notice that changing the framings on the boundaries of $M$ without modifying the degeneracy slopes simply induces a translation in $\R^n$ by an integer vector.

\medskip

The master result that will allow us to apply techniques from contact geometry to the study of ziggurats is the following \emph{ziggurat intersection theorem}, which is a consequence of \zcref{thmintro:ET, thmintro:ETconverse}.

\begin{thmintro} \label{thmintro:contactzigg}
    The foliation ziggurat $\mathcal Z(\Phi)$ is exactly the intersection of the negative and positive contact ziggurats: 
    $$\mathcal{Z}^-(\Phi) \cap \mathcal{Z}^+(\Phi) = \mathcal{Z}(\Phi).$$
\end{thmintro}

\begin{figure}[htpb]
    \centering
    \includegraphics[width=0.4\linewidth]{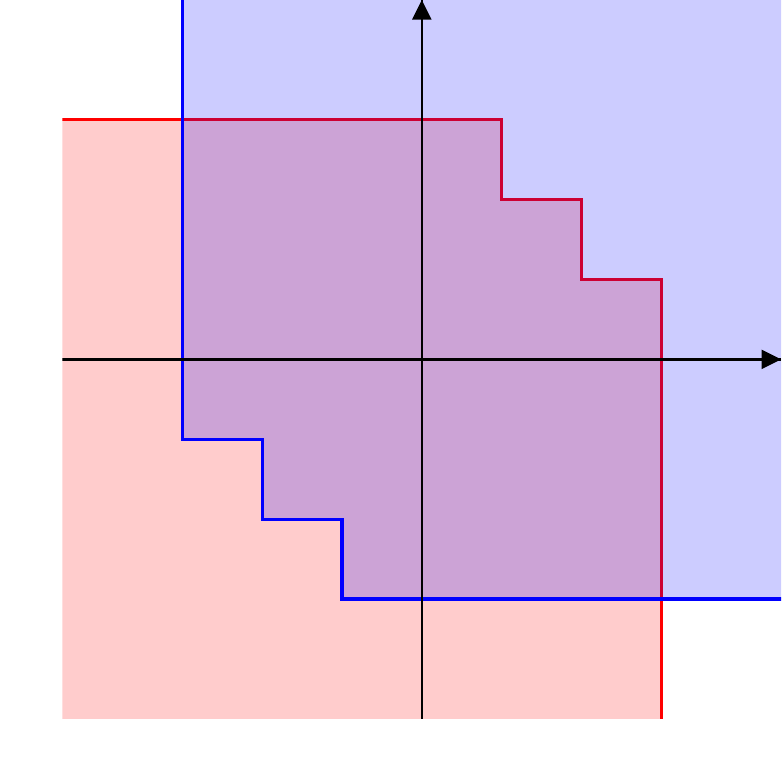}
    \caption{The intersection of $\Zg^-(\Phi)$ (blue) and $\Zg^+(\Phi)$ (red) is $\Zg(\Phi)$.}
    \label{fig:contactzigg}
\end{figure}

The next few results will be rather immediate applications of \zcref{thmintro:contactzigg}. First, we show that the ziggurat $\Zg(\Phi)$ satisfies a strong convexity property.

\begin{defn} \label{def:boxconv}
    A subset $A \subset \R^n$ is \textbf{box-convex} if for all $\bm{x}, \bm{y} \in A$ satisfying $\bm{x} \leq \bm{y}$, we have $[\bm{x}, \bm{y}] \subset A$.
\end{defn}

\begin{rem}
    Box-convexity is stable under taking intersections and under closure.
\end{rem}

\begin{defn}
    Let $S \subset \R^n$. The \textbf{box-closure of $S$}, denoted by $[S]$, is defined as
    $$[S] \coloneqq \left\{ \bm{s} \in \R^n \ \vert \ \exists \bm{s}^\pm \in S, \ \bm{s}^- \leq \bm{s} \leq \bm{s}^+\right\}.$$
    It is the smallest box-convex subset of $\R^n$ containing $S$.
\end{defn}

\begin{figure}[htbp]
\centering
    \begin{subfigure}{0.4\linewidth}
    \centering
    \includegraphics[width=\linewidth]{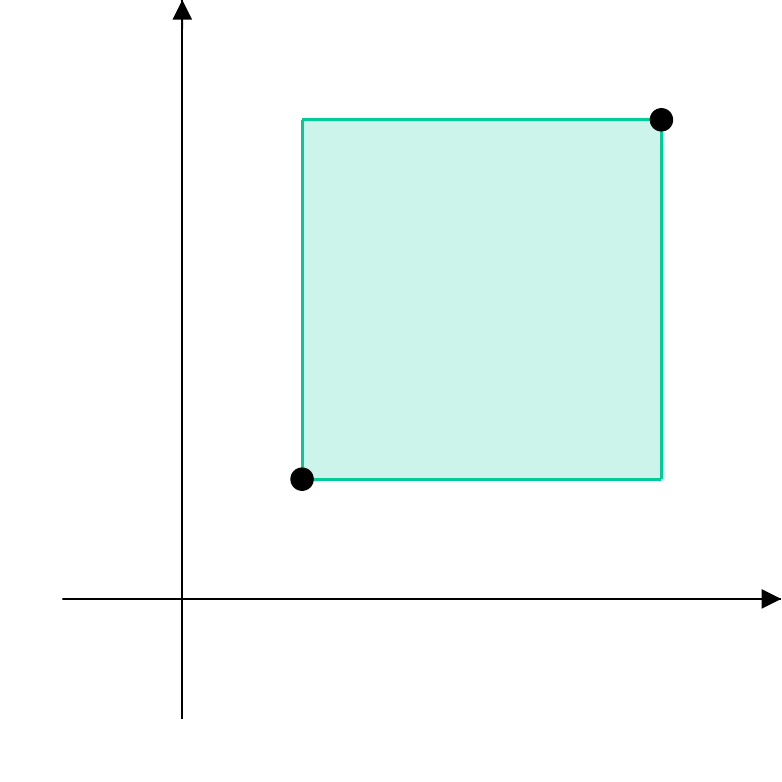}
    \caption{}
    \label{fig:nonbox}
    \end{subfigure}%
    \hspace{2em}%
    \begin{subfigure}{0.4\linewidth}
    \centering
    \includegraphics[width=\linewidth]{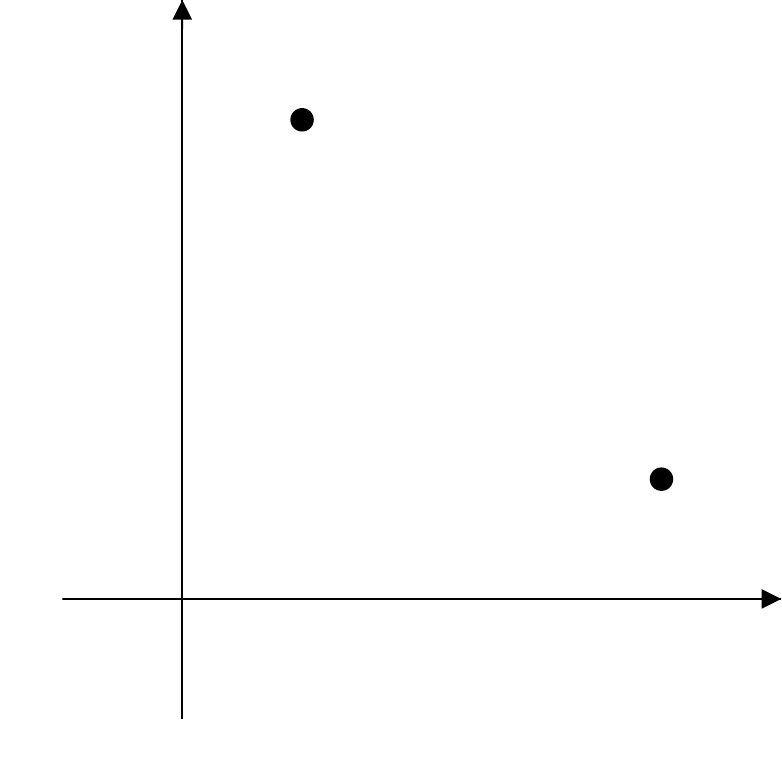}
    \caption{}
    \label{fig:emptybox}
    \end{subfigure}
    \caption{The box-closure of the pair of points on the left is a rectangle, while the box-closure of the pair on the right does not contain any other points.}\label{fig:boxconvex}
\end{figure}

\begin{thmintro}[Box-convexity] \label{thmintro:box-convexity}
    The set $\mathcal{Z}(\Phi)$ is box-convex.
\end{thmintro}

When $n=1$, i.e., when $M$ has a unique boundary component, this result simply means that $\Zg(\Phi) \subset \R$ is an interval. When $n \geq 2$, more complicated shapes can appear.

\medskip

We next establish some properties of rational and irrational slopes.

\begin{thmintro}[Dehn filling] \label{thmintro:rationalfilling}
Suppose $\bm{s}\in \Q^n$. The multislope $\bm{s}$ is realized by a $\Phi$-horizontal foliation with linear boundary type (foliation by circles) if and only if one of the following holds:
    \begin{itemize}
        \item $\bm{s}\in \interior \Zg(\Phi)$, or
        \item $\Phi$ is a suspension flow of a fibration with compact planar fibers with multislope $\bm{s}$.
    \end{itemize} 
\end{thmintro}

We also state a version of this theorem relating partial fillings of $M$ to slices of $\mathcal{Z}(\Phi)$; see \zcref{thm:partialfilling}.

\medskip

For $1 \leq k \leq n$, we define 
$$\bm{1}_k \coloneqq (\underset{k}{\underbrace{1, \dots, 1}}, \underset{n-k}{\underbrace{0, \dots, 0}}).$$
\begin{thmintro}[Instability of irrational slopes] \label{thmintro:irrationality}
    Let $\bm{s} = (s_1, \dots, s_k, s_{k+1}, \dots, s_n) \in \Zg(\Phi)$ such that $s_i \in \R \setminus \Q$ for $1 \leq i \leq k$. Then there exists $\varepsilon > 0$ such that 
    $$[\bm{s} - \varepsilon \bm{1}_k, \bm{s} + \varepsilon \bm{1}_k] \subset \Zg(\Phi).$$
\end{thmintro}

\begin{longrem}
    With the further assumption that $\Zg(\Phi)$ is closed, the previous results already explain the staircase phenomenon visible in \zcref{fig:ziggurat_examples}. For simplicity, we restrict to the case $n=2$. Suppose that $\Zg(\Phi)$ is closed, and that its projection onto the $x$-axis is surjective. By box-convexity, $\Zg(\Phi)$ is of the form
    $$\big\{(x,y) \in \R^2 \ \vert\  f_-(x) \leq y \leq f_+(x)\big\}, $$
    where $f_\pm : \R \rightarrow \R\cup\{\pm \infty\}$ are two nonincreasing functions. Combined with the instability property of irrational multislopes, we also deduce that $f_\pm$ only take rational values (or $\pm \infty)$, and their discontinuity points are rational as well. Of course, these functions can still have complicated discontinuity sets without any additional assumptions. Below we will prove that $\Zg(\Phi)$ is closed when some natural assumptions are satisfied (\zcref{thmintro:closed}), and we precisely identify its closure otherwise (\zcref{corintro:closure}). We will also identify the points where corners may accumulate (\zcref{thmintro:slippery}).
\end{longrem}

\FloatBarrier
\begin{center}
\rule{0.3\textwidth}{0.4pt}
\end{center}

\medskip

Compact planar surfaces play a special role in the next few theorems. We also need to consider \textbf{$\T^2$-type planar minimal sets} (see \zcref{def:pems}), the model example of which is an irrational linear foliation on $\T^2\times I$. A recurring theme is that minimal sets of these kinds impose rigidity on the foliation, because of the small fundamental groups of their leaves.

\begin{defn}[Obstructed multislopes]
    A multislope $\bm{s}$ is $\textbf{spherical}$ if there is an embedded compact planar surface with nonempty boundary meeting each $\partial_i M$ in a (possibly empty) collection of parallel curves of (oriented) slope $s_i$. Note that any such surface must be essential and homologically nontrivial. We define \textbf{$\T^2$-type planar multislope} similarly. We call a multislope \textbf{obstructed} if it is a spherical or $\T^2$-type planar multislope.
\end{defn}

\begin{rem}
    Dehn filling along the rational coordinates of a spherical multislope always gives a reducible $3$-manifold: the compact planar surface extends to an sphere which has positive intersection numbers with the cores of the filling tori, and is therefore essential.
\end{rem}

\begin{rem}
    Accumulation points of spherical multislopes must be spherical or $\T^2$-type planar multislopes; see \zcref{prop:spherical_limits}.
\end{rem}

\begin{thmintro}[Rigidity of spherical multislopes] \label{thmintro:rigidityspherical}
    Assume $\bm{s} = (s_1, \dots, s_n) \in \mathcal{Z}(\Phi)$ is a spherical multislope or a $\T^2$-type planar multislope. Then either 
    $$\left((s_1, +\infty) \times \dots \times (s_n, +\infty)\right) \cap \mathcal{Z}(\Phi) = \varnothing$$
    or 
    $$\left((-\infty, s_1) \times \dots \times (-\infty, s_n)\right) \cap \mathcal{Z}(\Phi) = \varnothing.$$
    If Dehn filling $(M,\Phi)$ along the rational components of $\bm{s}$ results in a suspension flow on $S^1\times S^2$ or $\T^2\times I$, then both of the quadrants above are empty.
\end{thmintro}

\begin{figure}[htbp]
    \centering
    \includegraphics[width=0.4\linewidth]{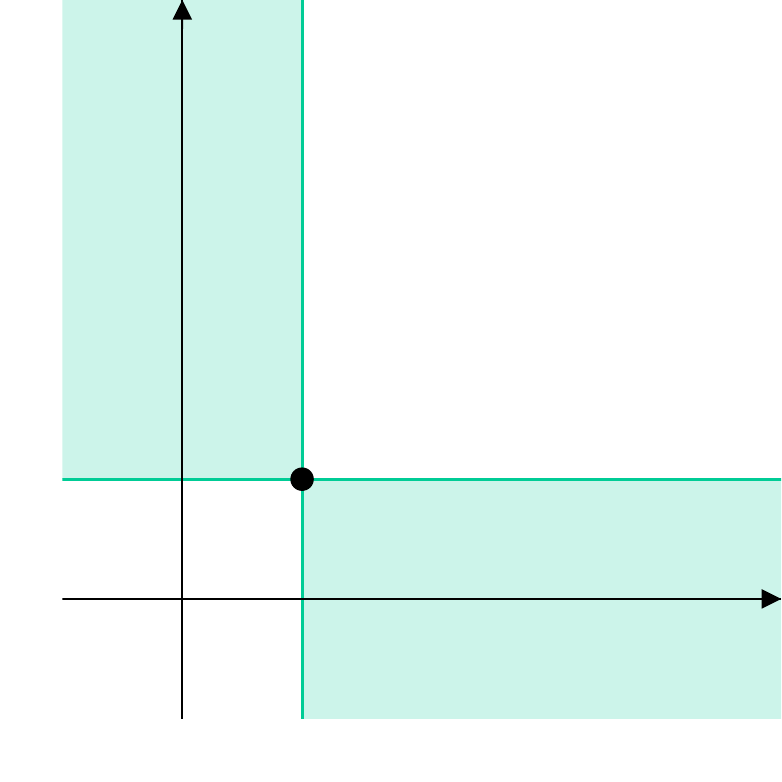}
    \caption{If the marked point corresponds to an $S^1\times S^2$ Dehn filling, then $\Zg(\Phi)$ is ``pinched'' there: it must be a subset of the green region.}
    \label{fig:sphericalrigidity}
\end{figure}

One can check this theorem in \zcref{fig:ziggurat_examples} (a) and (b); the largest blue dots correspond to spherical multislopes, and the ziggurat is ``pinched'' there. In \zcref{fig:magicziggurat}, the ziggurat is ``pinched'' along a hyperbola which corresponds to a family of spherical and $\T^2$-type planar multislopes. Indeed, one can fill a component of L6a5 to get a Hopf link, whose complement is $\T^2\times I$. We do not know whether the natural converse to \zcref{thmintro:rigidityspherical} holds: there might be points in $\Zg(\Phi)$ whose positive and negative quadrants are empty, but which are neither spherical nor $\T^2$-type planar multislopes.

\FloatBarrier

\begin{center}
\rule{0.3\textwidth}{0.4pt}
\end{center}

\medskip

In favorable situations, the ziggurat $\Zg(\Phi)$ is a closed subset of $\R^n$. We say that $\Phi$ is \textbf{branched surface finite} (or \textbf{BSF} for short) if there is a finite collection of $\Phi$-horizontal branched surfaces which carry all $\Phi$-horizontal foliations. Fully punctured pseudo-Anosov flows without perfect fits are BSF (\zcref{thm:pABSF}). We will prove our results for an even larger class of flows, the \textbf{locally branched surface finite} (\textbf{LBSF}) flows (\zcref{def:condLBSF}), which includes periodic flows as well.

\begin{thmintro} \label{thmintro:closed}
    If $\Phi$ is LBSF and does not admit any transverse compact planar surface with a boundary component of degeneracy slope, then $\Zg(\Phi)$ is a closed subset of $\R^n$.
\end{thmintro}

For example, this theorem applies to periodic flows. In general, however, $\Zg(\Phi)$ may fail to be closed at some exceptional points where the boundary foliations develop Reeb annuli; see \zcref{fig:approaching_aleph} for an example.\footnote{See also the painting \emph{Current} by Bridget Riley,\\ \href{https://bridget-riley.publications.britishart.yale.edu/catalogue/10/}{\texttt{https://bridget-riley.publications.britishart.yale.edu/catalogue/10/}}.} We call these boundary foliations \textbf{$\aleph$~type}. See \zcref{defn:type} for the precise definition. We say that $\partial \F$ is \textbf{admissible} if each component is either a suspension foliation or $\aleph$ type.

\begin{figure}[htbp]
    \centering
    \includegraphics[width=0.9\linewidth]{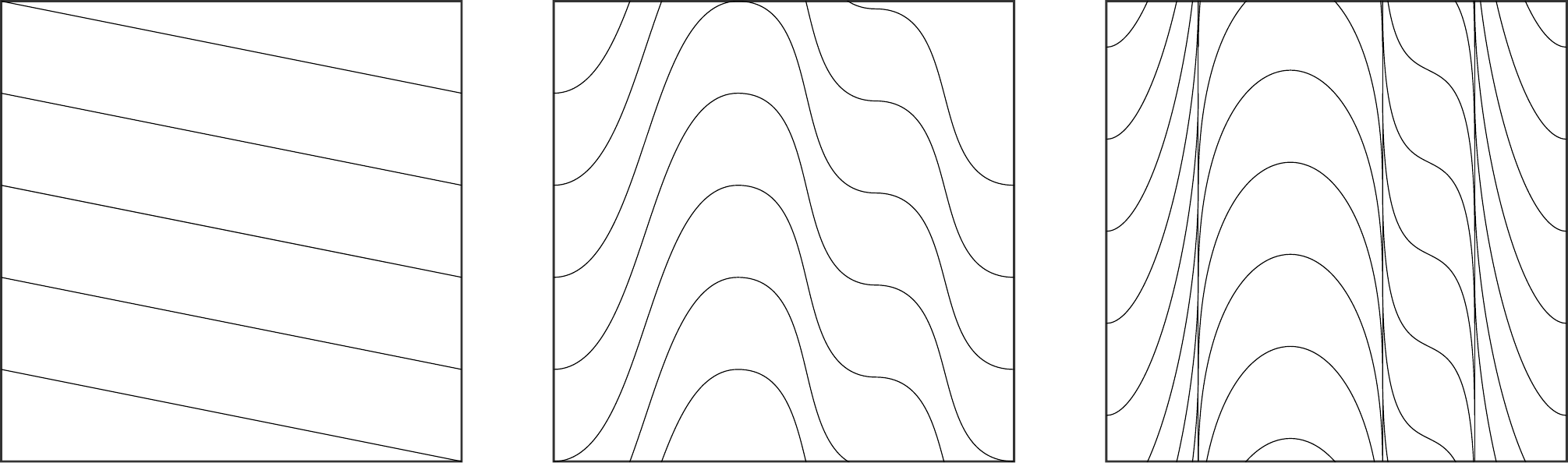}
    \caption{A sequence of foliations of $\T^2$ all of the same slope approaching an $\aleph$ type foliation.}
    \label{fig:approaching_aleph}
\end{figure}

\begin{rem}
    The appearance of $\aleph$ type boundaries should be no surprise for those familiar with the notion of almost-transversality in the pseudo-Anosov theory of foliations. Foliations with $\aleph$ type boundary components can be filled along any non-degeneracy slope to yield almost-transverse foliations in the sense of Mosher. See \zcref{prop:alephfilling} below.
\end{rem}

\begin{defn}
    Let $\F$ be a $\Phi$-horizontal foliation with admissible boundary. The \textbf{generalized boundary multislope} of $\F$ is a tuple $(s_1,\dots,s_n)$, where $s_i\in \R \cup \{\pm \infty\}$ if $\partial_i \F$ is a suspension foliation, and $s_i=\aleph$ if $\partial_i \F$ is of $\aleph$ type.
\end{defn}

\begin{defn} \label{def:alephcomp}
    The \textbf{$\aleph$-completion} of $\Zg(\Phi)$, denoted by $\mathcal{Z}_\aleph(\Phi)$, is the subset of $\R^n$ defined as follows. A multislope $\bm{s} = (s_1, \dots, s_n) \in \R^n$ is in $\mathcal{Z}_\aleph(\Phi)$ if and only if there is a generalized multislope $(\widetilde s_1,\dots,\widetilde s_n)$ realized by a $\Phi$-horizontal foliation such that each index $i$ satisfies either $\widetilde{s}_i = s_i$ or $\widetilde{s}_i=\aleph$; in other words, $\aleph$ may be treated as a wildcard.
\end{defn}

This definition may seem strange at first sight but is justified by the next sequence of results: we will show that $\Zg_\aleph(\Phi)$ is exactly the closure of $\Zg(\Phi)$, under some reasonable assumptions on $\Phi$.

\begin{thmintro}[$\aleph$-completion] \label{thmintro:alephclosure}
    If $\Phi$ is LBSF, then $\mathcal{Z}_\aleph(\Phi)$ is a closed subset of $\R^n$.
\end{thmintro}

\begin{rem}
    For \zcref{thmintro:closed, thmintro:alephclosure}, it is not necessary to assume that $\Phi$ is nonwandering.
\end{rem}

By definition, we always have $\Zg(\Phi) \subset \Zg_\aleph(\Phi)$, and under the hypothesis of the previous theorem, we further obtain 
\begin{align} \label{eq:inczig}
    \overline{\Zg(\Phi)} \subset \Zg_\aleph(\Phi).
\end{align}
To establish the converse inclusion, we would need to show that any $\Phi$-horizontal foliation realizing a multislope in $\Zg_\aleph(\Phi)$ can be \textbf{resolved}, i.e., modified so that the $\aleph$ components are replaced with arbitrary real slopes and the non-$\aleph$ boundary components suffer arbitrarily small changes in slope. We accomplish this using \zcref{thmintro:contactzigg} (and hence contact structures) in an essential way. We now identify some potential obstructions to the application of \zcref{thmintro:contactzigg}, and show that in their absence, the converse inclusion to~\eqref{eq:inczig} holds. Perhaps unsurprisingly, these obstructions take the form of certain compact planar surfaces transverse to $\Phi$ and with some degeneracy boundary components. 

\begin{defn}\label{defn:boundarysignature}
Let $\Sigma \subset M$ be a compact planar surface with $\partial \Sigma \subset \partial M$. Define
\begin{align*}
    b_{\mathrm{nd}}(\Sigma) &\coloneqq \text{the number of nondegeneracy boundary components of } \Sigma, \\
    b_{-\infty}(\Sigma) &\coloneqq \text{the number of $-\infty$ degeneracy boundary components of } \Sigma, \\
    b_{+\infty}(\Sigma) &\coloneqq \text{the number of $+\infty$ degeneracy boundary components of } \Sigma.
\end{align*}
We also define the \textbf{degeneracy signature} of $\Sigma$,
$$\bm{b}(\Sigma) \coloneqq \big(b_{\mathrm{nd}}(\Sigma), b_{-\infty}(\Sigma), b_{+\infty}(\Sigma)\big).$$
\end{defn}

\begin{defn}\label{def:alephunobstructed}
    We say that $\Phi$ is \textbf{$\aleph$-obstructed} if it has a transverse compact planar surface with any of the following degeneracy signatures:
    \begin{itemize}
        \item $(1, 1, 0)$ or $(1,0,1)$ (degenerate cylinders),
        \item $(0, 1, m)$ or $(0, m, 1)$ for $m \geq 1$.
    \end{itemize}
    Otherwise, we say that $\Phi$ is \textbf{$\aleph$-unobstructed}.
\end{defn}

Our main flows of interest, fully punctured pseudo-Anosov flows without perfect fits, are automatically $\aleph$-unobstructed. See \zcref{thm:pABSF} and \zcref{lem:fptpafwpf_unobstructed}.

\begin{thmintro}[Resolving $\aleph$ components] \label{thmintro:alephremove}
    Assume that $\Phi$ is $\aleph$-unobstructed. If a generalized multislope $\bm{s} = (s_1, \dots, s_k, \aleph, \dots, \aleph)$, $s_i \in \R$, is realized by a foliation transverse to $\Phi$, then there exist intervals $I_i \subset \R$ with $s_i \in \partial I_i$, $1 \leq i \leq k$, such that
    $$I_1 \times \dots \times I_k \times {\R}^{n-k} \subset \Zg(\Phi).$$
\end{thmintro}

Notice that the coordinates with entry $\aleph$ have been replaced with ${\R}$. We will prove a more precise statement below, see \zcref{thm:resolvealeph}. Combining \zcref{thmintro:alephclosure} and \zcref{thmintro:alephremove}, we obtain:

\begin{corintro} \label{corintro:closure}
    If $\Phi$ is LBSF and $\aleph$-unobstructed, then
    \begin{align}
        \overline{\mathcal{Z}(\Phi)} = \mathcal{Z}_\aleph(\Phi).
    \end{align}
\end{corintro}

Specializing to the case $n=1$, \zcref{thmintro:box-convexity,thmintro:irrationality,corintro:closure} readily imply: 

\begin{corintro} \label{corintro:closedinterval}
    Assume $M$ has a single boundary component. If $\Phi$ is LBSF and $\aleph$-unobstructed, then $\Zg(\Phi) \subset \R$ is a (possibly empty) \emph{closed} interval with rational endpoints (possibly lying at infinity).
\end{corintro}

\begin{center}
\rule{0.3\textwidth}{0.4pt}
\end{center}

\medskip

When $\Zg(\Phi)$ is closed, we can prove much finer structural properties of $\Zg(\Phi)$. We show that away from spherical multislopes, $\Zg(\Phi)$ is locally modeled on a finite union of boxes with rational corners.

\begin{defn}
    A \textbf{standard quadrant at $\bm{0} \in \R^n$} is a subset of the form 
    $$C =I_1 \times \dots \times I_n,$$
    where $I_i \subset \R$ is one of $(-\infty, 0]$, $\{0\}$, or $[0,+\infty)$. A \textbf{standard multiquadrant at $\bm{0}$} is a union of standard quadrants at $\bm{0}$.
\end{defn}

While standard quadrants are box-convex, a standard multiquadrant may not be. However, the box-closure of a multiquadrant is still a multiquadrant.

\begin{figure}[htbp]
\centering
    \begin{subfigure}{0.45\linewidth}
    \centering
    \includegraphics[width=\linewidth]{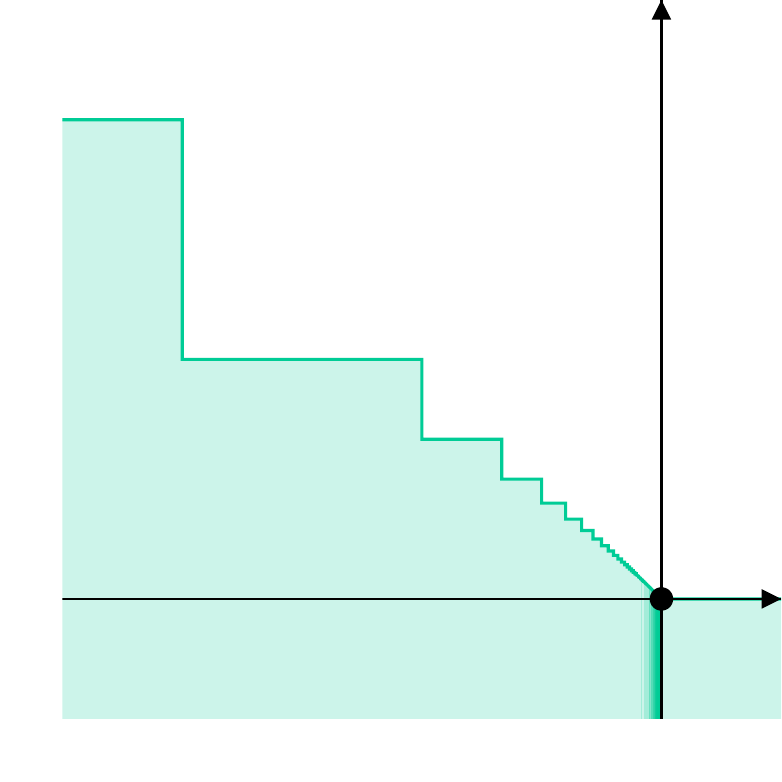}
    \caption{A closed box-convex subset of $\R^2$ with a single slippery point at the origin.}
    \end{subfigure}
    \hfill
    \begin{subfigure}{0.45\linewidth}
    \centering
    \includegraphics[width=\linewidth]{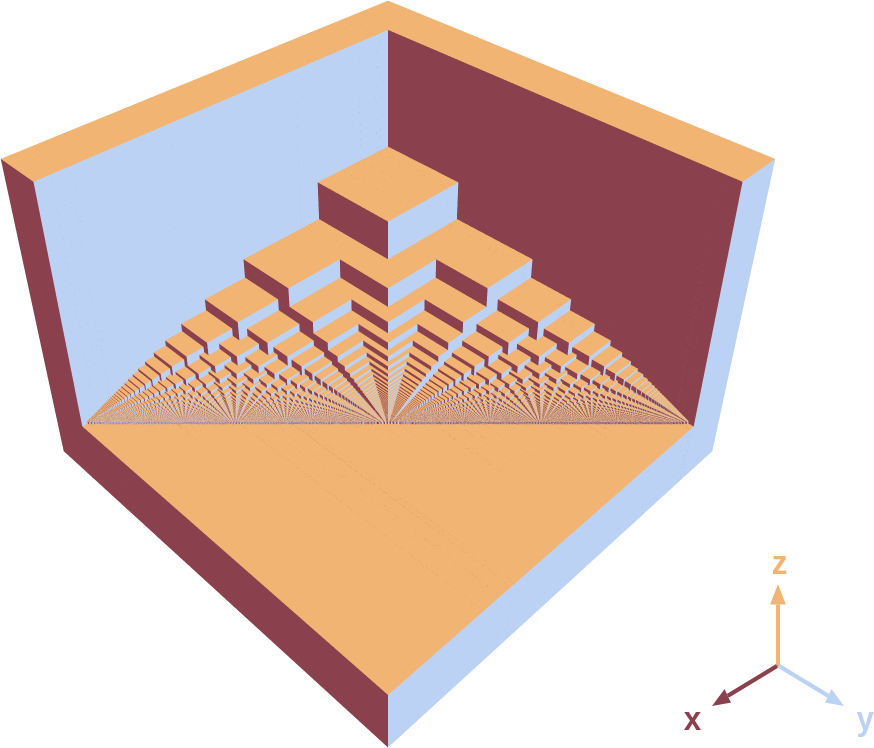}
    \caption{A closed box-convex subset of $\R^3$ with a line segment of slippery points.}
    \end{subfigure}
    \caption{Slippery points}
    \label{fig:slipperyexamplesintro}
\end{figure}

\begin{defn} \label{def:terraced}
    Let $A \subset \R^n$ be a closed box-convex set. A point $\bm{x} \in A$ is \textbf{terraced} if $A$ is locally modeled on a standard multiquadrant near $\bm{x}$, i.e., there exist a (box-convex) multiquadrant $Q \subset \R^n$ at $\bm{0}$ and a neighborhood $U$ of $\bm{x}$ such that 
    $$A \cap U = \left( \bm{x} + Q\right) \cap U.$$
    Otherwise, $\bm{x}$ is called \textbf{slippery}.
\end{defn}

The term ``slippery'' was introduced in~\cite{calegari.walker.ZigguratsRotationNumbers}, and is justified by \zcref{prop:slippery_corners} below: slippery points are accumulation points of extremal corners, i.e., accumulations of stairs getting smaller and smaller; a clumsy pedestrian could easily slip and fall there. If $A$ is a closed box-convex set in $\R^n$, then for any compact set $K \subset \R^n$ such that $A \cap K$ consists of terraced points of $A$, $A\cap K=B\cap K$ where $B$ is a finite union of closed boxes; these boxes are the ``terraces''.

For suitable nonwandering flows, slippery points may only appear at special multislopes that we will shortly describe. This is already suggested in our examples above: the slippery points in \zcref{fig:ziggurat_examples} (a) and (b) lie at spherical multislopes, and there is a prominent $1$-parameter family of slippery points along a hyperbola in \zcref{fig:magicziggurat}. Since we wish to consider sequences of $\Phi$-horizontal foliations, we also need to be careful about potential $\aleph$ components appearing in the limit, and we shall impose some unobstructedness assumptions to avoid additional difficulties.

\begin{defn} \label{def:slipperyunobstructed}
    For $\bm{s} \in \R^n$, we say that $\bm{s}$ is \textbf{strongly $\aleph$-unobstructed} if there is no degeneracy (compact or $\T^2$-type) planar minimal set whose nondegeneracy boundary slopes match the coordinates of $\bm{s}$.
\end{defn}

\begin{thmintro}[Slippery points are spherical] \label{thmintro:slippery}
     Suppose $\Phi$ is LBSF and $\bm{s}$ is a strongly $\aleph$-unobstructed slippery point of $\overline{\Zg(\Phi)}$. Then at least one of the following happens:
        \begin{itemize}
            \item Some of the $s_i$'s are rational and the corresponding Dehn filling yields a flow with a transverse sphere. In particular, the filling has an $S^1\times S^2$ summand.
            \item Two of the $s_i$'s are irrational and Dehn filling some subset of the rational ones yields a flow with a transverse $\T^2$-type planar lamination. In particular, the filling has a $\mathbb{T}^2 \times I$ summand and is a limit of fillings with $S^1\times S^2$ summands.
        \end{itemize}
\end{thmintro}

\FloatBarrier
\begin{center}
\rule{0.3\textwidth}{0.4pt}
\end{center}

\medskip

In all of the examples we produced (mostly for $n\leq 3$ and for pseudo-Anosov flows), we observed that the ziggurats are connected subsets of $\R^n$. We expect this phenomenon to be more general:

\begin{conj}\label{conj:connectedness}
    For any $(M,\Phi)$ as in \zcref{assum:1} with $M$ atoroidal, the spaces $\Zg(\Phi)$ and $\Zg_\aleph(\Phi)$ are connected.\footnote{One could conjecture the stronger statement that $\Zg(\Phi)$ and $\Zg_\aleph(\Phi)$ are contractible. This should be thought of as an analogue of the fact that whenever $\Phi$ is a pseudo-Anosov flow, the set of homology classes of $\Phi$-horizontal surfaces is a cone over a convex subset of a face of the Thurston norm ball.}
\end{conj}

Using the technology of veering triangulations along with \zcref{thmintro:alephremove}, we make partial progress on this conjecture:

\begin{thmintro} \label{thmintro:boundedconnected}
    Suppose $\Phi$ is a fully punctured pseudo-Anosov flow without perfect fits. If $\Zg(\Phi)$ has a \emph{compact} connected component, then $\Zg(\Phi)$ is connected; in particular, it is compact.
\end{thmintro}

Under these hypotheses, \zcref{thmintro:alephremove} implies that a compact set is a connected component of $\Zg(\Phi)$ if and only if it is a connected component of $\overline{\Zg(\Phi)}$. Thus, \zcref{thmintro:boundedconnected} can be equivalently formulated as follows: if $\overline{\Zg(\Phi)}$ has a bounded connected component, then $\Zg(\Phi) = \overline{\Zg(\Phi)}$ is connected.

\begin{longrem}
In a similar vein, Mann's results in \cite{mann.SpacesSurfaceGroup} imply that the fibers of the map 
$$\bm \slope: \big\{\textrm{$\Phi$-horizontal foliations}\big\} \longrightarrow \Zg(\Phi)$$
may be disconnected. She distinguished connected components by the slopes of the induced foliations on various essential tori. We expect that similar techniques may yield examples where $\Zg(\Phi)$ is disconnected, which motivates the extra condition that $M$ is atoroidal in \zcref{conj:connectedness}. In an earlier paper, Bowden instead used approximating contact structures to distinguish connected components \cite[Theorem B]{bowden.ContactStructuresDeformationsa}. There are two obstacles to overcome in adapting this strategy to produce a counterexample to \zcref{conj:connectedness}. First, for the approximating contact structures to yield well-defined invariants in our setting, one would need a version of Vogel's uniqueness theorem for $C^0$-foliations~\cite{V16}. Second, we do not know many examples of pseudo-Anosov flows with multiple transverse contact structures (see for example the uniqueness theorem in the fibered case of \cite{honda.kazez.ea.TightContactStructuresa}).
\end{longrem}

\begin{center}
\rule{0.3\textwidth}{0.4pt}
\end{center}

\medskip

We end this section with some results and remarks on the regularity of the objects under consideration. We now assume that $\Phi$ is a \emph{topological} flow on $M$ without fixed points, whose boundary foliations are Reebless and have slope $+\infty$, in the chosen framing of $M$. We define the \emph{topological ziggurat}
$$\Zg^\mathrm{top}(\Phi) \subset \R^n$$
as the collection of multislopes realized by \emph{topological} foliations on $M$ which are \emph{topologically (positively) transverse} to $\Phi$. Furthermore, assume that there exists a smooth model $\Phi_\mathrm{sm}$ of $\Phi$, i.e., a \emph{smooth} flow on $M$ which is topologically equivalent to $\Phi$, via an equivalence isotopic to the identity.\footnote{For us, a topological equivalence is a homeomorphism $h : M \rightarrow N$ sending \emph{oriented} orbits of a flow $\Phi$ to \emph{oriented} orbits of a flow $\Psi$.}

\begin{propintro} \label{propintro:smoothing}
    We have the equality
    $$\Zg(\Phi_\mathrm{sm}) = \Zg^\mathrm{top}(\Phi).$$
\end{propintro}

Here, we do not need to assume that $\Phi$ is nonwandering. This result justifies two of our conventions:
\begin{itemize}
    \item First, we may restrict our attention to \emph{smooth} flows, since our flows of interest (pseudo-Anosov and periodic flows) have smooth models,
    \item Second, we may also restrict to $C^0$-foliations which have a well-defined plane field, and which are genuinely transverse to a (smooth) flow. This is crucial in order to define contact approximations.
\end{itemize}

Moreover, since $\Zg^\mathrm{top}$ is clearly invariant under topological equivalence, we immediately get:

\begin{corintro} \label{corintro:topequiv}
    Let $M$ and $N$ be two manifolds satisfying the first item of \zcref{assum:1}. Suppose $\Phi$ and $\Psi$ are two smooth flows on $M$ and $N$, respectively, whose restrictions to each boundary component are Reebless with slope $+\infty$. If $\Phi$ and $\Psi$ are topologically equivalent via a topological equivalence $h : M \rightarrow N$, then 
    $$\Zg(\Psi) = h_*\big(\Zg(\Phi)\big),$$
    where $h_*$ is the induced map on the framings and boundary multislopes of $M$ and $N$. 
\end{corintro}

We will prove a more general result on upgrading topological foliations to $C^0$-foliations, relative to a (topologically) transverse flow, see \zcref{sec:leafwisesmoothing}.

        \subsection{Further directions}

\paragraph{Rigidity of extremal foliations.} Given a circle bundle over a surface attaining the extremal value of the Euler number permitted by the Milnor--Wood inequality, Matsumoto proved that there is a \emph{unique} horizontal foliation up to monotone equivalence~\cite{matsumoto.RemarksFoliatedS1}. Does any similar rigidity property hold for foliations achieving extremal points of $\Zg(\Phi)$?

\paragraph{Multiple pseudo-Anosov flows.} Our results focus on foliations transverse to a \emph{single} flow. In view of \zcref{conj:almosttransverse} (for hyperbolic 3-manifolds), it is expected that one must consider \emph{all} pseudo-Anosov flows simultaneously to complete the picture and survey the full geography of taut foliations on a given $3$-manifold.

\begin{conj}\label{conj:multiplepA}
    Suppose $\interior M$ is hyperbolic. Then the union
    $$\bigcup_i \overline{\Zg(\Phi_i)}$$
    over all pseudo-Anosov flows on $M$ is connected and locally box-convex. Here, we choose a common framing of $\partial M$.
\end{conj}

Note that rationality and the sphericity of slippery points would follow from the finiteness conjecture for pseudo-Anosov flows combined with the corresponding theorem for the individual ziggurats $\Zg(\Phi_i)$. As a first step towards this conjecture, it would be interesting to understand how the ziggurats for pseudo-Anosov flows coming from adjacent fibered faces of the Thurston norm ball fit together.

\begin{figure}[htp]
    \centering
    \includegraphics[width=0.8\linewidth]{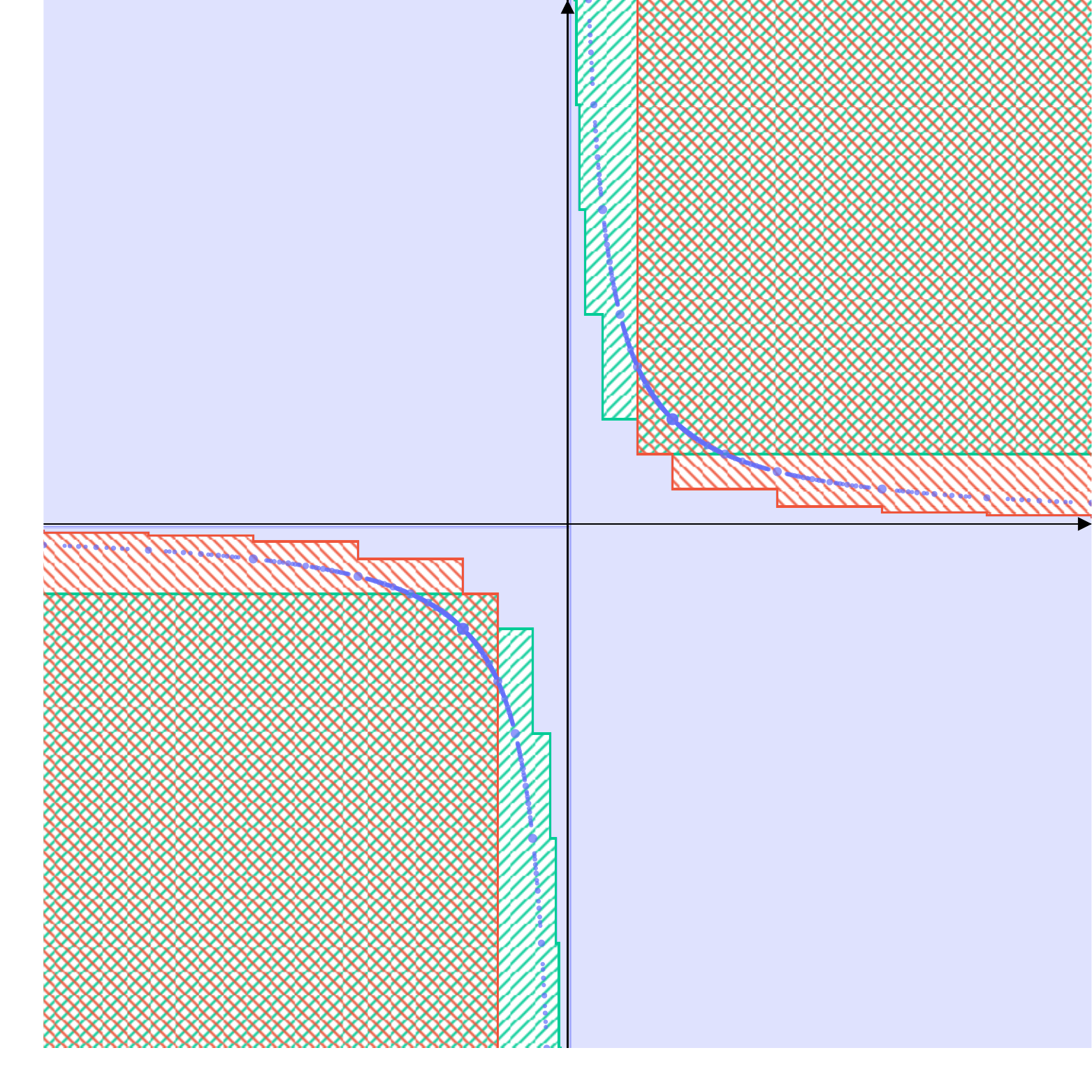}
    \caption{The landscape of surgeries on the complement of the $2$-component link L6a2. The locus of fillings with positive first Betti number is a hyperbola whose rational points are shown with blue dots. The $L$-space region, as computed in \cite{santoro.zhou.LspaceSurgeriesTwobridge}, is shown in solid blue. We found two pseudo-Anosov flows $\Phi_1$ and $\Phi_2$ on this $3$-manifold; $\Zg(\Phi_1)$ and $\Zg(\Phi_2)$ are shown in green and red hatching respectively. They overlap and the union of their interiors precisely fills out the complement of the $L$-space region. \zcref{conj:multiplepA} and \zcref{conj:HFziggurat} both appear to be correct in this example. The $L$-space region is closed, and the interiors of $\Zg(\Phi_1)$ and $\Zg(\Phi_2)$ give taut foliations by \zcref{thmintro:rationalfilling}. Thus, the $L$-space conjecture also holds for surgeries on this link. Note that unlike our conventions in the previous figures, we cannot choose the framing of $\partial M$ so that the degeneracy slopes of both $\Phi_1$ and $\Phi_2$ are $+\infty$. Both $\Zg(\Phi_1)$ and $\Zg(\Phi_2)$ have slippery points, but in the coordinates of this figure they lie at $(0,\infty)$ and $(\infty,0)$. We thank Hugo Zhou and Diego Santoro for sharing their work and helping us to generate this figure.}
    \label{fig:lspace_ziggurat}
\end{figure}

\paragraph{Heegaard Floer homology.} There are many tantalizing conjectures relating Heegaard Floer homology and taut foliations, chief among which is the \emph{$L$-space conjecture}~\cite{BoyerGordonWatson2013, Juhasz2015}. It predicts that the existence of a taut coorientable foliation on an irreducible rational homology $3$-sphere is equivalent to the nonvanishing of a certain flavor of Heegaard Floer homology. A step toward this conjecture would be to show that the space of non-$L$-space surgeries on a given link shares the same structural properties as the foliation ziggurats we study in this article. Let $\mathfrak s$ be a relative $\mathrm{spin}^c$ structure on $M$. Then $\mathfrak s$ induces a $\mathrm{spin}^c$ structure on every Dehn filling of $M$ up to an ambiguity of multiples of the Poincar\'e duals of the Dehn surgery cores. Let $\Zg_{HF}(M,\mathfrak{s})$ be the set of (rational) multislopes such that the corresponding Dehn filling is a non-$L$-space in at least one of the induced $\mathrm{spin}^c$ structures.

\begin{thm}[{\cite{kronheimer.mrowka.MonopolesContactStructures, ozsvath.szabo.HolomorphicDisksGenus}}+\zcref{thmintro:rationalfilling}]
If $\mathfrak{s}$ is the $\mathrm{spin}^c$ structure compatible with $\Phi^\perp$, then
$$\Q^n \cap \interior \Zg(\Phi) \subset \Zg_{\textit{HF}}(M,\mathfrak{s}).$$
\end{thm}

\begin{conj}\label{conj:HFziggurat}
$\overline{\Zg_{\textit{HF}}(M,\mathfrak{s})}$ has the same structural properties as foliation ziggurats: it locally satisfies \zcref{thmintro:box-convexity} (box-convexity) and \zcref{thmintro:rigidityspherical} (rigidity of spherical multislopes), and satisfies \zcref{thmintro:irrationality} (irrational instability), \zcref{thmintro:slippery} (slippery points are spherical), and \zcref{conj:connectedness} (connectedness).
\end{conj}

These results are known to hold when $M$ has a single boundary component; in this case, the set of $L$-space surgeries is known to be an interval with rational endpoints~\cite{rasmussen.rasmussen.FloerSimpleLspaceIntervals}. We do not know how to use the rational link surgery formula to say anything general about the case of multiple boundary components. However, in the case of two-bridge links, Santoro and Zhou have recently been able to completely compute the $L$-space surgery region, affirming \zcref{conj:HFziggurat} in this case \cite{santoro.zhou.LspaceSurgeriesTwobridge}. \zcref{fig:lspace_ziggurat} shows the specific example of the case of surgeries on the two-bridge link L6a2. 

        \subsection{Organization of the article}

The article is divided into three main parts and three appendices.
\begin{itemize}
    \item In \zcref{sec:foundations}, we collect the preliminaries on foliations, contact structures, and flows needed in the rest of the article. While most of the results in \zcref{sec:foundations} are not fundamentally new, we recommend reading \zcref{sec:T2foliations} about foliations on $\T^2$ for our notation and conventions.
    
    \item In \zcref{sec:ETandback}, we prove our version of the Eliashberg--Thurston theorem and its converse for $3$-manifolds with boundary (\zcref{thmintro:ET} and \zcref{thmintro:ETconverse}). This is the technical heart of the article and can be used as a black box for the next part.

    \item In \zcref{sec:mainresults}, we give applications to the structure of the ziggurat $\Zg(\Phi)$ and prove 
    \zcref[range]{thmintro:contactzigg,thmintro:boundedconnected}. Most of these applications are rather direct consequences of \zcref{sec:ETandback}, but this third part can be read without a detailed understanding of the proofs in \zcref{sec:ETandback}.

    \item In \zcref{sec:mcf}, we prove a technical result on mean curvature flow on the $2$-torus used in \zcref{sec:T2foliations}. In \zcref{sec:zigannulusdisk}, we study ziggurats in the special cases of disk and annulus suspensions, which are excluded from our main theorems. In \zcref{sec:beyondaleph}, we discuss a more general framework for describing ``inadmissible'' boundary foliations.
\end{itemize}

\FloatBarrier

        \subsection{Acknowledgments}

We thank Ian Agol, John Baldwin, Jonathan Bowden, Danny Calegari, Vincent Colin, Nathan Dunfield, Yasha Eliashberg, Ying Hu, Siddhi Krishna, Beibei Liu, Tao Li, Qingfeng Lyu, Rachel Roberts, Diego Santoro, Saul Schleimer, and Chi Cheuk Tsang for many helpful conversations and for their interest in this project.

Part of this work was completed while the authors were in residence at the Simons Laufer Mathematical Sciences Institute in Berkeley, California, during the Spring 2026 program \emph{Topological and Geometric Structures in Low Dimensions}; we thank the institute and the organizers for their hospitality and for providing a stimulating environment. The program was supported by the National Science Foundation under Grant No.~DMS-1928930.

T.~Massoni is supported by a Stanford Science Fellowship. J.~Zung was supported in part by a postdoctoral fellowship at MIT under Simons Foundation Award \#994330, \textit{Simons Collaboration on New Structures in Low-Dimensional Topology}. He would especially like to thank Tom Mrowka for all of his support at MIT.

\paragraph{Use of AI.} The mathematical content of this paper was produced without the assistance of AI tools. We used Anthropic's Claude for general proofreading: finding typos, suggestions on wording and exposition, spotting some gaps, and occasional clarifications. Our original computational experiments based on the code in \cite{zung.JonathanzungTransversefol} were performed without the assistance of AI. During the preparation of this paper, we used Anthropic's Claude to refactor and optimize our code. The pictures of ziggurats in this paper were produced using this optimized code.

\clearpage
\part{Foundations}\label{sec:foundations}

In this first part, we set the stage for our main results. We collect various definitions and prove a number of technical results that will be needed later. While most of this material is standard, some statements do not seem to appear in the literature in the form or generality required here, and particular care is needed to handle $C^0$-foliations.

    \section{\texorpdfstring{$C^0$}{C0}-foliations and laminations}

A \textbf{foliation} is a decomposition of an $n$-manifold into $k$-dimensional submanifolds (\textbf{leaves}) locally topologically modeled on $\R^k\times \R^{n-k}$. We are exclusively concerned with the cases $(n,k)=(3,2)$ (codimension 1 foliations on $M$) and $(n,k)=(2,1)$ (codimension 1 foliations on $\partial M$). In this article, a \textbf{$C^0$-foliation} is a foliation with $C^1$ leaves and a continuous tangent plane field. Every foliation is topologically isotopic to a $C^0$-foliation (and even with \emph{smooth} leaves) by~\cite[Theorem 3.4]{C01}. In general, the foliations we consider on manifolds with boundary are transverse to the boundary.

There are two natural topologies on the space of foliations.
Two foliations are \textbf{$\varepsilon$ tangentially close} if their tangent plane fields are $\varepsilon$-close. We call the induced topology the \textbf{tangential topology}. Two foliations $\F_1$ and $\F_2$ are \textbf{$\varepsilon$-$C^0$-close} if for every point $p\in M$ and every $1$-Lipschitz pointed $C^1$ immersion $f_1:(B_{1/\varepsilon},0)\to (M,p)$ with image contained in a leaf of $\F_1$, there is a pointed $C^1$ immersion $f_2:(B_{1/\varepsilon},0)\to (M,p)$ with image contained in a leaf of $\F_2$ which is $\varepsilon$-$C^1$ close to $f_1$, and similarly with $\F_1$ and $\F_2$ swapped. We call the induced topology the \textbf{$C^0$ topology}.

The tangential topology is coarser than the $C^0$ topology, because there can be different foliations tangent to the same continuous plane field.

\begin{defn} \label{def:isotopy}
    Let $\F$ be a $C^0$-foliation. A \textbf{$C^0$-isotopy} of $\F$ is a topological isotopy $h_t : M \rightarrow M$, $0 \leq t \leq 1$, such that $\F_t = h_t(\F)$ is a continuous path of $C^0$-foliations, in the $C^0$ topology. In particular, the plane fields $T \F_t$ exist and vary continuously with $t$. 
    
    This $C^0$-isotopy is \textbf{$\varepsilon$-small}, if 
    \begin{itemize}
        \item $d\big(h_t(p),p\big)<\varepsilon$ for all $p\in M$ and $t \in [0,1]$,
        \item $T\F_t$ is $\varepsilon$-close to $T\F$ for all $t \in [0,1]$.
    \end{itemize}
\end{defn}

Small $C^0$-isotopies, whose speed is also small, can be constructed from suitable topological isotopies of tangentially close foliations, as the next lemma shows.

\begin{lem}
    Let $M$ be a closed oriented manifold, equipped with some Riemannian metric, and $\F$ be a cooriented codimension-$1$ $C^0$-foliation on $M$. 
    
    For every $\varepsilon > 0$, there exists $\delta > 0$ such that the following holds. Let $\F'$ be another such foliation satisfying
    \begin{itemize}
        \item[(a)] $d\big(T\F, T\F' \big) < \delta$,
        \item[(b)] There is a homeomorphism $h : M \rightarrow M$ sending $\F$ to $\F'$ satisfying $d(h(p), p) < \delta$ for all $p \in M$.
    \end{itemize}
    Then there exists a $\varepsilon$-small $C^0$-isotopy $h_t$ of $\F$ such that $h_1(\F) = \F'$ and $\vert \partial_t h_t \vert < \varepsilon$.
\end{lem}

\begin{proof}
    Let us fix $\varepsilon > 0$. We construct the isotopy $h_t$ as a linear interpolation between $\F$ and $\F'$ with respect to a common transverse smooth flow.

    More precisely, let $v$ be a smooth vector field positively transverse to $T\F$, whose flow is denoted by $\varphi_v^t$. For $\delta > 0$ small enough, and for every foliation $\F'$ satisfying (a) and (b), there exists a map $\tau : M \rightarrow \R$ defined as follows: in a small foliated chart near $p \in M$, $\varphi_v^{\tau(p)}(p)$ belongs to the leaf of $\F'$ passing through $h(p)$ in that chart. We also define $h_1(p) \coloneqq \varphi_v^{\tau(p)}(p)$. By definition, $\tau$ satisfies:
    \begin{itemize}
        \item $\tau$ is $C^1$ along the leaves of $\F$,
        \item For every $p \in M$, the map $s \mapsto s + \tau\big(\varphi_v^s(p)\big)$ is strictly increasing.
    \end{itemize}
    These properties can easily be checked in small smooth charts in which $v$ is vertical and both $\F$ and $\F'$ are graphical. The second item follows from the fact that $h_1$ preserves the flow lines of $v$ and is strictly increasing along them, together with the identity 
    $$h_1\big(\varphi_v^s(p)\big) = \varphi_v^{\tau\big(\varphi_v^s(p)\big)}\big(\varphi_v^s(p)\big) = \varphi_v^{s + \tau\big(\varphi_v^s(p)\big)}(p).$$
    
    We then define, for $t \in [0,1]$,
    $$h_t(p) \coloneqq \varphi_v^{t \tau(p)}(p).$$
    By definition, $h_t$ preserves the flow lines of $v$, and using the second point above, one checks that it is strictly increasing along these flow lines. Therefore, one deduces that $h_t$ is a topological isotopy of $M$ such that $h_1(\F) = \F'$. Moreover, there exists a constant $C>0$, only depending on $\F$ and $v$, such that for all $p \in M$ and $t \in [0,1]$, $$d\big(h_t(p), p\big) \leq C \delta, \qquad \vert \tau(p)\vert \leq C \delta, \qquad \mathrm{and} \qquad \vert \partial_t h_t \vert \leq C \delta.$$

    Finally, we estimate the distance between $T \F$ and $T \F_t$, where $\F_t = h_t(\F)$. We may restrict to a smooth coordinate chart in which $v$ is vertical, and all these foliations are graphical. Then, $\F_t$ is simply obtained by a linear interpolation between the graphs of $\F$ and those of $\F'$. Let $q = h_t(p)$. Then we see that 
    \begin{align*}
        d\big(T\F_t(q), T\F(q)\big) &\leq d\big(T\F_t(q), T\F(p)\big) + d\big(T\F(q), T\F(p)\big) \\
        &\leq d\big( T\F'(h_1(p)), T\F(p)\big) + d\big(T\F(q), T\F(p)\big)\\
        &\leq d\big( T\F'(h_1(p)), T\F(h_1(p))\big) + d\big(T\F(h_1(p)), T\F(p)\big) + d\big( T\F(q), T\F(p)\big) \\
        &  \leq C'\big(\delta + \omega(\delta) \big),
    \end{align*}
    for some constant $C' > 0$ and a function $\omega : [0, 1) \rightarrow \R_{>0}$ with $\lim_{\delta \rightarrow 0} \omega(\delta) = 0$, which depends on the modulus of continuity of $T \F$. We then conclude by choosing $\delta$ small enough.
\end{proof}

\begin{rem}
    By the previous lemma, we can always replace a small $C^0$-isotopy by one with small speed. This will be useful in view of \zcref{lem:isotopy_extension} below.
\end{rem}

\begin{defn}
A foliation $\F$ is \textbf{taut} if for every point $p\in M$, there exists a closed loop positively transverse to $\F$ and passing through $p$.
\end{defn}

Whenever $\Phi$ is a flow and $\F$ is a foliation, we say that $\F$ is \textbf{ $\Phi$-horizontal} if $\F$ is transverse to $\Phi$. We say that an isotopy of foliations is horizontal if it is an isotopy through $\Phi$-horizontal foliations.

\begin{lem}[Isotopy extension]\label{lem:isotopy_extension} 
Let $\F$ be a $\Phi$-horizontal foliation on $M$. Let $(h_t)_{t\in [0,1]}$ be a $C^0$-isotopy of $\partial \F$ through $\Phi$-horizontal foliations on $\partial M$ which is smooth in both the $t$ and leafwise directions. Then $h_t$ extends to a $C^0$ isotopy $\overline h_t$ of $\F$ through $\Phi$-horizontal foliations on $M$. Moreover, if $h_t$ is $\varepsilon$-small (in the above sense) and $\vert \partial_t h_t \vert < \varepsilon$, then $\overline h_t$ is $O(\varepsilon)$-small, where the implicit constants may depend on $\Phi$ and $\F$ only.
\end{lem}

\begin{proof}
    Let $v_{\Phi}$ be a smooth vector field generating $\Phi$. In fact, for this argument we will only use that $v_{\Phi}$ is Lipschitz.
    
    Choose a tubular neighborhood of $\partial M$ parametrized as $\mathbb{T}^2_{x,y} \times [0,1]_r$. Our first preparation step is to modify the choice of this tubular neighborhood so that $\F$ and $v_{\Phi}$ are nearly invariant in the $r$ direction. Invariance of $\mathcal F$ in the $r$ direction can be achieved by a small isotopy and restricting to a small tubular neighborhood if necessary. To obtain near-invariance of $\Phi$ in the $r$ direction, let $\delta>0$ (to be determined later) and apply a reparametrization $r\mapsto f(r)$, where $f : [0,1]_r \rightarrow [0,1]$ is a diffeomorphism satisfying 
    $$f'(r) < \frac \delta {Kr} \qquad \mathrm{and} \qquad f(\delta/K) > \frac{1}{2},$$
    where $K$ is a Lipschitz constant for $v_\Phi$. This guarantees that 
    $$|v_{\Phi}(x,y,r) - v_{\Phi}(x,y,0)| <\delta$$
    for $r\in [0,1/2]$. 
    
    Choose a smooth cutoff function $g : [0,1]_r \to [0,1]$ satisfying that
    \begin{itemize}
        \item $g(r)=1$ for $r \leq  1/3$,
        \item $g(r)=0$ for $r \geq 1/2$,
        \item $\vert g'(r) \vert < 10$.
    \end{itemize}
    
    Define the isotopy 
    $$\overline h_t(x,y,r) \coloneqq \big(h_{tg(r)}(x,y),r\big).$$
    On $\partial M$, $\overline h_t$ agrees with $h_t$, and $\overline h_t = \mathrm{id}$ for $r>1/2$. When $\delta$ is small enough, $\overline h_t(\F)$ is $\Phi$-horizontal for all $0 \leq t \leq 1$. 
\end{proof}

We will also need a version of this result that we will apply to contact structures. The same proof works when $\F$ is replaced by a smooth plane field:

\begin{lem}[Isotopy extension for plane fields]\label{lem:isotopy_extension_contact} 
Let $\xi$ be a smooth $\Phi$-horizontal 2-plane field on $M$. Let $(h_t)_{t\in [0,1]}$ be a smooth isotopy of $\partial \xi$ through $\Phi$-horizontal foliations on $\partial M$. Then $h_t$ extends to a smooth isotopy $\overline h_t$ of $\xi$ through $\Phi$-horizontal $2$-plane fields on $M$.
\end{lem}

A \textbf{semiconjugacy} from a foliation $\F_0$ on $M$ to a foliation $\F_1$ on $M$ is a continuous, surjective map $h:M\to M$ which restricts to an orientation preserving immersion from leaves of $\F_0$ to leaves of $\F_1$ and restricts to a non-decreasing map on the local leaf spaces of $\F_0$ and $\F_1$. An elementary kind of semiconjugacy is a \textbf{Denjoy blowdown}: if there is an inclusion $\iota:(\lambda\times [0,1], \mathcal H) \hookrightarrow (M,\F_0)$, where $(\lambda\times[0,1],\mathcal H)$ is a foliated $I$ bundle over a (possibly noncompact) surface $\lambda$, then there is a semiconjugacy from $\F_0$ to a new foliation $\F_1$ which collapses the image of $\iota$ to a leaf homeomorphic to $\lambda$. We say that $\F_0$ is a \textbf{Denjoy blowup} of $\F_1$ along $\lambda$. For any choice of $\F_1$, $\lambda$ a leaf of $\F_1$, and $\mathcal H$ a foliated $I$-bundle over $\lambda$, one can construct a Denjoy blowup $\F_0$ such that the semiconjugacy $h:\F_0\to \F_1$ is $\varepsilon$-$C^0$ close to the identity and $\F_0$ and $\F_1$ are $\varepsilon$-$C^0$-close; one simply needs to split open along the leaf $\lambda$ to width much less than $\varepsilon$ and insert $\mathcal H$. We call such a Denjoy blowup a \textbf{Denjoy blowup of size $\varepsilon$}.

\begin{defn} \label{def:semiequi}
Two foliations $\F_0$ and $\F_1$ are \textbf{monotone equivalent} if there exists a third foliation $\mathcal{G}$ and semiconjugacies $h_0 : \mathcal{G} \twoheadrightarrow \F_0$ and $h_1 : \mathcal{G} \twoheadrightarrow \F_1$.
\end{defn}

One may check that if there are two semiconjugacies $\mathcal{G}_0 \twoheadrightarrow \F$ and $\mathcal{G}_1 \twoheadrightarrow \F$, then there exists a foliation $\mathcal{H}$ and semiconjugacies $\mathcal{H} \twoheadrightarrow \mathcal{G}_0$ and $\mathcal{H} \twoheadrightarrow \mathcal{G}_1$ so that the following diagram commutes:
$$\begin{tikzcd}[column sep=scriptsize]
	& {\mathcal{H}} \\
	{\mathcal{G}_0} && {\mathcal{G}_1} \\
	& {\mathcal{F}}
	\arrow[two heads, from=1-2, to=2-1]
	\arrow[two heads, from=1-2, to=2-3]
	\arrow[two heads, from=2-1, to=3-2]
	\arrow[two heads, from=2-3, to=3-2]
\end{tikzcd}$$
In particular, there is a semiconjugacy $\mathcal{H} \twoheadrightarrow \F$. This justifies that we don't need to consider zig-zags of semiconjugacies in \zcref{def:semiequi}.

\medskip

Compact leaves---especially planar (genus $0$) ones---will play a special role in our main theorems. In the presence of a transverse flow, we will be able to arrange that there are finitely many of those after blowdowns. 

\begin{lem} \label{lem:isolatedcompactleaves}
    Let $\Phi$ be a smooth flow tangent to $\partial M$, and let $\F$ be a $\Phi$-horizontal foliation. If $\F$ is not a fibration, then there exists a $\Phi$-horizontal $\F'$ obtained from $\F$ by blowdowns only, all of whose compact leaves are isolated.
\end{lem}

\begin{proof}
    As in Step 1 of~\cite[Section 8]{B16}, there is a finite collection of (smooth) embeddings $N_k = \Sigma_k \times [0, c_k] \hookrightarrow M$, where $c_k \geq 0$ and $\Sigma_k$ is a compact surface (possibly with boundary) so that each $N_k$ is a foliated $I$-bundle, $\partial_hN_k \coloneqq \Sigma_k \times \{0, c_k\}$ are leaves of $\F$, and every compact leaf of $\F$ is contained in some $N_k$. Furthermore, any two $N_k$'s only intersect along their horizontal boundaries $\partial_h N_k$. We can assume that these $N_k$'s are arbitrarily thin (possibly setting $c_k=0$), after suitable subdivisions. Thus, we may assume that the restriction of $\Phi$ to each $N_k$ is a product flow. Moreover, since $\F$ is not a fibration, we can assume that the $N_k$ do not cover $M$. 

    Now we blow down the $N_k$ one by one. If there is at least one $\ell$ such that $c_\ell > 0$, we may find such an $\ell$ so that either $\Sigma_\ell \times \{0\}$ or $\Sigma_\ell \times \{c_\ell\}$ lies on $\partial (\bigcup_k N_k)$. In particular, $\Sigma_\ell \times \{0\}$ is disjoint from $\Sigma_\ell \times \{c_\ell\}$. Since $\Phi$ is a smooth product flow on $N_\ell$, we may blow down $N_\ell$ by a semiconjugacy preserving flow lines of $\Phi$. Repeat this procedure until $c_k=0$ for all $k$. By hypothesis, at each step no $N_k$ covers $M$ entirely. Then each $N_k$ contains exactly one compact leaf, so the compact leaves are isolated and in finite number.
\end{proof}

\begin{defn}\label{defn:ribbon}
    Given a foliation $\F$ on $M$, a \textbf{ribbon} is a smoothly embedded strip $R =  [-T, T] \times [0,\varepsilon] \subset M$ transverse to $\F$ such that the arcs $\{\pm T\} \times [0,\varepsilon]$ lie on $\partial M$, $(-T,T)\times [0,\varepsilon]$ lies in $\interior M$, and the induced foliation $\F|_R$ is topologically isotopic rel.~boundary to the standard foliation by lines $[-T,T]\times \{y\}$, $y\in [0,\varepsilon]$. We use parameters $T$ and $\varepsilon$ for the width and height to emphasize that ribbons are usually long and thin.
\end{defn}

\begin{constr}\label{constr:ribbontwist}
    Given a ribbon $R$ and an orientation preserving homeomorphism $h:[0,\varepsilon]\to [0,\varepsilon]$, a \textbf{twisting modification} of $\F$ along $R$ is the operation of cutting along $R$ and regluing the interval's worth of leaves intersecting $R$ via the homeomorphism $h$. This operation can be performed so that the resulting foliation is a $C^0$-foliation.
\end{constr}

    \section{Foliations of \texorpdfstring{$\T^2$}{T2}}\label{sec:T2foliations}

Foliations on $\T^2$ naturally arise in the boundary of foliations on 3-manifolds. In this section we establish terminology for describing foliations on $\T^2$ as well as some basic properties. While the results may appear technical, all of the rationality, stability, and flexibility results we prove for $\Zg(\Phi)$ have their roots in phenomena on $\T^2$.

        \subsection{Slope and type}

When a leaf spirals to the right (resp.~left) onto a closed curve, we say that the closed curve has \textbf{$\sharp$ (resp.~$\flat$) type spiraling}. A \textbf{sharp ($\sharp$) annulus} is a $C^0$-foliated $S^1\times [0,1]$ with two compact leaves along the boundary, and whose remaining leaves are noncompact and spiral onto the compact leaves toward the right; see \zcref{fig:sharpannulus}. Similarly, a \textbf{flat ($\flat$) annulus} is a similar $C^0$-foliated annulus, except that the noncompact leaves spiral onto the compact leaves toward the left; see \zcref{fig:flatannulus}. Finally, a \textbf{Reeb annulus} is one for which leaves spiral to the left onto $S^1\times \{0\}$ and to the right along $S^1\times \{1\}$ (or vice versa). A \textbf{$\sharp$, $\flat$, or Reeb lamination} is a sublamination of the corresponding foliation with at least one noncompact leaf. Whenever $\lambda$ is a closed leaf of a foliation of $\T^2$, one may blow up $\lambda$ and insert a $\sharp$ or $\flat$ annulus. We call this operation a $\sharp$ or $\flat$ blowup.

\begin{defn}[Coarse slope]
Let $\G$ be an oriented foliation on $\T^2$ and $\ell$ be a leaf of $\G$ not lying in the interior of a Reeb annulus, we define its \textbf{(coarse) slope}, $\slope(\ell)$ which is an element of $(\R^2\setminus \{0\})/\R_{>0} \cong S^1$, where $\R_{>0}$ acts on $\R^2\setminus \{0\}$ by scaling. It is defined as 
$$\slope(\ell) \coloneqq \left[ \lim_{t\to +\infty} \frac{\varphi_t(x)-x}{t} \right],$$ 
where $\varphi_t$ denotes a flow directing the universal cover $\widetilde{\G}$ on $\R^2$ and $x\in \ell$; up to positive scalars, the limit does not depend on these choices. We are usually concerned with slopes having non-negative $x$-coordinate; we call these the \textbf{admissible} slopes. The admissible slopes can be identified with $\R \cup \{\pm \infty\}$ via $(1,x) \mapsto x$ and $(0, \pm1 ) \mapsto \pm \infty$.

If $\G$ has no Reeb annuli, all of its leaves have the same (coarse) slope which we denote by $\slope(\G)$. If $\G$ has Reeb annuli, $\slope(\G)$ is not well-defined.
\end{defn}

\begin{prop}\label{prop:poincare}
    Every Reebless $C^0$-foliation $\G$ on $\T^2$ of slope $s$ is monotone equivalent to a linear foliation $\G_s$ of slope $s$. When $s$ is irrational, the monotone equivalence can be taken to be a blowup along countably many leaves. When $s$ is rational, the monotone equivalence can be taken to be a blowup along countably many leaves, each inserting a $\sharp$ or $\flat$ annulus, followed by a blowdown of countably many $\natural$ annuli. The invariant measure on $\G_s$ pulls back to an invariant transverse measure on $\G$ whose support is the union of all minimal sets of $\G$.
\end{prop}

This follows from the existence of a closed transversal (see \cite[Ch 1, Sec 4.3]{hector.hirsch.IntroductionGeometryFoliations}) and the classification of circle homeomorphisms by Poincar\'e and Denjoy.

\begin{cor}
The coarse slope is a complete invariant of Reebless foliations of $\T^2$, up to monotone-equivalence. 
\end{cor}

\begin{figure}[ht]
\centering
    \begin{subfigure}{0.45\linewidth}
    \centering
    \includegraphics[width=\linewidth]{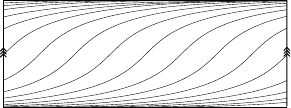}
    \caption{Standard $\sharp$ annulus.}
    \label{fig:sharpannulus}
    \end{subfigure}
    \hspace{0.05\linewidth}
    \begin{subfigure}{0.45\linewidth}
    \centering
    \includegraphics[width=\linewidth]{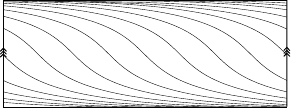}
    \caption{Standard $\flat$ annulus.}
    \label{fig:flatannulus}
    \end{subfigure}
    \caption{Foliations on an annulus.}
\end{figure}

\begin{figure}[ht]
    \centering
    \includegraphics[width=0.5\linewidth]{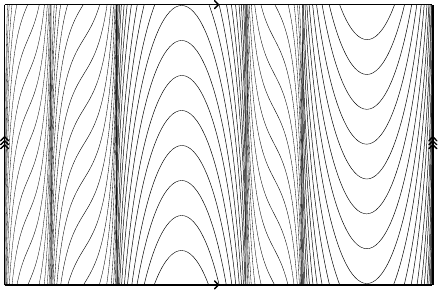}
    \caption{An $\aleph$ foliation of $\T^2$. If any of the sharp annuli were changed to flat annuli or vice versa, the foliation would fail to be $\aleph$ type.}
    \label{fig:aleph}
\end{figure}

\begin{defn}[Type]\label{defn:type}
Let $\G$ be an oriented foliation on $\T^2$. We define the \textbf{type} of $\G$ as follows:
\begin{enumerate}
    \item If $\G$ has no Reeb annuli,
        \begin{enumerate}
            \item If $\G$ has no $\sharp$ nor $\flat$ annuli, we say that it is of type $\natural$. Otherwise,
            \item If $\G$ has only $\sharp$ but no $\flat$ annuli, we say that it is of type $\sharp$,
            \item If $\G$ has only $\flat$ annuli but no $\sharp$ annuli, we say that it is of type $\flat$,
            \item Otherwise, $\G$ has both $\sharp$ and $\flat$ annuli, and we say that it is of type $\dagger$.
        \end{enumerate}
    \item If $\G$ has Reeb annuli,
        \begin{enumerate}
            \item If all leaves which spiral toward the right have the same slope, and all leaves which spiral toward the left have the same slope, then we say that it is of type $\aleph$.
            \item If the previous condition fails, we say that $\G$ is not of an admissible type.
        \end{enumerate}
\end{enumerate}
Besides the last case, we call these types \textbf{admissible}.
\end{defn}

Note that irrational foliations are necessarily of type $\natural$, and that $\natural$ foliations with rational slopes are exactly the foliations by closed curves. See \zcref{fig:aleph} for the typical shape of a foliation of type $\aleph$. We are forced to consider $\aleph$ type foliations because they may occur as limits of other admissible foliations; see \zcref{lem:slconv}.

\begin{defn}[Fine slope]
    A \textbf{(finite) fine slope} is a pair $(s,\heartsuit)$ (usually written $s^\heartsuit$), where $s\in \R$ and $\heartsuit \in \{\sharp,\flat,\natural,\dagger\}$. We say that $s$ is the underlying \textbf{coarse slope} of $s^\heartsuit$. If $\F$ is a foliation on $M$ with admissible boundary, its \textbf{fine (boundary) multislope} is the tuple $$\bm{\fineslope}(\F):=(\fineslope(\partial_1\F),\dots,\fineslope(\partial_n\F)).$$
\end{defn}

In \zcref{sec:beyondaleph}, we will also consider possibly infinite coarse and fine slopes, but some of our results will not immediately work for those, so we will mostly restrict to Reebless foliations with finite slope, or of $\aleph$ type.

\begin{defn}[Generalized slope]
    A \textbf{generalized} (coarse or fine) slope is either a (coarse or fine) slope or $\aleph$.
\end{defn}

    \subsection{Comparing slopes}\label{sec:comparing_slopes}

The goal of this section is to connect the geometry of foliations on $\T^2$ to their fine slopes.

\begin{defn}
    If $\G_1$ and $\G_2$ are two foliations on $\T^2$, then say that $\G_1$ has \textbf{lower slope} than $\G_2$ at a point $x\in \T^2$ if positively oriented tangent vectors in $T_x\G_1$ and $T_x\G_2$ combine to form a positively oriented basis of $T_x\R^2$. We say that $\G_1$ has \textbf{pointwise lower slope} than $\G_2$ if $\G_1$ has lower slope than $\G_2$ at every point in $\T^2$. In this case, we write $\G_1 \prec \G_2$. Note that if $\Phi$ is a smooth nonsingular flow on $\T^2$ with slope $+\infty$, then $\prec$ is a partial order when restricted to $\Phi$-horizontal foliations.
\end{defn}

\begin{defn}[Partial order on fine slopes]
    We define a partial relation $\prec$ on (finite) fine slopes generated by the following inequalities:
    \begin{equation}\label{eq:slope_ineq1}
        s^\heartsuit \prec t^\diamondsuit
    \end{equation}
    whenever $s < t$ and $\heartsuit, \diamondsuit \in \{\natural, \flat, \sharp, \dagger\}$, and
    \begin{equation}\label{eq:slope_ineq2}
        s^\flat \prec s^\dagger \prec s^\sharp, \qquad \text{and} \qquad s^\dagger \prec s^\dagger.
    \end{equation}
\end{defn}

\begin{prop} \label{prop:fineslope_inequality}
Suppose $\G_1$ and $\G_2$ are $\Phi$-horizontal, Reebless foliations with finite slopes. Then
$$\fineslope(\G_1) \prec \fineslope(\G_2)$$
if and only if there exist $\G_1'$ and $\G_2'$ isotopic to $\G_1$ and $\G_2$ respectively through horizontal foliations such that $\G_1'$ has pointwise lower slope than $\G_2'$.
\end{prop}

Throughout the rest of this section, we prove various results leading up to the proof of the forward direction of \zcref{prop:fineslope_inequality}. We defer the reverse direction to \zcref{cor:fineslope_inequality_converse} where it is proven as a consequence of \zcref{lem:phaselocking}.

\medskip

First, we may use the Mean Curvature Flow on $\T^2$ (see \zcref{sec:mcf}) to simultaneously straighten $C^0$-foliations on $\T^2$.

\begin{lem}[Straightening]\label{lem:straightening} 
    Let $\G_1,\dots,\G_k$ be Reebless $C^0$-foliations of $\T^2$, possibly with infinite slopes. Then there are $C^0$-isotopies $h_1^t,\dots,h_k^t$, $0\leq t \leq T$, of $\G_1,\dots,\G_k$ respectively such that the isotoped foliations $h_i^T (\G_i)$ are tangentially $C^0$-close to linear foliations. In particular, if $\G_i$ and $\G_j$ have different slopes, then $h_i^T(\G_i)$ is transverse to $h_j^T(\G_j)$. Moreover, if any pair of the foliations start out transverse, they remain transverse during the isotopy, and if any foliation is smooth, then the corresponding isotopy is smooth as well.
\end{lem}

\begin{proof}
    One can find such isotopies using simultaneous mean curvature flow on the leaves of $\G_1,\dots,\G_k$. The same strategy is used in \cite[Lemma 4.1]{Z24}. By \zcref{prop:mean_curvature_flow}, the foliations converge in the tangential sense to linear foliations (after some preparatory $C^0$-isotopies which are arbitrarily small in the nonsmooth case). The avoidance principle implies that mean curvature flow on a pair of curves may eliminate intersections, but never introduces new inessential intersections between them as desired. If $\G_i$ is smooth, then $h_i^t$ is smooth thanks to the smooth dependence of mean curvature flow on initial data. Restricting $h_i^t$ to $t \in [0,T]$ for $T > 0$ large enough gives the desired isotopies.
\end{proof}

\begin{lem} \label{lem:straightening2}
    Let $\Phi$ be a smooth, nonsingular and Reebless flow on $\T^2$, and let $\G_1$ and $\G_2$ be two $\Phi$-horizontal, Reebless foliations of $\T^2$. If $\slope(\G_1),\slope(\G_2)$ are distinct real numbers, then $\G_1$ and $\G_2$ can be put in transverse position by horizontal $C^0$-isotopies.
\end{lem}

\begin{proof}
    Apply \zcref{lem:straightening} to $\G_1$, $\G_2$, and $\Phi$ to obtain three isotopies $h_1^t$, $h_2^t$, and $h_\Phi^t$ bringing $\G_1$, $\G_2$ into transverse position. Then, \zcref{lem:straightening} guarantees that $h_i^t(\G_i)$ remains transverse to $h_\Phi^t(\Phi)$ for all $t$. Because $\Phi$ is smooth, the isotopy $h_\Phi^t$ is smooth as well. Therefore, ${(h_\Phi^t)}^{-1}h_1^t$ and ${(h_\Phi^t)}^{-1}h_2^t$ are isotopies of $\G_1$ and $\G_2$, respectively, through $\Phi$-horizontal foliations which bring $\G_1$ and $\G_2$ into transverse position.
\end{proof}

\begin{lem}[Accordion trick] \label{lem:accordion}
    Let $\G_1$ and $\G_2$ be two $\Phi$-horizontal, Reebless foliations of $\T^2$ with the same slope $s \in \Q$. Suppose that $\G_1$ has at least one $\flat$ annulus and $\G_2$ has at least one $\sharp$ annulus. Then $\G_1$ and $\G_2$ can be $C^0$-isotoped through $\Phi$-horizontal foliations so that $\G_2$ is positively transverse to $\G_1$. 
\end{lem}

\begin{proof}
    First apply the straightening lemma to bring $\Phi$, $\G_1$, and $\G_2$ tangentially $C^0$ close to linear foliations. In particular, $\Phi$ is nearly vertical and $\G_1$ and $\G_2$ are graphical. Compress the complement of the sharp region in $\G_1$ (like the bellows of an accordion) so the $\sharp$ region envelopes nearly all of $\T^2$, and similarly compress the complement of the flat region in $\G_2$. Then $\G_1$ and $\G_2$ can be put in transverse position as in \zcref{fig:accordion_trick}.
\end{proof}

\begin{figure}[H]
    \centering
    \includegraphics[width=0.6\textwidth]{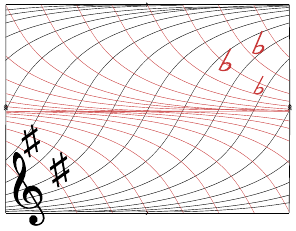}
    \caption{The accordion trick. $\G_2$ is shown in black and $\G_1$ is shown in red.}
    \label{fig:accordion_trick}
\end{figure}

\begin{proof}[Proof of forward direction of \zcref{prop:fineslope_inequality}]
    If $\fs(\G_1), \fs(\G_2)$ match \zcref{eq:slope_ineq1}, then we can put $\G_1$ and $\G_2$ in transverse position by \zcref{lem:straightening2}.
    If $\fs(\G_1), \fs(\G_2)$ match any of the inequalities in \zcref{eq:slope_ineq2}, $\G_1$ has a $\flat$ annulus and $\G_2$ has a $\sharp$ annulus so we can use the accordion trick (\zcref{lem:accordion}) to put $\G_1$ and $\G_2$ in transverse position.
\end{proof}

        \subsection{Making foliations transverse}

If $\G$ is a $C^0$-foliation on $\T^2$ and $\gamma$ is an arc contained in a leaf of $\G$, then it might not be possible to perturb this arc to an arc transverse to $\G$. However, this becomes possible after an arbitrarily small $C^0$-isotopy of $\G$, essentially by making it smooth near $\gamma$ using the methods from~\cite{B16}. We will extend this to suitable closed curves. First we define monotone curves, which are good substitutes for transversals in the theory of $C^0$-foliations.

\begin{defn}\label{def:monotonecurve}
Let $\G$ be a $C^0$-foliation on $\T^2$, and $\gamma : S^1 \rightarrow \T^2$ be a piecewise $C^1$ closed embedded curve on $\T^2$. We say that $\gamma$ is \textbf{positively (resp.~negatively) monotone} with respect to $\G$ if $\gamma$ decomposes into finitely many $C^1$ subarcs, each of which is either
\begin{itemize}
    \item Contained in a leaf of $\G$, or
    \item Positively (resp.~negatively) transverse to $\G$
\end{itemize}
and $\gamma$ is not entirely contained in a leaf of $\G$.
\end{defn}

\begin{lem}\label{lem:monotone_existence}
    Suppose $\G$ is a Reebless $C^0$-foliation on $\T^2$ of slope $s \in \R$.
    \begin{itemize}
        \item If $\G$ admits an embedded positively monotone closed curve of slope $p/q$, then $s \leq p/q$.
        \item For every rational slope $p/q>s$, $\G$ admits an embedded positively monotone closed curve of slope $p/q$.
    \end{itemize}
\end{lem}

\begin{proof}
    The first item follows from the definition of slope and is left to the reader.
    
    We prove the second item. Let $\widetilde{\G}$ be the lift of $\G$ to the universal cover $\R^2$. The leaf space $L$ of $\widetilde{\G}$ is homeomorphic to $\R$. Therefore, there is a positively oriented path $\widetilde{\gamma}$ from a point $x \in L$ to a point $y \in L$, corresponding to points $\widetilde{x}, \widetilde{y} \in \R^2$ such that $\widetilde{y} = \widetilde{x}+(q,p)$. Every sufficiently short arc in $L$ lifts to a transversal to $\widetilde \G$, so $\gamma$ may be divided into subarcs which lift to transversals $\widetilde{\gamma_1},\dots, \widetilde{\gamma_n}$ to $\widetilde \G$. Connecting these subarcs with paths in leaves of $\widetilde \G$ and then pushing them down to $\G$ yields a monotone curve $\sigma$ of the desired slope $(p,q)$.
    \begin{figure}[htbp]
    \centering
    \begin{subfigure}{0.3\linewidth}
    \centering
    \def\svgwidth{\textwidth}
    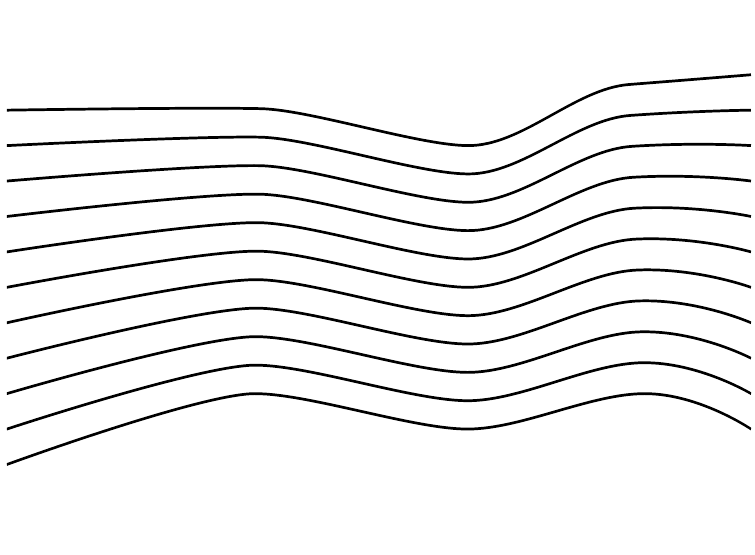
    \end{subfigure}
    \hspace{4em}
    \begin{subfigure}{0.3\linewidth}
    \centering
    \def\svgwidth{\textwidth}
    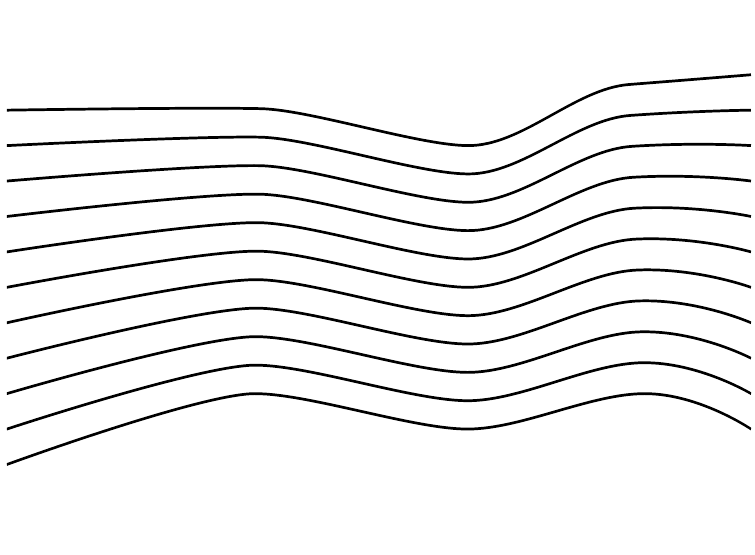
    \end{subfigure}
    \caption{Generic perturbation of a subarc of $\sigma$ contained in a leaf.}
    \label{fig:generic_staircase}
\end{figure}

Now we modify $\sigma$ so that it is embedded. First, we perturb $\sigma$ so that its self-intersections are transverse. For transverse subarcs of $\sigma$, this can be achieved by a generic perturbation. For the subarcs of $\sigma$ contained in leaves of $\G$, replace them with a short staircase with generic heights as in \zcref{fig:generic_staircase}. Now we can resolve self-intersections of $\sigma$. Resolution does not change the homology class of $\sigma$, and two embedded essential curves with different slopes must intersect, so some connected component of the resulting embedded, positively monotone curve must have slope $p/q$.
\end{proof}

\begin{lem}\label{lem:transversalizing}
    Let $\G$ be a $C^0$-foliation on $\T^2$ and $\gamma$ be a positively (resp.~negatively) monotone curve with respect to $\G$. Then there are arbitrarily small $C^0$-isotopies $\mathcal{G}_t$ of $\G$ and $\gamma_t$ of $\gamma$ such that $\gamma_1$ is a $C^1$ curve positively (resp.~negatively) transverse to $\G_1$.
\end{lem}

\begin{proof}
    For simplicity, let us assume that $\gamma$ is contained in a leaf $L_0$ of $\G$ for $t \in [0,1/2] \subset S^1 = \R \slash \Z$, and $\gamma$ is positively transverse to $\G$ on $(1/2, 1)$. The general case can be treated similarly and is left to the reader.

    First, we may find $C^1$ coordinates $(x,y) \in [-1,1]^2$ near $\Gamma=\gamma([0,1/2]) \subset \T^2$ in which $\G$ is graphical, and $L_0$ corresponds to $[-1,1] \times \{0\}$. Then by assumption, for $\delta > 0$ small enough, the point $\gamma(-\delta)$ lies below $L_0$ and the point $\gamma(1/2+ \delta)$ lies above $L_0$. We then perform a very small $C^0$-isotopy of $\G$ near a slightly larger neighborhood of $\Gamma$ to make $\G$ $C^1$ there using~\cite[Lemma 3.1]{B16}; we denote the resulting foliation by $\G'$. We may further arrange that $L_0$ is preserved by this operation, so that the relative positions of $L_0$ with $\gamma(-\delta)$ and $\gamma(1/2 + \delta)$ are preserved. We may now connect $\gamma(-\delta)$ and $\gamma(1/2 + \delta)$ with an arc positively transverse to $\G'$. This way, we obtain a curve $\gamma'$ positively transverse to $\G'$, which can be made arbitrarily $C^0$-close to $\gamma$. See \zcref{fig:smoothing_monotone}.
\end{proof}

\begin{figure}[ht]
    \centering
    \def\svgwidth{0.7\textwidth}
    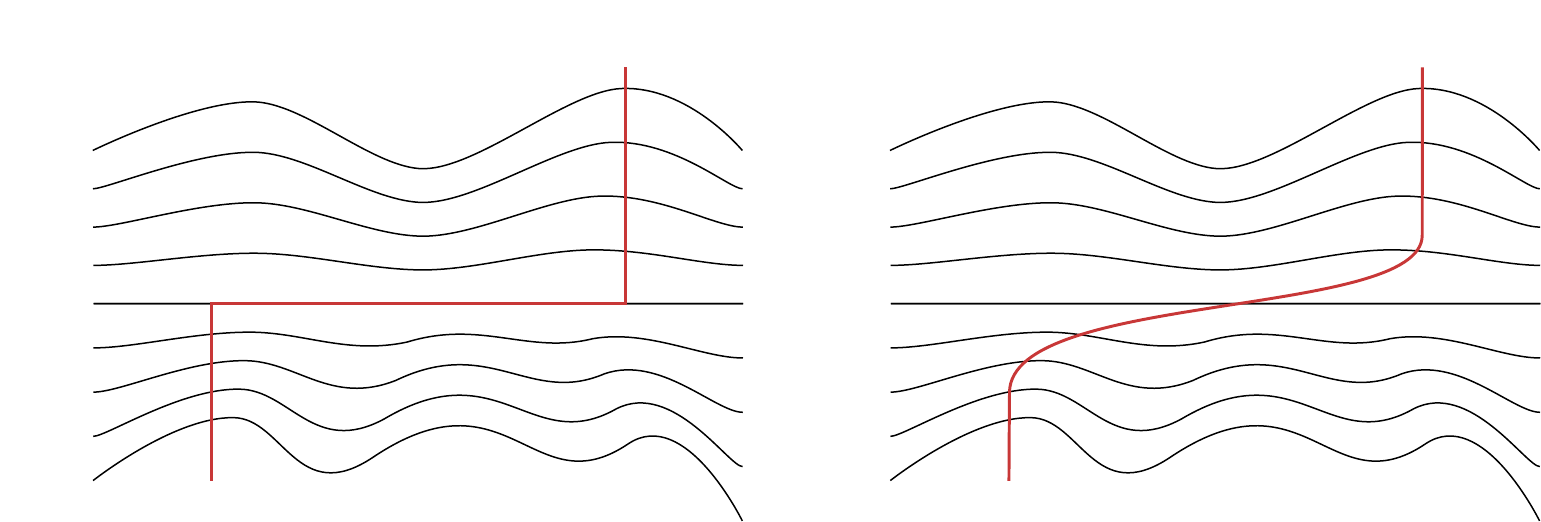
    \caption{Straightening a monotone curve.}
    \label{fig:smoothing_monotone}
\end{figure}

We now extend this result and consider two foliations on $\T^2$ which are everywhere tangent or positively transverse, such that the positively transverse region is sufficiently large. First, we make precise what it means for the positively transverse region to be sufficiently large.

\begin{defn} \label{def:mostlytrans}
    Let $\G$, $\G'$ be two $C^0$-foliations on $\mathbb{T}^2$. We say that $\G$ and $\G'$ are \textbf{mostly-transverse} if there exist open sets $W \Subset V \subset \mathbb{T}^2$ such that
    \begin{itemize}
        \item On a neighborhood of $\T^2 \setminus V$, the leaves of $\G$ and $\G'$ coincide,
        \item On a neighborhood of $\overline{V} \setminus W$, $\G$ and $\G'$ are smooth,
        \item On $V \setminus W$, $T\G$ and $T\G'$ coincide, and on $W$, $T\G$ is positively transverse to $T\G'$,
        \item For every $p \in \mathbb{T}^2$, the leaf of $\G$ passing through $p$ intersects $W$ on both sides of $p$.
    \end{itemize}
\end{defn}

\begin{figure}[ht]
    \centering
    \def\svgwidth{0.5\textwidth}
    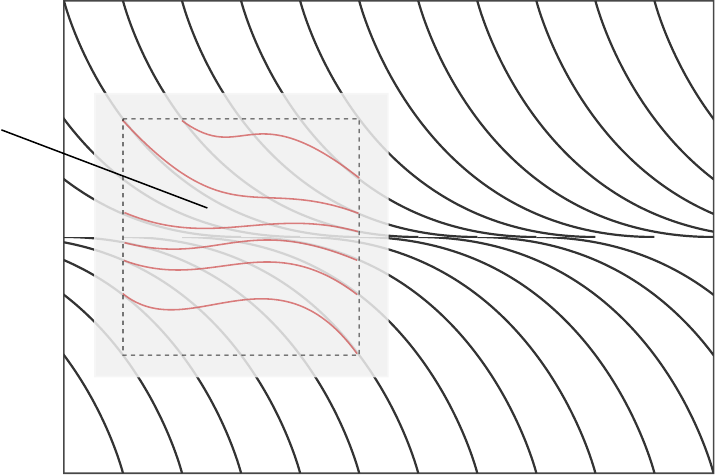
    \caption{Two mostly-transverse foliations. $\G'$ is drawn above in red when it differs from $\G$.}
    \label{fig:transversalizing2}
\end{figure}

We now show that two mostly-transverse foliations may be slightly perturbed to be made transverse.

\begin{lem}[Making boundary foliations transverse] \label{lem:mostlytrans}
    Let $\G$, $\G'$ be two mostly-transverse $C^0$-foliations on $\mathbb{T}^2$. Then, for every $\varepsilon>0$, there exist $C^0$-foliations $\widetilde{\G}$ and $\widetilde{\G}'$ such that
    \begin{itemize}
        \item $\widetilde{\G}$ and $\widetilde{\G}'$ are isotopic, via an $\varepsilon$-$C^{0}$-small isotopy, to $\G$ and $\G'$, respectively,
        \item $\widetilde{\G}$ is positively transverse to $\widetilde{\G}'$ everywhere.
    \end{itemize}
\end{lem}

\begin{proof}
    We proceed in two steps.
    \begin{itemize}[leftmargin=*]
        \item \textit{Step 1: pushing the nonsmoothness to the transverse region.} The goal is to modify $\G$ and $\G'$ by arbitrarily small $C^0$-isotopies so that the resulting foliations $\widehat{\G}$ and $\widehat{\G}'$ satisfy the same assumptions as $\G$ and $\G'$, and are furthermore smooth on a neighborhood of $\T^2 \setminus W$. 

        Let us introduce some useful terminology. A \emph{standard box} adapted to $\G$ and $\G'$ is a subset $B \subset \T^2$ with $C^1$ coordinates $B \cong [-1,1]^2_{x,y}$ such that the following properties are satisfied:
            \begin{enumerate}
                \item[(a)] On $B$, $\G$ and $\G'$ are graphical, i.e., transverse to $\partial_y$,
                \item[(b)] The intervals $[-1,1] \times \{\pm1/2\}$ are contained in leaves of $\G$,
                \item[(c)] The interval $\{-1\} \times [-1,1]$ is contained in $V$, and the interval $\{+1\} \times [-1,1]$ is contained in $W$,
                \item[(d)] All the leaves of $\G'$ intersecting $[-1,1] \times [-1/2, 1/2]$ remain in $\mathrm{int}B$,
                \item[(e)] There are finitely many closed, pairwise disjoint intervals $I_1, \dots, I_k \subset [-1,1]$ such that on each $I_j \times [-1,1]$, $\G$ and $\G'$ coincide and they are smooth near $\partial I_j \times [-1,1]$. Moreover, the complement of $\bigcup_j I_j \times [-1,1]$ in $B$ is contained in $V$.
            \end{enumerate}
        For $B$ such a standard box, we will write $B' \coloneqq [-1,1] \times [-1/2, 1/2]$. The last item is useful to ``isolate'' regions where $\G$ and $\G'$ differ, and will allow us to first modify them in regions where they coincide.

        By hypothesis, every point $p \in \T^2 \setminus V$ is contained in a set $B'$ as above. Indeed, let us consider the leaf of $\G$ passing through $p$, which intersects $W$ on both sides of $p$. In particular, there is a segment $\ell \cong [-1,1]$ in that leaf passing through $p$ such that $\partial \ell \subset W$. We may now trim $\ell$ so that it only intersects $W$ near its boundary, and subdivide $\ell$ into finitely many intervals which are either entirely contained in $V$, or contained in the region where $\G$ and $\G'$ coincide, and $\G$ and $\G'$ are smooth near the boundaries of the latter intervals. We may then thicken $\ell$ and consider a closed tubular neighborhood $B$ bounded by two leaves of $\G$ (which are $C^1$), and find a much smaller tubular neighborhood $B' \subset B$, also bounded by leaves of $\G$, such that item (d) is also satisfied. Choosing these neighborhoods small enough guarantees that condition (e) is satisfied as well.
        
        By compactness, we can cover $\T^2 \setminus V$ with finitely many sets of the form $B'$. Let us now consider one such pair of boxes $B' \subset B$. There is a $\delta > 0$ such that $(1-\delta, 1] \times [-1,1] \subset W$. We then modify $\G$ (and $\G'$) in each set $I_j \times [-1/2,1/2]$, by replacing it with a \emph{smooth} graphical foliation which is moreover arbitrarily $C^0$-close to $\G$. This modification is not an isotopy of $\G$ or $\G'$ a priori, since it modifies the parallel transport maps $\{-1\} \times [-1,1] \rightarrow \{1\} \times [-1,1]$ induced by following the leaves. However, this can be corrected on the set $(1-\delta, 1] \times [-1,1]$ by modifying $\G$ and $\G'$ there in a graphical way, so that the parallel transport maps match those of $\G$ and $\G'$. Moreover, this can be achieved so that $\G$ and $\G'$ remain transverse in $(1-\delta, 1] \times [-1,1]$, provided that the previous modifications are small enough. Since the overall modification is graphical, the new foliations are $C^0$-isotopic to the original ones, via arbitrarily small $C^0$-isotopies. 

        Note that this modification also affects $\G$ and $\G'$ in the standard boxes in the finite cover which intersect $B$. Importantly, our modification only \emph{enlarges} the subset of $\T^2 \setminus W$ on which $\G$ and $\G'$ are smooth. Moreover, for sufficiently small modifications, the conditions (a)--(e) are still satisfied for those nearby boxes, after the necessary adjustments. Therefore, after finitely many such modifications, the resulting foliations are smooth on a neighborhood of $\T^2 \setminus W$.
        
        \item \textit{Step 2: straightening the smooth part.} By the previous step, we may assume that $\G$ and $\G'$ are smooth on a neighborhood of $\T^2 \setminus W$. Therefore, there exists a smooth nonvanishing vector field $X$ such that the (oriented) leaves of $\G$ are tangent to $X$ on $\T^2 \setminus W$. By assumption, the flow line of $X$ passing through a point $p \in \T^2 \setminus V$ reaches a point in $W$ in positive and negative time. We may construct a smooth function $f : \T^2 \rightarrow \R$ satisfying $X \cdot f \geq 1$ on $\T^2 \setminus W$. It suffices to construct functions $f_i$'s such that $X \cdot f_i \geq 0$ and $X \cdot f_i > 0$ near a given flow line of $X$ in $\T^2 \setminus W$, and define $f$ as a scaling of a suitable finite sum of those. For $C > 0$, to be chosen large enough later, we set $g \coloneqq e^{Cf}$, so that $X \cdot g \geq  C g$ on $\T^2 \setminus W$.

        We then consider a smooth vector field $Z$ positively transverse to $X$, and we define a smooth isotopy of $\T^2$ by
        $$\phi_t(x) \coloneqq \varphi_Z^{tg(x)}(x).$$
        We have
        $$d\phi_t(X) = X + t \left( X \cdot g + g[X, Z]\right) + O(t^2),$$
        hence for $C$ large enough compared with $[X,Z]$, and for any $t> 0$ small enough, $\phi_t(\G)$ is positively transverse to $\G'$ on $\T^2 \setminus W$. For $t$ even smaller, $\phi_t(\G)$ remains positively transverse to $\G'$ on $W$ as well. \qedhere
    \end{itemize}
\end{proof}

    \subsection{(In)stability of slopes} \label{sec:instability_of_slopes}

Later in this article, we will perturb foliations on $3$-manifolds to contact structures. We will need to analyze how the induced foliations on $\partial M$ change under these perturbations. In this section, we give criteria to determine when the boundary slope of a foliation is stable or unstable.

\paragraph{Instability results.}

We show some elementary instability phenomena for $C^0$-foliations on $\T^2$ under tangential perturbations.

\begin{lem}[Instability of $\natural$ or $\sharp$-type foliations] \label{lem:phaselocking}
    Suppose $\G_1$ and $\G_2$ are two Reebless $C^0$-foliations of $\T^2$ with finite slopes, and $\G_1$ has pointwise greater slope than $\G_2$. Then 
    $$\slope(\G_1) \geq \slope(\G_2).$$
    If equality holds, then $\slope(\G_i)\in \Q$ and $\G_2$ has a $\flat$ annulus (i.e., $\G_2$ has $\flat$ or $\dagger$ type), and $\G_1$ has a $\sharp$ annulus (i.e., $\G_1$ has $\sharp$ or $\dagger$ type).
\end{lem}

\begin{proof}
    The non-strict inequality is immediate from the definition of $\slope$. For the strict inequality, let $\Lambda_1$ and $\Lambda_2$ be the limit sets of $\G_1$ and $\G_2$. They both admit measures of full support by \zcref{prop:poincare}. The difference in slope between $\Lambda_1$ and $\Lambda_2$ may be computed by an intersection pairing between their invariant measures. Since $\Lambda_1$ and $\Lambda_2$ have only intersections of positive sign, $\slope(\Lambda_1)=\slope(\Lambda_2)$ if and only if $\Lambda_1$ and $\Lambda_2$ are disjoint.

    If $\slope(\Lambda_1)=\slope(\Lambda_2)\in \R\setminus \Q$, then each leaf of $\Lambda_1$ lies in a gut of $\Lambda_2$. The guts of $\Lambda_2$ are ideal bigons, so leaves of $\Lambda_1$ must be asymptotic into the cusps of these bigons. But then there is no lower bound to the angle between leaves of $\Lambda_1$ and $\Lambda_2$ so $\G_1$ and $\G_2$ must have a tangency, contradicting our initial assumption.
    
    If $\slope(\Lambda_1)=\slope(\Lambda_2)\in \Q$, then each leaf of $\Lambda_1$ lies in a component of $\T^2\setminus \Lambda_2$ which is either a sharp or flat annulus. A sharp annulus does not admit a positive closed transversal, so $\G_2$ must have at least one $\flat$ annulus. Similarly, $\G_1$ must have at least one $\sharp$ annulus.
\end{proof}

\begin{cor}\label{cor:fineslope_inequality_converse}
    The reverse direction of \zcref{prop:fineslope_inequality} holds.
\end{cor}

\begin{lem}[Instability of annuli without $\flat$ spiraling] \label{lem:instability_sharpspiraling}
Let $\G$ be a foliation of an annulus $S^1 \times [0,1]$ which is positively transverse to $S^1 \times \{0\}$ and $S^1 \times \{1\}$. Suppose $\G$ has no closed leaf of $\flat$-type (resp.~$\sharp$-type) spiraling on either side. If $\G'$ is another foliation of the annulus which is similarly transverse to the boundary and has pointwise higher (resp.~lower) slope than $\G$, then $\G'$ has no closed leaf of degree $+1$ (resp.~$-1$) in the annulus. See \zcref{fig:instability}.
\end{lem}

\begin{proof}
    We consider the case where $\G$ has no closed leaf of $\flat$-type spiraling; the other case is the mirror image. Suppose for the sake of contradiction that $\G'$ admits a closed leaf $L$ with degree $+1$. Then $L$, together with $S^1\times \{0\}$, bounds an annulus into which $\G$ enters. Therefore, $\G$ has both $\sharp$ type and $\flat$ type spiraling, by the Poincar\'{e}--Bendixson theorem.
\end{proof}

\begin{figure}[ht]
    \centering
    \def\svgwidth{0.4\textwidth}
    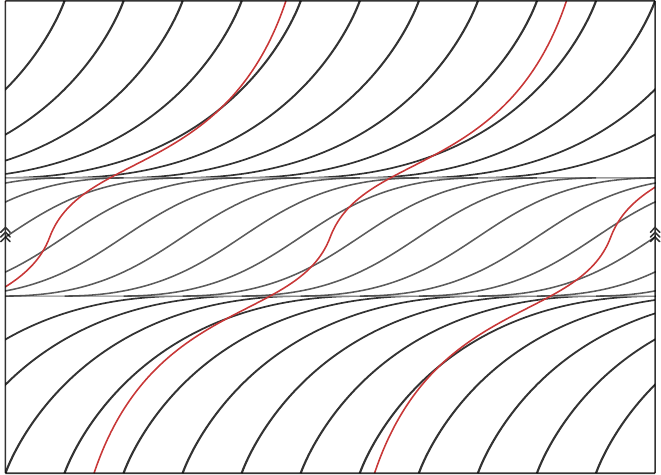
    \caption{$\G$ shown in black, and a few leaves of $\G'$ shown in red.}
    \label{fig:instability}
\end{figure}

\begin{lem}[Resolution of $\aleph$] \label{lem:resolution_of_aleph}
    Let $\G$ be a $\Phi$-horizontal $C^0$-foliation of $\aleph$ type, or more generally a foliation with a Reeb annulus and no $+\infty^\sharp$-spiraling (resp. no~$-\infty^\flat$-spiraling) curves. Then if $\G'$ is another $\Phi$-horizontal $C^0$-foliation which has pointwise higher slope (resp.~pointwise lower slope) than $\G$, then $\G'$ is Reebless and has slope $+\infty$ (resp.~$-\infty$).
\end{lem}

\begin{figure}[ht]
    \centering
    \def\svgwidth{0.7\textwidth}
    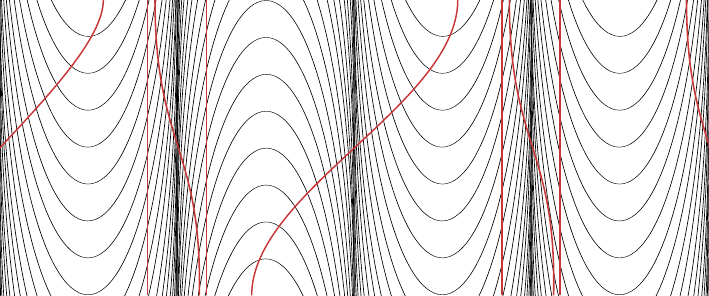
    \caption{$\G$ shown in black, and a few leaves of $\G'$ shown in red.}
    \label{fig:aleph_resolution}
\end{figure}

\begin{proof}
    We treat only the case of pointwise higher slope since the other case is similar. $\Phi$ must have a closed orbit in each Reeb annulus of $\G$; make a choice of one such orbit for each Reeb annulus and cut $\T^2$ along these closed curves. The resulting annuli $A_1,\dots,A_{2k}$ alternate between two types: $A_{2i+1}$ has only $\sharp$-type spiraling and has closed leaves of slope $-\infty$, while $A_{2i}$ has closed leaves of slope $+\infty$. Inside $A_{2i+1}$, $\G'$ does not have a closed leaf of slope $-\infty$ by \zcref{lem:instability_sharpspiraling}. Inside $A_{2i}$, $\G'$ has only closed leaves of slope $+\infty$; one way to see this is to glue up $A_{2i}$ to a torus and apply (the proof of) \zcref{lem:phaselocking} to the foliations induced by $\G$ and $\G'$ on the resulting torus. Therefore, $\G'$ is Reebless and has a closed curve of slope $+\infty$.
\end{proof}

\paragraph{Stability results.}

Unlike our instability results, the following stability results in the tangential topology will always require a preparatory $C^0$-small isotopy. This is due to the lack of unique integrability for the tangent line field of a $C^0$-foliation. The key role of this isotopy is to guarantee the existence of suitable smooth transversals.

\begin{lem}\label{lem:monotone_open}
    For any $s\in \Q$, the property of admitting a (positively or negatively) monotone curve of slope $s\in \Q\cup\{\pm\infty\}$ is open in the $C^0$ topology on foliations of $\T^2$. The property of admitting a smooth transversal of slope $s\in \Q \cup \{\pm \infty\}$ is open in the tangential topology on foliations of $\T^2$.
\end{lem}

\begin{proof}
    Suppose $\G$ is a foliation with a monotone curve $\gamma$. Decompose $\gamma$ into transverse subarcs $\gamma^\perp_1,\dots,\gamma^\perp_k$ and subarcs $\gamma^\parallel_1,\dots,\gamma^\parallel_k$ contained in leaves of $\G$. Extend each $\gamma_i^\perp$ slightly so that $\gamma_i^\parallel$ hits the interior of both $\gamma_i^\perp$ and $\gamma_{i+1}^\perp$. See \zcref{fig:perturbing_monotone}. If $\G'$ is sufficiently close to $\G$ in the $C^0$ topology, then $\gamma_i^\perp$ remains a transversal. Furthermore, $\G'$ also has a leaf connecting a point near the top of $\gamma_i^\perp$ to a point near the bottom of $\gamma_{i+1}^\perp$. Concatenating these curves together gives a new monotone curve for $\G'$ in the same homotopy class as $\gamma$.

    The corresponding statement in the tangential topology is easier: a transversal $\gamma$ to $\G$ remains transverse to all tangentially close foliations $\G'$.
\end{proof}

\begin{figure}[ht]
    \centering
    \def\svgwidth{0.4\textwidth}
    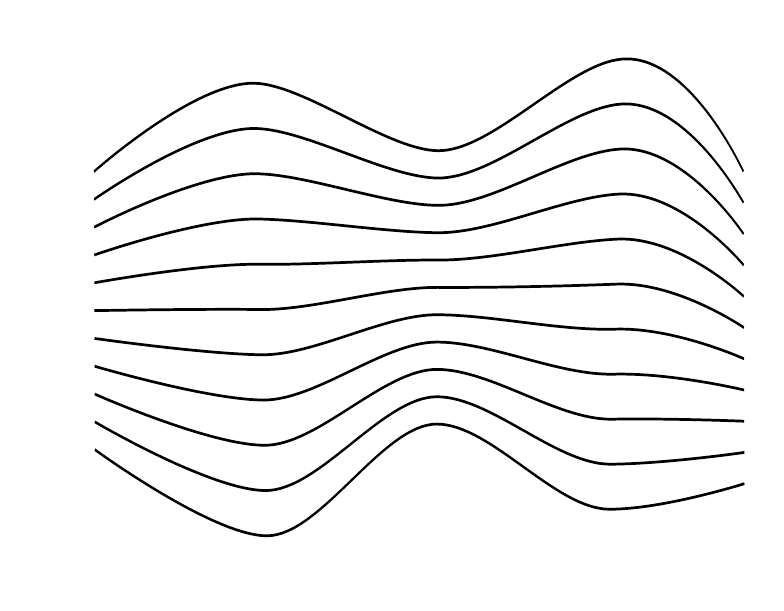
    \caption{Tangent and transverse subarcs of $\gamma$.}
    \label{fig:perturbing_monotone}
\end{figure}

\begin{lem} \label{lem:transversholo}
    Let $\gamma$ be a closed leaf of a foliation $\G$ on $\T^2$ with $\flat$ (resp. $\sharp$) spiraling on at least one side, and with slope $s$. Then $\G$ admits a positively (resp.~negatively) monotone curve of slope $s$. After an arbitrarily $C^0$-small isotopy, $\G$ admits a positively (resp. negatively) transverse curve of slope $s$.
\end{lem}

\begin{proof}
    Assume for instance that $\gamma$ has $\flat$ spiraling on its positive side. We consider $C^1$ coordinates $(x,y) \in S^1 \times [0, \delta)$ near $\gamma \cong S^1 \times \{0\}$, for some $\delta > 0$, in which $\G$ is graphical. In particular, it is transverse to the arc $c \coloneqq \{0\} \times [0, \delta)$. After an arbitrarily $C^0$-small isotopy of $\G$ near $c$, we may assume that $\G$ is smooth in a neighborhood $[-\varepsilon, \varepsilon] \times [0, \delta)$ of $c$ for some $\varepsilon > 0$, using~\cite[Section 3]{B16}. The leaf of $\G$ passing through a point $(0, \delta')$ for $0 < \delta' <\delta$ sufficiently small stays in $S^1 \times [0, \delta)$ and comes back to $c$ at a point $(0, \eta)$ for $0 < \eta < \delta'$ by our holonomy assumption. Therefore, it is easy to construct a $C^1$ closed curve $\gamma'$ which is positively monotone to $\G$ near $\gamma$: follow such a leaf from $\{\varepsilon\} \times [0, \delta)$ to $\{-\varepsilon\} \times [0, \delta)$ and close it up by adding a smooth arc positively transverse to $\G$ on $[-\varepsilon, \varepsilon] \times [0, \delta)$. We may arrange that the resulting $C^1$ curve is arbitrarily $C^0$-close to $\gamma$. We then apply \zcref{lem:transversalizing} to slightly modify $\G$ and $\gamma'$ to find the desired transversal; see \zcref{fig:transversalizing}.
\end{proof}

\begin{figure}[ht]
    \centering
    \begin{subfigure}{0.45\linewidth}
    \centering
    \def\svgwidth{\textwidth}
    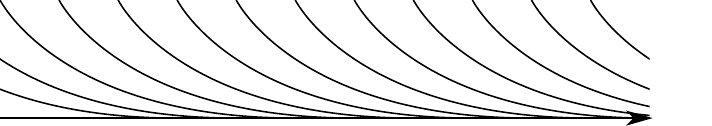
    \caption{}
    \label{fig:mostlytransversecurve}
    \end{subfigure}
    \hfill
    \begin{subfigure}{0.45\linewidth}
    \centering
    \def\svgwidth{\textwidth}
    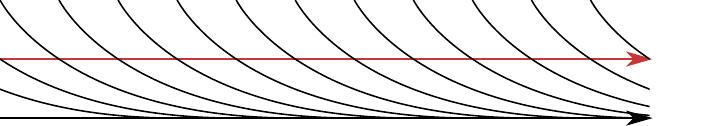
    \caption{}
    \end{subfigure}
    \caption{$\gamma'$ before and after the application of \zcref{lem:transversalizing}.}
    \label{fig:transversalizing}
\end{figure}

\begin{cor}\label{cor:annulus_transversals}
    Up to an arbitrarily $C^0$-small isotopy, every $\sharp$ or $\flat$ annulus admits a closed transversal.
\end{cor}

\begin{cor}[Stability of $\sharp$ and $\flat$ annuli]
Let $\G$ be a Reebless $C^0$-foliation on $\T^2$ with a $\sharp$ (resp.~$\flat$) annulus. Then there is a $C^0$-small isotopy of $\G$ to a new foliation $\G'$ whose coarse slope cannot be decreased (resp.~increased) by tangentially $C^0$-small perturbations.
\end{cor}

\begin{proof}
    Call the $\sharp$ annulus $A$. Apply \zcref{lem:transversholo} to a component of $\partial A$ to obtain a $C^0$-close foliation $\G'$ with a closed transversal $\gamma \subset A$. Any tangentially $C^0$ small perturbation of $\G'$ remains positively transverse to $\gamma$, and therefore has slope at least $\slope(\G)$.
\end{proof}

\begin{cor}[Stability of $\dagger$ foliations]
Let $\G$ be a $\dagger$-type foliation of $\T^2$. Then there is a $C^0$-small isotopy of $\G$ to a new foliation $\G'$ whose coarse slope is stable to tangentially $C^0$-small perturbations.
\end{cor}

Let $\G$ be a $C^0$-foliation on $\T^2$ which is horizontal (i.e., graphical) with respect to coordinates $\T^2 \cong S^1_x \times S^1_y$ and which has positive (resp.~negative) finite slope as defined above. We say that $\G$ is \textbf{stably positive (resp.~negative)} (with respect to these coordinates) if any $C^0$-foliation $\G'$ with $T\G'$ sufficiently close to $T\G$ also has positive (resp.~negative) slope. While slope is continuous in the $C^0$ topology, it is not continuous in the tangential topology. Because of the lack of unique integrability of $T\G$, a foliation with positive slope might not be stably positive, as $T\G$ itself might be tangent to a foliation with \emph{negative} slope! However, we have:

\begin{lem}[Stability of positive slope] \label{lem:stablypositive}
    Let $\G$ be a foliation on $\T^2$ which is graphical with respect to coordinates $\T^2 \cong S^1_x \times S^1_y$ and has $\slope(\G)>0$. Then $\G$ is isotopic, via an arbitrarily small $C^0$-isotopy, to a stably positive $C^0$-foliation.
\end{lem}

\begin{proof}
    Our strategy is to modify $\G$ in order to make it positively transverse to another graphical foliation $\G'$ which also satisfies $\slope(\G') > 0$. That way, any foliation sufficiently tangentially close to $\G$ will also be positively transverse to $\G'$, and will have slope at least $\slope(\G') > 0$.
    
    First, smooth $\G$ in a $2\delta$-neighborhood of the transversal $\{x=0\}$ by an arbitrarily small $C^0$-isotopy, for some $\delta > 0$, using~\cite[Section 3]{B16}. After a smooth change of coordinates, we may further assume that $\G$ is tangent to $\partial_x$ on that neighborhood. Let 
    $$V \coloneqq (-\delta, \delta) \times S^1, \qquad W \coloneqq (-\delta/2, \delta/2)\times S^1.$$ 
    Let $\varepsilon > 0$ be small enough, and $\G'$ be another horizontal foliation such that 
    \begin{itemize}
        \item $\G'$ agrees with $\G$ outside of $W$,
        \item The holonomy of $\G'$ from $\{-\delta\}\times S^1$ to $\{\delta\}\times S^1$ is (a rotation of angle) $-\varepsilon$, and
        \item $\G$ is positively transverse to $\G'$ in $W$.
    \end{itemize}
    For $\varepsilon$ sufficiently small, we have $\slope(\G') > 0$. Moreover, $\G$ is positively mostly-transverse to $\G'$, with $V$ and $W$ as in \zcref{def:mostlytrans}. By \zcref{lem:mostlytrans}, we may apply an arbitrarily small $C^0$-isotopy to $\G$ and $\G'$ to obtain foliations $\widetilde \G$ and $\widetilde \G'$ such that $\widetilde \G$ is positively transverse to $\widetilde \G'$, and $\slope(\widetilde \G')> 0$. Hence, $\widetilde \G$ has stably positive slope.
\end{proof}

        \subsection{Limits of foliations}

The following lemma justifies the definition of $\aleph$ type foliations and will be useful in \zcref{sec:connectedness}. The related convergence result that we prove in \zcref{sec:alephclosure} below will rely on different techniques. We will also state a generalization of this result for \emph{extended fine slopes} in \zcref{sec:beyondaleph}.

\begin{lem} \label{lem:slconv}
    Suppose we have a sequence of $\Phi$-horizontal $C^0$-foliations $\G_k$, $k \geq 0$, on $\T^2$, converging to a $\Phi$-horizontal foliation $\G$ in the $C^0$ topology. Furthermore, assume that each $\G_k$ is Reebless with slope $s_k \in \R$, and that $(s_k)_k$ converges to some slope $s_\infty \in \R$.
    Then $\G$ is admissible, and either
    \begin{itemize}
        \item $\G$ is of $\aleph$ type, or
        \item $\G$ has no Reeb annuli and has slope $s = s_\infty$.
    \end{itemize}

    Moreover, in the second case, if $(s_k)_k$ is strictly increasing (resp.~decreasing), then $\fs(\G) \in \{s^\natural, s^\flat\}$ (resp.~$\fs(\G) \in \{s^\natural, s^\sharp\}$).
\end{lem}

\begin{proof}
    Let  $s \in [-\infty, s_\infty) \cap (\Q \cup \{-\infty\})$. By \zcref{lem:monotone_open}, the property of not admitting a positive monotone curve of slope $s$ is closed in the $C^0$ topology. This property is satisfied for $\G_k$, $k$ sufficiently large by \zcref{lem:monotone_existence}. Therefore, it is also satisfied by $\G$. Similarly, for $s\in (s_\infty, +\infty]\cap (\Q \cup \{+\infty\})$, $\G$ does not admit a negatively monotone curve of slope $s$.

    If $\G$ is a suspension foliation, then this nonexistence of the particular monotone curves above combined with \zcref{lem:monotone_existence} implies that $\slope(\G)=s_\infty$. If on the other hand $\G$ has a Reeb annulus, then it is of $\aleph$ type since $\G$ does not admit a negatively monotone curve of slope $+\infty$ or a positively monotone curve of slope $-\infty$. If $s_k$ is strictly increasing, then none of the $\G_k$ admit a negatively monotone curve of slope $s_\infty$. Hence neither does $\G$, and it follows from \zcref{lem:transversholo} that $\G$ must have fine slope $s^\natural$ or $s^\flat$.
\end{proof}

    \section{Contact structures} \label{sec:contact}

In this section, we collect elementary facts about contact structures that will be useful later. We refer to~\cite{Etnyre2003, Geiges2008, massot.TopologicalMethods3dimensional} for excellent introductions to contact geometry.

A three-dimensional \textbf{positive (resp.~negative) contact structure}, usually denoted $\xi$, is a cooriented $2$-plane field on an oriented 3-manifold $M$ which is ``completely nonintegrable''. Concretely, it may be expressed as $\ker \alpha $ for some smooth 1-form $\alpha$ satisfying $\alpha \wedge d\alpha > 0$ (resp.~$\alpha \wedge d\alpha < 0$), where the sign involves the orientation on $M$. We make the following convention: for a \emph{negative} contact structure, the correspondence between its orientation and coorientation is obtained from the orientation on $M$ induced by the contact structure itself, namely, the \emph{opposite} of the chosen orientation on $M$. Moreover, a flow or vector field is \emph{positively} transverse to it if it is positively transverse with respect to its \emph{orientation} and not its coorientation. In other words, a vector field $X$ is positively transverse to a negative contact structure $\xi_-$ with contact form $\alpha_-$ if $\alpha_-(X) < 0$.

If $\partial M \neq \varnothing$, we will typically require that $\xi$ is transverse to $\partial M$. 

A contact structure $\xi$ on $M$ is \textbf{overtwisted} if there exists an embedded disk in $\interior M$ which is transverse to $\xi$, except along its boundary and its center where they are tangent. Equivalently, there is a standard model for overtwisted disks, see for instance~\cite[Example 2.1.6]{Geiges2008}. A contact structure that is not overtwisted is called \textbf{tight}.

Importantly, the contact condition is \emph{open} in the $C^1$-topology, and contact structures are \emph{stable}, in the sense of Gray stability: if $(\xi_t)_{t\in[0,1]}$ is a (smooth) path of contact structures which is constant near $\partial M$, then there exists an isotopy $(\phi_t)_{t\in [0,1]}$ of $M$ supported away from $\partial M$ satisfying $\phi_t^*\xi_t = \xi_0$ for all $t \in [0,1]$. However, the situation is more complicated when we allow modifications at the boundary.

By Darboux's theorem, every (positive) contact structure is locally of the form $\ker\big(dz+xdy \big)$. If $T \cong \T^2$ is a boundary component of $M$ and $\xi \pitchfork T$, we may also find a normal form of $\xi$ near $T$:

\begin{lem}
   There exists a neighborhood $U$ of $T$ and coordinates $U \cong \T^2_{x,y} \times (-1,0]_r$ such that the positive contact structure $\xi$ is defined by a $1$-form $\alpha_\partial$ of the form
    \begin{align} \label{eq:alphabdry}
        \alpha_\partial = \lambda + r \mu
    \end{align}
    where $\lambda$ and $\mu$ are $1$-forms on $\T^2$ satisfying $\mu \wedge \lambda > 0$.
\end{lem}

\begin{proof}[Sketch of proof]
    We may first find a smooth vector field $X$ transverse to $T$ and tangent to $\xi$, and use its flow to define coordinates on a neighborhood $U\cong \T^2_{x,y} \times (-1,0]_r$ so that $\xi$ is tangent to $\partial_r$ in this neighborhood. Therefore, it is defined by a contact form $\alpha$ of the form $\alpha = \lambda_r$, where $(\lambda_r)_{r \in (-1,0]}$ is a family of $1$-forms on $\T^2$ satisfying $\partial_r \lambda_r \wedge \lambda_r > 0$ (this is simply the contact condition). We may write 
    $$\lambda_r = \lambda_0 + r \underset{\eqqcolon \mu}{\underbrace{\partial_r\vert_{r=0}\lambda_r}} + r^2 \delta_r,$$
    for some family of $1$-forms $\delta_r$. For $r$ close to $0$, the contact condition is equivalent to $\mu \wedge \lambda_0 > 0$. We may now find new coordinates near $T$ in which $\xi$ is of the desired form, either by using a suitable reparametrization of the flow of $X$, or by using a version of Gray stability relative to $T$.
\end{proof}

We now use this normal form to modify the slope of $\xi$ near $T$.

\begin{lem}[Trim/twist]\label{lem:trim_twist}
    Let $\xi$ and $T$ be as above.
    \begin{enumerate}
        \item (Trim) There exists a positive contact structure $\xi'$ which is $C^\infty$-close to $\xi$ and such that $\partial \xi \prec \partial \xi'$ along $T$.
        \item (Twist) Suppose $\xi$ is $\Phi$-horizontal for some smooth nonsingular flow $\Phi$ on $M$. For any smooth $\Phi$-horizontal foliation $\G$ on $\partial M$ satisfying $\G \prec \partial \xi$, there is a $\Phi$-horizontal positive contact structure $\xi'$ with $\partial \xi' = \G$.
    \end{enumerate}
    In both cases, we can arrange that $\xi'$ coincides with $\xi$ away from an arbitrarily small neighborhood of $T$. 
\end{lem}

\begin{proof}
Let $U \cong \T^2 \times (-1,0]$ be a neighborhood of $T$, together with coordinates in which $\xi$ is defined by a $1$-form $\alpha_\partial$ as in~\eqref{eq:alphabdry}.

For the first item, we choose a smooth cutoff function $\tau : (-1,0] \rightarrow [0,1]$ such that $\tau' \geq 0$, $\tau \equiv 1$ near $r=0$ and $\tau(r) = 0$ for $r \leq -1/2$. Then for $\varepsilon > 0$ small enough, the $1$-form 
$$\lambda + \big(r - \varepsilon \tau(r)\big) \mu$$
is a contact form whose contact structure $\xi'$ coincides with $\xi$ for $r \leq -1/2$ and satisfies $\partial \xi \prec \partial \xi'$ along $\T^2$.

For the second item, we may first arrange that $\mu$ vanishes along $\Phi$. Indeed, let $\widetilde{\mu}$ be a $1$-form on $\T^2$ vanishing along $\Phi$ and such that $\widetilde{\mu} \wedge \lambda > 0$. For a constant $C > 0$ large enough, we consider the $1$-form
$$\lambda+ r \left\{ (1-\tau(r)) \mu + \tau(r) C\widetilde{\mu}\right\}.$$
It is easy to check that for $C$ sufficiently large, this defines a $\Phi$-horizontal contact structure on $\T^2 \times (-1, 0]$ which coincides with $\lambda + r \mu$ on $\T^2 \times (-1, 1/2]$ and with $\lambda + r C\widetilde{\mu}$ near $\T^2 \times \{0\}$.

We now assume that $\mu$ vanished along $\Phi$. Then, we consider a $1$-form $\beta$ on $\T^2$ defining $\G$. After rescaling, we might decompose it as $\beta = \lambda + f\mu$ for some function $f : \T^2 \rightarrow \R_{>0}$, using the condition $\G \prec \partial \xi$ and the condition on $\mu$. After possibly shrinking the neighborhood $U$, we further assume that both $\lambda$ and $\beta$ evaluate positively along the flow lines of $\Phi$ in $U$. We now consider the $1$-form 
$$\lambda + \big(r + \tau(r) f\big) \mu,$$
where $\tau$ is as before. It is immediate to check that this $1$-form defines a contact structure $\xi'$ transverse to $\Phi$ (by convexity) and with $\partial \xi' =\G$.
\end{proof}

\begin{rem}
    Notice that the first modification can be obtained by ``trimming'' $M$ along $T$, i.e., removing a small tubular neighborhood of $T$, while the second item may be obtained by ``extending'' $M$ along $T$ by attaching a suitable tubular neighborhood along it and twisting $\xi$ along the extension of $\partial_r$.
\end{rem}

For the purpose of studying foliations, we will consider \emph{pairs} of contact structures with opposite signs. The relationship between foliations and contact pairs was originally studied in~\cite{CF11} and the theory was further developed by the first author in~\cite{M24}, in the setting of \emph{closed} manifolds.

A \textbf{positive contact pair} on $M$ is a pair of (cooriented) contact structures $(\xi_-, \xi_+)$ such that $\xi_-$ and $\xi_+$ have no negative tangencies, or equivalently, if there is a smooth vector field positively transverse to both contact structures. When $\partial M \neq \varnothing$, we say that $(\xi_-, \xi_+)$ is \textbf{adapted to the boundary} if $\partial \xi_+$ is positively transverse to $\partial \xi_-$ (see \zcref{fig:contactplanes}). 

A contact pair $(\xi_-,\xi_+)$ is \textbf{taut} (\emph{``fortement tendue''} in~\cite{CF11}) if for every point $p\in M$, there exists a closed loop positively transverse to both $\xi_-$ and $\xi_+$ and passing through $p$. See~\cite{M24} for other characterizations of tautness, which also hold for manifolds with boundary. Moreover, we have:

\begin{lem} \label{lem:universallytight}
    If $(\xi_-, \xi_+)$ is a taut contact pair on a \textbf{closed} $3$-manifold, then $\xi_-$ and $\xi_+$ are both universally tight.
\end{lem}

A proof is sketched in~\cite[Section 5]{CH20}, see also~\cite[Proposition 3.2]{M24}.

    \section{Flows}

In this article, a \textbf{flow} is an $\R$-action on $M$ without stationary points whose orbits are $C^1$ and tangent to a nowhere vanishing $C^0$ vector field. In particular, flows are tangent to $\partial M$.

    \subsection{Blowup and blowdown}

\begin{defn}\label{def:flowblowdown}
    Let $\Phi$ be a smooth flow on a compact $3$-manifold $M$. A \textbf{blowup} of $(M, \Phi)$ along a closed orbit $\gamma \subset \mathrm{int} M$ of $\Phi$ is a pair $(M^\circ,\Phi^\circ)$, where $M^\circ$ is a compact $3$-manifold with torus boundary and $\Phi^\circ$ is a flow on $M^\circ$ tangent to the boundary, together with a map $\pi:M^\circ \to M$ such that 
        \begin{enumerate}
        \item There exist smooth local coordinates $(r, \theta, \phi)$ around $\partial M^\circ$ (where $r$ parametrizes the normal direction, $\phi$ parametrizes the longitudinal direction, $\theta$ parametrizes the meridional direction) and cylindrical coordinates $(r,\theta,\phi)$ in a neighborhood of $\gamma$ (where $\gamma$ lies along $\{r=0\}$), such that $\pi$ has the form 
        $$\pi(r,\theta,\phi) = (f(r),\theta,\phi)$$
        for some smooth increasing function $f:[0,1]\to [0,1]$ satisfying $f(0)=0$. In particular, $\pi$ sends $\partial M^\circ$ to $\gamma$, collapsing leaves of a linear foliation on $\partial M^\circ$ to points.
        \item $\pi$ restricts to a smooth orbit equivalence between $(\interior M^\circ, \Phi^\circ\vert_{\interior M^\circ})$ and $(M \setminus \gamma, \Phi\vert_{M\setminus \gamma})$.
        \end{enumerate}
    We say that $(M, \Phi)$ is a \textbf{blowdown} of $(M^\circ, \Phi^\circ)$ along the slope $\phi=const$. The induced flow $\Phi^\circ_{\partial M^\circ}$ for the blowup is Reebless and has slope $+\infty$ in the induced framing; importantly, it depends on the choice of coordinates near $\gamma$.
\end{defn}

This definition readily extends to finite collections of orbits of $\Phi$.

\begin{rem}
    We emphasize that the blowup is generally \emph{not} unique. The blowdown is unique up to topological conjugacy, but may have different smooth models. This nonuniqueness does not affect our results; they hold for any choice of blowup and blowdown.
\end{rem}

$\Phi$-horizontal foliations on $M$ can be lifted to blowups along any closed orbit. One can simply put such a foliation in standard form in the local coordinates above by a $\Phi$-horizontal isotopy. On the other hand, a foliation may be blown down along slope $\bm{s}$ \emph{only if} it has linear boundary foliations, i.e., fine boundary multislope $\bm{s}^\natural$. 

Finally, a contact structure can be blown down provided that its boundary multislope exceeds $\bm{s}$:

\begin{lem}[Filling contact structures]\label{lem:fillingcontact}
    Suppose $\Phi$ admits a transverse positive (resp.~negative) contact structure $\xi_+$ (resp.~$\xi_-$) of boundary slope $\bm{s}$. For any rational multislope $\bm{s'} < \bm{s}$ (resp.~$\bm{s}'> \bm{s}$), any blowdown $\Phi(\bm{s'})$ on the Dehn filling $M(\bm{s'})$ admits a transverse positive (resp.~negative) contact structure.

    Conversely, if $\Phi(\bm{s'})$ admits a transverse positive (resp.~negative) contact structure, then for some $\bm{s}' < \bm{s}$ (resp.~$\bm{s}' > \bm{s}$) $\Phi$ admits a transverse positive (resp.~negative) contact structure of slope $\bm{s}$.
\end{lem}

\begin{proof} Recall the local model of blowdown from \zcref{def:flowblowdown}: there exists a tubular neighborhood $N$ of $\partial_i M$ smoothly parametrized as $\mathbb{T}^2_{\theta,\phi} \times [0,1]_r$ where $\partial_i M$ lies along $\{r=0\}$, the slope $s'_i$ corresponds to the line $\{\phi=0\}$, and the blowdown $\pi:(M,\Phi) \to (M(\bm{s}), \Phi(\bm{s}))$ has the form $$\pi(r,\theta,\phi) = (f(r),\theta,\phi)$$ for some smooth function $f:[0,1]\to [0,1]$. 

Apply \zcref{lem:isotopy_extension_contact} and \zcref{prop:fineslope_inequality} to modify $\xi_+$ so that $\partial_i \xi_+$ has slope $>s_i$ pointwise. By a smooth $\Phi$-horizontal isotopy not changing $\partial \xi_+$, we may arrange that the lines $\partial_r$ is Legendrian in a neighborhood of $\partial M$. As in \zcref{lem:trim_twist}, add twisting to $\xi_+$ in a neighborhood of $\partial M$ so that for $r_0$ sufficiently small, $\xi_+|_{r=r_0}$ is a linear foliation of slope $f(r_0)^2$. Then the pushforward of $\xi_+$ along $\pi$ is a $\Phi$-horizontal contact structure on the blowdown (after possibly correcting the twisting near $\{r=0\}$ to ensure that the contact condition is achieved there).

Conversely, if $\Phi(\bm{s'})$ admits a transverse positive contact structure, we can put it in the standard form $\xi_+=d\phi + r^2 d\theta$ near each blown down orbit. Blowing up, we obtain a $\Phi$-horizontal contact structure of fine multislope $\bm{s'}^\natural$, meaning that all of its boundaries are linear. Trimming $\xi_+$ (\zcref{lem:trim_twist}) and using the instability of $\natural$ foliations, we can increase the boundary slope of $\xi_+$ from $\bm{s'}$ to $\bm{s}$ for some $\bm{s}>\bm{s'}$.
\end{proof}

    \subsection{Nonwandering flows}

Ultimately, we are interested in pseudo-Anosov flows and periodic flows. However, we prove many of our results for the larger class of nonwandering flows. For some results, we will impose additional conditions that are automatically satisfied for suitable pseudo-Anosov and periodic flows.

\begin{defn} \label{def:NW}
    Let $\Phi = (\varphi_t)$ be a (topological) flow on $M$. We say that $\Phi$ is a \textbf{nonwandering} flow if for every nonempty open set $U \subset M$ and every $T > 0$, there exists $t \geq T$ such that $\varphi_t(U) \cap U \neq \varnothing$.
\end{defn}

The following classes of flows are nonwandering:
\begin{itemize}
    \item Topologically transitive flows,
    \item Volume preserving flows (by the Poincar\'{e} recurrence theorem),
    \item Flows whose orbits are all periodic.
\end{itemize}

Importantly, nonwandering flows are ``witnesses'' of tautness. Indeed, any (everywhere) taut foliation or taut contact pair admits a transverse \emph{volume preserving} flow, which is in particular nonwandering. Conversely, we have:

\begin{lem}
Let $\Phi = (\varphi_t)$ be a nonwandering flow on $M$ tangent to a continuous nonvanishing vector field $X$.
\begin{enumerate}
    \item Any $C^0$-foliation transverse to $\Phi$ is taut,
    \item Any contact pair (positively) transverse to $\Phi$ is taut.
\end{enumerate}
\end{lem}

\begin{proof} 
    The proof is the same in both cases. Let $C$ be a continuous cone field on $M$, made of cones which have non-empty interiors. In the case of a foliation, $C$ will be the positive open half space cut out by $T\F$, and in the case of a contact pair, $C$ will be the positive open cone cut out by the two contact structures. Given a point $p \in M$, we would like to construct a $C$-positive curve (i.e., a curve whose derivative lies in $C$) which passes through $p$.

    Reparametrize the flow to move at unit speed. Let $U$ be an $\varepsilon$-ball around $p$. Let $\delta$ be small enough that a $\delta$-cone around the line field tangent to $\Phi$ is contained in $C$. For $\varepsilon$ small enough, every point in $\varphi_{-10\varepsilon/\delta}(U)$ is connected to $p$ by a $C$-positive path and $p$ is connected to every point in $\varphi_{10\varepsilon/\delta}(U)$ by a $C$-positive path. By the nonwandering hypothesis, there exists a flow line of $\Phi$ from $\varphi_{10\varepsilon/\delta}(U)$ to $\varphi_{-10\varepsilon/\delta}(U)$. Concatenating these three $C$-positive paths (and smoothing the corners) gives a $C$-positive curve passing through $p$.
\end{proof}

    \subsection{Periodic flows}

A flow $\Phi$ is \textbf{periodic} if it has no fixed points, and there exists $T>0$ such that $\Phi^T$ is the identity.\footnote{By~\cite{Eps72}, a smooth nonsingular flow on a compact $3$-manifold, all of whose orbits are periodic, is a periodic flow, and hence is diffeomorphic to a Seifert fibration.} Periodic flows are linear on $\partial M$, so they satisfy $\Zg_\aleph(\Phi)=\Zg(\Phi)$. Periodic flows do not in general satisfy condition BSF, but they satisfy the following weaker condition:

\begin{defn}\label{def:condLBSF}
     A flow $\Phi$ is \textbf{locally branched surface finite} (\textbf{LBSF}) if for every $p\in \R^n$, there is a neighborhood $U$ of $p$ and a finite collection of $\Phi$-horizontal branched surfaces $\Sigma_1,\dots, \Sigma_m$ such that, up to the action of flow preserving homeomorphisms of $(M,\Phi)$ rel boundary, every $\Phi$-horizontal foliation with generalized boundary multislope in $U$ is carried by some $\Sigma_i$. Here, a generalized boundary multislope is in $U$ if some substitution of its $\aleph$ components with real numbers lies in $U$.
\end{defn}

\begin{rem}
    The simplest example of an LBSF but not BSF flow is a translation flow on $\T^2\times I$. In this example, $\Zg_\aleph(\Phi)=\Zg(\Phi)$ is noncompact and slopes $\pm\infty$ are never achievable, but simple adaptations of the techniques of \zcref{sec:alephclosure} show that the set of generalized boundary multislopes of foliations carried by a branched surface is a compact subset of $[-\infty, \infty]^n$.
\end{rem}

\begin{lem}
    Periodic flows satisfy condition LBSF.
\end{lem}

\begin{proof}
    Let $\mathcal O$ be the base orbifold of $\Phi$. Choose a finite collection of disks $D_1,\dots,D_m$ covering $\mathcal O$ so that all $k$-fold intersections of the disks are also disks. Given a transverse (taut) foliation $\mathcal F$, its restriction to the solid torus over $D_i$ is a trivial foliation of $S^1\times D_i$ (this works even when $D_i$ crosses an orbifold singularity). Choose a lift $\widetilde D_i$ of $D_i$ to a leaf of $\mathcal F$. Now $\cup_i{\widetilde D_i}$ intersects every flow line. Therefore, $\mathcal F$ is carried by a branched surface $B$ whose $I$-fibered neighborhood is obtained by blowing air into each of the $D_i$'s. The intrinsic topology of $B$ is determined by the ordering of the $D_i$'s over each $k$-fold intersection, which is a finite amount of data. The embedding of $B$ into $M$ is determined up to the action of $\Mod(M,\Phi)$, the mapping class group generated by Dehn twists along vertical tori and annuli. Say that a $\Phi$-horizontal branched surface $B$ \textbf{covers} $A\subset \Zg(\Phi)$ if $A$ is the set of multislopes of foliations carried by $B$. Thus far, we have constructed a finite collection of branched surfaces whose $\Mod(M,\Phi)$ translates cover $\Zg(\Phi)$.
    
    Any train track on $\T^2$ transverse to $\partial_y$ carries only slopes in a bounded subset of $\R$. Therefore, any $\Phi$-horizontal branched surface covers a bounded subset of $\Zg(\Phi)\subset \R^n$. Let $\Mod_\partial(M,\Phi)$ be the subgroup of $\Mod(M,\Phi)$ preserving $(\partial M, \Phi)$. Since $\Mod(\T^2, \partial_y)$ acts faithfully on slopes in $\T^2$ by translation, the quotient $$\Mod(M,\Phi)/\Mod_\partial(M,\Phi)$$ acts faithfully by translations on $\R^n$. Therefore, up to the action of $\Mod(M,\Phi)/\Mod_\partial(M,\Phi)$, finitely many branched surfaces cover any bounded subset of $\Zg(\Phi)$.   
\end{proof}

\begin{rem}
    When $H^1(\mathcal O)$ is nontrivial, the space of transverse foliations up to isotopy is not compact in any reasonable sense. This is because one can perform high powers of a Dehn twist along a vertical torus or annulus to approach a foliation containing the vertical torus or annulus as a leaf.
\end{rem}

\begin{rem}
    One can also give an analytic proof of closedness of $\mathcal Z(\Phi)$ when $\Phi$ is periodic. Let $\pi_1(\mathcal O)$ be the orbifold fundamental group of the base. Let $p_1/q_1,\dots,p_k/q_k$ be the Seifert invariants of the singular orbits, and let $\gamma_1,\dots,\gamma_k \in \pi_1(\mathcal O)$ be short loops around the orbifold points. We have a correspondence between $\Phi$-horizontal foliations and their holonomy representations:
     \begin{equation}\frac{\{\text{$\Phi$-horizontal foliations}\}}{\text{isotopy} + \Mod(M,\Phi)} \longleftrightarrow  \frac{\{ \rho:\pi_1(\mathcal O)\to\Homeo^+(S^1) \mid \forall i,\, \rot(\rho(\gamma_i)) = p_i/q_i\}}{\text{conjugation}}.\label{eqn:homeoreps}
     \end{equation}
    
     Every action of $\pi_1(\mathcal O)$ on $S^1$ can be upgraded to a Lipschitz action via a conjugation, where the Lipschitz constants are bounded by a constant $k$ depending only on $\mathcal O$~\cite[Theorem 2.118]{C07}. The space of $k$-Lipschitz actions on the circle is compact and the rotation number condition is closed. Therefore, the right side of \zcref{eqn:homeoreps} is compact by Arzel\`a-Ascoli. The (reduction mod $1$ of the) boundary multislope is a continuous, conjugation invariant function on $\Hom(\pi_1(\mathcal O), \Homeo^+(S^1))$, so $\Zg(\Phi)/\Mod(M,\Phi)$ is compact. Since $\Mod(M,\Phi)$ acts on $\Zg(\Phi)$ by integer translations, $\Zg(\Phi)$ is closed.
\end{rem}

    \subsection{Pseudo-Anosov flows}

A pseudo-Anosov flow is a flow on a 3-manifold locally modeled on the suspension flow of a pseudo-Anosov surface homeomorphism. In this article, we consider pseudo-Anosov flows on 3-manifolds with torus boundary. We always assume that $\Phi$ is locally modeled on the \textbf{boundary blowup} near $\partial M$: $\Phi|_{\partial M}$ has $2k\geq 2$ closed orbits alternating between attracting and repelling, the attracting closed orbits bound unstable half-leaves, and the repelling closed orbits bound stable half-leaves. See \zcref{fig:pablowup}.

\begin{figure}[ht]
    \centering
    \includegraphics[width=0.3\linewidth]{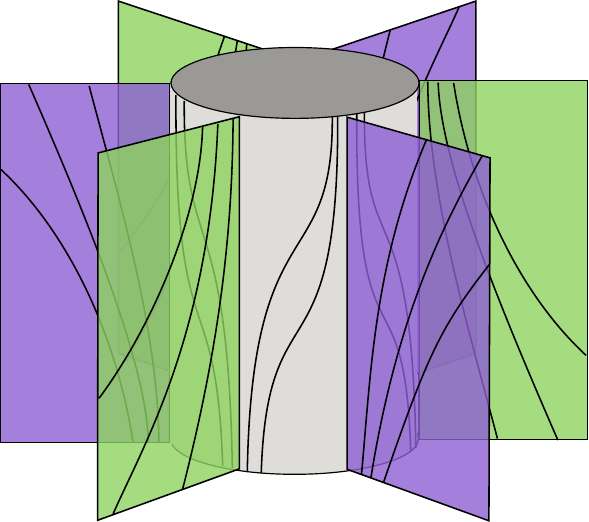}
    \caption{A pseudo-Anosov flow near $\partial M$. The unstable half-leaves are drawn in purple and the stable half leaves are drawn in green.}
    \label{fig:pablowup}
\end{figure}

We refer the reader to \cite{barthelme.mann.PseudoAnosovFlowsPlane} for general background on pseudo-Anosov flows. In this article, we will not need precise definitions, as we will only use some special properties of these flows. Among pseudo-Anosov flows, the easiest to study are the \textbf{fully punctured pseudo-Anosov flows without perfect fits} because they are adapted for the combinatorial machinery of veering triangulations. In particular, they satisfy condition BSF and are $\aleph$-unobstructed.

In theory, it is sufficient to study $\Zg(\Phi)$ for such flows because every transitive pseudo-Anosov flow can be obtained from some fully punctured pseudo-Anosov flow without perfect fits by Dehn filling along some non-degeneracy slopes. In fact, any transitive pseudo-Anosov flow can be drilled along closed orbits to obtain a suspension flow by work of
Fried \cite{fried.TransitiveAnosovFlows} and Brunella \cite{brunella.SurfacesSectionExpansive}.

\begin{thm}[\cite{Z24}]\label{thm:pABSF}
    If $\Phi$ is a fully punctured pseudo-Anosov flow without perfect fits, then $\Phi$ is BSF; in fact, there is a single branched surface transverse to $\Phi$ which carries all $\Phi$-horizontal foliations.
\end{thm}

\begin{lem}\label{lem:fptpafwpf_unobstructed}
    If $\Phi$ is a fully punctured pseudo-Anosov flow without perfect fits, then $\Phi$ is $\aleph$-unobstructed.
\end{lem}

\begin{proof}
    Suppose $\Sigma$ is a $\Phi$-horizontal compact planar surface. The stable and unstable foliations of $\Phi$ restrict to foliations $\F^s$ and $\F^u$ on $\Sigma$. Upon capping off the boundary components of $\Sigma$ to obtain a closed surface $\Sigma^\bullet \cong S^2$, $T\F^s$ and $T\F^u$ become singular line fields, with each degeneracy boundary contributing an index $+1$ singularity and each $k$-pronged boundary component with $k\geq 1$ contributing an index $1-k/2$ singularity. The Poincar\'e--Hopf index theorem implies that the total index equals $\chi(\Sigma^\bullet)=2$.
    
    If $\Sigma$ is a degenerate cylinder then the total index would be at most $2-k/2 < 2$, which violates Poincar\'e--Hopf. If all components of $\partial \Sigma$ are degeneracy, then by Poincar\'e--Hopf, $\Sigma$ must have exactly two boundary components. Thus, $\Sigma$ is a parallelism between two orbits, and $\Phi$ has a perfect fit.
\end{proof}

\begin{prop}[$\aleph$ filling] \label{prop:alephfilling}
Let $\F$ be a $\Phi$-horizontal foliation such that $\partial_i \F$ is $\aleph$ type along some boundary component $\partial_i M$. Then $\F$ blows up to a $\Phi$-horizontal foliation $\F'$ which has the same generalized fine boundary multislope as $\F$ and which extends to a foliation $\F'(s)$ on the Dehn filling of $\partial_i M$ along any (nondegeneracy) slope $s\in \Q$. Moreover, $\F'(s)$ is transverse to an extension $\Phi(s)$ of $\Phi$ to the Dehn filling. Every leaf of $\F'(s)$ is (the collapse of) a leaf of $\F'$; hence, if $\F'$ is taut then so is $\F'(s)$.
\end{prop}

\begin{proof}
    Around $\partial_i \F$ we see a period 4 repeating pattern of annuli: $A_1,A_2,\dots$, where $A_{4m}$ is downward-opening Reeb, $A_{4m+1}$ has only $\sharp$ or $\natural$ annuli, $A_{4m+2}$ is upward-opening Reeb, and $A_{4m+3}$ has only $\flat /\natural$ annuli. 

\begin{figure}[ht]
    \centering
    \def\svgwidth{0.6\textwidth}
    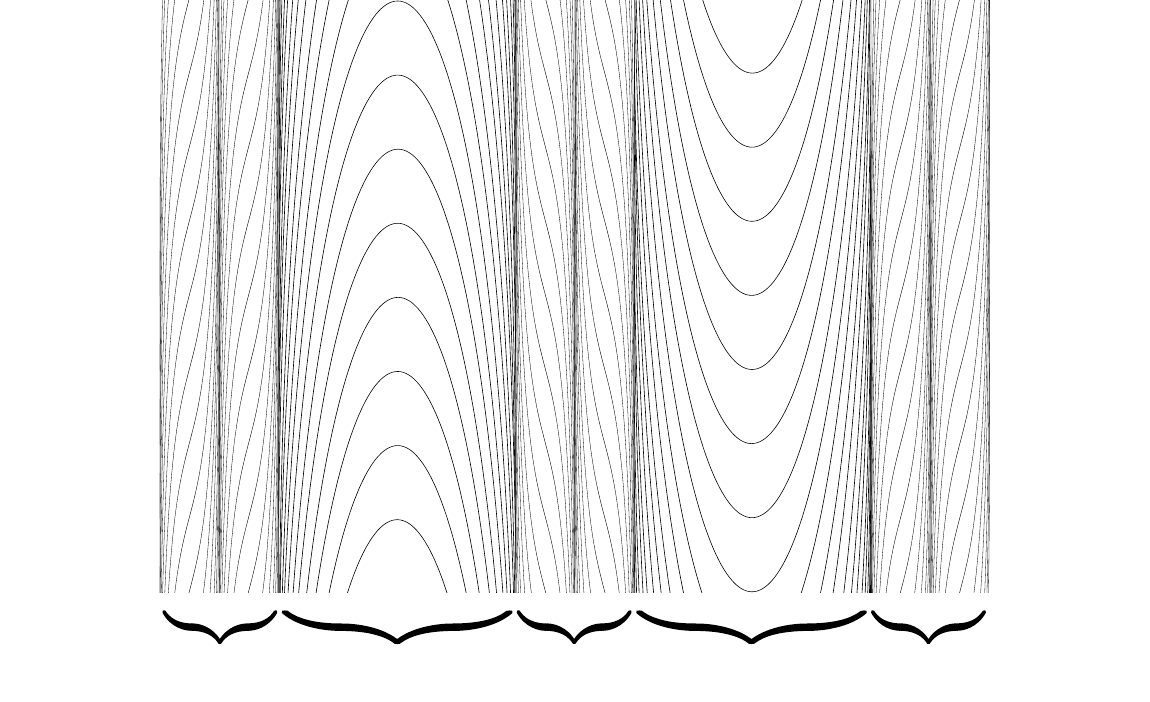
    \caption{$\aleph$ pattern.}
    \label{fig:alephfilling}
\end{figure}

    We say that a foliated annulus is a \textbf{Cantor $\sharp$ (resp.~$\flat$) annulus} if its closed leaves constitute the product of a Cantor set with $S^1$, and each complementary region is a $\sharp$ (resp.~$\flat$) annulus. Since Cantor sets are unique up to homeomorphisms of the interval, there is a unique Cantor $\sharp$ (resp.~$\flat$) annulus up to topological isotopy.
    
    We will construct $\F'$ by performing a blowup on $\F$ so that each $A_{4m+1}$ is a Cantor $\sharp$ annulus while each $A_{4m+3}$ is a Cantor $\flat$ annulus. To do this, perform a trivial blowup on the boundaries of any $\sharp$ or $\flat$ annuli in $A_{4m+1}$ or $A_{4m+3}$ to ensure that all these annuli are isolated. Then choose a countable collection of leaves $\{\lambda_j\}_{j\geq 1}$ whose boundaries are dense in each $\natural$ annulus in $A_{4m+1}$ or $A_{4m+3}$. Blow up each $\lambda_j$ to insert a $\sharp$ annulus along any boundary component in $A_{4m+1}$ and a $\flat$ annulus along any boundary component in $A_{4m+3}$, and the generalized boundary type of $\F$ does not change. This is always possible by perfectness of $\Homeo_+(I)$ unless $\lambda_j$ is a compact planar leaf. In the compact planar case, consider the maximal fibered neighborhood $\lambda_j \times [0,1]$ of $\lambda_j$ (the maximal fibered neighborhood cannot be a fibration over $S^1$ since one of the boundaries of $\lambda_j$ lies on an $\aleph$ type boundary). Then $\lambda_j \times \{0\}$ must have nontrivial holonomy above. On some boundary component it has weakly attracting holonomy (i.e., non-strictly decreasing, and not equal to the identity in any neighborhood of $0$), so it has weakly repelling holonomy on some other component $\lambda_j \times \{0\}$. Therefore, $\lambda_j$ has some boundary component where a $\sharp$ annulus can be added without changing the generalized fine boundary multislope of $\F$, and similarly some boundary component where a $\flat$ annulus can be added. Thus, the desired blowup is possible using a representation $\rho:\pi_1(\lambda_i)\to \Homeo_+(I)$ with ascending holonomy near the $\sharp$-permissible boundary components, descending holonomy near the $\flat$-permissible boundary components, and trivial holonomy around the rest.

\begin{figure}[ht]
    \centering
    \begin{subfigure}{0.6\linewidth}
        \centering
        \def\svgwidth{0.9\linewidth}
        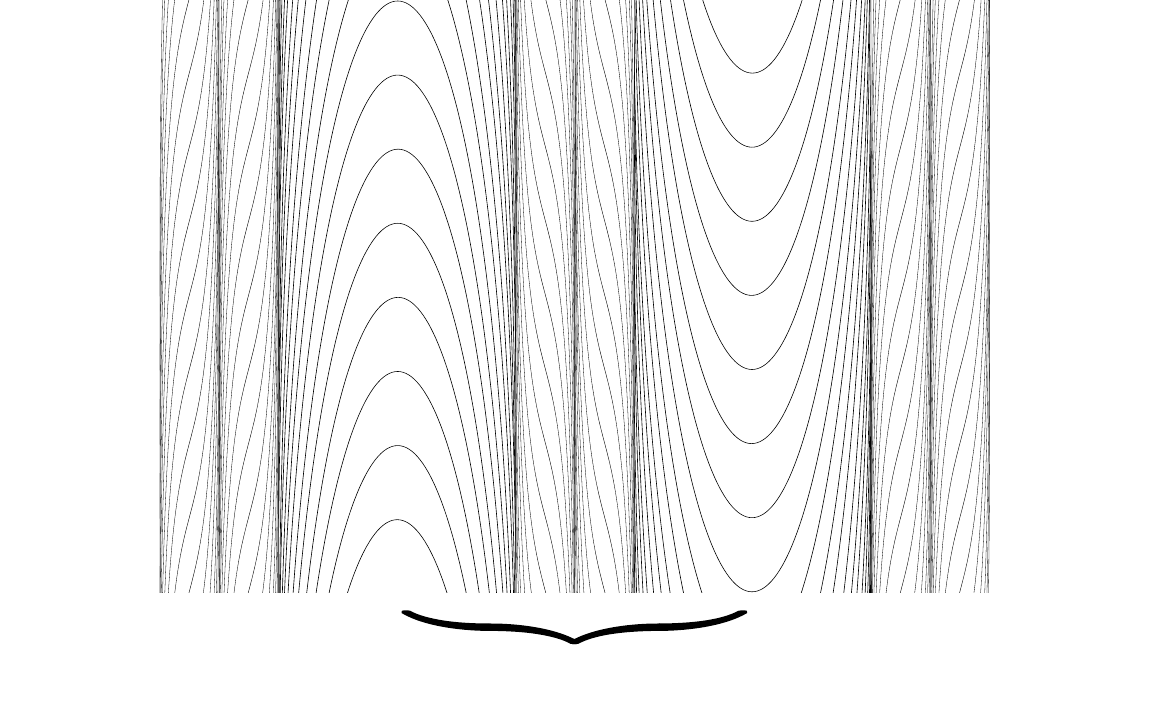
        \caption{Folding instructions...}
        \label{fig:aleph_filling_origami}
    \end{subfigure}
    \begin{subfigure}{0.39\linewidth}
        \centering
        \def\svgwidth{0.75\linewidth}
        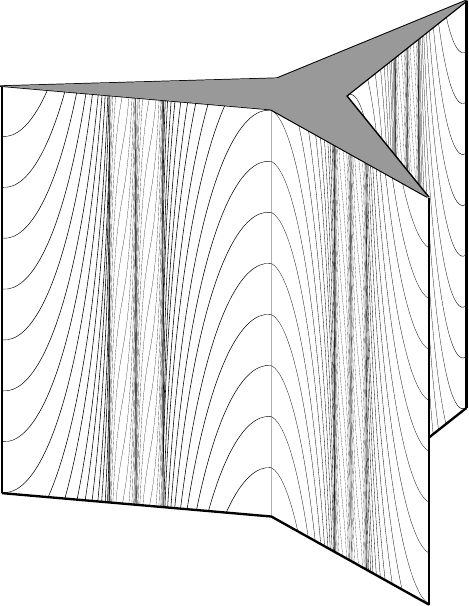
        \caption{...Result.}
        \label{fig:fortune_teller}
    \end{subfigure}
    \caption{Folding the origami pattern. Unlike our usual convention, it is better to imagine that we are looking at $\partial M$ from inside $M$.}
    \label{fig:origami}
\end{figure}

    We are ready to perform the Dehn filling.
    Imagine that $\partial_i M$ is a sheet of origami paper. Perform a mountain fold along a closed transversal in each of the $k$ upward Reeb annuli and a valley fold along a closed transversal in each of the $k$ downward Reeb annuli (there are several topologically distinct options that work just as well; for example, one could instead perform valley folds inside the upward Reeb annuli and mountain folds inside the downward Reeb annuli).
    Now the paper annulus may be collapsed along slope $s_i=p_i/q_i$ onto a $2$-complex which is a $kq_i$-pointed star graph bundle over $S^1$ as in \zcref{fig:origami}. Denote the annular panels of the origami $B_1,\dots,B_{2k}$ as in \zcref{fig:aleph_filling_origami}. As arranged above, the restriction of $\F'$ to $B_{2i}$ is the mirror of the foliation on $B_{2i+1}$: one has a Cantor $\sharp$ annulus between two half $\sharp$ annuli, and the other has a Cantor $\flat$ annulus between two half $\flat$ annuli. Therefore, the foliations extend across the 2-complex. 

    We might not be able to arrange that $\Phi\vert_{B_{2i}}$ matches exactly with the mirror of $\Phi\vert_{B_{2i+1}}$, so $\Phi$ does not immediately extend over the 2-complex. However, we can interpolate between $\Phi\vert_{B_{2i+1}}$ and $\mirror(\Phi\vert_{B_{2i}})$ with flows transverse to $\F'$. This interpolation defines an extension $\Phi(s)$ of $\Phi$ across the Dehn filling. 
\end{proof}

\begin{rem}
    If we assume $\Phi$ is nonwandering, $\Phi(s)$ might not strictly satisfy the definition of nonwandering: it might be wandering in a neighborhood of the Dehn filling core. However, every leaf of $\F'(s)$ extends away from the Dehn filling core and touches a nonwandering set of $\Phi$, so this Dehn filling does not destroy tautness. Also note that $\Phi(s)$ is in the same homotopy class of vector fields as the standard blowdown described in \zcref{def:flowblowdown}.
\end{rem}

\begin{rem}
    If $\Phi$ can be arranged so that $\Phi|_{B_{2i}}$ matches with the mirror of $\Phi|_{B_{2i+1}}$, then $\Phi$ extends immediately across the 2-complex. In this case, the Dehn filled foliation is transverse to a \emph{dynamic blowdown} of $\Phi$ in the sense of Mosher, see \cite[Section 1]{mosher.SurfacesBranchedSurfaces}.
\end{rem}

    \section{Leafwise smoothing of foliations} \label{sec:leafwisesmoothing}

In~\cite{C01}, Calegari showed that every cooriented \emph{topological} foliation (or lamination) on a compact $3$-manifold can be topologically isotoped to a $C^0$-foliation, whose leaves are smoothly immersed. In the presence of a transverse smooth flow, the proof simplifies and it is possible to make the resulting foliation genuinely transverse to the flow:

\begin{lem} \label{lem:smoothingfol}
    Let $\Phi$ be a smooth nonsingular flow on $M$, and let $\F$ be a topological foliation on $M$ which is topologically transverse to $\Phi$. Then $\F$ is topologically isotopic to a $C^0$-foliation with smooth leaves, which is genuinely transverse to $\Phi$. Moreover, the isotopy can be made arbitrarily small, and along the flow lines of $\Phi$.
\end{lem}

\begin{proof}
    First, it is easy to smooth $\F$ \emph{locally}. Indeed, consider a flow box for $\Phi$ near a point $p \in M$, i.e., smooth coordinates $(x,y,z) \in [-1,1]^3$ near $p \cong (0,0,0)$ in which $\Phi$ is tangent to $\partial_z$. Then by assumption, the leaves of $\F$ can be realized as the (partial) graphs of a family of continuous maps $(x,y,z) \mapsto f_z(x,y)$, which satisfy that $z \mapsto f_z(x,y)$ is increasing for every $(x,y)$. One can easily smooth this family in a smaller box, so that the new foliation is smooth and genuinely transverse to $\Phi$. More precisely, we may first approximate $f$ with a smooth $\widetilde{f} : (x,y,z) \mapsto \widetilde{f}(x,y,z)$ so that $\partial_z \widetilde{f} > 0$.\footnote{Smoothing continuous increasing function is an easy task, and it can be done in families and relatively to boundaries, see for instance~\cite[Appendix A.1]{BM25}.} Then, we may interpolate between $f$ and $\widetilde{f}$ with a suitable , so that the resulting map $\overline{f}$ is smooth in a slightly smaller flow box and coincides with $f$ near the boundary of the flow box. 
    
    To obtain the desired approximation \emph{globally}, we may proceed as outlined in~\cite[Remark 3.3]{B16}: cover $M$ with finitely many flow boxes, and apply the previous local smoothing one box at a time. Notice that this operation preserves leafwise smoothness, but not transversal smoothness. Moreover, the foliation remains graphical in each box, so that the resulting $C^0$-foliation is isotopic to the original one via a small continuous isotopy along the flow lines of $\Phi$.
\end{proof}

\begin{proof}[Proof of \zcref{propintro:smoothing}]
  Let $\Phi$ be a topological flow on $M$, $\Phi_\mathrm{sm}$ a smooth model related by a topological equivalence $h : M \rightarrow M$ isotopic to the identity, and let $\F$ be a topological foliation, topologically transverse to $\Phi$, and realizing a boundary multislope $\bm{s} \in \Zg^\mathrm{top}(\Phi)$. Transporting $\F$ along $h$ yields a topological foliation $\F'$ which is topologically transverse to the smooth flow $\Phi_\mathrm{sm}$, and realizes the same boundary multislope $\bm{s}$. We then apply \zcref{lem:smoothingfol} to $\F'$ to obtain a $C^0$-foliation transverse to $\Phi_\mathrm{sm}$ realizing $\bm{s}$, hence $\bm{s} \in \Zg(\Phi_\mathrm{sm})$. The inclusion $\Zg(\Phi_\mathrm{sm}) \subset \Zg^\mathrm{top}(\Phi)$ is immediate by applying $h^{-1}$.
\end{proof}

    \section{Smoothing foliations without holonomy}

In this section we prove a few results about smoothing foliations without holonomy, to be later used in the analysis of planar minimal sets in \zcref{sec:t3typeminimalsets}.

\begin{thm}\label{thm:smoothing_measured_foliations}
    Let $M$ be a compact 3-manifold and $\F$ a taut, transversely oriented $C^{1,0}$ foliation of $M$ transverse to $\partial M$ and having only trivial holonomy. Then $M$ fibers over $S^1$, and $\mathcal F$ is $C^0$-close to a smooth fibering of $M$.
\end{thm}

This is proven in \cite[Theorem 8.11]{KR17} in the case where $M$ is closed. Their proof works verbatim when $M$ has boundary, but we will give a brief sketch below for the benefit of the reader. We use the same technology to prove \zcref{lem:smoothing_measured_laminations}, a smoothing lemma for measured sublaminations.

\begin{defn}
    A \textbf{smooth branching lamination} $\Lambda$ on a 3-manifold $M$ is a (branching) foliation on a compact codimension 0 submanifold $K\subset M$ such that $\partial K$ is tangent to leaves of $\Lambda$ except near branch loci modeled on \zcref{fig:branching_foliation}, and $\Lambda$ is a smooth foliation away from these branching points.
\end{defn}

Any branched surface in $M$ with a transverse measure $\mu$ gives a smooth branching lamination by thickening each sector $\lambda$ into a trivially foliated product $\lambda \times [0,\mu(\lambda)]$.

\begin{defn}
    The \textbf{interstices} of a smooth branching lamination $\Lambda$ supported on $K\subset M$ are the intersections of branching leaves with $\interior K$ (or more precisely, the leafwise metric closures thereof since we want the interstices to contain the branching locus.) The \textbf{guts} of $\Lambda$ are the closures of the components of $M\setminus K$.
\end{defn}

\begin{defn}
    Let $\Lambda$ be a smooth branching lamination supported on $K\subset M$ and let $\phi$ be a flow transverse to $\Lambda$. A \textbf{Denjoy blowup of size $\varepsilon$} of $\Lambda$ along $\phi$ is a lamination $\Lambda'$ supported on $K$ obtained by blowing air into the interstices of $\Lambda$ such that the semiconjugacy $\Lambda' \to \Lambda$ is $\varepsilon$-$C^0$ small, and $\Lambda'$ is transverse to $\phi$.
\end{defn}

\begin{figure}[ht]
    \centering
    \def\svgwidth{0.3\textwidth}
    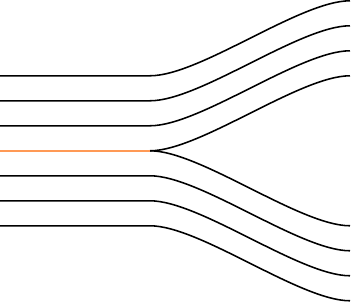
    \caption{A smooth branching lamination. One of the interstices is shown in orange.}
    \label{fig:branching_foliation}
\end{figure}

\begin{proof}[Proof sketch of \zcref{thm:smoothing_measured_foliations}]
$\F$ must be a trivial Denjoy blowup of a measured foliation. If the latter is not a fibration already, we consider $\Lambda$ the exceptional measured minimal set for $\F$. Choose a branched surface $B$ carrying $\Lambda$. Up to splitting $B$ further, we may assume that $B$ is $\varepsilon$-flat, i.e., it admits an $I$-fibered neighborhood which has a $\varepsilon$-flat flow box decomposition in the sense of \cite{KR17}. In short, the condition of $\varepsilon$-flatness is that every flow box is much thinner in the transverse direction than in the tangential direction. This condition guarantees that any lamination $\F'$ carried by $B$ is, up to isotopy, $\varepsilon$ tangentially $C^0$ close to $\F$.

Our branched surface $B$ is measured, so it gives rise to a smooth branching lamination without holonomy supported on our $\varepsilon$-flat $I$-fibered neighborhood $K$ of $B$. For each branch sector $\lambda$ on $\partial K$, shave a thin neighborhood $\lambda \times [0,\delta]$ away from $K$ as in \zcref{fig:branching_foliation_shaved}. The newly exposed vertical annuli are trivially foliated. Since the original foliation had no holonomy, we can glue this up to the old foliation on the guts of $\Lambda$. Now we have obtained a measured foliation $\F'$ which is $\varepsilon$ tangentially $C^0$ close to $\F$ near $B$ and equal to $\F$ away from $B$. Finally, we can approximate the transverse measure of $\F'$ with a smooth $1$-form to obtain a smooth foliation.
\begin{figure}
    \centering
    \def\svgwidth{0.6\textwidth}
    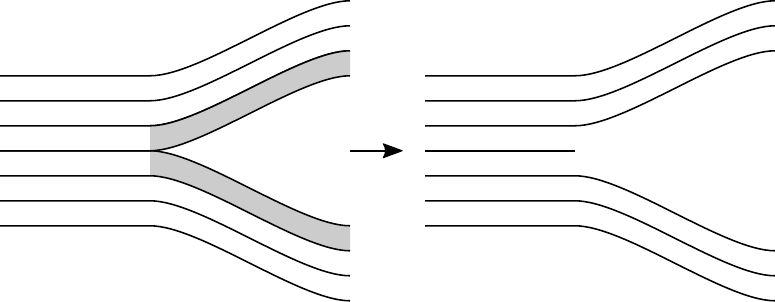
    \caption{Shave from the horizontal boundary of $K$.}
    \label{fig:branching_foliation_shaved}
\end{figure}
\end{proof}

\begin{lem} \label{lem:smoothing_measured_laminations}
    Suppose $\Lambda$ is a measured sublamination of a foliation $\F$. For any $\varepsilon > 0$, there is a smooth branching lamination $\Lambda'$ supported on a neighborhood of $\Lambda$, and a semiconjugacy $\Lambda' \to \Lambda$ which is $\varepsilon$-close to the identity.
\end{lem}

\begin{proof}
    Choose an $\varepsilon$-flat branched surface $B$ carrying $\Lambda$. It is measured, so there is a smooth branching lamination $\Lambda'$ supported on the $\varepsilon$-flat $I$-fibered neighborhood of $B$. The condition of $\varepsilon$-flatness implies that the semiconjugacy $\Lambda'\to \Lambda$ is $\varepsilon$-close to the identity.
\end{proof}

    \section{Typology of minimal sets} \label{sec:minimalsets}

In this section, we analyze and classify all the possible types of minimal sets for a cooriented $C^0$-foliation on $M$ transverse to $\partial M$. To that end, it is convenient to classify the topological types of \emph{laminations} on $M$ which are transverse to $\partial M$.

Each boundary component of a leaf of a lamination transverse to $\partial M$ is either open (homeomorphic to $\R$) or closed (homeomorphic to $S^1$). Each open boundary component may lie in a boundary component of $\partial M$ with irrational slope, or it may spiral on both sides towards two closed curves bounding an annulus (or a single curve bounding a ``degenerate annulus'') on a boundary component of $M$, in the fashion of a $\sharp$, $\flat$, or Reeb annulus.

We first classify the possible topological types for a single leaf of such a lamination:

\begin{lem} \label{lem:surftype}
    Let $\lambda$ be a connected surface (i.e., $2$-manifold), possibly non-compact and possibly with boundary. Then at least one of the following holds:
    \begin{enumerate}
        \item $\pi_1(\lambda)$ remains nontrivial after capping off all the closed boundary components of $\lambda$.
        \item $\lambda$ is a compact genus 0 surface.
        \item $\lambda$ becomes homeomorphic to $\R^2$ after filling the closed boundary components with disks and deleting the open boundary components. In that case, $\pi_1(\lambda)$ is freely generated by its closed boundary components.
    \end{enumerate}
\end{lem}

\begin{proof}
    Let us assume that 1 does not hold, i.e., the abstract surface $\lambda^\bullet$ obtained from $\lambda$ by capping off all of its closed boundary components has trivial fundamental group. If $\lambda^\bullet$ has no boundary, then by the classification of simply connected surfaces, it is either homeomorphic to $S^2$ or to $\R^2$. In the first case, compactness implies that $\lambda$ has finitely many boundary components and is obtained from $S^2$ by removing finitely many open disks, i.e., $\lambda$ is a compact genus $0$ surface. In the second case, $\lambda$ is of the type of case 3. If $\lambda^\bullet$ still has noncompact boundary components, then the result of removing these boundaries is not homeomorphic to $S^2$ since it is noncompact, so it is homeomorphic to $\R^2$ as in case 3.
\end{proof}

We now describe some special cases of laminations by planar surfaces on $M$. These will be pathological cases to consider, especially for the finer version of the Eliashberg--Thurston theorem with boundary in the next section.

\begin{defn}\label{def:pems}
    A \textbf{model $\T^3$-type planar lamination} is any lamination which can be constructed starting from an irrational linear foliation by planes on $\T^3$, possibly splitting open some leaves, and then possibly removing a tubular neighborhood of a transverse link to obtain a lamination on a manifold with boundary. Similarly define a \textbf{model $\T^2$-type planar lamination}.

    If $\Lambda$ is a lamination with a neighborhood which is lamination-preserving, boundary preserving homeomorphic to a neighborhood of some model planar lamination, then we say that $\Lambda$ is a \textbf{$\mathbb{T}^3$-type} or \textbf{$\mathbb{T}^2$-type} planar lamination, respectively.
\end{defn}

\begin{figure}[ht]
    \centering
    \def\svgwidth{0.35\textwidth}
    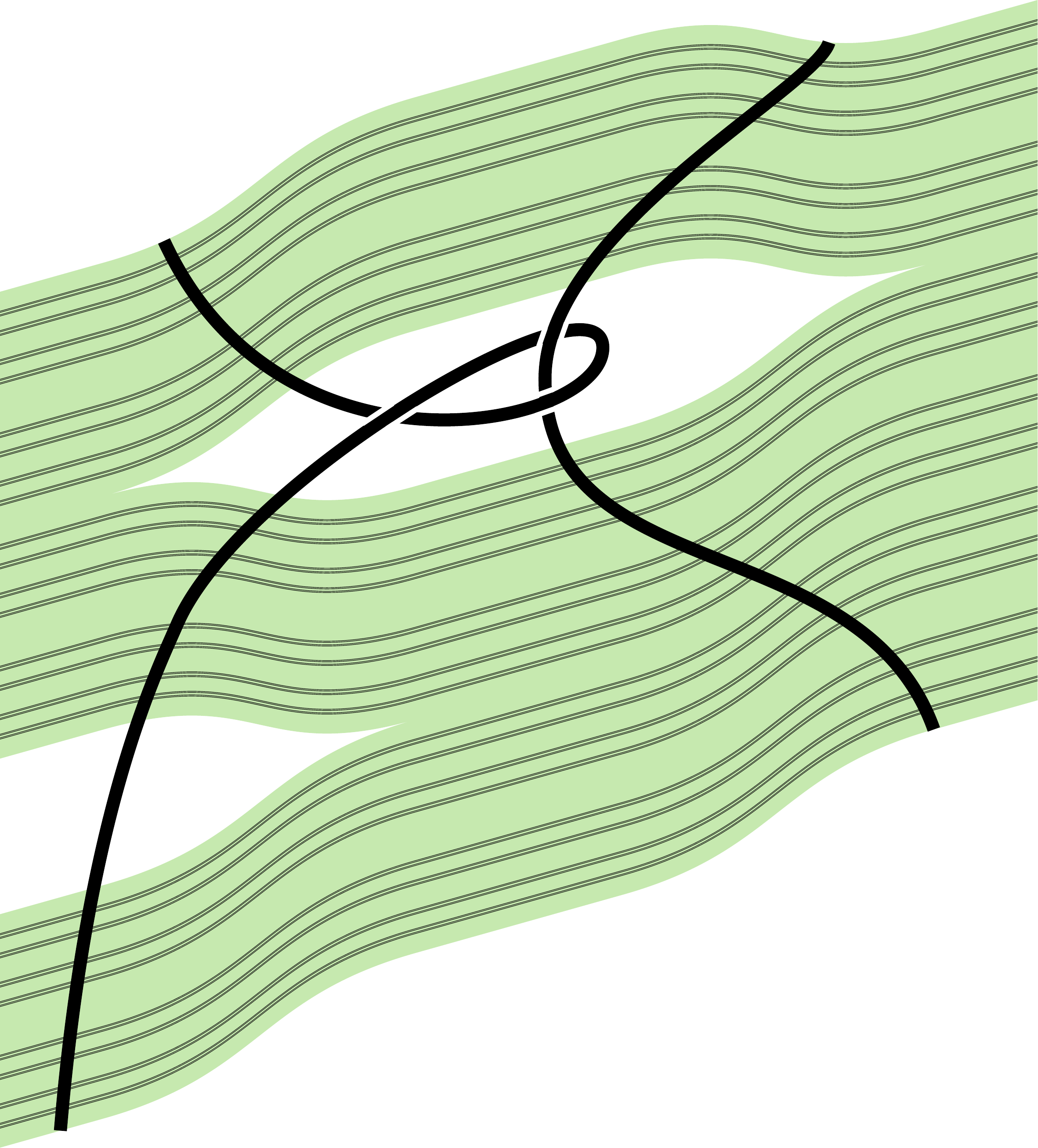
    \caption{Sketch of a model planar lamination. The transverse link may have nontrivial topology inside the guts.}
    \label{fig:t2planarminimalset}
\end{figure}

Notice that the transverse link can have complicated topology in the guts of the lamination; see \zcref{fig:t2planarminimalset} for a sketch. We now classify laminations by simply connected leaves.

\begin{lem} \label{lem:planarclassification}
    Suppose that $M$ admits a minimal lamination $\Lambda$ with simply connected leaves. Then either $\Lambda$ is a homologically trivial sphere, or $M$ is homeomorphic to a connected sum
    $$M \cong M'\#K_1\#\dots \#K_m,$$
    where the $K_i$'s are compact 3-manifolds, $M'$ is one of $S^1\times S^2$, $S^1 \times D^2$, $\T^3$, or $\T^2 \times I$, and $\Lambda$ is contained in $M'$ and is semiconjugate to a linear foliation on $M'$. In particular, $\Lambda$ is measured.
\end{lem}

\begin{proof}
Gabai proved this result in the case where $M$ is closed, see~\cite{G91} and the exposition in \cite[Lemma 7.21]{C07}. We follow his argument in the case where $M$ has boundary. If $\Lambda$ is a sphere, then the result is immediate: either the sphere is homologically trivial, or it cuts an $S^1\times S^2$ summand. If $\partial \Lambda$ contains any closed leaf in $\partial M$, that leaf bounds a disk. Therefore, we are in the second case $S^1\times D^2$. 

Now consider the case that all leaves of $\Lambda$ are planes. Since $\partial M$ is a collection of tori and $\partial \Lambda$ has no compact leaves, the guts of $\partial\Lambda$ are all products. Therefore, we can collapse these products to be very thin so that the boundary of the guts do not intersect $\partial M$. Then every component of the guts of $\Lambda$ has $R_+$ and $R_-$ homeomorphic to disks. Cut out these guts and glue in balls to get a lamination $\Lambda'$ on a 3-manifold $M'$ whose guts are homeomorphic to the product sutured manifold $D^2\times I$. Blow down the guts to get a minimal foliation $\F'$. Since $\F'$ has no holonomy, it is $\R$-covered. Therefore, $\pi_1(M')$ acts on $\R$ without fixed points. By H\"older's classification of such actions (see the exposition in \cite[Theorem 2.90]{C07}), $\pi_1(M')$ is free abelian and preserves an invariant transverse measure on $\F'$. Since $\mathcal F'$ is minimal, the invariant transverse measure has full support. By \zcref{thm:smoothing_measured_foliations}, $\mathcal F'$ is topologically isotopic to a smooth measured foliation, which may in turn be approximated by a fibration $\Sigma \to M' \to S^1$. Fundamental groups of fibers inject into $\pi_1(M')$, so the fiber $\Sigma$ must have abelian fundamental group. Therefore, $\Sigma$ is a sphere, torus, annulus, or a disk. These give the four possibilities appearing in the lemma statement.
\end{proof}

\begin{lem}[Reeb stability for sutured manifolds]\label{lem:suturedreeb}
    Suppose $(M,R_+,R_-)$ is a connected sutured $3$-manifold and $\F$ is a foliation compatible with the sutures. Suppose that $R_-$ has a simply connected component. Then $(M,R_+,R_-)$ is either the product sutured manifold $D^2\times I$ or $S^2 \times I$, and $\F$ is the product foliation.
\end{lem}

\begin{proof}
    If the simply connected component of $R_-$ is a sphere, the Reeb stability theorem directly applies. If the simply connected component of $R_-$ is a disk, then the argument proceeds as in the proof of the usual Reeb stability theorem: one shows that the maximal trivially foliated region parallel to the simply connected component of $R_-$ is both open and closed. To avoid repeating this proof, we can reduce our lemma to the standard Reeb stability theorem by doubling $M$ along its sutures. The disk component of $R_-$ doubles to become a spherical leaf of the doubled foliation, $D\F$. By the Reeb stability theorem, $D\F$ is the product foliation $S^2\times I$. The only orientation reversing involution of $D\F$ which does not interchange the boundary components of $D\F$ is, up to homeomorphism, reflection in the $S^2$ factor. Therefore, $\F$ was homeomorphic to a product foliation by disks.
\end{proof}

\begin{cor}[Simply connected minimal sets] \label{cor:simplyconnleaves}
    Let $\F$ be a cooriented $C^0$-foliation on $M$ transverse to $\partial M$ (if nonempty). If $\Lambda$ is a minimal set of $\F$ consisting of simply connected leaves, then one of the cases in the following table holds:

\begin{center}
\begin{tblr}{
  colspec = {|Q[c,m]|Q[l,m,wd=6cm]|Q[l,m,wd=6.5cm]|},
  row{1} = {bg=gray!20},
  column{1} = {bg=gray!20},
  vlines,
  hlines,
}
  & \centering $\Lambda$ is a compact leaf& \centering $\Lambda$ is not a compact leaf \\
  $\partial M = \varnothing$ & $\Lambda \cong S^2$, and $\F$ is conjugate to the standard foliation by spheres on $S^1 \times S^2$. & $\Lambda$ is a $\T^3$-type planar minimal set, and $\F$ is semiconjugate to an irrational foliation on $\T^3$.\\
  $\partial M \neq \varnothing$ & $\Lambda \cong D^2$, and $\F$ is conjugate to the standard foliation by disks on $S^1 \times D^2$. & $\Lambda$ is a $\T^2$-type planar minimal set, and $\F$ is semiconjugate to an irrational foliation on $\T^2 \times I$.\\
\end{tblr}
\end{center}
\end{cor}

\begin{proof}
    Apply \zcref{lem:planarclassification} to $\Lambda$. The leaves on the horizontal boundary of $\Lambda$ are simply connected. Therefore, the guts of $\Lambda$ are sutured manifolds whose $R_+$ and $R_-$ are all disks or spheres. By \zcref{lem:suturedreeb}, the guts of $\Lambda$ are trivially foliated. The interstices of $\Lambda$ are all homeomorphic to $S^1 \times \R \times I$. They must also be trivially foliated, because the holonomy around each suture of the guts of $\Lambda$ is trivial. Therefore, $M \setminus \Lambda$ is trivially foliated. This implies that the $K_i$'s in the conclusion of \zcref{lem:planarclassification} are trivial and $\F$ is semiconjugate to a linear foliation on $S^1\times S^2$, $S^1\times D^2$, $\T^3$, or $T^2\times I$. These are the four cases in the table above.
\end{proof}

We can now establish the main result of this section.

\begin{prop} \label{prop:lamclass}
    Suppose $\Lambda$ is a minimal lamination on $M$ transverse to $\partial M$. Then at least one of the following holds:
    \begin{enumerate}
        \item $\Lambda$ has a leaf $\lambda$ such that $\pi_1(\lambda)$ is not normally generated by closed boundary components.
        \item $\Lambda$ is a compact planar surface.
        \item $\Lambda$ has a leaf with an open boundary component inside a Reeb, $\sharp$, or $\flat$ annulus on $\partial M$.
        \item $\Lambda$ is a $\mathbb{T}^3$-type planar lamination.
        \item $\Lambda$ is a $\mathbb{T}^2$-type planar lamination.
    \end{enumerate}
\end{prop}

\begin{proof}
    Given a leaf $\lambda$ of $\Lambda$, we denote by $\lambda^\bullet$ the abstract surface obtained by capping off the closed boundary components of $\lambda$ by disks. If $\pi_1(\lambda^\bullet)\neq 1$ for some leaf $\lambda$ of $\Lambda$, then option 1 holds. If $\lambda^\bullet \cong S^2$, then option 2 holds. Otherwise, every leaf $\lambda$ of $\Lambda$ satisfies $\interior \lambda^\bullet \cong \R^2$ by \zcref{lem:surftype}. If no leaf of $\Lambda$ satisfies option 3, then every component of $\partial M$ is either a lamination on $\T^2$ by circles, or a lamination on $\T^2$ without circles. Dehn fill the boundary components of $\partial M$ of the former type to obtain a new lamination $\Lambda^\bullet$ on a new 3-manifold $M^\bullet$. Now, $\Lambda^\bullet$ is a lamination with all leaves homeomorphic to $\R^2$. By \zcref{lem:planarclassification}, $\Lambda$ satisfies either option 4 or option 5.
\end{proof}

We introduce some terminology for the various minimal sets that we will encounter.

\begin{defn} \label{def:minimalsets}
    A minimal lamination $\Lambda$ transverse to $\partial M$ is \textbf{nonperipheral} if it is as in item 1 of \zcref{prop:lamclass}. It is \textbf{peripheral} otherwise. If it is as in item 3, we call it a \textbf{spiral} minimal set. These definitions extend to minimal sets of foliations.
\end{defn}

\clearpage
\part{From foliations to contact structures and back} \label{sec:ETandback}

In~\cite{ET}, Eliashberg and Thurston proved that any $C^2$-foliation different from the fibration by spheres on $S^1 \times S^2$ on a \textbf{closed} manifold can be approximated by positive and negative contact structures. This result has been extended to $C^0$-foliations by Bowden~\cite{B16} and independently Kazez--Roberts~\cite{KR17}. In~\cite{M24}, the first author proved a converse result on the construction of $C^0$-foliations from suitable contact pairs, still in the setting of closed manifolds.

The goal of this part is to extend these theorems to manifolds with toroidal boundary. In the case of the Eliashberg--Thurston approximation theorem, we will carefully keep track of how the boundary foliation is modified; this is perhaps the most difficult result of this article. The converse theorem for manifolds with boundary will be reduced to the closed manifold case by an elementary reflection trick.

    \section{Eliashberg--Thurston with boundary}

As before, we assume that $M$ is a connected, compact, oriented $3$-manifold with nonempty boundary $\partial M \neq \varnothing$.

Let $\F$ be a cooriented $C^0$-foliation on $M$ which is transverse to $\partial M$. Our goal is to approximate $\F$ by positive and negative contact structures $\xi_\pm$ with precise control over the characteristic foliations of $\xi_\pm$ along $\partial M$. In particular, if $\partial M$ is framed and $\F$ realizes a multislope $\bm s$ along $\partial M$, we would like to construct contact approximations realizing the same multislope $\bm s$ along $\partial M$. We refer to this version of the Eliashberg--Thurston theorem with boundary as the \textbf{coarse} version. It applies to any foliation $\F$ as above which is not homeomorphic to the standard foliation by disks on $S^1 \times D^2$, or to a blowup of an irrational foliation on $\T^2 \times I$.

This coarse version does not allow one to control the \emph{boundary type} of the contact approximations along $\partial M$. For some applications, we will need a more precise result which will allow us to \emph{improve} the boundary type of the foliation (e.g., turn a $\sharp$ boundary foliation into one with strictly higher slope), provided that certain obstructions (compact genus $0$ surfaces and $\T^2$-type planar minimal sets, see below) do not appear. We will refer to this version as the \textbf{fine} version of the Eliashberg--Thurston theorem with boundary. In particular, this will allow us to remove certain $\flat$, $\sharp$, $\natural$, and Reeb regions appearing in $\partial \F$ in the unobstructed setting.

\begin{rem}
    After some minor adjustments near $\partial M$, one can \emph{double} the foliation $\F$ on the doubled manifold $\overline{M} = M \cup_{\partial M} (-M)$, and apply the standard version of Eliashberg--Thurston theorem (for $C^0$-foliations) to the closed manifold $\overline{M}$, provided that $\F$ is not homeomorphic to the standard foliation by disks on $S^1 \times D^2$ (which would yield the standard foliation by spheres on $S^1 \times S^2$). This way, one obtains positive and negative contact approximations to $\F$ on $M$. However, this strategy provides no control over characteristic foliations of the contact structures along $\partial M$, beyond $C^0$-closeness. This is even more problematic when the original foliation is only $C^0$, as the boundary slope may be modified in a dramatic way.
\end{rem}

We now give a rough overview of the proof strategy for the coarse and fine versions. While the fine version essentially follows the same steps, it will require a much more careful analysis of various key ingredients; in particular, $\T^3$-type planar minimal sets will be handled separately using new techniques.

\begin{itemize}[leftmargin=*]
    \item \textbf{Step 1.} First, we make compact leaves isolated by performing suitable blowdowns and blowups. The new boundary foliation is monotone-equivalent to the original one. In the absence of compact genus $0$ leaves, we can arrange that the new boundary foliation is obtained by blowdowns and \emph{trivial} blowups, which is relevant to the fine version. This step is very similar to Step 1 from~\cite[Section 8]{B16}.
    
    \item \textbf{Step 2.} Then, we perform blowups near the (finitely many) minimal sets to create suitable holonomy there, as in Step 2 of~\cite[Section 8]{B16}. While the boundary foliation is only modified by some monotone-equivalence, more care is needed to prevent nontrivial blowups for the fine version, in the absence of compact genus $0$ leaves. Besides, $\T^3$-type planar minimal sets will be treated separately using a rather hands on strategy.

    \item \textbf{Step 3.} We proceed as in~\cite[Lemma 4.1]{B16} to create \emph{nice annular fences} near holonomy curves. For the fine version we also create ``toroidal holes'' by drilling out neighborhoods of some transverse closed curves in order to handle the $\T^3$-type planar minimal sets. We further consider suitable ``boundary windows'' along $\partial M$ (including on the newly created toroidal boundary components). We then remove neighborhoods of the annular fences to obtain a certain type of manifold with boundary and concave corners that we call \emph{manifolds with (annular boundary) fences and (boundary) windows}. Crucially, we will arrange that every point in the resulting manifold can be connected to a boundary fence or window by a smooth curve tangent to the foliation.

    \item \textbf{Step 4.} Adapting techniques developed in~\cite{B16}, we modify the foliation along \emph{ribbons} connecting fences and windows to the ``genuine'' boundary components and improve the boundary foliation along the latter, at the expense of also modifying the foliation along the boundary fences and windows. The reader should think of the fences and windows as ``drains'' which absorb the modifications at the genuine boundary along the ribbons.

    \item \textbf{Step 5.} We can now apply a suitable version of the main result of~\cite{B16} to approximate the foliation and obtain a contact structure which dominates the boundary foliation obtained from Step 4 along the genuine boundary components (after possibly making it slightly worse inside of the windows).

    \item \textbf{Step 6.} Finally, we extend the approximating contact structures by filling the annular fences as in~\cite[Lemma 5.6]{B16}. We also fill the toroidal boundary components corresponding to the $\T^3$-type planar minimal sets using the results in \zcref{sec:tori} below.
\end{itemize}

\begin{longrem}
    One key aspect of~\cite{B16}, which goes back to~\cite{ET}, is to approximate the foliations by \emph{confoliations} and propagate the contactness to the entire manifold. In this language, our strategy can be phrased as follows: first find holonomy \emph{away} from $\partial M$ and produce bubbles of contactness there. Then propagate contactness towards $\partial M$ which will improve the slope of $\partial \F$ as in \zcref{fig:propagate}.

    \begin{figure}[ht]
        \centering
        \begin{subfigure}{0.45\linewidth}
        \centering
        \includegraphics[width=\linewidth]{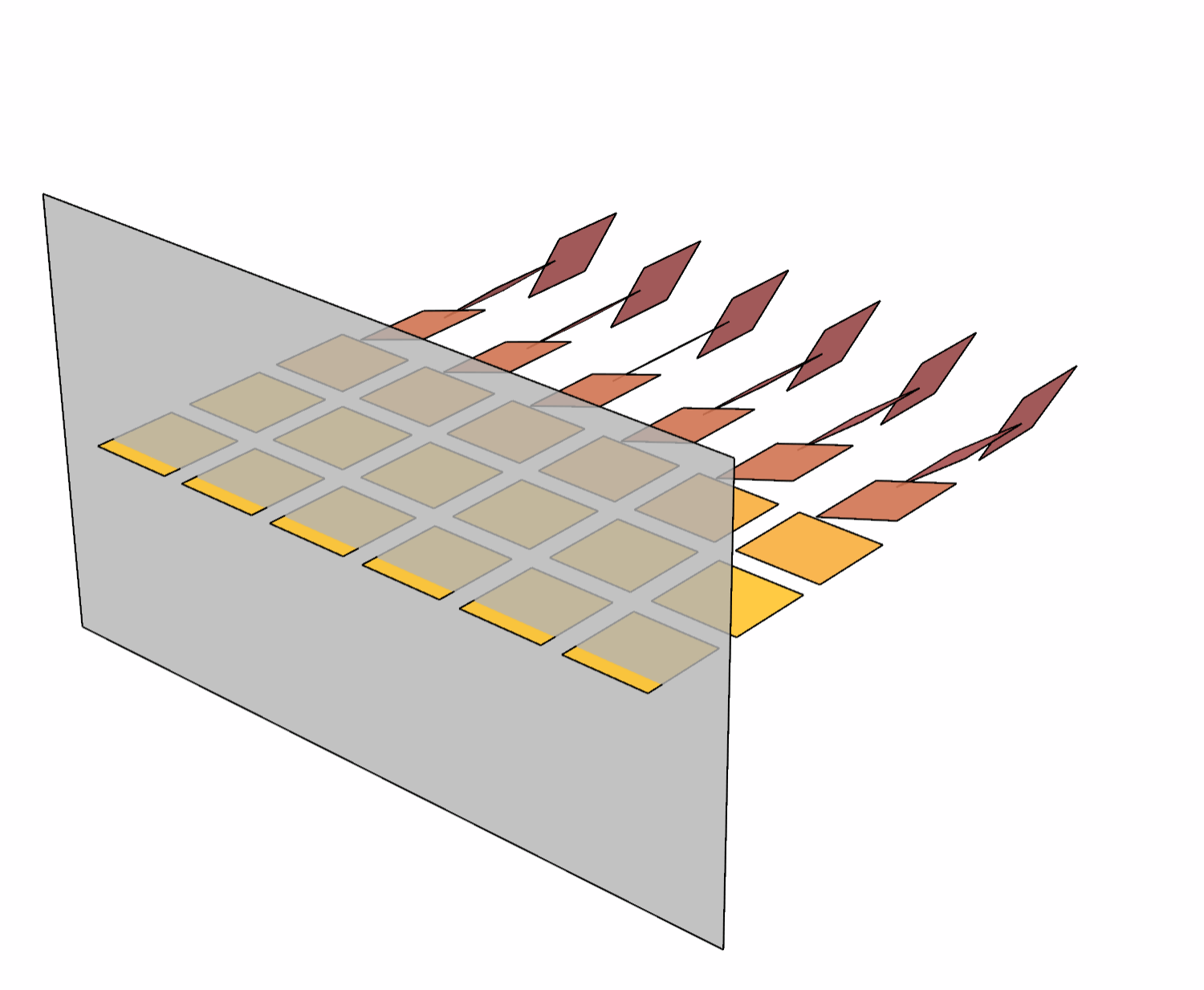}
        \caption{}
        \label{fig:propagate1}
        \end{subfigure}
        \hspace{0.05\linewidth}
        \begin{subfigure}{0.45\linewidth}
        \centering
        \includegraphics[width=\linewidth]{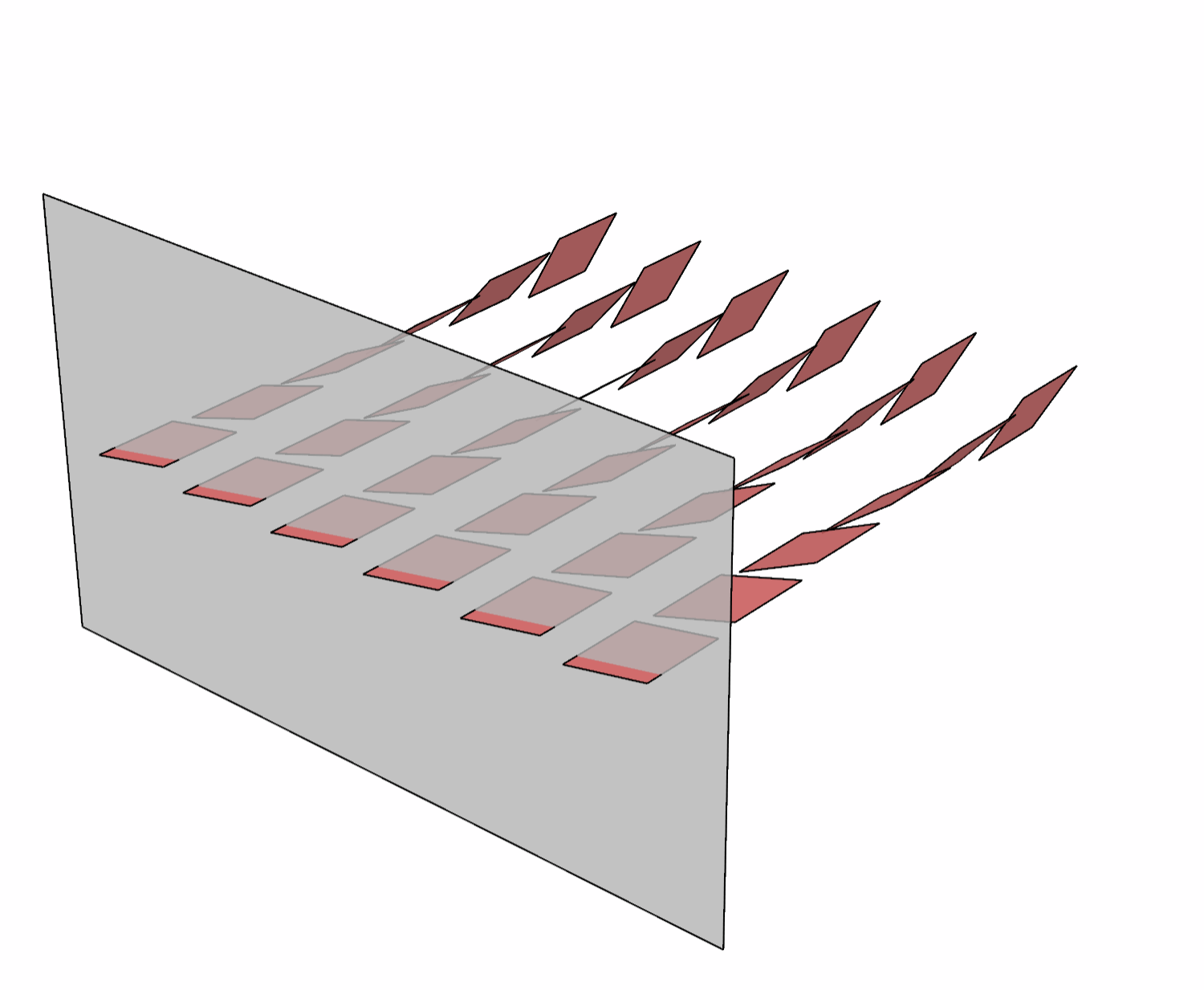}
        \caption{}
        \label{fig:propagate2}
        \end{subfigure}
        \caption{\ref{fig:propagate1} shows a smooth confoliation $\F$ which is a foliation near $\partial M$. The contact region is shown in red tones and the foliated region is shown in yellow. Propagating contactness towards $\partial M$ as in \ref{fig:propagate2} improves the slope of $\partial \F$.}
        \label{fig:propagate}
    \end{figure}

However, the notion of $C^0$-confoliation from~\cite{B16} is rather ad hoc and quite subtle to manipulate. The strategy outlined in the main text above has the advantage of completely bypassing the use of $C^0$-confoliations, except in Step 5 which closely follows~\cite{B16}. 
\end{longrem}

        \subsection{Manifolds with fences and windows} \label{sec:fencesandwindows}

In this section, we develop the main technical tools involved in the proof of both the coarse and the fine versions of Eliashberg--Thurston with boundary. One main result is a rephrasing (and slight extension) of the key technical steps from~\cite{B16}. We also use similar techniques to improve boundary foliations in suitable settings. To this end, we will drill parts of the manifold (annular boundary fences) where the foliation has a certain type of holonomy, and consider special regions along the boundary (windows) where the foliation is standard. We then use variations of \emph{smoothing ribbons} from~\cite{B16} to ``twist'' the foliation along the ribbons and transport modifications of the genuine boundary components to those drilled regions and to the windows. This will allow us to make the boundary foliation ``better'' than the original one, except inside the windows where it will become slightly ``worse''. These modifications will take place on what we call \emph{manifolds with fences and windows}.

\begin{defn} \label{def:manfence}
    A $3$-manifold with annular boundary fences and boundary windows, or \textbf{manifold with fences and windows} for simplicity, is a manifold with boundary and (concave) corners $M_\ssq$ of the form $\overline{M \setminus N}$, where $M$ is a usual $3$-manifold with boundary (and without corners), and $N$ is a disjoint union of codimension-$0$ submanifolds with boundary and corners, disjoint from $\partial M$, and diffeomorphic to $S^1 \times I \times I$. We call the boundary components of $M_\ssq$ contained in $\partial N$ the \textbf{fences} of $M_\ssq$.
    
    It also comes with the data of \textbf{windows} which are subsets of $\partial M_\ssq$ of two types: 
    \begin{itemize}
        \item A \textbf{square window} is a smoothly embedded square $I \times I$ inside a component of $\partial M_\ssq$ which is not a fence,
        \item A \textbf{toroidal window} is an entire boundary component of $\partial M_\ssq$ which is not a fence and that is diffeomorphic to $\T^2$.
    \end{itemize}
    Moreover, we require that the windows are pairwise disjoint.
\end{defn}

When $M$ has other boundary components than the fences (which will typically be the case for us), we denote by $\partial_\ssq M_\ssq$ the union of the boundary fences and toroidal windows, and by $\partial_* M_\ssq$ the union of the genuine boundary components, or \textbf{facades}, i.e., the ones coming from $M$ in \zcref{def:manfence} which are not toroidal windows. By definition,
$$\partial M_\ssq = \partial_\ssq M_\ssq \sqcup \partial_* M_\ssq.$$
We further denote by $\partial^\mathrm{an}_\ssq M_\ssq$ the union of the boundary fences, and by $\partial^\mathrm{tor}_\ssq M_\ssq$ the union of the toroidal windows, so that $\partial_\ssq M_\ssq = \partial^\mathrm{an}_\ssq M_\ssq \sqcup \partial^\mathrm{tor}_\ssq M_\ssq$.

We will typically assume that $\partial_* M_\ssq$ is a disjoint union of tori. We denote by $\mathcal{A}_1, \dots, \mathcal{A}_k \subset \partial M_\ssq$ the annular fences of $M_\ssq$, and by $\mathcal{W}_1, \dots, \mathcal{W}_\ell \subset \partial_\ssq M_\ssq$ the boundary windows. We will also consider \textbf{standard neighborhoods (and coordinates)} of those fences and windows:

\begin{defn}
    Let $M_\ssq$ be a manifold with fences and windows.
    \begin{itemize}
        \item A standard neighborhood of the annular fence $\mathcal{A} \subset \partial_\ssq M_\ssq$ is a neighborhood $N(\mathcal{A})$ of the form $S^1 \times [-1,2]^2 \setminus (S^1 \times (0,1)^2)$ with smooth coordinates $(x, y, z)$. It only intersects $\partial M_\ssq$ along $\mathcal{A}$. We call the annuli $S^1 \times \{0,1\} \times [0,1]$ the \textbf{vertical sides} of the fence $\mathcal{A}$.
        \item A standard neighborhood of the boundary square window $\mathcal{W} \subset \partial_* M_\ssq$ is a neighborhood $N(\mathcal{W})$ of the form $[-1,2] \times I \times [-1,2]$ with coordinates $(x,y,z)$, in which $\mathcal{W}$ corresponds to $[0,1] \times \{1\} \times [0,1]$. It only intersects $\partial M_\ssq$ along $[-1,2] \times \{1\} \times [-1,2]$ and is contained in the component of $\partial_*M$ containing $\mathcal{W}$ there.
        \item A standard neighborhood of the boundary toroidal window $\mathcal{W} \subset \partial_* M_\ssq$ is a neighborhood $N(\mathcal{W})$ of the form $\mathcal{W} \times [0,1]_y \cong \T_{x,z}^2 \times [0,1]_y$ in which $\mathcal{W}$ corresponds to $\mathcal{W} \times \{1\}$ 
    \end{itemize}
\end{defn}

For a choice of pairwise disjoint standard neighborhoods and coordinates around the annular fences and windows of $M_\ssq$, we write 
\begin{align*}
N_\ssq^\mathrm{an} \coloneqq \bigcup_{i=1}^k N(\mathcal{A}_i), \qquad
N_\ssq^\mathrm{win} \coloneqq \bigcup_{j=1}^\ell N(\mathcal{W}_j), \qquad
N_\ssq \coloneqq N_\ssq^\mathrm{an} \cup N_\ssq^\mathrm{win}.
\end{align*}
We also write
\begin{align*}
    N_\partial(\mathcal{W}_j) \coloneqq N(\mathcal{W}_j) \cap \partial_* M_\ssq, \qquad N^\mathrm{win}_\partial \coloneqq \bigcup_{j=1}^\ell N_\partial^\mathrm{win}(\mathcal{W}_j) = N^\mathrm{win}_\ssq \cap \partial_* M_\ssq.
\end{align*}

\begin{defn} \label{def:adapted}
    Let $\F_\ssq$ be a cooriented $C^0$-foliation on a manifold with fences and windows $M_\ssq$. We say that $\F_\ssq$ is \textbf{adapted} to the boundary of $M_\ssq$ if it is transverse to $\partial M_\ssq$, and there exist pairwise disjoint standard neighborhoods and coordinates of the fences and windows such that
    \begin{itemize}
        \item In $N(\mathcal{A}_i)$, $\F_\ssq$ is tangent to $\partial_y$, positively transverse to $\partial_z$ (i.e., horizontal), and it is
        \begin{itemize}
            \item Positively transverse to $\partial_x$ in $S^1 \times [-1,2] \times [1,2)$
            \item Negatively transverse to $\partial_x$ in $S^1 \times [-1,2] \times (-1,0]$,
        \end{itemize}
        
        \item If $\mathcal{W}_j$ is a square window, then in $N(\mathcal{W}_j)$, $\F_\ssq$ is the standard horizontal foliation by $(x,y)$-planes. In particular, it is smooth there. If $\mathcal{W}_j$ is a toroidal window, we only require that $\mathcal{F}_\ssq$ is tangent to $\partial_y$ in $N(\mathcal{W}_j)$.
    \end{itemize}
\end{defn}

Using techniques from~\cite[Sections 3--4]{B16}, it is furthermore possible to apply an arbitrarily $C^0$-small isotopy of $\F_\ssq$ near the annular boundary fences to make them \emph{very nice}, as well as \emph{smooth} in a neighborhood of 
\begin{align*}
    N^h(\mathcal{A}_i) \coloneqq S^1 \times [0,1] \times \big( [-1,0] \cup [1,2]\big).
\end{align*}
We also write
\begin{align} \label{eq:horizontalnbd}
    N^h_\ssq \coloneqq \bigcup_{i=1}^k N^h(\mathcal{A}_i).
\end{align}

We will need one last definition: 

\begin{defn} \label{defn:fencetransitive}
    $\F_\ssq$ is called \textbf{$\partial_\ssq$-transitive} if every point in $M_\ssq$ can be joined to a boundary fence or the interior of a square or toroidal window by a smooth curve tangent to $\F_\ssq$.
\end{defn}

In what follows, we will always assume that $C^0$-foliations on manifolds with fences and windows are cooriented.

            \subsubsection{Improving \texorpdfstring{$\partial_\ssq$}{partialsq}-transitive foliations}

The next technical result allows us to improve the boundary foliation along the genuine boundary components of a manifold with fences and windows in a controlled way, provided that the foliation is $\partial_\ssq$-transitive. More precisely, we first allow the boundary foliation to become slightly worse inside of the windows, and then improve it everywhere. The modification is modeled as follows:

\begin{defn}
     Let $\G$ denote the standard horizontal foliation on the neighborhood $N(\mathcal{W}) = [-1,2]^2_{x,z}$ of the square window $\mathcal{W} = [0,1]^2$, i.e., $\G$ is tangent to $\partial_x$. Let $\star \in \{ \sharp, \flat\}$. A \textbf{$\star$~modification of $\G$ in $N(\mathcal{W})$} is a smooth foliation $\G^\star$ on $N(\mathcal{W})$ satisfying
     \begin{itemize}
        \item At every point, $T\G^\star$ either coincides with $T \G$, or is positively transverse to it if $\star = \sharp$ or negatively transverse to it if $\star = \flat$,
         \item $\G^\star$ coincides with $\G$ near $\partial N(\mathcal{W})$,
         \item $\G^\star$ is transverse to $\G$ in a neighborhood of $\mathcal{W}$.
     \end{itemize}
\end{defn}

Such a modification can be obtained by considering a modification $\partial_x + f(x,z) \partial_z$ of $\partial_x$, for some function $f$ which vanishes near $\partial N(\mathcal{W})$, is nonnegative (resp.~nonpositive), and is positive (resp.~negative) near $\mathcal{W}$. This function can be chosen arbitrarily $C^\infty$ small, so the resulting modification $\G^\star$ can be made arbitrarily $C^\infty$ close to $\G$.

\begin{prop} \label{prop:betterfence}
    Let $M_\ssq$ be a manifold with fences and windows, and let $\F_\ssq$ be a $C^0$-foliation on $M_\ssq$ adapted to $\partial M_\ssq$ and $\partial_\ssq$-transitive. Furthermore, let $\partial^\flat \F_\ssq$ (resp.~$\partial^\sharp \F_\ssq$) be an arbitrarily small $\flat$ (resp.~$\sharp$) window modification of $\partial \F_\ssq$.
    
    Then $\F_\ssq$ can be $C^0$-approximated by $C^0$-foliations on $M_\ssq$, still adapted to $\partial M_\ssq$, which are positively (resp.~negatively) transverse to foliations isotopic to $\partial^\flat \F_\ssq$ (resp.~$\partial^\sharp \F_\ssq$) along the facades in $\partial_* M_\ssq$.
\end{prop}

Note that we do not impose any control along the toroidal boundary windows beyond the $C^0$-control. The main tool in the proof is an adaptation of the \emph{smoothing ribbons} from~\cite{B16}.

\paragraph{$\partial$-ribbons.} We fix $M_\ssq$ and $\F_\ssq$ as in the statement of \zcref{prop:betterfence}. We also consider an auxiliary smooth vector field $Z$ transverse to $\F_\ssq$, tangent to $\partial_* M_\ssq$ and which coincides with $\partial_z$ in the standard neighborhoods of the fences.

\begin{defn}[$\partial$-ribbons]
    A \textbf{system of $\bm \partial$-ribbons} adapted to $Z$ is a finite collection of pairwise disjoint smoothly embedded strips $R_m = \sigma_m \times [-1, 1] \subset M_\ssq$ such that
    \begin{itemize}
        \item The arcs $\sigma_m \times \{\pm1\}$ are tangent to the foliation $\F_\ssq$, and each ribbon is transverse to $\F_\ssq$ and tangent to $Z$,
        
        \item For the initial point $p_m$ of $\sigma_m$, the interval $p_m \times [-1,1]$ lies on $\partial_* M_\ssq$ away from $N_\partial^\mathrm{win}$, and for the end point $q_m$, the interval $q_m \times [-1,1]$ lies either in the interior of a vertical side of an annular fence, or in the interior of a window. Moreover, $R_m$ intersects $\partial M_\ssq$ along these intervals only, and transversally. 
        
        We further require that the ribbons are tangent to $\partial_y$ inside of $N_\ssq$.
    \end{itemize}

    This system is \textbf{$\bm \partial$-full} if for each point $p \in \partial_* M_\ssq$, the leaf of $\partial \F_\ssq$ passing through $p$ intersects interiors of initial intervals of $\partial$-ribbons, or interiors of windows, on both sides of $p$.
\end{defn}

If $R_m = \sigma_m \times [-1,1]$ is such a $\partial$-ribbon, we denote by $\partial^h R_m = \sigma_m \times \{-1, 1\}$ its horizontal boundary and by $\partial^v R_m = \{p_m, q_m\} \times [-1,1]$ its vertical boundary.

The $\partial_\ssq$-transitivity property allows us to connect points on $\partial_* M_\ssq$ to the fences and windows along the leaves of $\F_\ssq$. As a consequence, we have:

\begin{lem} \label{lem:fullribbon}
    If $\F_\ssq$ is $\partial_\ssq$-transitive, then there exists a $\partial$-full system of $\partial$-ribbons for $\F_\ssq$ adapted to $Z$.
\end{lem}

\begin{proof}
    The proof closely follows the strategy from~\cite[Section 7]{B16}, which is itself based on Vogel's article~\cite{V16}; we only highlight the main ideas and refer the reader to these articles for further details.

    The first step is to construct a finite collection of (not necessarily disjoint) $\partial$-ribbons satisfying the listed properties using a compactness argument, by joining points on the facade components of $\partial_* M_\ssq$ to the fences and windows along smooth curves tangent to $\F_\ssq$, and using the flow of $Z$ to create ribbons passing through those curves. Then, these ribbons can be made pairwise disjoint by making them transverse and resolving the intersections as explained in~\cite[Section 7]{B16}. During this last operation, the $\partial$-fullness property is preserved.
\end{proof}

\begin{figure}
    \centering
    \def\svgwidth{0.95\textwidth}
    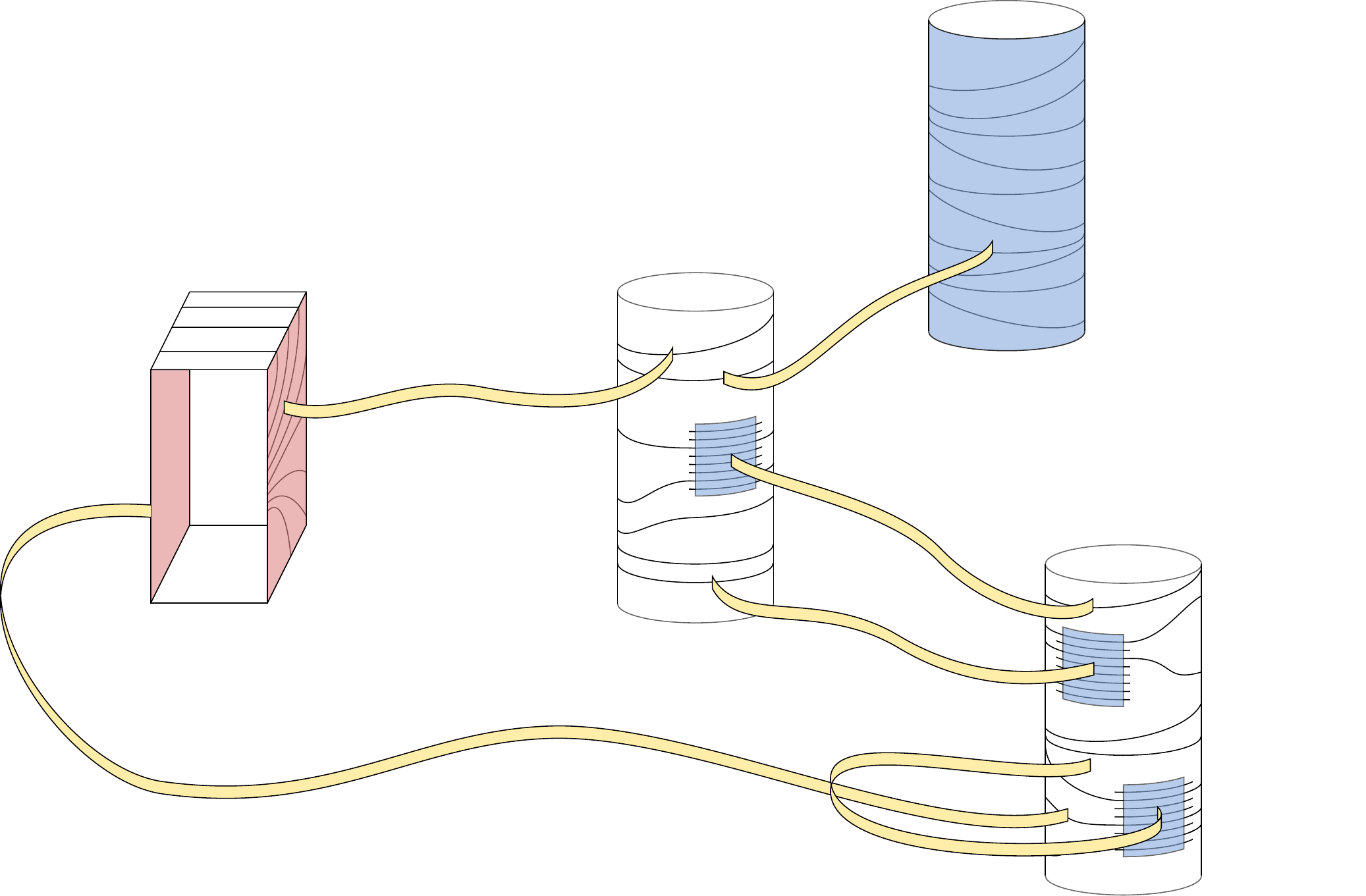
    \caption{Ribbons (yellow), fences (red), windows (blue).}
    \label{fig:fencesandwindows}
\end{figure}

\paragraph{Proof of \zcref{prop:betterfence}.}

Let $M_\ssq$ and $\F_\ssq$ as in the statement of the proposition, and let $\mathcal{R} = \{R_m\}_m$ be a $\partial$-full system of ribbons provided by \zcref{lem:fullribbon}. 

Let $\partial^\flat \F_\ssq$ be a small $\flat$ window modification of $\partial \F_\ssq$, so that $\partial \F_\ssq$ is positively transverse to $\partial^\flat \F_\ssq$ inside of the square windows. 

We can proceed as in Step 3 of~\cite{B16} and smoothen $\F_\ssq$ in small pairwise disjoint neighborhoods of the ribbons, via an arbitrarily $C^0$-small isotopy of $\F_\ssq$. The resulting $C^0$-foliation, denoted by $\widetilde{\F}_\ssq$, is still adapted to the fences and windows. Along the facades of $\partial_* M_\ssq$, these modifications happen away from $N_\partial^\mathrm{win}$ (since the foliation is already smooth in $N_\ssq^\mathrm{win}$ so the foliation can be preserved there), so we may perform them on $\partial^\flat \F_\ssq$ as well and we denote the resulting boundary foliation by $\partial^\flat \widetilde{\F}_\ssq$.

Let us consider a $\partial$-ribbon $R_m \in \mathcal{R}$, together with a small (closed) neighborhood of the form $N(R_m) = [-\delta, \delta] \times \sigma_m \times [-1-\delta, 1+\delta]$ for some small $\delta > 0$, such that $\widetilde{\F}_\ssq$ is a product foliation in $N(R_m)$. More precisely, we write $N(R_m) =  [-\delta, \delta]_x \times \sigma_m \times [-1-\delta, 1+\delta]_z \cong \sigma_m \times \mathsf{R}_\delta$ so that $\widetilde{\F}_\ssq$ is tangent to $[-\delta, \delta]_x \times \sigma_m  \times \{z\}$ for $z \in [-1-\delta, 1+\delta]$. On the rectangle $\mathsf{R}_\delta = [-\delta, \delta] \times [-1-\delta, 1+\delta]$, we consider a smooth foliation $\G_0$ satisfying:
\begin{itemize}
    \item At every point, $T\G_0$ is either tangent or positively transverse to the standard horizontal foliation,
    \item Near $\partial \mathsf{R}_\delta$, $\G_0$ coincides with the standard horizontal foliation,
    \item In $\mathsf{R}_{\delta/2} \subset \mathsf{R}_\delta$, $\G_0$ is positively transverse to the standard horizontal foliation.
\end{itemize}
Furthermore, we may arrange that $\G_0$ is arbitrarily $C^0$-close to the standard horizontal foliation. We may now modify $\widetilde{\F}_\ssq$ by replacing it with $\sigma_j \times \G_0$ inside of $N(R_m)$. Performing this operation along all the $\partial$-ribbons, we obtain a new $C^0$-foliation $\widetilde{\F}^+_\ssq$ which is arbitrarily $C^0$-close to $\F_\ssq$, and such that along $\partial_* M_\ssq$, $\partial \widetilde{\F}^+_\ssq$ is mostly (positively) transverse to $\partial^\flat \widetilde{\F}_\ssq$. Moreover, choosing these modifications small enough ensures that $\widetilde{\F}^+_\ssq$ is still adapted to the boundary, after smooth change of coordinates in the neighborhoods of windows.

\begin{figure}[ht]
    \centering
    \def\svgwidth{0.6\textwidth}
    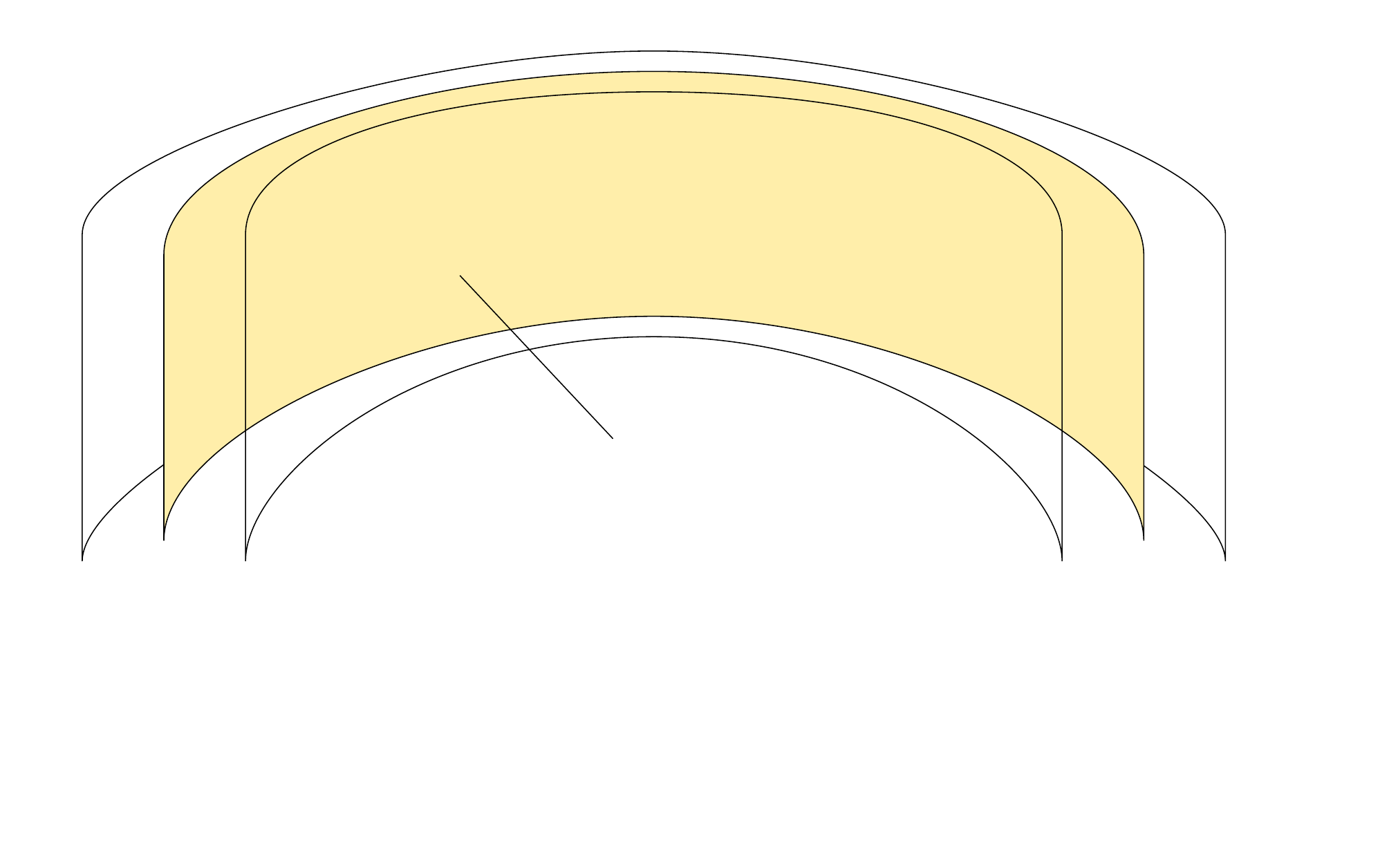
    \caption{Twisting along a thickened ribbon.}
    \label{fig:twisting_ribbon}
\end{figure}

Finally, we apply \zcref{lem:mostlytrans,lem:isotopy_extension} to modify $\widetilde{\F}^+_\ssq$ near $\partial_* M_\ssq$ to make it positively transverse to a foliation which is isotopic to $\partial \widetilde{\F}^\flat_\ssq$, and which is itself isotopic to $\partial \F^\flat_\ssq$, via $C^0$-small isotopies. While the boundary foliation is slightly modified in the standard neighborhood of the windows and might not be smooth anymore, it can easily be corrected by a $C^0$-small isotopy preserving the transversality condition, and one easily constructs new coordinates near the windows in which the resulting foliation is horizontal. The resulting $C^0$-foliation satisfies all the desired requirements and the proposition is proved in the ``positive'' case. The ``negative'' case is obtained similarly. \qed

            \subsubsection{Approximating \texorpdfstring{$\partial_\ssq$}{partialsq}-transitive foliations}

We restate one of the main results in~\cite{B16} on the approximation of \emph{transitive} foliations in our current setting. The main difference is that in~\cite{B16}, the author does not drill out the annular fences and does not consider boundary components nor windows. However, the overall strategy works \emph{mutatis mutandis}.

\begin{prop}[\cite{B16}] \label{prop:contapprox}
    Let $M_\ssq$ be a manifold with fences and toroidal windows,\footnote{We do not technically need square windows for that statement, as we do not impose any control over the boundary foliation beyond a tangential $C^0$-control.} and let $\F_\ssq$ be a $C^0$-foliation on $M_\ssq$ adapted to the boundary and $\partial_\ssq$-transitive. Then $\F_\ssq$ is $C^0$-approximated by positive and negative contact structures.
\end{prop}

\begin{proof}
    We closely follow Step 3 from~\cite[Section 8]{B16} and explain how the strategy needs to be modified in our setup. The main difference is that in~\cite{B16}, the author considers annular fences $A_i$ which are embedded annuli in the manifold, and considers a slightly smaller annulus $A'_i \subset A_i$ as well as standard neighborhoods $N(A'_i) \subset N(A_i)$. In our setting, the reader shall imagine that the neighborhood $N(A'_i)$ has been removed from the manifold, and the neighborhood of the corresponding boundary fence $N(\mathcal{A}_i)$ shall be identified with $\overline{N(A_i) \setminus N(A'_i)}$. Recall that $\F_\ssq$ is assumed to be tangent to $\partial_y$ in $N(\mathcal{A}_i)$, so it is essentially a product foliation there, except that the two induced foliations on the vertical sides of the fence might differ. Recall that in the chosen standard neighborhoods of the windows, the foliation is the standard foliation by horizontal planes.

    The vertical parts of the annular boundary fences may be made very nice as in~\cite[Section 4]{B16}, via some $C^0$-small isotopies.

    We may treat all the boundary components of $M_\ssq$ which are not fences as toroidal boundary components. We may extend $M_\ssq$ slightly along  $\partial^\mathrm{tor}_\ssq M_\ssq$, by adding a small collar and extending $\F_\ssq$ by a product there (after some arbitrarily small $C^0$-isotopy near $\partial^\mathrm{tor}_\ssq M_\ssq$ to make it tangent to a smooth vector field transverse to $\partial^\mathrm{tor}_\ssq M_\ssq$). We denote the extended manifold by $M^+_\ssq$, and we treat the new boundaries as toroidal windows.
    
    Proceeding as in the first paragraph of Step 3 of~\cite[Section 8]{B16}, we consider a sufficiently fine triangulation of a complement of a small neighborhood of $\partial M^+_\ssq$ which covers $\partial N_\ssq$ and $\partial^\mathrm{tor}_\ssq M_\ssq$, in general position with respect to $\F_\ssq$ and with the chosen coordinates in the neighborhood of the fences and windows. This triangulation can then be modified into a polyhedral decomposition as in~\cite[Lemma 6.3]{B16}. 

    We then smoothen $\F_\ssq$ near the $1$-skeleton exactly as in the second paragraph of Step 3 of~\cite[Section 8]{B16}.

    Now, we can choose a full collection of smoothing ribbons as in the third paragraph of Step 3 of~\cite[Section 8]{B16}, with the only difference that the ribbons in~\cite{B16} are technically allowed to cross the regions $N(A'_i)$ which are missing in our setup, and we also consider ribbons ending in the (toroidal) windows. However, the $\partial_\ssq$-transitivity condition ensures that we can still connect every point in $M_\ssq$ to the interior of a vertical annulus in $\partial N_\ssq^\mathrm{an}$, or to the interior of a window as before, and the construction of a full collection of ribbons as in~\cite[Definition 7.1]{B16} goes through. We may further arrange that these ribbons are tangent to $\partial_y$ in the standard neighborhoods of the annular fences.

    The rest of Step 3 of~\cite[Section 8]{B16} goes through essentially verbatim, as all the modifications happen away from $N(A'_i)$, so they can be performed on $M^+_\ssq$ in our setup. We may start the contact approximation by adding contactness in neighborhoods of $N^h_\ssq$ defined by \zcref{eq:horizontalnbd}, so that the contact structure there is tangent to $\partial_y$. We then apply twisting along the ribbons to create contactness, ``correct'' the holonomy around the polyhedra, and fill them with appropriate contact structures. Notice that some modification also happens along $\partial^\mathrm{tor}_\ssq M^+_\ssq$, but this is not relevant for our construction as we only want a contact approximation of $\F_\ssq$ on $M_\ssq$ without fine control at the boundary.
    
    We obtain a (positive) contact structure $\xi$ defined on a neighborhood of
    $$\left(M_\ssq \setminus N_\ssq \right) \cup N^h_\ssq \subset M^+_\ssq,$$
    which is moreover tangent to $\partial_y$ in $N^\mathrm{an}_\ssq$, and which is arbitrarily $C^0$-close to $T \F_\ssq$. This contact structure may be extended to $M_\ssq$ by applying a very small twisting along $\partial_y$. This concludes the proof.
\end{proof}

        \subsection{Coarse version}

In this section, we adapt the strategy of Bowden~\cite{B16} to the case of manifolds with boundary and prove:

\begin{thm}[Coarse Eliashberg--Thurston with boundary] \label{thm:coarseET}
    Let $\mathcal{F}$ be a cooriented $C^0$-foliation on $M$ transverse to $\partial M$. Assume that $\F$ is neither homeomorphic to the standard foliation by disks on $S^1 \times D^2$, nor to a foliation on $\T^2 \times I$ semiconjugate to an irrational foliation. 
    
    Then for every $\varepsilon > 0$, there exists a positive (resp.~negative) contact structure which is $\varepsilon$-$C^0$-close to $T \F$ and which dominates a $1$-dimensional $C^0$-foliation $\G$ on $\partial M$ which is monotone equivalent to $\partial \mathcal{F}$.
\end{thm}

In this theorem, the contact approximations might not dominate $\partial \F$ itself, but they dominate a foliation $\G$ on $\partial M$ obtained from $\partial \F$ by small Denjoy blowups and blowdowns, and $C^0$-small isotopies. Notice that $\G$ depends on $\varepsilon$ a priori. Nevertheless, this readily implies:

\begin{cor}[Boundary multislope] \label{cor:coarseETslope}
    Under the hypothesis of \zcref{thm:coarseET}, if $\F$ has a well-defined boundary multislope $\bm{s}$, then it is approximated by positive and negative contact structures realizing the same boundary multislope $\bm{s}$.
\end{cor}

As mentioned before, this coarse version \emph{does not} allow us to control the precise \emph{boundary types} of the contact approximations, only their multislopes. However, this is enough for most of our applications.

\medskip

We note that the hypotheses on $\F$ in \zcref{thm:coarseET} are necessary for the purpose of approximating foliations with contact structures:

\begin{lem}\label{lem:nec_hyp}
    \begin{enumerate}
        \item The standard foliation by disks on $S^1 \times D^2$ cannot be $C^0$-approximated by contact structures with boundary slope $0$ for the standard boundary framing.
        
        \item A linear irrational foliation $\F$ on $\T^2 \times I$ cannot be $C^0$-approximated by contact structures realizing the same boundary slopes as $\F$. 
    \end{enumerate}
\end{lem}

\begin{proof}
    We first argue that there is no contact structure on $S^1 \times D^2$ transverse to the $S^1$-direction and with slope $0$ along $\partial(S^1 \times D^2)$. Indeed, such a contact structure would have a closed characteristic along $\partial(S^1 \times D^2)$ which would be a Legendrian curve bounding an embedded disk, implying that it is overtwisted. However, a horizontal contact structure on $S^1 \times D^2$ is necessarily tight, see for instance~\cite[Theorem 2.27]{V16}.
        
    More generally, if $\F$ is a $C^1$-foliation by disks on $S^1 \times D^2$, then there exists a \emph{smooth} diffeomorphism $\varphi : S^1 \times D^2 \rightarrow S^1 \times D^2$ such that $\varphi(\F)$ is transverse to the $S^1$-direction. Indeed, there exists a $C^1$ diffeomorphism sending $\F$ to the standard foliation by disks, and any smoothing of it satisfies the desired conditions.\footnote{Note that this step might fail for $C^0$-foliations!}

    To prove the second item, we may assume that $\F$ is a product $\G \times I$, where $\G$ is an irrational foliation on $\T^2 = S^1 \times S^1$. In particular, it is transverse to the first $S^1$ factor. We now assume by contradiction that such an approximating contact structure exists. Then, ``trimming'' the boundary yields a contact structure $\xi$ with strictly better boundary slopes, which we can arrange to be rational and realized by \emph{linear} characteristic foliations. Then, $\F$ can be approximated by a smooth foliation by cylinders $\mathcal{C}$ transverse to the $S^1$ direction. After passing to a finite cover and applying a suitable diffeomorphism, we may assume that $\mathcal{C}$ coincides with the foliation $\{\mathrm{pt}\} \times S^1 \times I$ on $\T^2 \times I$, that $\xi$ is still transverse to the first $S^1$ component, and that $\xi$ has positive slope along each boundary component of $\T^2 \times I$. We may fill one boundary component with a solid torus and extend $\xi$, so that the resulting manifold is diffeomorphic to $S^1 \times D^2$, and the resulting contact structure is transverse to the $S^1$ direction and has positive slope along the boundary. We may twist it along the boundary to achieve slope $0$, which contradicts the previous item.
\end{proof}

\begin{rem}
    Perhaps surprisingly, the previous lemma fails for $C^0$-foliations. Indeed, there exist $C^0$-foliations by disks on $S^1 \times D^2$ which are approximated by contact structures; however, our proof shows that those are necessarily overtwisted. In fact, one can construct a $C^0$-foliation by spheres on $S^1 \times S^2$ which \emph{can} be approximated by positive and negative contact structures (although this is impossible for $C^1$-foliations). The idea is to start with the standard foliation by spheres, and consider a transverse torus $T$ given by the product of $S^1$ the equator of $S^2$. Then, one may modify the foliation near this torus in order to insert a \emph{phantom Reeb component} as in~\cite{CKR19}; the resulting $C^0$-foliation $\F$ is still a foliation by sphere, but $T \F$ is tangent to $T$, while $T$ is \emph{not} a leaf of $\F$. Furthermore, $\F$ is approximated by foliations on $S^1 \times S^2$ with two Reeb components glued together along their boundaries. In particular, this foliation may be approximated by positive and negative contact structures. Puncturing the resulting contact structures along one or two $S^1$ factors yield approximating contact structures to $C^0$-foliations by disks on $S^1 \times D^2$, and to $C^0$-foliations by cylinders on $\T^2 \times I$. We may use a similar strategy to modify an irrational foliation on $\T^2 \times I$ by inserting a phantom Reeb component, so that it can be approximated by $C^0$-foliations with the same boundary foliations but with a torus leaf. Then, \zcref{thm:coarseET} applies to the latter foliations and provides contact approximations with the desired boundary slopes.

    However, note that the above foliations with phantom Reeb components may also be $C^0$-approximated (via a small $C^0$-isotopy) by \emph{smooth} foliations (by disks on $S^1 \times D^2$, and irrational on $\T^2 \times I$), which cannot be approximated by contact structures with the required boundary slopes. In some sense, those pathological $C^0$-foliations are not ``stably approximable'' by contact structures.
\end{rem}

            \subsubsection{Making compact leaves isolated}

    We adapt the argument from Step 1 in~\cite[Section 8]{B16} to the case with boundary. We state a version that will also be useful later when we discuss the fine version.

\begin{lem} \label{lem:finiteleaves}
    Let $\F$ be a cooriented $C^0$-foliation on $M$ transverse to $\partial M$. Assume that no leaf of $\F$ is a disk. Then it can be $C^0$-approximated by $C^0$-foliations which have finitely many compact leaves, and whose boundary foliations are monotone-equivalent to $\partial \F$. 
    
    If $\F$ has finitely many compact genus $0$ leaves, we can further arrange that the new boundary foliation is obtained from $\F$ by some blowdowns followed by \emph{trivial} blowups, supported away from the boundaries of compact genus $0$ leaves and $\T^2$-type planar minimal sets.
\end{lem}

\begin{proof}
    As in Step 1 of~\cite[Section 8]{B16}, there is a finite collection of (smooth) embeddings $N_k = \Sigma_k \times [0, c_k] \hookrightarrow M$, where $c_k \geq 0$ and $\Sigma_k$ is a compact surface (possibly with boundary) so that each $N_k$ is a foliated $I$-bundle, $\partial_hN_k \coloneqq \Sigma_k \times \{0, c_k\}$ are leaves of $\F$, and every compact leaf of $\F$ is contained in some $N_k$. Furthermore, any two $N_k$'s only intersect along their horizontal boundaries $\partial_h N_k$. We can assume that these $N_k$'s are arbitrarily thin, after suitable subdivisions. We shall distinguish three cases depending on the topology of $\Sigma_k$:
        \begin{enumerate}
            \item $\Sigma_k$ is closed, in which case it has positive genus by Reeb stability (recall that $\partial M \neq \varnothing$ is a union of tori),
            \item $\Sigma_k$ has nonempty boundary and positive genus,
            \item $\Sigma_k$ has nonempty boundary and genus $0$.
        \end{enumerate}
    
    The first case is handled exactly as in Step 1 of~\cite[Section 8]{B16}: replace the interior of $N_k$ with a $C^0$-foliation without closed leaves. This operation does not modify the boundary foliation.

    For the second case, we proceed similarly but the boundary foliation is modified. Using an elementary commutator trick, we can modify it in a controlled way and realize a \emph{linear} boundary foliation (i.e., a foliation by closed circles) on $N_k$. Indeed, it is easy to construct a fixed-point free representation $\pi_1(\Sigma_k) \to \mathrm{Homeo}^+\big((0,1)\big)$ which restricts to a trivial action along $\partial \Sigma_k$, by constructing such an action of $\Z^2 \cong \pi_1(\T^2)$ first (it suffices to pick a single homeomorphism of $(0,1)$ without fixed points), attaching trivial handles and puncturing it. 

    Finally, in the last case, we are forced to modify the boundary foliation in a less controlled way. However, we can achieve that the new boundary foliation in $N_k$ has one $\sharp$ annulus, one $\flat$ annulus, and all the other boundaries have $\natural$ annuli.

    In all those cases, the modification of the boundary foliation $\partial \F$ can be described as a blowdown immediately followed by a blowup. In case 2, the blowup is moreover trivial.
    \end{proof}

            \subsubsection{Inserting holonomy}

We now adapt the arguments from Step 2 of~\cite[Section 8]{B16}.

\begin{defn}[Infinitesimal holonomy]
    Let $\F$ be a cooriented $C^0$-foliation on $M$, and $\gamma : S^1 \rightarrow L$ be a smooth closed curve tangent to a leaf $L$ of $\F$, inducing a holonomy map 
    $$h : (-\varepsilon, \varepsilon) \rightarrow (-1,1)$$
    for some $\varepsilon > 0$. We say that $\F$ has \textbf{infinitesimal attracting holonomy} around $\gamma$ if for every open interval $\{0\} \subset I \subset (-\varepsilon, \varepsilon)$, there exists a smaller open interval $\{0\} \subset J \subset I$ such that 
    $$\overline{h(J)} \subset J.$$
\end{defn}

Technically speaking, the holonomy map $h$ is only defined as a germ and up to conjugation, but the notion of having infinitesimal attracting holonomy is well-defined and only depends on the homotopy class of $\gamma$ inside of $L$. This notion is equivalent to the one of \emph{sometimes attracting holonomy} from~\cite[Section 5]{B16}, but we prefer to use a different terminology.

\begin{defn} \label{def:enoughholo}
    Let $\F$ be a cooriented $C^0$-foliation on $M$ transverse to $\partial M$ and with finitely many compact leaves. We say that $\F$ has \textbf{enough (infinitesimal) holonomy} if each of its minimal sets has a curve with infinitesimal attracting holonomy.
\end{defn}

Similarly as in~\cite{B16} in the closed case, we can always modify a foliation $\F$ on $M$ by some arbitrarily small $C^0$-perturbation so that the resulting foliation admits plenty of curves with attracting holonomy:

\begin{lem} \label{lem:enoughholo}
    Let $\F$ be as in the setup of \zcref{def:enoughholo}. Moreover, assume that $\F$ is not semiconjugate to an irrational foliation on $\T^2 \times I$. Then $\F$ can be $C^0$-approximated by $C^0$-foliations with enough holonomy, and whose boundary foliations are monotone-equivalent to $\partial \F$.
\end{lem}

\begin{proof}
We closely follow the strategy of Step 2 of~\cite[Section 8]{B16}. 

Let $\Lambda$ be a minimal set of $\F$. By \zcref{cor:simplyconnleaves} and our assumptions, $\Lambda$ has a leaf $L_0$ which is not simply connected. Let $\gamma \subset L_0$ be an embedded homotopically nontrivial curve with a corresponding holonomy map $h : (-\varepsilon, \varepsilon) \rightarrow (-1,1)$, where $h(0)=0$ corresponds to $\gamma$.

If $h$ has a unique fixed point (after possibly shrinking $\varepsilon$) or has germinally nontrivial holonomy on both sides of $L_0$, we proceed exactly as in \textit{Case 1} and \textit{Case 2} in~\cite{B16} and perform an appropriate blowup of $L_0$ to create infinitesimally attracting holonomy around two parallel copies of $\gamma$. Note that the number of minimal sets may increase in this procedure. If $L_0$ intersects $\partial M$, the new boundary foliation is a blowup of $\partial \F$, and it remains unchanged otherwise.

The third and remaining case, when $h$ has a nontrivial interval of fixed points containing $0$ (potentially on its boundary), is slightly more subtle.

Following~\cite{B16}, we consider the leaves near $L_0$ intersecting a neighborhood of $\gamma$. If the curves parallel to $\gamma$ in those leaves are all leafwise homotopically trivial, we claim that $L_0$ is either a torus or a cylinder. Indeed, we may consider the doubling of $\F$ along $\partial M$, together with (a connected component of) the doubling $\widehat{L}_0$ of $L_0$ whose nearby leaves satisfy a similar property. Using~\cite[Proposition 2.9]{B16}, we deduce that $\widehat{L}_0$ is a torus, which implies our claim. Moreover, the holonomy assumption on $h$ prevents $L_0$ from being an annulus since it is an isolated compact leaf. Therefore, $L_0$ is an isolated torus leaf and must admit another homotopically nontrivial curve with nontrivial holonomy \emph{on both sides}; this falls into the previously considered cases. 

We may now blow up countably many leaves near $L_0$ to create infinitesimal attracting holonomy around $\gamma$. If $\F$ is not minimal, we choose these leaves away from the (finitely many) minimal sets of $\F$, so that no new minimal set is created. Otherwise, if $\F$ is minimal, one blowup creates germinally nontrivial holonomy around $\gamma$ on both sides, and we may then blowup $L_0$ as in the previous cases. Once again, the boundary foliation is only modified by some monotone-equivalence, and these blowups can be performed so that the new foliation is arbitrarily $C^0$-close to $\F$.
\end{proof}

\begin{rem}
    We will need to be much more careful later in the proof of the \emph{fine} version of the Eliashberg--Thurston theorem with boundary, as we won't be able to perform arbitrary blowups that might alter the boundary foliation in uncontrolled ways.
\end{rem}

            \subsubsection{Proof of \texorpdfstring{\zcref{thm:coarseET}}{thm:coarseET}}

Let $\F$ be a $C^0$-foliation on $M$ as in the statement of \zcref{thm:coarseET}. After applying \zcref{lem:finiteleaves} and \zcref{lem:enoughholo}, we obtain another $C^0$-foliation $\widetilde{\F}$ which is $C^0$-close to $\F$, whose boundary $\partial \widetilde{\F}$ is monotone-equivalent to $\partial \F$, and which has finitely many compact leaves and enough holonomy. 

We may now create annular fences as in Step 2 of~\cite[Section 8]{B16}, which are embedded annuli $A_1, \dots, A_k$ inside of $M$ and disjoint from the boundary, constructed near holonomy curves. We may further assume that the heights of these annuli are very small, since the holonomy curves we consider have infinitesimal attracting holonomy. We can then find very nice neighborhoods of the $A_i$'s as in~\cite[Section 4]{B16}, and remove slightly smaller neighborhoods to create annular boundary fences as in \zcref{def:manfence}. We denote by $M_\ssq$ the resulting manifold with fences (and without windows). The induced $C^0$-foliation $\widetilde{\F}_\ssq$ on $M_\ssq$ is adapted to the boundary of $M_\ssq$ by definition, and coincides with $\widetilde{\F}$ along $\partial_* M_\ssq$. We can now apply \zcref{prop:betterfence} to obtain a new $C^0$-foliation $\widetilde{\F}^+_\ssq$ on $M_\ssq$, $C^0$-close to $\widetilde{\F}_\ssq$, whose (genuine) boundary is moreover positively transverse to a $C^0$-foliation $\G$ on $\partial M$ which is isotopic to $\partial \widetilde{\F}$. We can further apply \zcref{prop:contapprox} to obtain a positive contact structure on $M_\ssq$, which is arbitrarily $C^0$-close to $\widetilde{\F}^+_\ssq$; in particular, we can arrange that it is still positively transverse to $\G$ along $\partial_* M_\ssq$. Finally, we can fill the annular holes using~\cite[Lemma 5.6]{B16} to obtain a contact structure on $M$ with the desired properties along $\partial M$. Notice that in that last step, the distance between the contact structures and $\F$ in the filled annular holes depends on the height on the annuli, which was already chosen to be very small. \qed

        \subsection{Fine version}

In this section, we prove the second item of \zcref{thmintro:ET} from the introduction, namely, the ``fine'' version of the Eliashberg--Thurston theorem with boundary. We will actually prove a slightly more general result which allows compact planar leaves and $\T^2$-type planar minimal sets, at the expense of making the boundary foliation slightly ``worse'' near some boundary components of those sets. This controlled modification is explained in the following definition.

\begin{defn}[$\sharp$ and $\flat$ modifications]
    Let $\G$ be a cooriented $C^0$-foliation on the $2$-torus $\T^2$, and let $p \in \T^2$. We choose a smooth closed arc $c$ transverse to $\G$ and containing $p$ in its interior. After some arbitrarily small $C^0$ isotopy, we may assume that $\G$ is smooth in a neighborhood of $c$, and we choose coordinates $N(c) \cong (-\varepsilon, \varepsilon)_x \times (-2,2)_y$ near $c$, in which $c$ is identified with $\{0\} \times [-1,1]$, $p=(0,0)$, and $\G$ is tangent to $\partial_x$. 

    Let $h : [-1,1] \rightarrow [-1,1]$ be a smooth diffeomorphism satisfying 
    \begin{itemize}
        \item $h \geq \mathrm{id}$ (resp.~$h \leq \mathrm{id}$),
        \item $h = \mathrm{id}$ near $\partial [-1,1]$,
        \item $h > \mathrm{id}$ (resp.~$h < \mathrm{id}$) on $[-1/2, 1/2]$.
    \end{itemize}

    We modify $\G$ in $(-\varepsilon/2, \varepsilon/2) \times (-2,2) \subset N(c)$ so that the parallel transport along the leaves of $\G$ from left to right is $h$, and moreover the new $C^0$-foliation $\widetilde{\G}$ is everywhere either tangent or positively (resp.~negatively) transverse to $\partial_x$ in $N(c)$. Choosing $h$ sufficiently $C^0$-close to $\mathrm{id}$, we can further arrange that $T\widetilde{\G}$ is $C^0$-close to $T\G$. 
    
    We call the resulting $C^0$-foliation a $\sharp$ (resp.~$\flat$) modification of $\G$ near the point $p$. This modification can be made arbitrarily $C^0$-small and close to $p$.
\end{defn}

Concretely, a $\sharp$ (resp.~$\flat$) modification replaces $\G$ with a new foliation that has strictly higher (resp.~lower) slope in a neighborhood of $p$, without introducing negative (resp.~positive) transverse points.

\begin{rem} \label{rem:modif}
    If $p$ and $q$ lie on the same leaf of $\G$, then a sufficiently small $\sharp$ (resp.~$\flat$) modification of $\G$ near $p$ is isotopic to a $\sharp$ (resp.~$\flat$) modification of $\G$ near $q$, via a small $C^0$ isotopy. Indeed, we can first smooth $\G$ along a leaf segment containing $p$ and $q$, and work in appropriate smooth coordinates near there.
\end{rem}

\begin{defn}
    A compact planar leaf or a $\T^2$-type planar minimal set of $\F$ is \textbf{stable} if it has a closed boundary curve with infinitesimal attracting holonomy. Otherwise, it is called \textbf{unstable}.
\end{defn}

\begin{thm}[Finest version of Eliashberg--Thurston] \label{thm:fineET}
    Let $\mathcal{F}$ be a cooriented $C^0$-foliation on $M$ transverse to $\partial M$ with finitely many compact planar leaves. Let $L_1, \dots, L_m$ denote its unstable compact planar leaves, and let $\Lambda_1, \dots, \Lambda_k$ denote its unstable $\T^2$-type planar minimal sets. For each of those, choose a boundary point $p_i \in \partial L_i \subset \partial M$ and $q_j \in \partial \Lambda_j \subset \partial M$.
    
    Then $\mathcal{F}$ can be $C^0$-approximated by positive (resp.~negative) contact structures which dominate some foliation $\G$ on $\partial M$, obtained from $\partial \F$ by the following operations only:
    \begin{itemize}
        \item Isotopies,
        \item Blowdowns and blowups of type $\natural/\sharp$ (resp.~$\natural/ \flat$) away from the boundaries of the aforementioned minimal sets,
        \item $\flat$ (resp.~$\sharp$) modifications near $p_1, \dots, p_m, q_1, \dots, q_k$.
    \end{itemize}
    Moreover, all of these operations can be made arbitrarily small, so that $\G$ is tangentially $C^0$-close to $\partial \F$.
\end{thm}

We emphasize that the foliation $\G$ \emph{depends} on the size of the contact approximations; it cannot be chosen uniformly for all contact approximations a priori. In the absence of unstable compact planar leaves and $\T^2$-type planar minimal sets, we obtain:

\begin{cor} \label{cor:fineET}
    Let $\mathcal{F}$ be a cooriented $C^0$-foliation on $M$ transverse to $\partial M$ without compact planar leaves nor $\T^2$-type planar minimal sets. Then $\F$ can be $C^0$ approximated by positive and negative contact structures which dominate some foliation $\G$ on $\partial M$ which is isotopic to $\partial \F$, via a $C^0$-small isotopy.
\end{cor}

Under the hypothesis of the corollary, if $\F$ has a well-defined boundary type, then it can approximated by (positive and negative) contact structures with the \emph{same} boundary type, as in the second item of \zcref{thmintro:ET}.

\begin{rem}
    In general, unstable compact planar leaves and $\T^2$-type planar minimal sets constitute obstructions to finding dominating contact structures for $\partial \F$, up to isotopy, since we might not be able to add attracting holonomy near them, see \zcref{lem:enoughholo2} below. However, we do not know if they always constitute obstructions in practice: a foliation with such minimal sets could still be approximable by contact structures dominating its boundary foliation (or rather a different boundary foliation isotopic to it).
\end{rem}

Before diving into the proof, let us explain how to refine the arguments for the coarse version to this finer setting. The main difference is that we are only allowed to add holonomy near minimal sets in a controlled way as this operation impacts the boundary foliation. We need to distinguish five different types of minimal sets as in the flowchart in \zcref{fig:flowchart}. The most delicate ones to deal with are the \emph{$\T^3$-type planar minimal sets}, as the leaves have little topology. However, these enjoy a precise local model that can be used to modify them: we first drill a transverse curve, thus creating a new boundary component (the \emph{maypole}) along which the foliation has slope $0$, and then modify the foliation near the minimal set so that the new foliation has \emph{positive slope} along the maypole. We may then treat this boundary component as a toroidal window for the contact approximation, before filling it at the end.

\begin{figure}
    \centering
        \begin{tikzpicture}[node distance=1.8cm, auto]
  \tikzset{
    start/.style={diamond, draw, fill=gray!40, text centered, inner sep=3pt, font=\bfseries, align=center, minimum width=1cm, minimum height=1cm, aspect=2},
    type/.style={ellipse, draw, fill=gray!20, text centered, inner sep=3pt, align=center, minimum width=0cm, minimum height=1cm},
    action/.style={rectangle, draw, fill=blue!20, text centered, rounded corners, minimum height=2.5em, inner sep=3pt, align=center, text width=4cm},
    action2/.style={rectangle, draw, fill=blue!20, text centered, rounded corners, minimum height=2.5em, inner sep=3pt, align=center, text width=5cm},
    actionr/.style={rectangle, draw, fill=red!20, text centered, rounded corners, minimum height=2.5em, inner sep=3pt, align=center, text width=4cm},
    line/.style={draw, -Latex, shorten >=2pt, shorten <=2pt},
  }
  \node [start] (start) {Minimal set $\Lambda$};
  \node [type, above right of=start, xshift=1cm, yshift=2cm] (nonperipheral) {\textbf{Nonperipheral}};
  \node [action, right of=nonperipheral, xshift=4cm] (action1) {Perform blowups to add holonomy. The boundary is modified by \emph{controlled} blowups (\zcref{lem:enoughholo2}).};
  \node [type, below right of=start, xshift=1cm, yshift=-3cm] (peripheral) {\textbf{Peripheral}};
  \node [type, right of=peripheral, yshift=4cm, xshift=1cm] (spiral) {Spiral};
  \node [action, right of=spiral, xshift=4cm] (action2) {Add $\flat/\sharp$-modification at boundary; the new boundary is isotopic to the original one (\zcref{lem:spiral}).};
  \node [type, right of=peripheral, yshift=1.5cm, xshift=2cm] (t2) {$\mathbb{T}^2$-type planar};
  \node [type, right of=peripheral, yshift=-1.5cm, xshift=2cm] (genus0) {Compact planar};
  \node [actionr, right of=peripheral, xshift=7cm, yshift=0cm] (action3) {Add $\flat/\sharp$-modification along a boundary component (closed or not). The boundary \emph{is} modified!};
  \node [type, right of=peripheral, yshift=-5cm, xshift=1cm] (t3) {$\mathbb{T}^3$-type planar};
  \node [action2, right of=t3, xshift=4cm] (action4) {Drill a \emph{maypole} and ``twist'' the foliation near $\Lambda$ (\zcref{prop:maypoletwist}). Treat the maypole as an additional boundary component with \emph{positive} or \emph{negative} slope, and allow any $C^0$-modifications along it. Fill the maypole at the end (\zcref{lem:maypolefilling}).};
  \path [line] (start) -- (nonperipheral);
  \path [line] (nonperipheral) -- (action1);
  \path [line] (start) -- (peripheral);
  \path [line] (peripheral) -- (spiral);
  \path [line] (spiral) -- (action2);
  \path [line] (peripheral) -- (t2);
  \path [line] (t2) -- (action3);
  \path [line] (peripheral) -- (genus0);
  \path [line] (genus0) -- (action3);
  \path [line] (peripheral) -- (t3);
  \path [line] (t3) -- (action4);
\end{tikzpicture}
    \caption{Operations on the different types of minimal sets. Actions in blue do not affect the boundary type in a harmful way, while the red one does.}
    \label{fig:flowchart}
\end{figure}
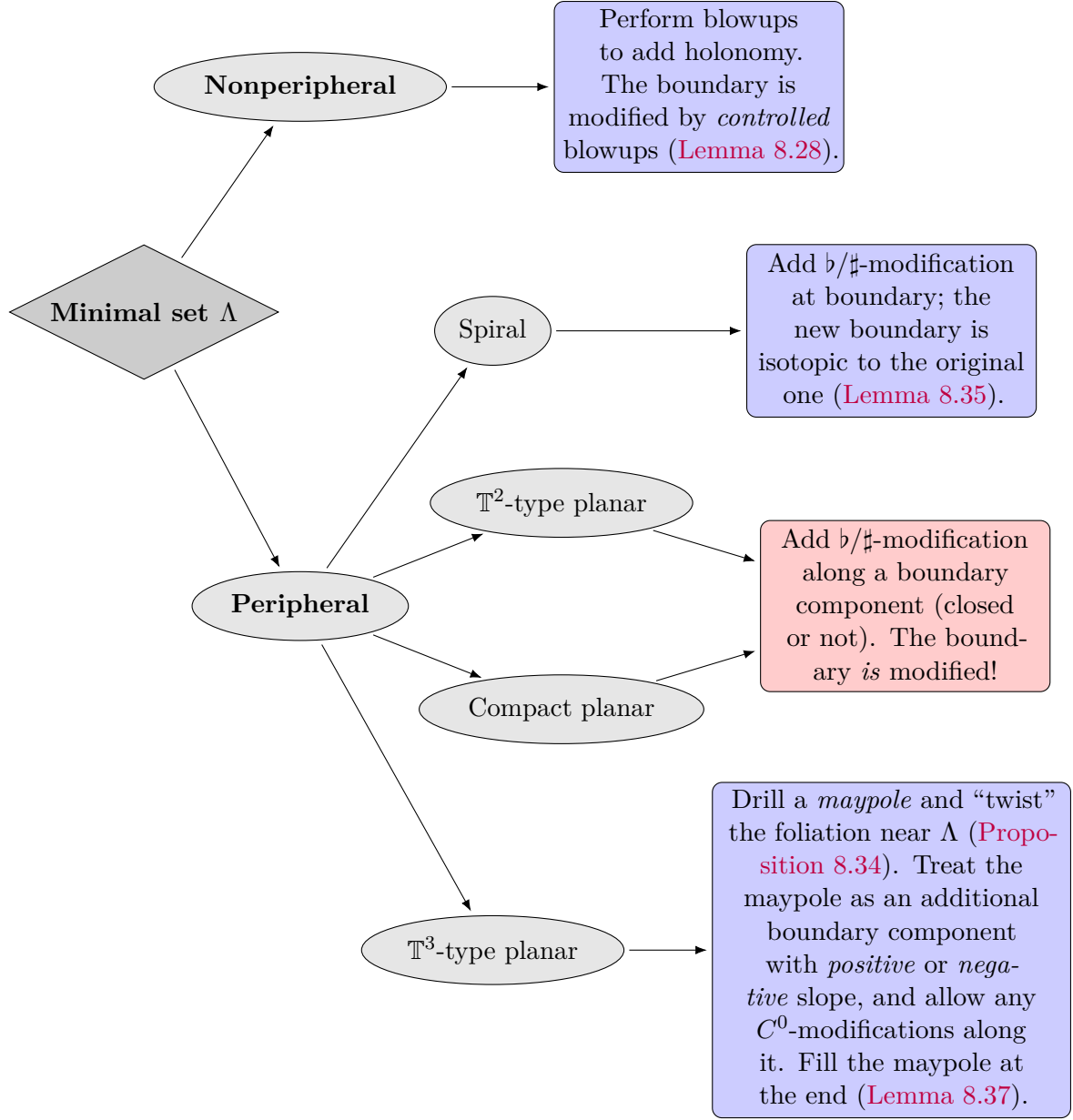

            \subsubsection{Inserting controlled holonomy}

Recall from \zcref{def:minimalsets} that a minimal set $\Lambda$ of $\F$ is \textbf{nonperipheral} if it has a leaf $\lambda$ such that $\pi_1(\lambda)$ is not normally generated by boundary parallel curves, as in the first item of \zcref{prop:lamclass}. Otherwise, we say that it is \textbf{peripheral}.

\begin{defn}
    We say that $\F$ has \textbf{enough nonperipheral (infinitesimal) holonomy} if each of its nonperipheral minimal sets admit a curve with infinitesimal attracting holonomy.
\end{defn}

\begin{lem} \label{lem:enoughholo2}
    Let $\F$ be as in the setup of \zcref{def:enoughholo}, and let $\Upsilon$ be the union of its peripheral minimal sets. Then  $\F$ can be $C^0$-approximated away from $\Upsilon$ by $C^0$-foliations with enough nonperipheral holonomy, and whose boundary foliations are obtained from $\partial \F$ by blowdowns followed by only $\natural$ and $\sharp$ (resp.~$\natural$ and $\flat$) blowups. In particular, the peripheral minimal sets remain unchanged.
\end{lem}

\begin{proof}
This is a variation on the proof of \zcref{lem:enoughholo}. We perform similar blowups, with precise control over the type of blowups we allow. For that, we need to be particularly careful about the curve we choose to add holonomy along.

If $\Sigma$ is a connected orientable surface with boundary, we denote by $\pi_1^\partial(\Sigma) \triangleleft \pi_1(\Sigma)$ the normal subgroup generated by its (closed) boundary components. We have a short exact sequence 
\begin{align} \label{eq:ses}
    1 \longrightarrow \pi_1^\partial(\Sigma) \longrightarrow \pi_1(\Sigma) \longrightarrow \pi_1(\Sigma^\bullet) \longrightarrow 1,
\end{align}
where $\Sigma^\bullet$ is obtained from $\Sigma$ by capping off its closed boundary components by disks.

Let $\Lambda$ be a nonperipheral minimal set of $\F$. By definition, it admits a leaf $L_0$ and an embedded curve $\gamma \subset L_0$ such that $\gamma$ is nontrivial in $\pi_1(L_0^\bullet)$. As in the proof of \zcref{lem:enoughholo}, we consider the holonomy $h$ of $\F$ around $\gamma$.

If this holonomy is nontrivial on both sides of $\gamma$, we proceed as before: we blow up $L_0$ and insert the appropriate holonomy near $\gamma$. Moreover, this can be achieved so that the boundary foliation is only modified by a \emph{trivial} blowup. Indeed, the blowup of $L_0$ is essentially determined by a representation $\rho : \pi_1(L_0) \longrightarrow \mathrm{Homeo}^+(I)$, and we can choose it so that it factors through $c$ as 
$$\begin{tikzcd}
	{\pi_1^\partial(L_0)} \\
	{\pi_1(L_0)} & {\mathrm{Homeo}^+(I)} \\
	{\pi_1(L_0^\bullet)}
	\arrow[from=1-1, to=2-1]
	\arrow["\mathrm{triv}", bend left=12, from=1-1, to=2-2]
	\arrow["\rho", from=2-1, to=2-2]
	\arrow["c"', from=2-1, to=3-1]
	\arrow["{\rho^\bullet}"', bend right=12, from=3-1, to=2-2]
\end{tikzcd}$$
where $\rho^\bullet$ sends $c(\gamma)$ to the appropriate homeomorphism of $I$. This ensures that the image of $\partial$ under $\rho$ is trivial.

We now assume that the holonomy $h$ is trivial on an interval containing $0$. As in the proof of \zcref{lem:enoughholo}, we may assume that there are noncompact leaves arbitrarily close to $\gamma$ where $h$ is trivial, and in which the corresponding parallel curve is noncontractible. Let $L$ be such a leaf, with a corresponding curve $\gamma'$ with trivial holonomy. By \zcref{lem:surftype}, we consider two cases:
\begin{itemize}
    \item[(a)] $\pi_1(L^\bullet) \neq 1$.
    \item[(b)] $\pi_1(L^\bullet) = 1$ and $\pi_1^\partial(L) \cong \pi_1(L)$ is free.
\end{itemize}

Let us consider Case (a) first. If $c(\gamma')$ itself is nontrivial, we proceed as before to add holonomy around $\gamma'$, while adding only trivial holonomy at the boundary. Otherwise, $c(\gamma')$ is trivial and $\gamma' \in \pi_1^\partial(L)$. Since $\gamma'$ bounds a disk in $L^\bullet$, we can write it as a finite product $\gamma' = \beta_1 \beta_2 \dots \beta_m$ of classes represented by boundary components of $L$ \emph{with their orientations coming from the orientation of $L$}.

We claim that there is a representation $\rho : \pi_1(L) \longrightarrow \mathrm{Homeo}^+(I)$ such that all the $\rho(\beta_i)$ are $\sharp$ (resp.~$\flat$), so that $h' \coloneqq \rho(\gamma')$ has no interior fixed points, and $\rho$ applied to the other boundary components of $L$ is trivial. After cutting $L$ along $\gamma'$, we might assume without loss of generality that $\gamma'$ itself is a boundary component of $L$. Moreover, since $\pi_1(L^\bullet) \neq 1$ by assumption, $L$ must have a puncture or a handle. Let $\beta \subset L$ be an embedded curve bounding such a puncture or handle, and let $\beta'$ be an embedded curve on $L$ such that $\gamma'$, $\beta$, and $\beta'$ bound a pair of pants $P \subset L$. In particular, for appropriate choices of orientations, we have $\beta' = \beta \gamma'$. Let $\Sigma$ be the subsurface bounded by $\beta$ disjoint from $P$, which is either a once-punctured disk or a torus with one boundary component, and let $\Sigma'$ be the subsurface bounded by $\beta'$ disjoint from $P$. We may define $\rho$ as above so that $\rho(\beta) = (h')^{-1}$, and the restriction of $\rho$ on $\pi_1(\Sigma')$ is trivial. Indeed, it is trivial to define $\rho$ that way on $\pi_1(\Sigma)$ if $\Sigma$ is a punctured disk, and it follows from a standard commutator trick it is a handle; see for instance the proof of ~\cite[Lemma 2.16]{B16}. Now that $\rho(\beta') = \rho(\beta) \rho(\gamma') = \mathrm{id}$, we readily extend it by the identity on $\pi_1(\Sigma')$, and our claim is proved. 

\begin{figure}[ht]
    \centering
    \def\svgwidth{0.6\textwidth}
    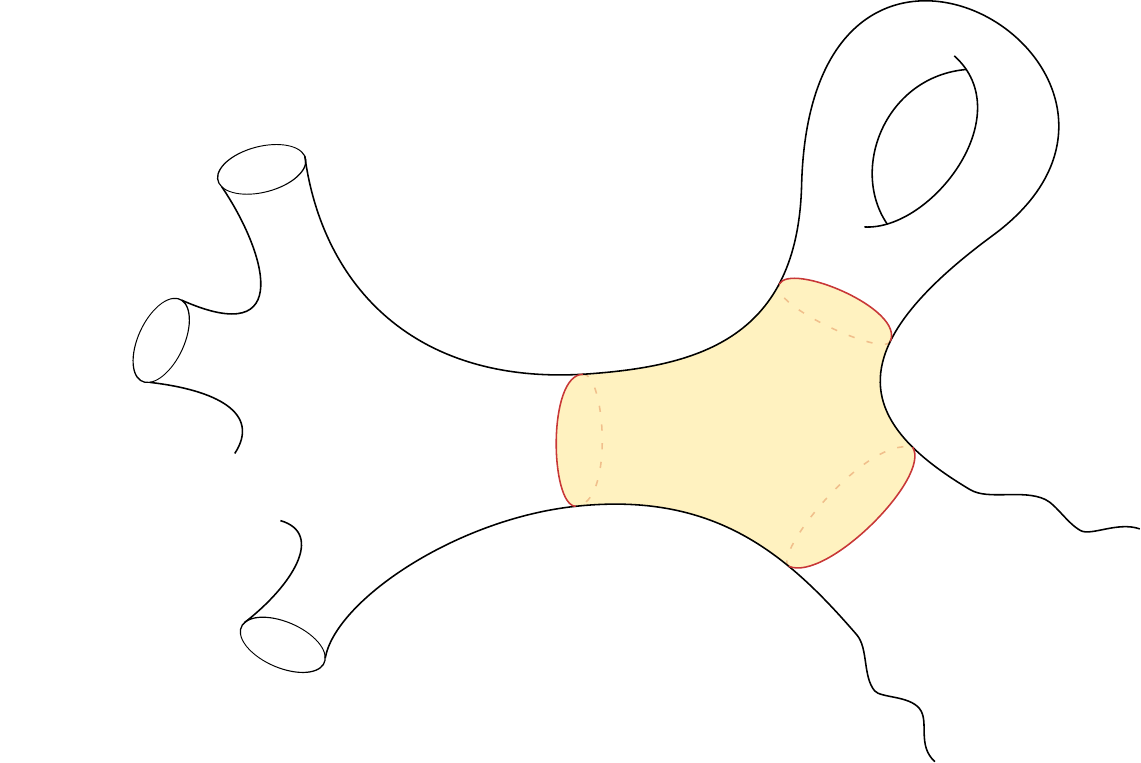
    \caption{Decomposition of the leaf $L$.}
    \label{fig:surface}
\end{figure}

In Case (b), $\pi_1(L)$ is freely generated by its boundary components, and $\gamma'$ is a product of \emph{oriented} boundary components with all the same orientations. We may simply perform a blowup of $L$ inducing only $\sharp$ (resp.~$\flat$) blowups along its boundary components, making the new holonomy around $\gamma'$ nontrivial.

As in the proof of \zcref{lem:enoughholo}, if $\F$ is not minimal we perform such blowups near countably many leaves near $L_0$ which are away from the minimal sets of $\F$, creating germinally nontrivial holonomy on both sides of $\gamma$. We may then blowup $L_0$ itself to create infinitesimally attracting holonomy, while only creating trivial blowups at the boundary. If $\F$ is minimal, we only perform one such blowup along a nearby leaf, making the holonomy germinally nontrivial on both sides of $\gamma$, followed by a blowup along $L_0$ as before.
\end{proof}

        \subsubsection{\texorpdfstring{$\T^3$}{T3}-type planar minimal sets} \label{sec:t3typeminimalsets}

The goal of this section is to prove \zcref{prop:maypoletwist}, which informally says that we may modify the foliation in a small neighborhood of a $\T^3$-type minimal set, after drilling a thickened closed transversal (a \emph{maypole}), so that the new foliation has \emph{positive} or \emph{negative} slope along the new boundary component, in a suitable framing. This operation has the effect of destroying the minimal set by perforating it, and the slope condition allows us to ``add contactness'' there by treating the new boundary component as a toroidal window. These boundaries will eventually be filled by appropriate contact structures at the end of the proof.

\medskip

Let us first warm up to the general case of $\T^3$-type minimal sets by starting with the simpler example of a rational linear foliation on $\T^3$; while it is not strictly speaking a $\T^3$-type planar minimal set, a similar construction applies. In this case, one could even write down an explicit formula for a \emph{contact} perturbation, but we follow a different strategy based on a cut and paste construction that will generalize well to more complicated minimal sets.

The existence of contact structures on $\T^3$ is closely related to the fact that the commutator of two homeomorphisms of the circle can have positive rotation number. We will need a quantitative version of this fact:

\begin{defn}
    For $f\in \Homeo^+([a,b])$, we say that $f$ is \textbf{$\varepsilon$-strongly ascending} if $f(a+\varepsilon) > b-\varepsilon$.
\end{defn}

\begin{lem}[Commutator trick]\label{lem:commutatortrick}
    Let $N$ be an even positive integer. Let 
    $$\Homeo_{N\Z}^+(\R) \coloneqq \big\{f\in \Homeo_+(\R) \mid \forall x \in \R, \ f(x+N) = f(x) + N \big\}.$$
    Let $0 < \varepsilon < 1/100$, and suppose $f$, $g$, $\overline{f}$, and $\overline{g}$ are four elements of $\Homeo^+_{N\Z}(\R)$ satisfying the following properties:
    \begin{enumerate}
        \item[(a)] $f$ and $\overline{f}$ have $N$ fixed points $x_1,\dots x_N$ (and possibly other fixed points as well), and $g$ and $\overline{g}$ have $N$ fixed points $y_1,\dots y_N$,
        \item[(b)] $x_i \in (i-\varepsilon, i+\varepsilon)$,
        \item[(c)] $y_i \in (i+0.5-\varepsilon,i+0.5+\varepsilon)$,
        \item[(d)] On the interval $[x_i,x_{i+1}]$, if $i$ is odd then $f$ is $\varepsilon$-strongly ascending. For $i$ even, $\overline{f}$ is $\varepsilon$-strongly ascending,
        \item[(e)] On the interval $[y_i,y_{i+1}]$, if $i$ is odd then $g$ is $\varepsilon$-strongly ascending. For $i$ even, $\overline{g}$ is $\varepsilon$-strongly ascending.
    \end{enumerate}
    Then $\overline{g}\overline{f}gf$ has positive translation number (in fact, a translation number $\geq 2$).
\end{lem}

\begin{proof}
    By the strongly ascending property, we have the following inequalities:
    \begin{itemize}
        \item $f(i + 2\varepsilon) > i+1-2\varepsilon$ for $i$ odd,
        \item $g(i+0.5+2\varepsilon) > i+1.5-2\varepsilon$ for $i$ odd,
        \item $\overline{f}(i + 2\varepsilon) > i+1-2\varepsilon$ for $i$ even,
        \item $\overline{g}(i+0.5+2\varepsilon) > i+1.5-2\varepsilon$ for $i$ even.
    \end{itemize}
    See \zcref{fig:commutator}. Combining these inequalities, we find that for $i$ odd, $$\overline{g}\overline{f}gf(i+2\varepsilon)> i+2.5 - 2\varepsilon > (i + 2\varepsilon) +2.$$ Therefore, $\overline{g}\overline{f}gf$ has translation number greater or equal than 2.
\end{proof}

\begin{figure}[ht]
    \centering
    \def\svgwidth{0.4\textwidth}
    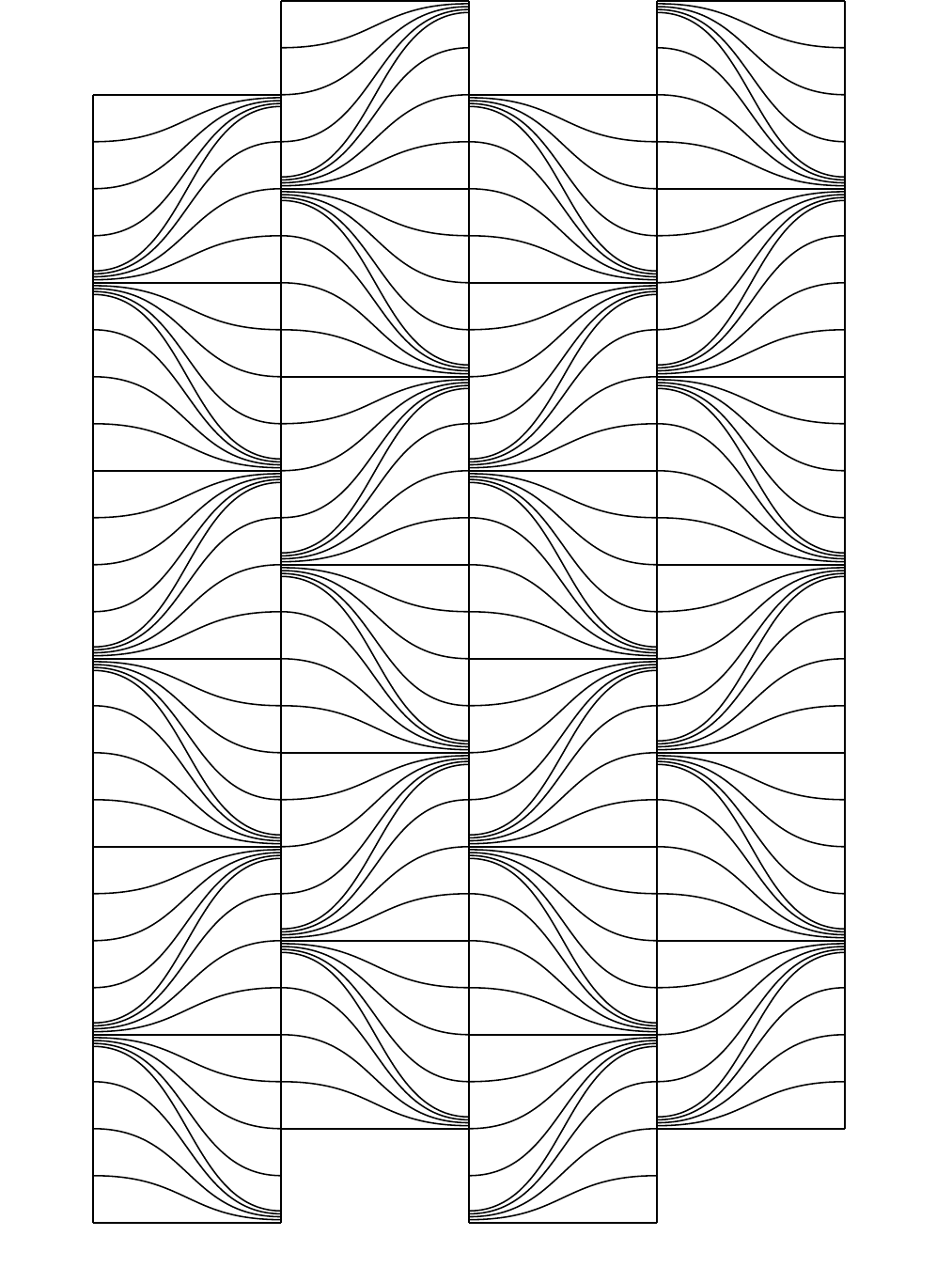
    \caption{The product $\overline{g}\overline{f}gf$ drawn as a concatenation of horizontal foliations on $\R\times[0,1]$. The strongly-ascending constraints imply the existence of the orange leaves which in turn force the translation number of the product to be positive.}
    \label{fig:commutator}
\end{figure}

Recall that a (positive) confoliation is a (smooth) plane field $\xi$ of the form $\xi = \ker \alpha$, where the $1$-form $\alpha$ satisfies $\alpha \wedge d\alpha \geq 0$. The next example illustrates how to use the previous commutator trick to produce a confoliation approximating a rational foliation on $\T^3$. While we will not directly use this construction in the rest of the article, it serves as an insightful warm up and highlights the strategy of the proof of \zcref{prop:maypoletwist}. 

\begin{constr}\label{constr:rationalapprox}
    Let $\F$ be the foliation of $\T^3$ tangent to $\ker(dz + p_1/q \, dx + p_2/q\,  dy)$, where $p_1, p_2$ are integers and $q$ is a nonzero integer. We will construct a $C^0$-small perturbation of $\F$ to a confoliation which is contact in a neighborhood of the $z$-axis.

    Let $P$ be an $\varepsilon$-neighborhood of the $z$-axis. We will colloquially refer to $P$ as the \textbf{maypole}, for many ribbons will be affixed to it. Choose some integer $N\gg 1/\varepsilon^2$ which is an even multiple of $q$. We now define $2N$ parallelograms $R_1,\dots,R_N, S_1,\dots,S_N$, which we call \textbf{ribbons}. For $1\leq i \leq N$, $R_i$ is a trivially foliated parallelogram with two sides on $\partial P$ and two sides on the line segments 
    $$\left\{ \left(t,0,\frac i N - \frac {p_1}{q}t \right),  0 \leq t \leq 1\right\} \quad \mathrm{and} \quad \left\{ \left(t,0,\frac {i+1} N - \frac {p_1}{q}t\right),  0 \leq t \leq 1\right\}.$$ 
    The $S_i$ are similarly defined parallelograms in the $yz$ plane with a $\frac 1 {2N}$ offset: $S_i$ is bounded above and below by the line segments
    $$\left\{ \left(0,t,\frac {i-1/2} N - \frac {p_2}{q}t \right),  0 \leq t \leq 1\right\} \quad \mathrm{and} \quad \left\{ \left(0,t,\frac {i+1/2} N - \frac {p_2}{q}t\right),  0 \leq t \leq 1\right\}.$$ 
    See \zcref{fig:fence}.

    Thicken each ribbon $R_i$ to width $\varepsilon$, so that it a trivially foliated box $R_i \times [0,\varepsilon]$ of dimensions $\approx 1\times 1/N \times \varepsilon$. Since $N \gg 1/\varepsilon^2$, these thickened ribbons are actually much thicker than they are tall. However, there is plenty of room to accommodate their thickness on the maypole since the maypole has width $\approx \varepsilon$. Now replace the foliated $R_i\times [0,\varepsilon]$ by a new foliation with agrees with the old one on all the faces of $R_i \times [0,\varepsilon]$, except the ones which touch the maypole. We choose this foliation so that for even $i$, the holonomy from $R_i\times \{0\}$ to $R_i \times \{\varepsilon\}$ is $\varepsilon$-strongly ascending, and for odd $i$ the holonomy from $R_i\times \{0\}$ to $R_i \times \{1\}$ is $\varepsilon$-strongly descending. See \zcref{fig:slat_operation}. Since $R_i$ is much thicker than it is tall, we can choose the new foliation so that it is $C^0$ close to $\F$. Perform the same operation on the $S_i$'s. Let $\F'$ be the resulting foliation on $\T^3 \setminus P$.

    The restriction of $\F$ to $\partial P$ is modified near where the ribbons attach, in a way so that the holonomy around $\partial P$ satisfies the conditions of the commutator trick (\zcref{lem:commutatortrick}). Thus, $\F'$ has positive slope along $\partial P$. We may now fill $P$ with a contact structure by radial twisting to obtain the desired confoliation.
\end{constr}

\begin{figure}[ht]
    \centering
    \def\svgwidth{0.5\textwidth}
    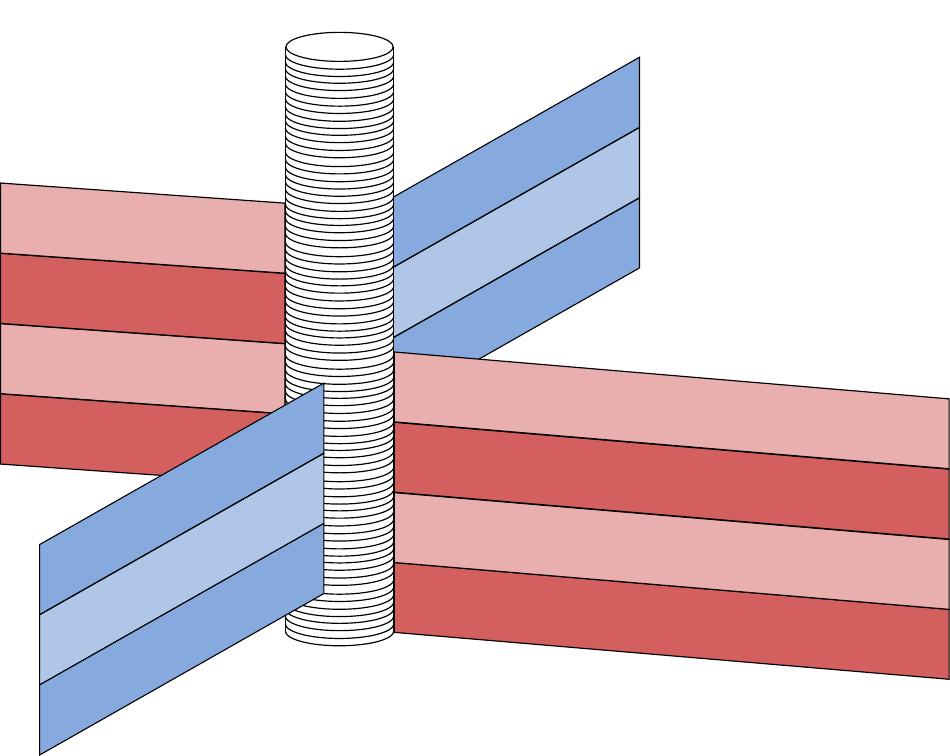
    \caption{The maypole as well as some of the ribbons. The $R_i$'s are drawn in red and the $S_i$'s are drawn in blue. Note how the red ribbons are offset from the blue ribbons by half a unit. Note also that even ribbons line up with even and odd ribbons line up with odd.}
    \label{fig:fence}
\end{figure}

\begin{figure}[ht]
    \centering
    \def\svgwidth{0.4\textwidth}
    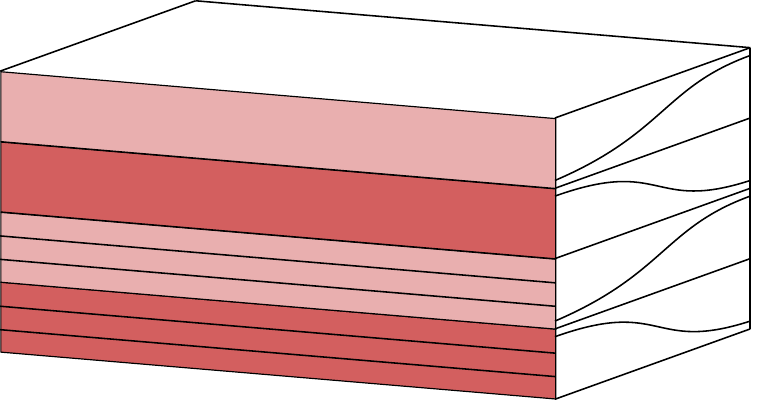
    \caption{Replacing the foliation on $R_i \times [0,\varepsilon]$.}
    \label{fig:slat_operation}
\end{figure}

Here are some adjustments that need to be made in order to generalize \zcref{constr:rationalapprox} to arbitrary $\T^3$-type planar minimal sets:
\begin{enumerate}
    \item The foliation may have irrational slope, so the ribbons might not line up with multiples of $1/q$. However, if we choose a good enough rational approximation using Dirichlet's theorem (see \zcref{thm:dirichlet} below), then the ribbons will nearly line up. The commutator trick \zcref{lem:commutatortrick} is robust enough to deal with small height mismatches in the ribbons. Another approach closer to the original argument of Kazez--Roberts~\cite{KR17} would be to try to replace the $\T^3$-type minimal set with a new set having rational slope. However, in the generality we seek, there would be additional complications to insert the foliations of the guts into the interstices of this new foliation with rational slope.
    \item The minimal set might have complicated guts. Therefore the pole might need to swerve in order to stay in an $\varepsilon$-neighborhood of the minimal set. The individual ribbons might also need to swerve to avoid the guts. See \zcref{fig:wiggledfence}.
    \item The invariant transverse measure for the $\T^3$-type minimal set might be different from the Lebesgue measure. The ribbons need to be of height $1/N$ with respect to the former measure, and of height $<\varepsilon$ with respect to the latter measure. This will be achieved by taking $N$ much larger than $1/\varepsilon$.
    \item The holonomy around the pole at the end of our small perturbation might not be smooth, so we must take more care in filling the pole with a contact structure. This filling is achieved in \zcref{lem:maypolefilling}.
\end{enumerate}

\begin{figure}
        \centering
        \def\svgwidth{0.65\textwidth}
        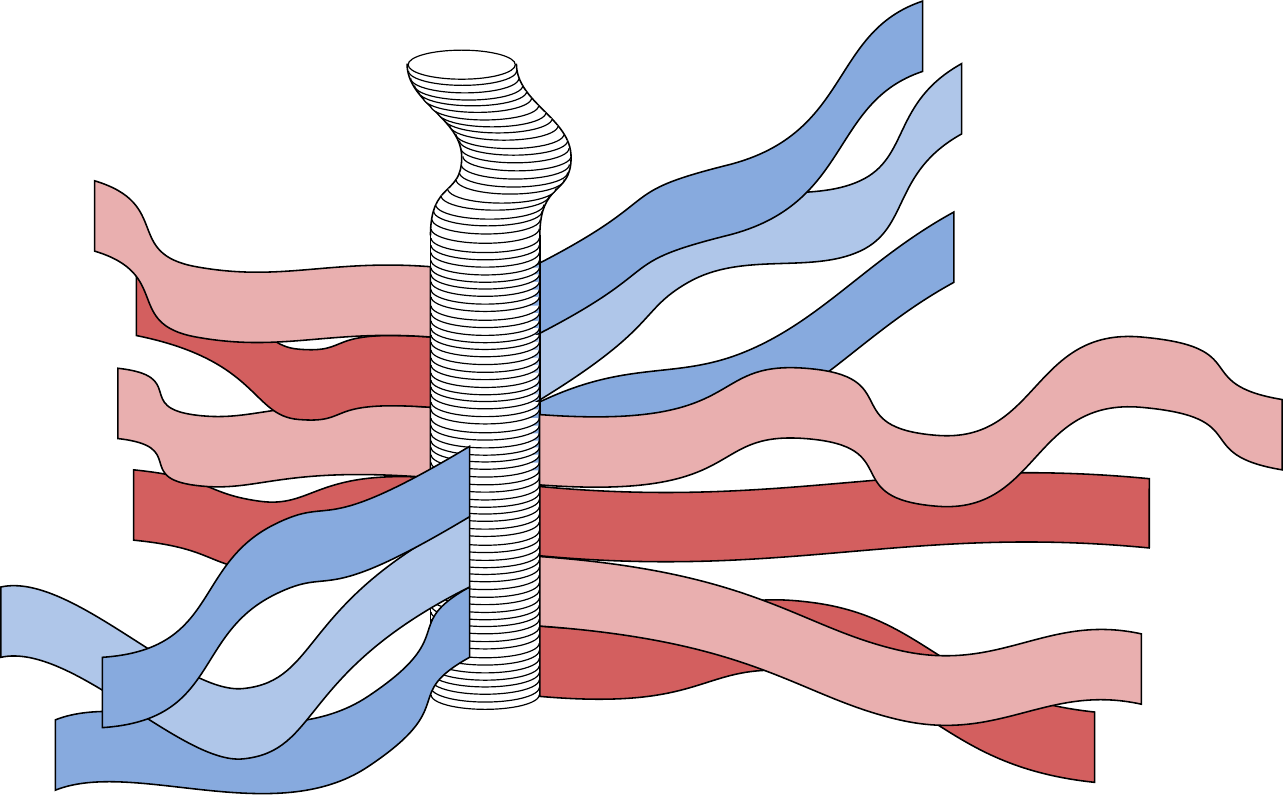
        \caption{There is freedom in the embedding of the maypole and the ribbons.}
        \label{fig:wiggledfence}
\end{figure}

\begin{thm}[Dirichlet's approximation theorem, simultaneous version]\label{thm:dirichlet} Given real numbers  ${\displaystyle \alpha _{1},\ldots ,\alpha _{d}}$ and a natural number ${\displaystyle A}$ then there are integers ${\displaystyle p_{1},\ldots ,p_{d},q\in \mathbb {Z} ,1\leq q\leq A}$ such that $${\displaystyle \left|\alpha _{i}-{\frac {p_{i}}{q}}\right|\leq {\frac {1}{qA^{1/d}}}.}$$
\end{thm}

See~\cite[Corollary 1B]{schmidt.DiophantineApproximation} for a proof.

\begin{defn}
    Let $\F$ be a $C^0$-foliation on $M$ with a $\T^3$-type planar minimal set $\Lambda$, and $\varepsilon > 0$. A \textbf{maypole green} of size $\leq \varepsilon$ for $\Lambda$ is a triple $\mathfrak{G} = (N, \gamma, P)$ where
    \begin{itemize}
        \item $N$ is an open neighborhood of $\Lambda$ of size $\leq \varepsilon$,
        \item $\gamma \subset N$ is a closed loop transverse to $\F$,
        \item $P \subset N$ is a closed tubular neighborhood of $\gamma$ equipped with a \emph{smooth} coordinate system with respect to which $\F$ is the product foliation by disks.
    \end{itemize}
    We call $P$ the \textbf{maypole} with core $\gamma$.
\end{defn}

\begin{figure}[ht]
    \centering
    \def\svgwidth{0.4\textwidth}
    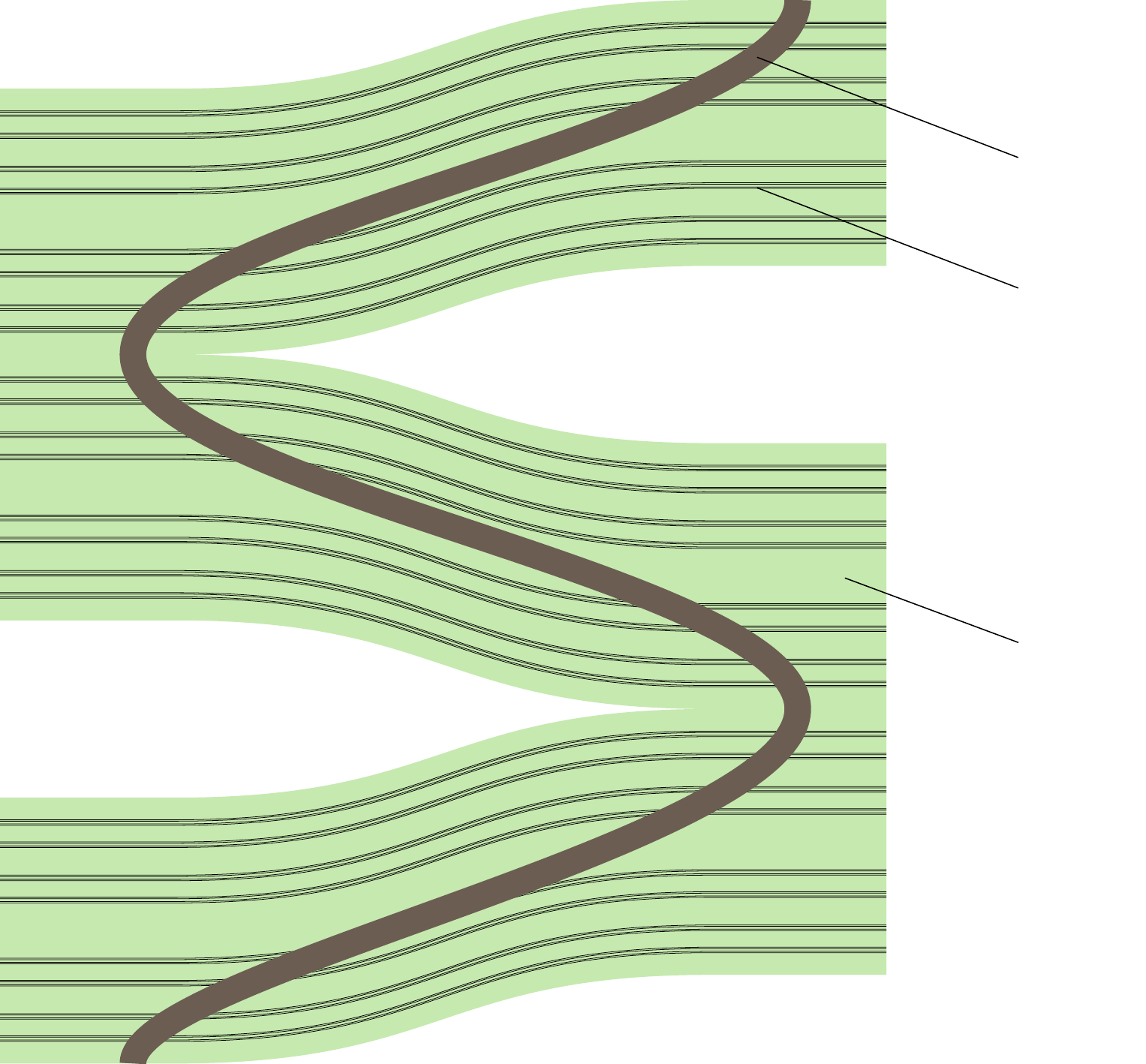
    \caption{A maypole green. $N$ is shown in green. Note that $P$ intersects leaves of $\F$ in disks.}
    \label{fig:maypoleweaving}
\end{figure}

Note that for $\varepsilon > 0$ small enough, the neighborhood $N$ is disjoint from the other minimal sets of $\F$.

\begin{prop}\label{prop:maypoletwist}
Let $\F$ be a $C^0$-foliation on $M$ with a $\T^3$-type planar exceptional minimal set $\Lambda$. For any $\varepsilon > 0$, and after an arbitrarily $C^0$-small isotopy of $\F$ whose size is independent of $\varepsilon$, there exists a maypole green $\mathfrak{G} = (N, \gamma, P)$ for $\Lambda$ of size $\leq \varepsilon$ (with a smooth framing of $P$) such that for every $\delta > 0$, there exist two $C^0$-foliations $\F^\pm_\delta$ on $M \setminus \interior P$ such that
\begin{itemize}
    \item $\F^\pm_\delta$ coincide with $\F$ away from $N$ and near $\partial M$,
    \item $T\F^\pm_\delta$ is $\varepsilon$-close to $T\F$ everywhere, and $\delta$-close to it near $\partial P$,
    \item $\F^\pm_\delta$ has \emph{positive} (resp.~\emph{negative}) slope along $\partial P$ in the chosen smooth framing.
\end{itemize}
In particular, for $\varepsilon$ small enough, the minimal sets of $\F^\pm_\delta$ disjoint from $P$ correspond to minimal sets of $\F$, and $\F^\pm_\delta$ coincide with $\F$ near those.
\end{prop}

\begin{proof}
The construction proceeds in three steps.
\begin{itemize}[leftmargin=*]
    \item \textit{Step 1: smoothing the minimal set.} Since $\Lambda$ is a $\T^3$-type planar minimal set, $\Lambda$ is a measured lamination. Using \zcref{lem:smoothing_measured_laminations}, $\F$ is semiconjugate to a branching foliation $\F_{\mathrm{branch}}$ with a smooth branching sublamination $\Lambda_{\mathrm{branch}}$, and $\F_{\mathrm{branch}}$ is $\varepsilon/10$ tangentially close to $\F$. Moreover, there exists a smooth branching lamination $\Lambda_{\mathrm{model}}$ on $\T^3 \setminus N(L)$ for some link $L$ such that
    \begin{enumerate}
        \item After filling in $L$, $\Lambda_{\mathrm{model}}$ blows down to a smooth irrational foliation $\Lambda_{\mathrm{lin}}$ of $\T^3$,
        \item There is a smooth inclusion 
        $$u : \big(\T^3\setminus N(L), \Lambda_{\mathrm{model}}\big) \hookrightarrow (M,\Lambda_{\mathrm{branch}}).$$
    \end{enumerate}
    
    Since the guts of $\Lambda_{\mathrm{model}}$ in $\T^3$ are balls and $L$ is a link, a neighborhood of $\guts(\Lambda_{\mathrm{model}})\cup N(L)$ is homeomorphic to a handlebody $H \subset \T^3$ intersecting every leaf of $\Lambda_{\mathrm{model}}$ in a collection of disks. In sum, we obtain a diagram
    $$\begin{tikzcd}[column sep=scriptsize]
    	  & &{(M, \mathcal F)} \\
    	{\big(\T^3 \setminus f^{-1}(H), \Lambda_{\mathrm{lin}}\big)} & {(\T^3\setminus H, \Lambda_{\mathrm{model}})} & {(M, \F_{\mathrm{branch}})}
        \arrow["\pi", two heads, from=1-3, to=2-3]
    	\arrow["u"', hook, from=2-2, to=2-3]
        \arrow["\cong", "f"', from=2-1, to=2-2]
    \end{tikzcd}$$
    where 
    \begin{enumerate}
    \item $\pi$ is the blowdown from $\F$ to $\F_{\mathrm{branch}}$,
    \item $u$ is a smooth embedding,
    \item $f$ is a (smooth) diffeomorphism. 
    \end{enumerate}

    \item \textit{Step 2: finding the maypole and the ribbons.} We consider standard coordinates $(x,y,z)$ on $\T^3$. Without loss of generality, we may assume that $\Lambda_{\mathrm{lin}}$ is transverse to $\partial_z$. Up to changing the diffeomorphism $f$, we may shrink the handlebody $f^{-1}(H)$ to lie in a small neighborhood of an embedded graph $\Gamma$. By an isotopy of $\Gamma$, we can arrange that it is in general position with respect to $\{0\}\times \{0\} \times S^1$ and the tori $\{x=0\}$ and $\{y=0\}$. We now define $\gamma_{\mathrm{lin}} \subset \T^3 \setminus f^{-1}(H)$ to be $\{0\} \times \{0\} \times S^1$ and declare $P_{\mathrm{lin}}$ to be a small closed tubular neighborhood of $\gamma_{\mathrm{lin}}$ diffeomorphic to a solid torus transverse to $\Lambda_{\mathrm{lin}}$. Transport $\gamma_{\mathrm{lin}}$ and $P_{\mathrm{lin}}$, from $(\T^3\setminus f^{-1}(H),\Lambda_{\mathrm{lin}})$ to $(M,\F)$ by the composition $\pi^{-1} \circ u \circ f$ to obtain $\gamma$ and $P$.

    Then, smooth $\gamma$ and $\F$ in a neighborhood of $P$ by an arbitrarily $C^0$-small isotopy, so that there are smooth coordinates $P\cong D^2_{r,\theta}\times S^1_\phi$ in which the leaves of $\F$ are of the form $D^2\times \{\mathrm{pt}\}$ and $\gamma$ lies along $\{r=0\}$. Let $\F'$ be a new foliation on $M$ which is topologically isotopic to $\F$, fitting into the diagram $$\begin{tikzcd}[column sep=scriptsize]
    	  {(M, \mathcal F)} & {(M, \mathcal \F')}\\
    	    & {(M, \mathcal F_{\mathrm{branch}})}
    	\arrow["\pi'", two heads, from=1-2, to=2-2]
        \arrow["\pi", two heads, from=1-1, to=2-2]
    \end{tikzcd}$$
    
    where $\pi'^{-1}$ is a blowup satisfying the following properties:
    \begin{enumerate}
        \item $\pi'^{-1}$ is a blowup of size $\sigma > 0$, where $\sigma$ will be chosen small enough later,
        \item $\pi'$ preserves $P$ along with its product structure; in particular, $\F$ and $\F'$ coincide in $P$,
        \item $\F'$ is $\varepsilon/10$-tangentially close to $\F_{\mathrm{branch}}$, and therefore $\varepsilon$-tangentially close to $\F$.
    \end{enumerate}
    
    We are now ready to attach ribbons to the maypole $P_{\mathrm{lin}}$. Let $\alpha_1$ and $\alpha_2$ be the slopes of $\Lambda_{\mathrm{lin}}$, i.e., the numbers such that $\Lambda_{\mathrm{lin}}$ is tangent to $\ker(dz + \alpha_1 \, dx + \alpha_2\,  dy)$. Using \zcref{thm:dirichlet}, we may choose rational numbers $p_1/N$ and $p_2/N$ so that
    \begin{align} \label{eq:rational_approx}
        \left|\alpha _{i}-{\frac {p_{i}}{N}}\right|=o(1/N),
    \end{align}
    and $N \gg 1/\varepsilon^2$ is even.
    Now we attach ribbons $\big\{R^{\mathrm{lin}}_i\big\} \cup \big\{S^{\mathrm{lin}}_i\big\}$ of width $1/N$ in $(\T^3\setminus f^{-1}(H), \Lambda_{\mathrm{lin}})$ as in \zcref{constr:rationalapprox}, except that whenever $\Gamma$ passes through the $\{x=0\}$ torus or the $\{y=0\}$ torus, the ribbons simply make a small detour around $\Gamma$ in $\T^3$ while remaining embedded and disjoint.

    Finally, we transport $\big\{R^{\mathrm{lin}}_i\big\} \cup \big\{S^{\mathrm{lin}}_i\big\}$ from $(\T^3\setminus f^{-1}(H),\Lambda_{\mathrm{lin}})$ to $(M,\F)$ by the composition $\pi'^{-1} \circ u \circ f$ to obtain $\{R_i\} \cup \{S_i\}$. Since $\pi'^{-1}$ is a blowup of size $\sigma$, it modifies the ribbons by a $\sigma$-$C^0$-small isotopy.

    \item \textit{Step 3: perturb $\F'$ along the ribbons.} We construct a new foliation $\F^+$ on $M\setminus \interior P$ from $\mathcal F'$ as follows. First, we smooth $\F'$ along each of the ribbons, and perform the modification from \zcref{constr:rationalapprox} which cuts open along these ribbons and reglues with alternating strongly ascending or strongly descending holonomy. By \zcref{eq:rational_approx} and taking $\sigma \ll 1/N$ so that the blowup $\pi'^{-1}$ does not modify the ribbons too much, the ends of the ribbons line up closely enough on $\partial P$ that the holonomy of $\mathcal F^+$ around $\partial P$ satisfies the hypotheses of \zcref{lem:commutatortrick}. That lemma tells us that the holonomy of $\F^+$ around $\partial P$ is ascending, for the given choice of smooth coordinates.

    We now check that our perturbations are as small as claimed. We already know that $\F_{\mathrm{branch}}$ is $\varepsilon$-close to $\F$. Then, the heights of the ribbons depend on the maps $u$, $f$, as well as the parameters $N$ and $\sigma$. However, we are free to choose $N$ and $\sigma$ after choosing $u$ and $f$. For any $\delta > 0$, taking $N$ sufficiently large and $\sigma$ sufficiently small ensures that the ribbons all have height less than $\delta$. Thus, we can arrange that $\F^+$ is tangentially $\delta$-close to $\F'$. In turn, $T\F'$ is $\varepsilon$-close to $T\F$ and equal to it near $\partial P$. Thus, $\F^+$ has the desired approximation property. The construction of $\F^-$ is similar. \qedhere
    \end{itemize}
\end{proof}

            \subsubsection{Opening windows on spiral minimal sets}

\begin{lem} \label{lem:spiral}
    Let $\G$ be a sharp, flat, or Reeb $C^0$-foliation on the annulus $A = S^1 \times [0,1]$. Let $p \in \mathrm{int} A$. For every $\varepsilon > 0$, there exists a $\sharp$ (resp.~$\flat$) modification of $\G$ near $p$ which is isotopic to $\G$ via an $\varepsilon$-$C^0$-small isotopy.
\end{lem}

\begin{proof}
Without loss of generality, we may assume that $\G$ has contracting holonomy around $S^1 \times \{0\} \subset A$. We proceed in 3 steps.
\begin{itemize}[leftmargin=*]
    \item \textit{Step 1: preparation of $\G$.} 
    Using \zcref{lem:transversholo}, we may perform an arbitrarily $C^0$-small isotopy near $S^1 \times \{0\}$ such that $\G$ admits a (smooth) transversal $\gamma$ in a neighborhood $U_\delta \coloneqq S^1 \times [0, \delta)$ of that boundary curve for some small $\delta> 0$. By construction, $\gamma$ is horizontal, i.e., transverse to the vertical intervals $\{s\} \times [0,\delta)$, $s \in S^1$. We may assume that $p \notin U_\delta$, so that $\G$ is not modified near $p$. By slight abuse of notation, we still denote by $\G$ the modified foliation.

    \item \textit{Step 2: $\flat/\sharp$-modification.} Since $p$ is connected to a point $q \in U_\delta$ below $\gamma$ along a leaf of $\G$, we may as well perform the $\flat/\sharp$ modification near $q$; see \zcref{rem:modif}. We can further arrange that the resulting foliation $\widehat{\G}$ is still very nice in $U_\delta$, by making this modification extremely small.
    
    \item \textit{Step 3: correction.} We now claim that for a sufficiently small choice of $\delta > 0$ and sufficiently small choices of $C^0$-isotopies and $\flat/\sharp$ modifications as above, $\widehat{\G}$ is isotopic to $\G$ via a $C^0$-small, graphical isotopy.

    By construction, the subset of $A$ where $\G$ and $\widehat{\G}$ differ is contained in the connected component $W$ of $A \setminus \gamma$ containing $S^1 \times \{0\}$, and both of these foliations are graphical in the sub-annulus $A_\delta \coloneqq S^1 \times [0,\delta)$. Since $\gamma \subset A_\delta$ is transverse to both $\G$ and $\widehat{\G}$, we can use it to parametrize the leaves of these foliations in $A_\delta$, in the following sense. Let $\pi : \widetilde{A_\delta} = \R \times [0,\delta) \rightarrow A_\delta$ denote the universal cover of $A_\delta$, and $\widetilde{\gamma} \subset \widetilde{A_\delta}$ the universal cover of $\gamma$. Since $\gamma$ is horizontal, we parametrize it as a graph $s \mapsto (s,f(s)) \in S^1 \times [0,\delta)$, and identify it with $S^1_s$. We lift it to a parametrization $x \mapsto \big(x, \widetilde{f}(x)\big)$ of $\widetilde{\gamma}$. Then, there exist two continuous families of $C^1$ functions $g_x : [x, \infty) \rightarrow (0,\delta)$ and $\widehat{g}_x : [x, \infty) \rightarrow (0,\delta)$, $x \in \R$, such that $g_x(x) = \widehat{g}_x(x) = \widetilde{f}(x)$, and the leaf of $\G$ (resp.~$\widehat{\G}$) passing through $\gamma(s) = (s, f(s))$, $s \in S^1$, is the image under $\pi$ of the graph of $g_x$ (resp.~$\widehat{g}_x$), where $\pi(x) = s$. By construction, $g_x$ and $\widehat{g}_x$ coincide near $x$, and satisfy the periodicity condition
    $$g_{x_0+1}(x) = g_{x_0}(x-1), \quad \widehat{g}_{x_0+1}(x) = \widehat{g}_{x_0}(x-1),$$
    for all $x_0 \in \R$ and $x \geq x_0+1$.

    We then define 
    $$g^\tau_x \coloneqq \tau g_x + (1-\tau) \widehat{g}_x$$
    for all $\tau \in [0,1]$ and $x \in \R$. The graphs of these functions define a family of $C^0$-foliations interpolating between $\G$ and $\widehat{\G}$ on $\overline{W}$, after adding the leaf $S^1 \times \{0\}$, and can be extended to a family $\G_\tau$ of $C^0$-foliations on $A$ since this family is constant and coincides with $\G$ near $\gamma$. By all the previous choices and the properties of linear interpolations, $\G_\tau$ is a $C^0$-small graphical isotopy between $\G$ and $\widehat{\G}$, as desired. \qedhere
    \end{itemize}
\end{proof}

\begin{rem}
    In the language of \zcref{sec:fencesandwindows}, the previous lemma allows us to open a square window near the point $p$ and apply a small $\sharp$ or $\flat$ modification there, while only changing the foliation by an arbitrarily small $C^0$-isotopy.
\end{rem}

        \subsubsection{Filling the maypoles} \label{sec:tori}

\begin{lem}\label{lem:maypolefilling}
    Let $P \coloneqq D^2 \times S^1$ denote the solid torus endowed with its natural Euclidean metric. We write $H \coloneqq T D^2 \times \{0\}$ for the horizontal plane field on $P$.
    
    There exists a constant $C > 0$ such that the following holds. Let $\varepsilon > 0$ and let $\xi_\partial$ be a germ of contact structure near $\partial P \subset P$ such that 
    \begin{itemize}
        \item $\xi_\partial$ is $\varepsilon$-close to $H$,
        \item The slope of $\xi_\partial$ along $\partial P$ is positive.
    \end{itemize}
Then, there exists a contact structure $\xi$ on $P$ such that 
    \begin{itemize}
        \item $\xi$ coincides with $\xi_\partial$ near $\partial P$,
        \item $\xi$ is $C\varepsilon$-close to $H$.
    \end{itemize}
\end{lem}

\begin{proof}
    We sketch the main steps of the proof. We will write ``$\lesssim \varepsilon$'' to mean ``$\leq C \varepsilon$'' for some constant $C$ independent of $\varepsilon$ and $\xi_\partial$.

    Let $\varepsilon > 0$ and $\xi_\partial$ near $\partial P$ satisfying the hypotheses of the lemma. We will construct a foliation $\mathcal{D}$ on $P$ by horizontal disks such that the boundary of each disk is transverse to $\xi_\partial$, and the distance between $T \mathcal{D}$ and $H$ is $\lesssim \varepsilon$. Then, the desired extension of $\xi_\partial$ is obtained by twisting along the disks in $\mathcal{D}$ radially.
    
    Let $\mathcal G_+$ denote the characteristic foliation of $\xi_\partial$ along $\partial P$. It is horizontal and has positive slope by assumption. We construct a foliation by horizontal circles $\mathcal{D}_\partial$ on $\partial P$ which are transverse to $\mathcal{G}_+$ and stay at distance $\lesssim \varepsilon$ from the horizontal line field $H_\partial \coloneqq H \cap \partial P$ on $\partial P$.

    Let $g : S^1 \rightarrow S^1$ denote the holonomy of $\mathcal{G}$ around $\partial P$, where we identified $S^1 \cong \{1\} \times S^1 \subset \partial P$. Then by assumption, $g > \mathrm{id}$ (this inequality makes sense since $g$ is very close to $\mathrm{id}$), and $\vert g(t) - t \vert < \varepsilon$ for all $t \in S^1$.

    Let $X$ be a vector field spanning $T \mathcal G_+$ of the form
    $$X = \partial_x + f(x,t) \, \partial t,$$
    where we identify $\partial P \cong S^1_x \times S^1_t$. Note that $\vert f(x,t)\vert < \varepsilon$ by assumption. For $\delta \in [-2\varepsilon, 0]$, we consider the vector field $X_\delta \coloneqq X + \delta \, \partial_t$ which is transverse to $X$. Then $X_{-2\varepsilon}$ is tangent to a horizontal foliation on $\partial P$ with descending holonomy (and no closed leaves), denoted by $\mathcal{G}_-$, and by a standard continuity argument, there exists $\delta \in (-2\varepsilon, 0)$ such that $X_\delta$ has a closed orbit of slope $0$; denote this orbit by $\gamma$. Note that $\gamma$ is transverse to both $\mathcal G_+$ and $\mathcal G_-$ (with opposite signs), and both $\gamma$ and $\mathcal G_-$ are at a distance $\lesssim \varepsilon$ from $H_\partial$. We now construct a horizontal foliation by circles on $\partial P$ containing $\gamma$ and which is transverse to $\mathcal G_+$ and $\mathcal G_-$, with appropriate signs. The latter condition also ensures that it will stay sufficiently close to $H_\partial$.

    We now cut $\partial P$ along $\gamma$ to obtain a cylinder $C$ bounded by two copies of $\gamma$, $\gamma_-$ and $\gamma_+$. Since $\mathcal G_+$ is transverse to both $\gamma_-$ and $\gamma_+$ and has positive slope, we can find a diffeomorphism $\varphi : C \rightarrow S^1 \times [0,1]$ such that
    $$\varphi (\gamma_-) = S^1 \times \{0\}, \qquad \varphi (\gamma_+) = S^1 \times \{1\},$$
    and $\varphi_* \mathcal{G}_+$ is the vertical foliation by intervals $\{u\} \times [0,1]$ on $S_u^1 \times [0,1]_v$. Setting $\widetilde{\mathcal{G}}_- \coloneqq \varphi_* \mathcal{G}_-$, then $\widetilde{\mathcal{G}}_- $ is tangent to a vector field of the form 
    $$\widetilde{X}(u,v) = \partial_u + h(u,v) \, \partial_v,$$
    where $h(u, 0) > 0$ and $h(u, 1) > 0$ for all $u \in S^1$. Moreover, every flow line of $\widetilde{X}$ starting on $S^1 \times \{0\}$ reaches $S^1 \times \{1\}$. Therefore, there exists a ``graphical'' diffeomorphism $\psi : S^1 \times [0,1] \rightarrow S^1 \times [0,1]$ preserving the vertical foliation by intervals, and sending $\widetilde{\mathcal{G}}_-$ to a foliation which is transverse to both the vertical foliation by intervals and the horizontal foliation by circles. Here, we are using that $\widetilde{\G}_-$ has no closed leaves. Pulling back the latter foliation by circles along $\psi$ and then along $\varphi$ yields a foliation by circles on $C$ extending $\gamma$, satisfying the desired transversality conditions. It is then easy to extend it to a smooth foliation by circles on $\partial P$, which is transverse to $\mathcal{G}_+$ and at a distance $\lesssim \varepsilon$ from $H_\partial$. We denote it by $\mathcal{D}_\partial$.

    Now that $\mathcal{D}_\partial$ is constructed, it is easy to extend it to a foliation by disks $\mathcal{D}$ on $P$ with the desired properties, using some cutoff function. Finally, $\xi_\partial$ can be extended by choosing a ``radial'' vector field tangent to $\mathcal{D}$ and tangent to $\xi_\partial$ near $\partial P$ and twisting towards the center of each disk. The strategy is fairly standard and the details are left to the reader.
\end{proof}

            \subsubsection{Proof of  \texorpdfstring{\zcref{thm:fineET}}{thm:fineET}}

Let $\F$ be a cooriented $C^0$-foliation on $M$ as in the statement of the theorem. We focus on constructing a positive contact approximation with the desired properties along $\partial M$. We will modify the foliation in many sequential steps; at each step, the induced modification can be made arbitrarily $C^0$-small, independently of the previous steps. We will carefully keep track of how the boundary foliation is modified in the process.

\paragraph{Making leaves isolated.}

First, we make (nonplanar) compact leaves isolated using \zcref{lem:finiteleaves}. The boundary foliation is modified by blowdowns and $\natural$ blowups supported away from the boundaries of compact planar leaves and $\T^2$-type planar minimal sets.

\paragraph{Creating holonomy.}

Then, we create holonomy near nonperipheral minimal sets using \zcref{lem:enoughholo2}. The boundary foliation is modified by only $\natural$ and $\sharp$ blowups supported away from the boundaries of compact planar leaves and $\T^2$-type planar minimal sets, but could accumulate there. We denote the resulting foliation by $\F_+$.

\paragraph{Drilling maypoles and twisting the foliation.}

Now, we drill maypoles through the $\T^3$-type planar minimal sets as in \zcref{prop:maypoletwist} to obtain a manifold with more boundary components $\check{M} \subset M$ and a foliation $\check{\F}_+$ on $\check{M}$ which is arbitrarily $C^0$-close to $\F_+$ on $\check{M}$, whose minimal sets disjoint from the maypoles coincide with minimal sets of $\F_+$, and which coincides with $\F_+$ near those. Moreover, each component $P$ of $\partial_\parallel \check{M} \coloneqq \partial \check{M} \setminus \partial M$ comes with smooth coordinates $P \cong S^1_x \times S^1_z$ in which $\partial \check{\F}_+$ is $\varepsilon$-close to $\partial_x$, for an arbitrarily small $\varepsilon > 0$, and $\check{\F}_+$ has \emph{positive} slope along $P$ in those coordinates. After some arbitrarily $C^0$-small isotopy of $\check{\F}_+$ near $P$, we may further arrange that its slope is \emph{stably} positive by \zcref{lem:stablypositive}. Near the boundary components of $\partial \check{M}$ coming from $\partial M$, $\check{\F}_+$ and $\F_+$ coincide.

By construction, every point in $\check{M}$ can be joined to either $\partial_\parallel \check{M}$ or to an arbitrarily small neighborhood of a minimal set of $\F_+$ which is not a $\T^3$-type planar minimal set. 

\paragraph{Drilling boundary fences and opening windows.}

We now create annular boundary fences and windows to apply \zcref{prop:betterfence}. For the fences, we first construct \emph{attracting annular fences} for the nonperipheral minimal sets, which admit curves with infinitesimally contracting holonomy by construction, exactly as in~\cite[Section 4]{B16}. These annular fences are embedded annuli whose core are curves with infinitesimally attracting holonomy, and whose top and bottom boundaries are positively transverse to the foliation. Note that we can make these fences arbitrarily thin in the $z$ direction. We also construct fences near finitely many boundary components of \emph{stable} compact planar or $\T^2$-type planar minimal sets; by definition, any such minimal set has a closed boundary curve with infinitesimally attracting holonomy, and we may push this curve inside of $M$ and construct a thin annular fence near it. After some arbitrarily small $C^0$-isotopies, we can arrange that $\F_+$ is smooth near boundaries of the fences, and we can remove neighborhoods of the fences to produce boundary annular fences as in \zcref{def:manfence}, so that $\F_+$ is adapted to the boundary fences as in \zcref{def:adapted}. For the windows, we choose points on $\partial M \subset \partial \check M$ for each spiral minimal set, which lie on the boundary of leaves in the interior of a $\sharp$, $\flat$, or Reeb annulus, and we create windows near them by some arbitrarily small $C^0$-isotopies of $\F_+$ making it smooth near those points. We do the same near each chosen point in the boundaries of the unstable compact planar leaves and $\T^2$-type planar minimal sets. We treat the components of $\partial_\parallel \check M$, corresponding to the maypoles, as toroidal windows. 

\paragraph{Improving boundary foliations and contact approximation.}

We have now constructed a manifold with fences and windows $M_\ssq$ from $\check M$, such that the foliation $\F_\ssq$ induced by (the small modifications of) $\F_+$ is adapted to the boundary of $M_\ssq$ and is moreover $\partial_\ssq$-transitive. Applying \zcref{prop:betterfence}, we approximate $\F_\ssq$ by a $C^0$-foliation $\widetilde{\F}_\ssq$ whose boundary foliations have stably positive slopes along the components corresponding to the maypoles, and are positively transverse to foliations which are isotopic to arbitrarily small $\flat$-modifications of $\partial \F_\ssq$ in the windows. For such a boundary component $C$, which is a boundary component of $M$, we denote by $\G_C$ the corresponding foliation on $C$ transverse to $\partial \F_\ssq$. By \zcref{lem:spiral}, we can ignore the $\flat$-modifications occurring at the boundaries of spiral minimal sets. Then, treating the boundary components containing square windows as toroidal windows, we can apply \zcref{prop:contapprox} to obtain a positive contact structure $\xi_\ssq$ on $M_\ssq$ which is arbitrarily $C^0$-close to $\widetilde{\F}_\ssq$, which has positive slopes along the components of $\partial M_\ssq$ corresponding to maypoles, and which induces characteristic foliations positively transverse to $\G_C$ along the other boundary components which are not fences. 

\paragraph{Filling the boundary fences and maypoles.}

By construction, we can assume that the boundary annular fences are arbitrarily thin, and that $\xi_\ssq$ is arbitrarily close to the standard horizontal foliation in the chosen smooth coordinates near the maypole boundary components; we can therefore apply~\cite[Lemma 5.6]{B16} and  \zcref{lem:maypolefilling}, respectively, to obtain a contact structure $\xi$ on $M$ which is $C^0$-close to $\F$, and whose boundary characteristic foliation is positively transverse to $\G_C$ along each component $C$ of $\partial M$. By construction, $\G_C$ is obtained from $\F$ by only applying the listed modifications in the statement of the theorem, and the proof is finally complete. \qed

    \section{Converse to Eliashberg--Thurston with boundary}

Let us recall some notions from~\cite{M24}, which immediately adapt to the case of manifolds with boundary.

Let $\xi_-$ and $\xi_+$ be cooriented contact structures on $M$ which are negative and positive, respectively, and transverse to $\partial M$. We associate to the pair $(\xi_-, \xi_+)$ a smooth vector field $X$ called the \textbf{axis} of the pair as follows. First, we choose contact forms $\alpha_\pm$ for $\xi_\pm$ satisfying the balancing condition
$$\alpha_+ \wedge d\alpha_+ = - \alpha_- \wedge d\alpha_- = \mathrm{dvol} > 0,$$
for some auxiliary volume form $\mathrm{dvol}$, and we define $X$ by
$$\iota_X \mathrm{dvol} = \alpha_- \wedge \alpha_+.$$
While this a priori depends on the choice of $\mathrm{dvol}$, an immediate computation shows that any other volume form yields the same vector field $X$. Moreover, it satisfies:
\begin{itemize}
    \item $\{X = 0\} = \{ \xi_- = \xi_+\} \eqqcolon \Delta$,
    \item Away from $\Delta$, $\xi_- \pitchfork \xi_+ = \langle X \rangle$.
\end{itemize}
Recall from \zcref{sec:contact} that the contact pair $(\xi_-, \xi_+)$ is \textbf{positive} if along $\Delta$, $\xi_-$ and $\xi_+$ agree with the \emph{same} orientation (hence \emph{opposite} coorientations). Equivalently, it means that there exists a smooth flow positively transverse to both $\xi_\pm$.

\medskip

Let $\Phi$ be a nonsingular nonwandering flow on $M$ tangent to $\partial M$. The main result of this section is

\begin{thm}[Converse of Eliashberg--Thurston with boundary]\label{thm:ETconverse}
    Let $(\xi_-, \xi_+)$ be a positive contact pair on $M$ transverse to $\partial M$ and positively transverse to $\Phi$. Assume that the axis $X$ of $(\xi_-, \xi_+)$ points inward along $\partial M$. Furthermore, let $\mathcal{G}$ be a smooth foliation on $\partial M$ which is transverse to $\Phi \vert_{\partial M}$ and which is dominated by ${\xi_-}\vert_{\partial M}$ and ${\xi_+}\vert_{ \partial M}$ along $\partial M$.

    Then, there exists a taut $C^0$-foliation $\mathcal{F}$ on $M$ transverse to $\Phi$ and to $\partial M$, such that $\partial \mathcal{F}$ is a blowup of $\mathcal{G}$.
\end{thm}

In fact, we construct a \emph{branching foliation} transverse to $\Phi$ and $\partial M$, whose restriction to $\partial M$ is tangent to $T\mathcal{G}$. We will then ``separate its leaves'', which might modify its restriction to the boundary, but the resulting (genuine) foliation will be semiconjugate to $\mathcal{G}$ and not necessarily isotopic to it.

Note that the existence of such a ``sandwiched'' foliation $\G$ is guaranteed by (and equivalent to) the inward transversality of the axis $X$ along $\partial M$.

While the proof in~\cite{M24} for the closed setting could be adapted to this setting with boundary without too much complication, we will instead use a reflection trick to reduce to the situation without boundary.

The construction of foliations in~\cite{M24} holds for more general positive contact pairs, see~\cite[Theorem B]{M24}. A similar result would hold in our setup with boundary, but we focus on the case of positive pairs transverse to a nonwandering flow as it is closer to the main applications of the present article.

        \subsection{Reflection trick}

In this section, we extend the contact pair to the doubling of $M$ along its boundary to obtain a positive contact pair on a \emph{closed} manifold, in order to apply the main results of~\cite{M24}. We fix a nonwandering flow $\Phi$ and positive contact pair $(\xi_-, \xi_+)$ as in \zcref{thm:ETconverse}.

The flow of $X$ gives coordinates $[-1,0]_\tau \times \partial M$ near $\partial M$, where we identify $\{0\} \times \partial M$ with $\partial M$. In those coordinates, $\xi_\pm$ are of the form 
$$\xi_\pm = \langle \partial_\tau \rangle \oplus \ell^\pm_\tau$$
for families of line fields $(\ell^\pm_\tau)_\tau$ on $\partial M$ which are twisting positively/negatively as $\tau$ increases. Dually, we choose smooth $1$-forms $\alpha$ and $\beta$ on $\partial M$ such that $\ker \alpha = T \G$ and $\alpha \wedge \beta > 0$, and there exists maps $f_\pm : [-1,0] \times \partial M \rightarrow \R_{>0}$ such that
$$\alpha_\pm = \pm \alpha + f_\pm \beta$$
is a contact form for $\xi_\pm$. The signs of $f_\pm$ are guaranteed by the domination hypothesis along $\partial M$. Note that the contact conditions are simply
$$ \partial_\tau f_\pm < 0.$$
We now modify $f_\pm$, hence $\xi_\pm$, near $\{0\} \times \partial M$, into $\widehat{f}_\pm : [-1, 0] \times \partial M \rightarrow \R_{\geq 0}$ satisfying
\begin{itemize}
    \item $\partial_\tau \widehat{f}_\pm < 0$,
    \item $\widehat{f}_\pm = f_\pm$ near $\{-1\} \times \partial M$,
    \item $\widehat{f}_\pm = -\tau$ near $\{0\} \times \partial M$.
\end{itemize}
We then set 
$$\widehat{\alpha}_\pm \coloneqq \pm \alpha + \widehat{f}_\pm \beta, \qquad \widehat{\xi}_\pm \coloneqq \ker \widehat{\alpha}_\pm.$$
In the collar region, the new axis $\widehat{X}$ is of the form $\widehat{X} = h \partial_\tau$, where $h \leq 0$, $h^{-1}(0)= \{0\} \times \partial M$, and $h= 2 \tau$ near $\{0\} \times \partial M$.

We can now extend $\big(\widehat{\xi}_-, \widehat{\xi}_+\big)$ to the doubling of $M$ along its boundary, denoted by $\overline{M}$. We have a canonical inclusion
$$M \subset \overline{M} = M \cup_{\partial M}(-M), \qquad C \coloneqq [-1,1] \times \partial M \subset \overline{M}.$$
Note that on $-M$, $\widehat{\xi}_+$  (resp.~$\widehat{\xi}_-$) becomes a \emph{negative} (resp.~\emph{positive}) contact structure. Therefore, we can glue $\widehat{\xi}_+$ (resp.~$\widehat{\xi}_-$) on $M$ to $\widehat{\xi}_-$  (resp.~$\widehat{\xi}_+$) on $-M$ and with reversed coorientation to obtain a positive (resp.~negative) contact structure $\overline{\xi}_+$ (resp.~$\overline{\xi}_-$) on $\overline{M}$. In the collar $C$, they admit contact forms
$$\overline{\alpha}_\pm = \pm \alpha + \overline{f}_\pm \beta,$$
where
$$\overline{f}_\pm(x, \tau) = \begin{cases}
    \widehat{f}_\pm(x,\tau) & \mathrm{if \ } \tau \leq 0,\\
    -\widehat{f}_\mp(x,-\tau) & \mathrm{if \ } \tau \geq 0.
\end{cases}$$
In particular, the axis $\overline{X}$ of $\big(\overline{\xi}_-, \overline{\xi}_+\big)$ satisfies $\overline{X} = 2 \tau \partial_\tau$ near $\{0\} \times \partial M \subset C$.

By construction, $\big(\overline{\xi}_-, \overline{\xi}_+\big)$ is a positive contact pair which is moreover transverse to the doubling $\overline{\Phi}$ of $\Phi$, which is still nonwandering. Moreover, along $\{0\} \times \partial M$, we have
\begin{align} \label{eq:alongpartialM}
    \overline{\xi}_- \cap T\left(\{0\} \times \partial M\right) = \overline{\xi}_+ \cap T\left(\{0\} \times \partial M\right)= T \G.
\end{align}

We are now in the situation of constructing a $C^0$-foliation on a \emph{closed} $3$-manifold $\overline{M}$ from a positive contact pair $\big(\overline{\xi}_-, \overline{\xi}_+\big)$, which is the setup of~\cite{M24}. Recall that any such pair has an associated \emph{unstable plane field} defined as the common limit
\begin{align} \label{eq:etabar}
    \overline{\eta} \coloneqq \lim_{t \rightarrow +\infty} (\varphi^t_{\overline{X}})_* \overline{\xi}_\pm.
\end{align}
Here, $\overline{\eta}$ is a continuous plane field which is tangent to $\overline{X}$ and invariant under its flow. A direct computation using the special form of $\overline{X}$ in the collar and~\eqref{eq:alongpartialM} further shows that it is tangent to $\G$ in the collar region $C = [-1,1] \times \partial M$. In summary, we have proved:

\begin{prop} \label{prop:contpairdouble}
    There exists a positive contact pair $\big(\overline{\xi}_-, \overline{\xi}_+\big)$ on $\overline{M}$, the doubling of $M$, which is transverse to $\overline{\Phi}$, and whose unstable plane field $\overline{\eta}$ satisfies \begin{align*}
        \overline{\eta} = \langle \partial_\tau \rangle \oplus T \G
    \end{align*}
    in the collar region $C$.
\end{prop}

\begin{rem}
    While the construction of $\big(\overline{\xi}_-, \overline{\xi}_+\big)$ involved some choices of modifications of $(\xi_-, \xi_+)$, the resulting unstable plane field $\overline{\eta}$ does not. Indeed, it is prescribed in the collar region, hence in the union of all the flow lines of $X$ intersecting the collar region, and away from it, it is uniquely determined by~\eqref{eq:etabar} with $\overline{X} = X$.
\end{rem}

We obtain another useful byproduct of the reflection trick:
\begin{cor} \label{cor:universallytight}
    Both $\xi_\pm$ are universally tight.
\end{cor}

\begin{proof}
    The contact structures $\overline{\xi}_\pm$ on the \emph{closed} manifold $\overline{M}$ are universally tight by \zcref{lem:universallytight}. Moreover, $\widehat{\xi}_\pm$ are isotopic rel.\ boundary to \emph{extensions} of $\xi_\pm$ obtained by gluing a collar neighborhood along $\partial M$ and adding radial twisting. Therefore, $\xi_\pm$ become \emph{restrictions} of $\widehat{\xi}_\pm$ and are universally tight as well.
\end{proof}

        \subsection{Proof of \texorpdfstring{\zcref{thm:ETconverse}}{thm:ETconverse}}

Let $(\xi_-, \xi_+)$ and $\G$ be as in \zcref{thm:ETconverse}, and let $\big(\overline{\xi}_-, \overline{\xi}_+\big)$ be the contact pair on $\overline{M}$ constructed in \zcref{prop:contpairdouble}, with associated unstable plane field $\overline{\eta}$. Since the latter contact pair is transverse to a nonwandering flow, namely, the doubling $\overline{\Phi}$ of $\Phi$, it is taut. We can then apply~\cite[Theorem C]{M24} (and~\cite[Remark 0.7]{M24}) to deduce that $\overline{\eta}$ is tangent to a \emph{branching foliation}; this is a collection of complete, immersed surfaces in $\overline{M}$ tangent to $\overline{\eta}$ which cover $\overline{M}$, and which are allowed to intersect but not to \emph{topologically cross} (see~\cite[Definition 4.1]{BI08}). This branching foliation can then be resolved into a genuine foliation, up to a small $C^0$-modification of $\overline{\eta}$. More precisely, it is shown in~\cite[Section 7]{BI08} that for every $\varepsilon > 0$, there exists a genuine $C^0$-foliation $\overline{\F}_\varepsilon$ on $\overline{M}$ and a continuous surjective map $\overline{h}_\varepsilon : \overline{M} \rightarrow \overline{M}$ satisfying:
\begin{itemize}
    \item The $C^0$-distance between $\overline{\eta}$ and $T \overline{\F}_\varepsilon$ is less than $\varepsilon$,
    \item The $C^0$-distance between $\overline{h}_\varepsilon$ and $\mathrm{id}$ is less than $\varepsilon$,
    \item $\overline{h}_\varepsilon$ is a leafwise immersion, and it sends the leaves of $\overline{\F}_\varepsilon$ to surfaces tangent to $\overline{\eta}$.
\end{itemize}
Recall that in the collar $C = [-1,1] \times \partial M \subset \overline{M}$, $\overline{\eta}$ is smooth and tangent to the product of $\G$ with $[-1,1]$; in particular, it is \emph{uniquely} integrable. Choosing $\varepsilon$ small enough, the third item above implies that the restriction of $\overline{\F}_\varepsilon$ to $C$ is a blowup of $[-1,1] \times \G$. The restriction of $\overline{\F}_\varepsilon$ to $M$ is the desired foliation. \qed

\clearpage
\part{Architecture of ziggurats}\label{sec:mainresults}

Throughout this part, and unless stated otherwise, we adhere to \zcref{assum:1}: $M$ is a compact, oriented 3-manifold with non-empty, framed boundary consisting of tori, different from $S^1 \times D^2$ and $\T^2 \times I$. Moreover, $\Phi$ is a smooth, nonwandering flow on $M$ whose restriction to each boundary component of $\partial M$ is a Reebless foliation with slope $+\infty$. The exceptional cases $S^1\times D^2$ and $\T^2\times I$ are treated in \zcref{sec:zigannulusdisk}.

The main object of interest in this part is the ziggurat $\mathcal{Z}(\Phi)\subset \R^n$ associated with $\Phi$, defined as the set of boundary multislopes realized by foliations transverse to $\Phi$. We combine the various results from \zcref{sec:foundations} and \zcref{sec:ETandback} to prove many structural properties of this set.

    \section{Basic properties of contact ziggurats}

Recall that $\mathcal{Z}^\pm(\Phi) \subset \R^n$ denotes the set of boundary multislopes realized by \emph{positive/negative contact structures} transverse to $\Phi$.

\begin{defn}[Downward closed]  A subset $A \subset \R^n$ is \textbf{downward closed} if for all $\bm{x} \in A$ we have $(\bm{-\infty}, \bm{x}] \subset A$. We say that $A\subset \R^n$ is \textbf{upward closed} if for all $\bm{x} \in A$ we have $[\bm{x},\bm{+\infty})\subset A$.
\end{defn}

\begin{prop} \label{prop:contactboxconv}
    $\Zg^+(\Phi)$ is downward closed and $\Zg^-(\Phi)$ is upward closed. In particular, they are box-convex.
\end{prop}

\begin{proof}
By \zcref{lem:trim_twist}, we can add twisting to any $\Phi$-horizontal positive contact structure to decrease its slope. In other words, if $\bm{s}\in \Zg^+(\Phi)$ and $\bm{s'} \leq \bm{s}$, then $\bm{s}'\in \Zg^+(\Phi)$. Similarly, we can add twisting to any $\Phi$-horizontal negative contact structure to increase its slope.
\end{proof}

    \section{Ziggurat intersection theorem}

In this section, we prove \zcref{thmintro:contactzigg} from the introduction.

\begin{proof}
    The inclusion $\Zg(\Phi) \subset \Zg^-(\Phi)\cap \Zg^+(\Phi)$ follows from \zcref{cor:coarseETslope}. For the reverse direction, assume there are contact structures $\xi_+$ and $\xi_-$ with $$\bm\slope(\xi_+)=\bm\slope(\xi_-)=\bm{s}.$$ By trimming (\zcref{lem:trim_twist}) combined with \zcref{lem:phaselocking}, we can modify $\xi_+$ so that for each $i$, either $\slope_i(\xi_-)> s_i$ or $\fineslope_i(\xi_+)\in \{s^\sharp, s^\dagger\}$. Similarly, we can modify $\xi_-$ so that for each $i$, either $\slope_i(\xi_+)< s_i$ or $\fineslope_i(\xi_+)\in \{s^\flat, s^\dagger\}$. Therefore $\fineslope_i(\xi_-) \prec \fineslope_i(\xi_+)$. It follows that there is a foliation $\G$ of slope $\bm{s}$ such that $$\partial \xi_- \prec \G \prec\partial\xi_+.$$ Applying the converse to Eliashberg--Thurston (\zcref{thmintro:ETconverse}), there exists a $\Phi$-horizontal foliation $\F$ with $\partial \F$ monotone equivalent to $\G$. In particular, $\bm{\slope}(\partial \F) = \bm{s}$, so $\bm{s}\in \Zg(\Phi)$.
\end{proof}

\begin{rem} \label{rem:intersection}
    In the exceptional cases where $\Phi$ is a nonwandering disk or annulus suspension flow, we still have $$\mathcal{Z}^-(\Phi) \cap \mathcal{Z}^+(\Phi) \subset \mathcal{Z}(\Phi).$$
    However, the reverse inclusion should not hold: if $\Phi$ is a nonwandering disk or annulus suspension flow, then we expect 
    $\Zg^-(\Phi) \cap \Zg^+(\Phi) = \varnothing$ (see \zcref{rem:contactsuspension}), while $\Zg(\Phi) \neq \varnothing$ using \zcref{propintro:smoothing}.
\end{rem}

    \section{Box-convexity}

Recall from \zcref{def:boxconv}: A subset $A \subset \R^n$ is box-convex if for all $\bm{x}, \bm{y} \in A$ satisfying $\bm{x} \leq \bm{y}$, we have $[\bm{x}, \bm{y}] \subset A$.

\begin{proof}[Proof of \zcref{thmintro:box-convexity}]
    Since $\Zg^\pm(\Phi)$ are box-convex by \zcref{prop:contactboxconv}, their intersection is also box-convex, and it equals $\Zg(\Phi)$ by \zcref{thmintro:contactzigg}.
\end{proof}

In view of \zcref{rem:intersection}, the proof actually shows a related statement that will be useful later:

\begin{prop} \label{prop:boxconvex}
    If $\bm{s}^- \in \Zg^-(\Phi)$ and $\bm{s}^+ \in \Zg^+(\Phi)$ satisfy $\bm{s}^- \leq \bm{s}^+$, then $[\bm{s}^-, \bm{s}^+] \subset \Zg(\Phi)$.
\end{prop}

\begin{rem}
    If $\Phi$ is an annulus or disk suspension, then box-convexity trivially holds by the explicit descriptions in \zcref{prop:special_ziggurats} and \zcref{rem:homlong}.
\end{rem}

    \section{Irrational instability}

        \begin{proof}[Proof of \zcref{thmintro:irrationality}]
        Let $\F$ be a $\Phi$-horizontal foliation. Apply coarse Eliashberg--Thurston (\zcref{thm:coarseET}) to find a foliation $\G$ on $\partial M$ which is transverse to $\Phi\vert_{\partial M}$, and $\Phi$-horizontal contact structures $\xi_-$ and $\xi_+$ such that $\bm{\slope}(\G) = \bm{\slope}(\partial\F)$ and 
        $$\partial \xi_- \prec \G \prec \partial \xi_+.$$ 
        It follows that for all $1\leq i \leq n$, 
        $$\slope_i(\xi_-) \leq \slope_i(\G) \leq \slope_i(\xi_+).$$ Moreover, whenever $\slope_i(\G)$ is irrational, the instability of irrational slopes (\zcref{lem:phaselocking}) implies that both inequalities are strict, and \zcref{prop:boxconvex} finishes the proof.
        \end{proof}

    \section{Filling rational multislopes}

In this section, we show that the ziggurat $\Zg(\Phi)$ controls the existence of those taut foliations on Dehn fillings of $M$ which are transverse to blowdowns of $\Phi$.

\begin{defn}[Slices]
    Let $A \subset \R^n$ and $\bm{t} \in \R^{n-k}$ for $0 \leq k \leq n-1$. 
    \begin{itemize}
        \item The \textbf{(right) $\bm{t}$-slice of $A$} is defined as
    $$Sl_{\vert \bm{t}}A \coloneqq \left(\R^{k} \times \{\bm{t}\}\right) \cap A.$$
    It can be naturally identified with a subset of $\R^k$.
        \item The set of \textbf{transversally interior points} in the (right) $\bm{t}$-slice $Sl_{\vert \bm{t}}A$, denoted by $Sl^\circ_{\vert\bm{t}} A$, is the set of $\bm s \in Sl_{\vert \bm{t}}A$ such that there exists $\varepsilon > 0$ satisfying 
        $$\{\bm{s}\} \times (\bm{t}-\bm{\varepsilon}, \bm{t} +\bm{\varepsilon}) \subset A.$$
    \end{itemize}
\end{defn}

For $k=0$, the slice $Sl_{\vert \bm{t}}A \subset \R^0 = \{0\}$ is either a single point or is empty depending on whether $\bm{t} \in A$ or not.

\begin{figure}[ht]
        \centering
        \def\svgwidth{0.4\textwidth}
        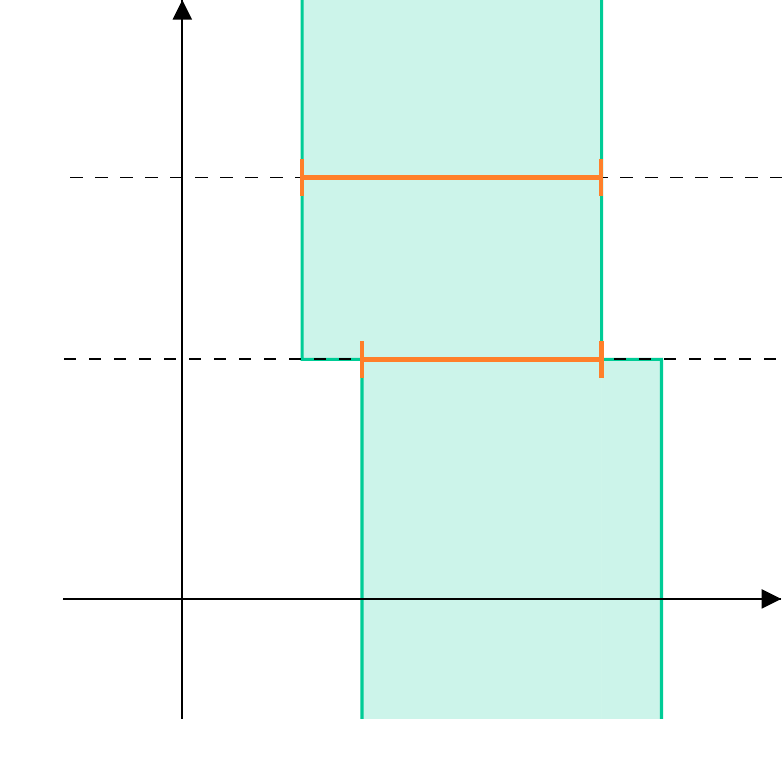
        \caption{A box-convex set $A$ in green and $Sl^\circ_{\vert t} A$ shown in orange for two values of $t\in \R$.}
        \label{fig:slice}
\end{figure}

Let $\Phi\vert_{\bm{t}}$ be a blowdown of $\Phi$ along multislope $\bm{t}$. We always have the inclusion
$$ \Zg(\Phi \vert_{\bm{t}}) \subset Sl_{\vert \bm{t}}\Zg(\Phi) \subset \R^k.$$
The converse inclusion does not hold in general, but the subset $\Zg(\Phi \vert_{\bm{t}})$ can be entirely recovered from $\Zg(\Phi)$, except for the cylindrical or toroidal multislope if $\Phi$ is a planar suspension flow.

\begin{thm}[Ziggurats of Dehn fillings] \label{thm:partialfilling}
    Let $\bm{t} \in \Q^{n-k}$. Let $\Phi\vert_{\bm{t}}$ be a blowdown of $\Phi$ along the multislope $\bm{t}$. If $\Phi\vert_{\bm{t}}$ is not an annulus, disk, or sphere suspension flow, then
    \begin{align} \label{eq:partialslice}
        Sl^\circ_{ \vert \bm{t}} \Zg(\Phi) =  \Zg(\Phi\vert_{\bm{t}}).
    \end{align}
\end{thm}

\begin{proof}
    First, we prove the inclusion $\Zg(\Phi\vert_{\bm{t}}) \subset Sl^\circ_{ \vert \bm{t}} \Zg(\Phi)$. Let $\F$ be any $\Phi|_{\bm{t}}$-horizontal foliation with boundary multislope $\bm{s}$. Apply the coarse Eliashberg--Thurston theorem \zcref{thm:coarseET} to $\F$ to obtain a contact structure $\xi_+'$ whose coarse boundary multislope agrees with that of $\F$. Applying \zcref{lem:fillingcontact}, we may blow up $\xi_+'$ to obtain a contact structure $\xi_+$ with 
    \begin{align*}
    \slope_i(\xi_+)&= s_i &&\textrm{for }i=1,\dots,k,\\
    \slope_{k+i}(\xi_+) &> t_i &&\textrm{for }i= 1,\dots,n-k,
    \end{align*}
    Similarly there is a $\Phi$-horizontal contact structure $\xi_-$ satisfying
    \begin{align*}
    \slope_i(\xi_-)&=s_i &&\textrm{for }i=1,\dots,k,\\
    \slope_{k+i}(\xi_-) &< t_i &&\textrm{for }i= 1,\dots,n-k,
    \end{align*}
    By \zcref{prop:boxconvex}, $[\bm{\slope}(\xi_-), \bm{\slope}(\xi_+)]\subset \Zg(\Phi)$, so $\bm{s'}$ is a transversally interior point in $Sl_{\vert \bm{t}}(\Zg(\Phi))$.

    To prove the reverse inclusion $Sl^\circ_{ \vert \bm{t}} \Zg(\Phi) \subset \Zg(\Phi\vert_{\bm{t}})$, we follow the previous argument backwards. Contact structures $\xi_-$, $\xi_+$ satisfying the slope inequalities above exist by \zcref{thmintro:contactzigg}. Then, filling them with \zcref{lem:fillingcontact} and applying \zcref{prop:boxconvex} guarantees that $\bm{s} \in \Zg(\Phi\vert_{\bm{t}})$.
\end{proof}

\begin{proof}[Proof of \zcref{thmintro:rationalfilling}]
Apply \zcref{thm:partialfilling} to the case $k=0$.
\end{proof}

The results of this section control fillings of foliations with suspension-type (i.e., Reebless) boundaries. Dehn filling of foliations with $\aleph$-type boundaries is treated in \zcref{prop:alephfilling}.

    \section{Limits of spherical and \texorpdfstring{$\mathbb{T}^2$}{T2}-type multislopes}\label{sec:sphericallimits}

The two obstructions to the application of the fine version of the Eliashberg--Thurston with boundary are spherical multislopes and $\T^2$-type planar multislopes. The goal of this section is to show:

\begin{prop}\label{prop:spherical_limits}
    The closure of the spherical multislopes is precisely the union of the spherical multislopes and the $\T^2$-type planar multislopes.
\end{prop}

We emphasize that this is a statement about slopes realized by surfaces, not a statement about ziggurats themselves. Only those compact planar surfaces transverse to $\Phi$ are obstructions to the application of the fine Eliashberg--Thurston theorem to $\Phi$-horizontal foliations, but the theorems of this section concern all the compact planar surfaces properly embedded in $M$.

We define the \textbf{filled Euler measure} of a compact surface $\Sigma$ to be $$\chi^\bullet(\Sigma) \coloneqq \chi(\Sigma) + \#_\partial(\Sigma).$$
This is the Euler characteristic of the surface after filling all boundary components with disks. The definition extends to measured laminations as follows. If a measured lamination $\Lambda$ has irrational boundary slope on $\partial_i M$ then the contribution to $\#_\partial(\Lambda)$ is 0, and if it has rational boundary slope, then the contribution to $\#_\partial(\Lambda)$ is the transverse measure of a curve in $\partial_i M$ which intersect each leaf of $\partial_i \Lambda$ exactly once.

We say that a measured lamination $\Lambda$ has \textbf{coherently oriented boundary} if for each $i$, all the leaves of $\partial_i \Lambda$ have the same (oriented) slope. In this section, the topology on the space of measured laminations is as follows: two measured laminations $\Lambda_1$ and $\Lambda_2$ are $\varepsilon$-close if they are carried by a common branched surface $B$ with at most $1/\varepsilon$ branch sectors, each of diameter $\leq 1$, and the weights induced by $\Lambda_1$ and $\Lambda_2$ on any given sector of $B$ differ by at most $\varepsilon^2$.
\begin{lem}\label{lem:uppersemicontinuous}
    $\chi^\bullet$ is upper semi-continuous on the space of coherently oriented laminations, and is continuous at laminations with irrational boundary slope.
\end{lem}

\begin{proof}
First, note that $\chi$ is a continuous function on the space of measured laminations. Indeed, the Euler measure of a measured lamination can be computed as a linear function of the weights induced on a carrying branched surface. If two measured foliations are carried by a common branched surface $B$ having $<1/\varepsilon$ branch sectors and with weights differing by at most $\varepsilon^2$ on each sector, then their Euler measures differ by at most $O(1/\varepsilon\cdot \varepsilon^2)=O(\varepsilon)$.

 For measured laminations with coherently oriented boundary, the function $\#_\partial$ can be expressed as 
 $$\#_\partial=\sum_{i=1}^n \# \circ \partial_i$$
 where $\partial_i$ is the boundary map $\Lambda \mapsto \partial_i \Lambda \in H_1(\partial_i M,\R) \cong \R^2$, and $\#:\R^2\to \R$ is a cousin of the ``popcorn function'':
\[
\#(x,y) \coloneqq 
\begin{cases}
\left\vert \frac{x}{\denominator(y/x)} \right\vert&\text{if } x\neq 0 \text{ and } y/x \in \Q, \\
0 & \text{if } x\neq 0 \text{ and } y/x \in \R \setminus \Q, \\
\vert y\vert & \text{if } x=0.
\end{cases}
\]
The boundary map is easily checked to be continuous. Also, $\#$ is upper semi-continuous, and continuous only at laminations with all boundary slopes irrational. Therefore, the same is true of $\chi^\bullet$.
\end{proof}

\begin{lem}\label{lem:non-negative_classification}
    Let $\Lambda$ be a minimal measured lamination on $M$ with coherently oriented boundary. If $\chi^\bullet(\Lambda)>0$, then $\Lambda$ is a compact genus $0$ surface. If $\chi^\bullet(\Lambda)=0$, then either $\Lambda$ is a $\T^2$-type planar minimal set, or $\Lambda$ has no irrational boundary components. 
\end{lem}

\begin{proof}
   If $\Lambda$ has any rational boundary components, then we can fill them without changing the filled Euler measure and without changing the conclusion of the lemma. Thus, we may assume without loss of generality that $\Lambda$ has no rational boundary components. Then the Euler measure of $\Lambda$ equals its filled Euler measure. By Candel's theorem \cite{candel.UniformizationSurfaceLaminations}, $\Lambda$ has a leafwise Riemannian metric of constant curvature equal to $\chi(\Lambda)$. If $\chi(\Lambda)>0$, then $\Lambda$ must be a sphere. If $\chi(\Lambda)=0$, then $\Lambda$ has a leafwise Euclidean metric. Since $\Lambda$ is a measured minimal set, all of its leaves have the same topology. There are three options for the topology of leaves of $\Lambda$. If $\Lambda$ has toroidal or cylindrical leaves, then the boundary of any leaf is compact; in particular, $\Lambda$ has no irrational boundary components. If $\Lambda$ has planar leaves, then by \zcref{cor:simplyconnleaves} $\Lambda$ is either a $\T^2$-type planar minimal set or a $\T^3$-type planar minimal set. A $\T^3$-type planar minimal set does not have irrational boundary components, so we are done.
\end{proof}

\begin{lem} \label{lem:spherical_limits1}
If $\Lambda$ is a measured lamination on $M$ which can be approximated by coherently oriented measured laminations of positive filled Euler measure, then it is a union of compact planar surfaces or $\T^2$-type planar minimal sets.
\end{lem}

\begin{proof}
   Suppose we have a sequence $(\Lambda_j)_{j\geq 1}$ of compact genus zero surfaces with coherently oriented boundary approximating $\Lambda$ as measured laminations. By the upper semicontinuity property of $\chi^\bullet$, we must have $\chi^\bullet(\Lambda)\geq 0$. By \zcref{lem:non-negative_classification}, each component of $\Lambda$ is either a compact planar surface or a $\T^2$-type planar minimal set (in which case we are done), or has no irrational boundary components.
   
   Suppose $\Lambda$ has no irrational boundary slopes. Since the function $\#$ as defined in \zcref{lem:uppersemicontinuous} is strictly upper semicontinuous at rational slopes, there is a neighborhood of $\Lambda$ in the space of measured laminations and $c>0$ such that every lamination $\Lambda'$ in this neighborhood either has $\chi^\bullet(\Lambda') < \chi^\bullet(\Lambda)-c$ or $\Lambda'$ has the same rational slope as $\Lambda$. Therefore, the slopes of the $\Lambda_j$ are eventually equal to the slopes of $\Lambda$. Then we may fill these rational boundary components, obtaining a sequence of spheres $(\Lambda^\bullet_j)_j$ approximating a measured lamination $\Lambda^\bullet$. But then we must have $\chi(\Lambda^\bullet) > 0$. Therefore, $\Lambda^\bullet$ itself must have positive Euler measure, and therefore is a linear combination of essential spheres.
\end{proof}

\begin{lem}[Nearby compact planar surfaces]\label{lem:spherical_limits2}
    Assume that $\Lambda$ is a $\T^2$-type planar minimal set $\Lambda$. Then there exists a sequence of compact planar surfaces converging to $\Lambda$ as measured laminations.
\end{lem}

\begin{proof}
    It suffices to check this for $\Lambda$ a model $\T^2$-type planar lamination. Recall the definition of such from \zcref{def:pems}: start with an irrational linear foliation on $\T^2\times I$, apply a Denjoy blowup to get a lamination $\Lambda^\bullet$, and then drill out the tubular neighborhood of a transverse link $L$ to get $\Lambda$. Choose an $I$-invariant branched surface $B^\bullet$ carrying $\Lambda^\bullet$ which is fine enough that $L$ is transverse to $B^\bullet$. Then after drilling out $L$, we get a branched surface $B$ carrying $\Lambda$. This drilling operation does not change the measures on $B^\bullet$. The measure on $\Lambda$ induces a measure on $B$ and equivalently on $B^\bullet$. We may approximate this irrational measure with a rational one. By $I$-invariance of $B$, all measures on $B^\bullet$ with rational weights correspond to unions of annuli. Therefore, these rational weights on $B$ correspond to unions of compact planar surfaces approximating $\Lambda$, as desired.
\end{proof}

\begin{proof}[Proof of \zcref{prop:spherical_limits}]
    Suppose $\bm{s}$ is an accumulation point of spherical multislopes $\{\bm s^j\}_{j\geq 1}$. For each $j$, let $\Sigma_j$ be a compact planar surface with coherently oriented boundary certifying that $\bm{s}^j$ is a spherical multislope.
    
    Choose a triangulation of $M$. All Thurston norm minimizing surfaces in $M$, and in particular all the $\Sigma_j$'s, can be put in normal position with respect to this triangulation. The normal weights of some subsequence of the $\Sigma_j$'s converges; call the corresponding limiting measured lamination $\Lambda$. By \zcref{lem:spherical_limits1}, $\Lambda$ has a component which is a compact planar surface or a $\T^2$-type planar minimal set. Furthermore, $\Lambda$ has coherently oriented boundary because all the $\Sigma_j$'s do. Thus, $\bm{s}$ is a spherical or $\T^2$-type planar multislope.

    Conversely, if $\bm{s}$ is a $\T^2$-type planar multislope, then  \zcref{lem:spherical_limits2} shows that $\bm{s}$ is a limit of spherical multislopes.
\end{proof}

    \section{Rigidity of spherical and \texorpdfstring{$\mathbb{T}^2$}{T2}-type multislopes}

    \subsection{Spherical multislopes and contact structures} \label{sec:S2contact}
        
\begin{lem}[Horizontal $S^2$] \label{lem:transverseS2}
Suppose $\Phi$ is nonsingular flow on $M$ admitting a transverse sphere. Then $M$ does not admit a tight (positive or negative) $\Phi$-horizontal contact structure.
\end{lem}

\begin{proof}
Let $\xi$ by a $\Phi$-horizontal contact structure on $M$. After a $C^\infty$-small perturbation, we can arrange that the characteristic foliation of $\xi$ along the transverse sphere $S$ has nondegenerate singularities, which must all be \emph{positive} by assumption. The Poincar\'{e}--Bendixson theorem and the Poincar\'{e}--Hopf theorem then imply that $\xi$ must have a closed characteristic along $S$, which guarantees the existence of an overtwisted disk.
\end{proof}

We now prove a more refined statement in the case of nonwandering suspension flows on $S^1 \times S^2$:

\begin{prop} \label{prop:contactS2suspension}
Let $\Phi$ be a nonwandering suspension flow on $S^1\times S^2$. Then $\Phi$ admits no horizontal contact structure, tight or not.
\end{prop}

We will in fact prove a slightly stronger result: if $\Phi$ is \emph{any} suspension flow on $S^1 \times S^2$, and if $\Phi$ admits a transverse contact structure $\xi$, then there exists an embedded torus $T$ transverse to $\Phi$ whose complement is disconnected. In particular, $\Phi$ cannot be nonwandering. 

Before giving the proof, we introduce some tools and explain the strategy. By \zcref{lem:transverseS2}, we already know that any $\Phi$-horizontal contact structure must be overtwisted. However, the nonwandering hypothesis implies some extra rigidity. 

We now fix a flow $\Phi$ as in the proposition, as well as a $\Phi$-horizontal contact structure $\xi$. For $t \in S^1$, we set $S_t \coloneqq \{t\} \times S^2$. In particular, $\xi$ has \emph{no negative tangencies} along $S_t$. We fix a distinguished point $0 \in S^1 \cong \R \slash \Z$, and identify $S^1 \setminus \{0\}$ with the interval $(0,1)$.

\begin{lem} \label{lem:convexbifurcation}
    After some arbitrarily $C^\infty$-small perturbation of $\xi$, there exists finitely many times $0 < t_1 < \dots < t_N < 1$ such that 
    \begin{itemize}
        \item If $t \neq t_i$, then $S_t$ is $\xi$-convex ,
        \item At $t=t_i$, a pair of parallel closed characteristics of $\xi|_{S_t}$ is born or dies.
    \end{itemize}
    Moreover, when a pair of characteristics is born (in forwards time), a $\flat$ annulus appears between them, and when a pair of characteristics dies, a $\sharp$ annulus disappears between them (see \zcref{fig:birthdeath}).
\end{lem}

    \begin{figure}[ht]
        \centering
        \def\svgwidth{\textwidth}
        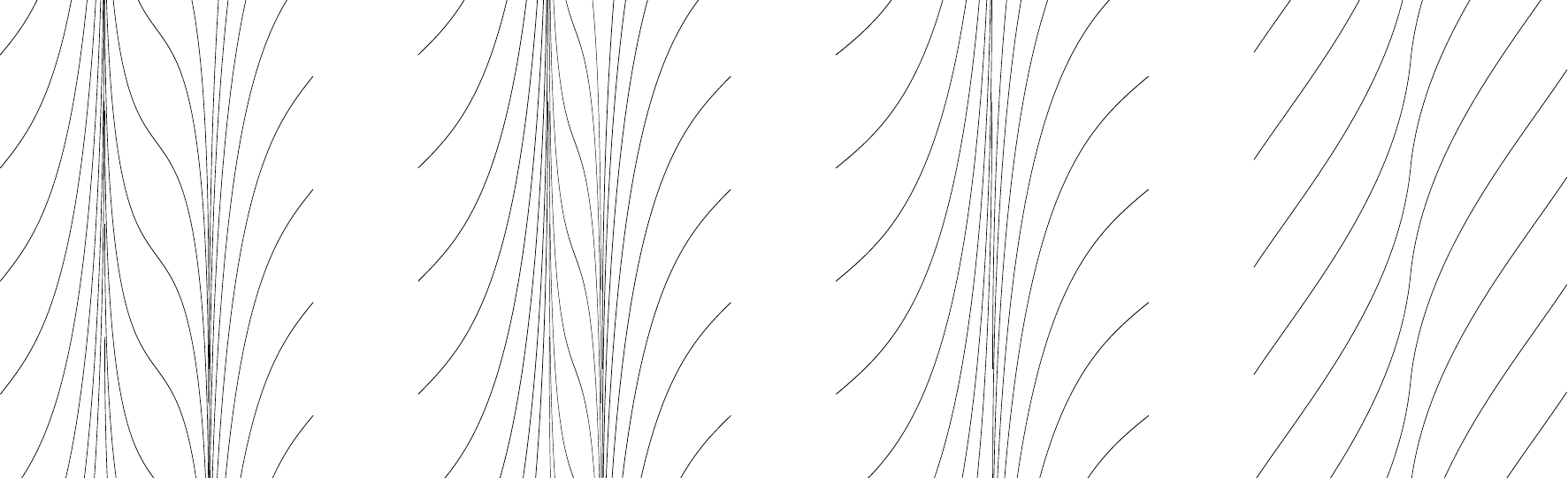
        \caption{Permissible death of a pair of closed characteristics. Note how the tangent line marked in red rotates clockwise.}
        \label{fig:birthdeath}
    \end{figure}

\begin{proof}
    The first part follows from the analysis in~\cite{ giroux.StructuresContactDimension,massot.TopologicalMethods3dimensional}, and the fact that there are no ``retrograde connection'' from a negative to a positive saddle point of $\xi|_{S_t}$, since this characteristic foliation has no negative singularities at all. Note that this argument crucially uses that the fibers are spheres.

    We now justify the second part on the type of annuli that can appear or disappear. At $t = t_i$, the characteristic foliation of $\xi$ along $S_{t_i}$ has a single degenerate closed characteristic $\gamma_0$, which is repelling on one side and attracting on the other side. Let us assume that for $t>t_i$ slightly larger, $\gamma_0$ gives rise to two parallel (nondegenerate) closed characteristics, one attracting and one repelling. Denote them by $\gamma_-$ and $\gamma_+$, respectively. We want to show that $\gamma_+$ is on the left of $\gamma_-$, so that they bound a $\flat$ annulus. Equivalently, we want to show that $\gamma_0$ is repelling on the left, and attracting on the right. Since $\xi$ is transverse to $\gamma_0$, we may find a smooth vector field $X$ defined in a neighborhood of $\gamma_0$ in $S^1 \times S^2$, which is tangent to $\xi$ and positively transverse to the sphere $S_t$ for $t$ close to $t_i$. We may rescale this vector field so that its flow preserves the foliation by spheres (at least in a neighborhood of $\gamma_0$). We may then use this flow to define coordinates $(x,y,z)$ near $\gamma_0$, such that $X$ becomes $\partial_z$ and the spheres $S_t$ are tangent to the $(x,y)$-planes. In these coordinates, $\xi$ is tangent to $\partial_z$, and the contact condition implies that $\xi$ twists in a clockwise direction in the $(x,y)$-planes when $z$ increases. Therefore, $\gamma_0$ must be repelling on the left and attracting on the right since otherwise it would disappear when $t$ increases. The case of a death of two canceling nondegenerate orbits is similar. 
\end{proof}

After applying the lemma, we consider a finite sequence of directed graphs $G_0, \dots, G_N$ on $S^2$, constructed as follows. For each $i$, choose a $\tau_i \in (t_i, t_{i+1})$. 
\begin{itemize}
    \item The vertices of $G_i$ are the connected components of the set $R_+$ of the characteristic foliation of $\xi$ on $S_{\tau_i}$,
    \item The edges of $G_i$ are in correspondence with the connected components of $R_-$. Since $\xi$ has no negative tangencies with $S_{\tau_i}$, each component of $R_-$ is an annulus containing an isolated closed characteristic with attracting holonomy. The corresponding edge connects the neighboring components of $R_+$.    
    \item The edges are oriented according to the orientation of the characteristics of the components of $R_-$ they enclose as in \zcref{fig:edge_expansion}.
\end{itemize}

Note that $G_i$ does not depend on the specific choice of $\tau_i$, since the dividing sets and $R_\pm$ regions of convex surfaces are only modified by isotopies when varying them among a family of convex surfaces. Moreover, the orientations of the edges are also preserved since there are no negative singularities, and the attracting closed characteristics in $R_-$ are also modified by isotopies only.

\begin{figure}[ht]
    \centering
    \begin{subfigure}{0.9\textwidth}
        \centering
        \def\svgwidth{\textwidth}
        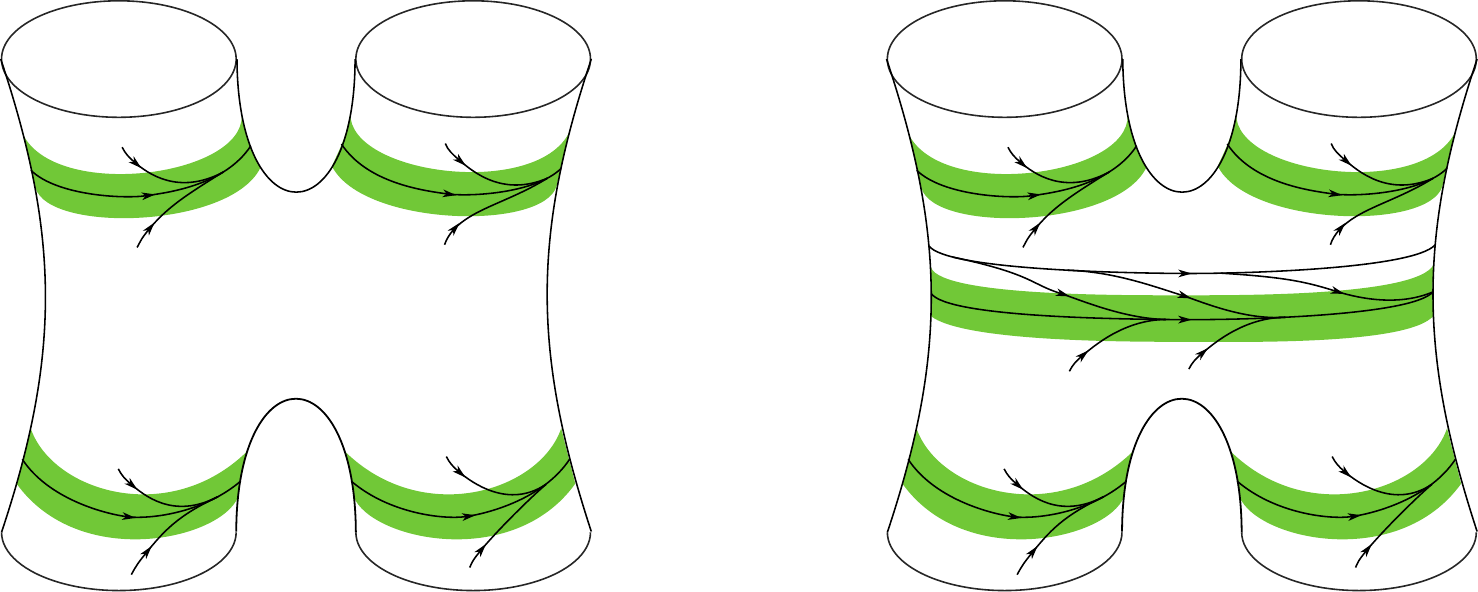
        \caption{Birth of an attracting orbit on the surface.}
        \label{fig:edge_expansion1}
    \end{subfigure}
    
    \vspace{1em}
    \begin{subfigure}{0.5\linewidth}
        \centering
        \def\svgwidth{\textwidth}
        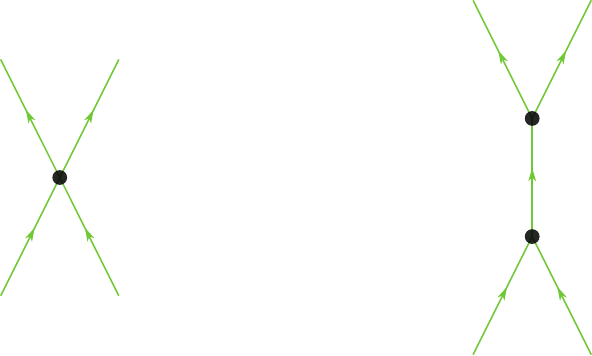
        \caption{Corresponding edge expansion.}
        \label{fig:edge_expansion2}
    \end{subfigure}
    \caption{A birth of a pair of closed characteristics and the corresponding edge expansion. $R_-$ is shown in green.}
    \label{fig:edge_expansion}
\end{figure}

Since the $S_{\tau_i}$'s are convex spheres, the graphs $G_i$ are \emph{trees} with at least one edge each. Moreover, there are elementary operations (edge expansions or edge contractions)\footnote{While an edge contraction is a graph morphism, an edge expansion is not. An edge expansion can be thought of as the dual of an edge contraction.}
\begin{equation} \label{eq:bifurcation}
\begin{tikzcd}
	{G_0} & {G_1} & \dots & {G_{N-1}} & {G_N \cong G_0}
	\arrow["{\varphi_0}", squiggly, from=1-1, to=1-2]
	\arrow["{\varphi_1}", squiggly, from=1-2, to=1-3]
	\arrow["{\varphi_{N-2}}", squiggly, from=1-3, to=1-4]
	\arrow["{\varphi_{N-1}}", squiggly, from=1-4, to=1-5]
\end{tikzcd}
\end{equation}
which keep track of the birth of death of attracting closed characteristics. Note that there is a canonical identification between $G_0$ and $G_N$.

Let us describe these operations in more detail:
\begin{itemize}
    \item If the convex surface structure on $S_{\tau_{i+1}}$ differs from the one on $S_{\tau_i}$ by the birth of a pair of parallel characteristics, one attracting and one repelling, then $G_{i+1}$ differs from $G_i$ by the insertion of an edge, corresponding to the new attracting closed characteristic. The operation $\varphi_i$ is the corresponding edge expansion, with the appropriate orientation. See \zcref{fig:edge_expansion}.
    \item If the convex surface structure on $S_{\tau_{i+1}}$ differs from the one on $S_{\tau_i}$ by the death of a pair of parallel characteristics, one attracting and one repelling, then $G_{i+1}$ differ from $G_i$ by the deletion of an edge, corresponding to the vanishing attracting closed characteristic. The operation $\varphi_i$ is the corresponding edge contraction.
\end{itemize}

Let $0 \leq i \leq N$. A leaf (degree-$1$ vertex) $v$ of $G_i$ is a \textbf{source leaf} (resp.~\textbf{sink leaf}) if its adjacent edge is oriented away from (resp.~towards) it. We denote by $\mathcal{L}^\mathrm{source}_i$ and $\mathcal{L}^\mathrm{sink}_i$ the sets of source and sink leaves of $G_i$, respectively. We set $\mathcal L_i=\mathcal{L}^\mathrm{source}_i \cup \mathcal{L}^\mathrm{sink}_i$. Concretely, a source (resp.~sink) leaf corresponds to a disk component of $R_+$ whose adjacent component of $R_-$ contains an attracting closed curve oriented clockwise (resp.~counterclockwise) with respect to $R_+$. See the leftmost panels in \zcref{fig:dividing_set_S2} for an example of a sink leaf.

\begin{lem} \label{lem:disallowed_expansions}\leavevmode
    \begin{itemize}[beginpenalty=10000]
        \item Sink leaves cannot be born; that is, an edge in $G_{i+1}$ pointing into a sink leaf cannot be created by an edge expansion $\varphi_i$.
        \item Source leaves cannot die; that is, an edge adjacent to a source leaf cannot be contracted by $\varphi_i$.
    \end{itemize}
\end{lem}

\begin{proof}
    Let us prove the first item. Assume by contradiction that a sink leaf is created by $\varphi_i$. We then obtain a disk component $D$ in $R_+(S_{\tau_{i+1}})$ whose characteristics near the boundary converge to an attracting curve $\gamma_-$ oriented counterclockwise with respect to the boundary of $D$. Moreover, another parallel repelling curve $\gamma_+$ is created, which must lie to the left of $\gamma_-$ and must belong to $D$, by \zcref{lem:convexbifurcation}. See \zcref{fig:forbidden_birth}. However, the characteristics emanating to the left of $\gamma_+$ remain in $D$ and cannot limit to a closed characteristic (it would be an attracting curve, which is impossible).  Therefore, at least one must converge to a sink, which is again impossible. 

    \begin{figure}[ht]
        \centering
        \def\svgwidth{0.5\textwidth}
        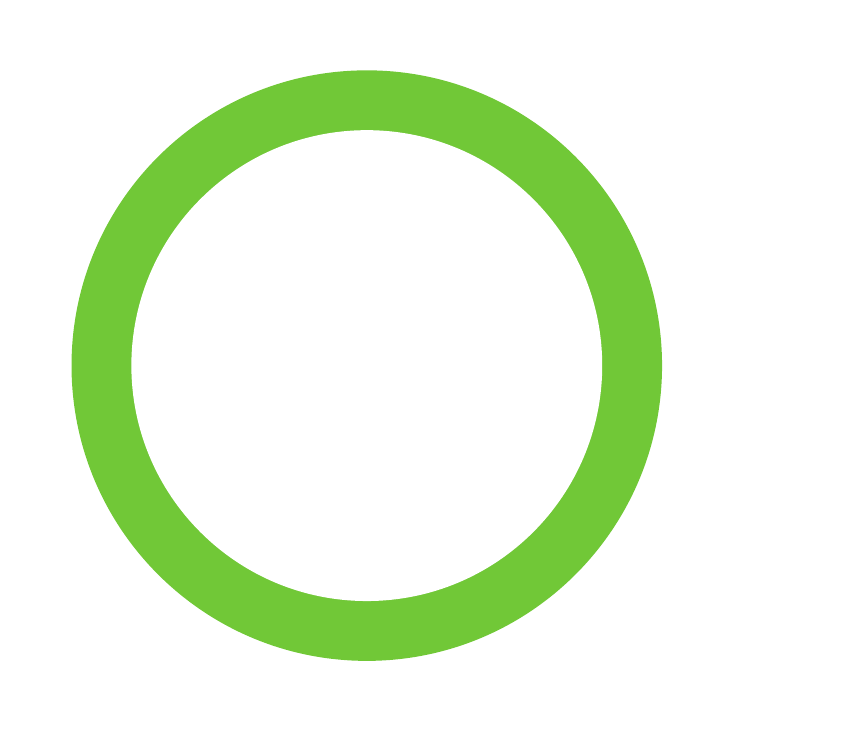
        \caption{A forbidden birth. There is no way to complete the characteristic foliation on the disk in a way compatible with the convex surface structure.}
        \label{fig:forbidden_birth}
    \end{figure}

    The proof of the second item follows from applying time reversal $t\to -t$ composed with reversing the orientation of the $S^2$ fibers; this symmetry preserves the orientation of $M$ and therefore maps $\xi$ to another positive contact structure, but changes births for deaths, sources for sinks, and edge contractions for edge expansions.
\end{proof}

\begin{rem}
    In the proof, we crucially use that the sink and source leaves correspond to attracting curves bounding a \emph{disk}. This guarantees that the characteristic foliation in the corresponding component of $R_+$ cannot have a repelling orbit; otherwise it must have a sink, contradicting that $\xi$ only has positive tangencies with the spheres $S_t$. However, this argument fails if that component of $R_+$ has more than one boundary component.
\end{rem}

\begin{cor}
    There are natural injections 
    $$\sigma_i^\mathrm{source}\colon \mathcal{L}^\mathrm{source}_i \hookrightarrow \mathcal{L}^\mathrm{source}_{i+1}, \qquad \sigma_i^\mathrm{sink}\colon \mathcal{L}^\mathrm{sink}_{i+1} \hookrightarrow \mathcal{L}^\mathrm{sink}_{i}.$$
    Moreover, $$\big\vert\mathcal{L}^\mathrm{source}_0\big\vert=\big\vert\mathcal{L}^\mathrm{source}_N\big\vert, \qquad \big\vert\mathcal{L}^\mathrm{sink}_0\big\vert=\big\vert\mathcal{L}^\mathrm{sink}_N\big\vert,$$
    hence both $\sigma_i^\mathrm{source}$ and $\sigma_i^\mathrm{sink}$ are bijections.
\end{cor}

Now we define bijections $\sigma_i: \mathcal L_i \to \mathcal L_{i+1}$ by $\sigma_i (v)= \sigma_i^{\textrm{source}}$ if $v$ is a source and $\sigma_i(v) = (\sigma_i^{\textrm{sink}})^{-1}$ if $v$ is a sink. We call these the \textbf{leaf tracking maps}. Composing these maps together, we get a bijection
\begin{align*}
    \Sigma& : \mathcal{L}_0 \rightarrow \mathcal{L}_N \cong \mathcal{L}_0\\
    \Sigma& \coloneq \sigma_{N-1} \circ \dots \circ \sigma_0 
\end{align*}

\begin{defn}
We call an edge connected to a leaf a \textbf{twig}.  It is \textbf{incoming} or \textbf{outgoing} depending on whether it is connected to a source or sink leaf. If $v$ is a leaf, we denote the adjacent twig by $\twig(v)$.
\end{defn}

We are now ready to prove \zcref{prop:contactS2suspension}.

\begin{proof}[Proof of \zcref{prop:contactS2suspension}]
    Let $\Phi$ be a suspension flow on $S^1 \times S^2$ and $\xi$ be a positive contact structure transverse to $\Phi$. We apply \zcref{lem:convexbifurcation} and use our bifurcation analysis captured by the trees $G_i$'s and the sequence of elementary operations~\eqref{eq:bifurcation} to construct an embedded torus $L$ in $S^1 \times S^2$, which is transverse to $\Phi$ and homologically trivial in $H_2(S^1 \times S^2; \Z)$. In particular, $S^1 \times S^2 \setminus L$ is disconnected, and $\Phi$ flows from one connected component to another, preventing it from being nonwandering.

    Let us consider a leaf $v$ of $G_0$ and the twig $e$ attached to it. We first illustrate our strategy in a simple case. Let us assume that $v$ is a sink and is not affected by the elementary operations $\varphi_i$. This means that $e$ corresponds to an attracting closed characteristic of $\xi\vert_{S_0}$ which bounds a disk and is not affected by the bifurcations of the characteristic foliations as $t$ increases. We also assume that $\Sigma(v) = v$, i.e., this closed characteristic comes back to itself after going around the $S^1$-direction of $S^1 \times S^2$. We obtain a family of closed curves $\gamma_t \subset S_t$, $t \in S^1$, which trace out an embedded torus $L$ which is homologically trivial. In this context, $L$ is obtained as a union of Legendrian curves bounding overtwisted disks, namely, a \emph{Lutz tube}.
    
    By construction, $L$ is both transverse to the spheres $S_t$'s and to the contact structure $\xi$, and we now argue that $\Phi$ must be transverse to $L$ by analyzing the positive cone bounded by $T S_t$ and $\xi$ along $\gamma_t$. We coorient $L$ so that the direction positively transverse to $\gamma_t$ in $S_t$ is positively transverse to $L$. This convention matches our convention for the orientation of the edges of the trees $G_i$'s.

    Let us consider a fixed characteristic $\gamma_{t_0} \subset S_{t_0}$ for some $t_0 \in S^1$. After a leafwise isotopy near $\gamma_{t_0}$, we may arrange that $\gamma_t$ is constant for $t$ close to $t_0$. The characteristic foliation on the sphere $S_t$ for $t$ close to $t_0$ defines a flow $\Psi$ near $\gamma_{t_0}$ which is tangent to $L$ and to the leaves $S_t$, and which has attracting holonomy around $L$. The linearization of the first return map of $\Psi$ at a point $p \in L$ has two invariant directions: $T_p S_{t_0}$ along which it is contracting, and $T_pL$, along which it is the identity. The contact condition and the effect of this first return map implies that $\xi$ lies in the positive quadrant bounded by $T_pS_{t_0}$ and $T_pL$. Moreover, the coorientation of $\xi(p)$ is negatively transverse to $T_pS_{t_0}$, by definition.\footnote{Recall the orientation convention for the characteristic foliation of a positive oriented contact structure $\xi$ along an oriented surface $\Sigma$: at a point $p \in \Sigma$ where $\xi$ and $\Sigma$ are transverse, a nonzero vector $v \in \xi(p) \cap T_p \Sigma$ is a positive generator of the characteristic foliation if for some/any vector $n \in T_p \Sigma$ positively transverse to $\xi(p)$, the pair $(v,n)$ is a positive basis of $T_p \Sigma$.} Therefore, the positive cone bounded by $T_p S_{t_0}$ and $\xi(p)$ is positively transverse to $T_pL$, and $\Phi$ is positively transverse to $L$ as claimed. See \zcref{fig:barrier_3d}.

    \begin{figure}[ht]
    \centering
    \includegraphics[width=0.6\linewidth]{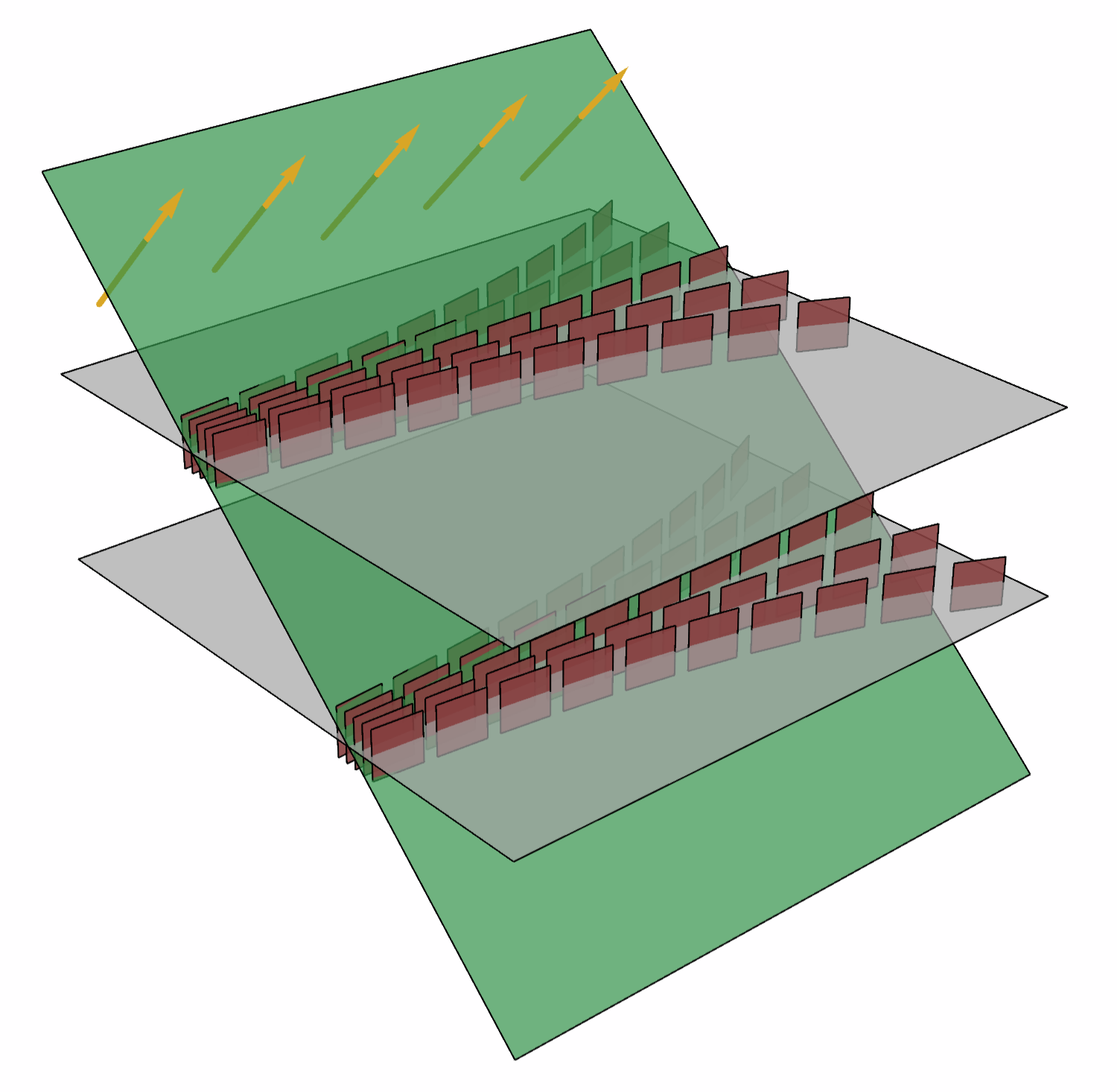}
    \caption{Local chart near an attracting characteristic where the lines parallel to the $t$ axis (vertical) are Legendrian. Two nearby fibers $S_t$ and $S_{t+\varepsilon}$ are shown in gray. The closed characteristic moves from the front right of the picture to the back left as $t$ increases. The coorientation of the contact structure (red) is towards us. The flow (yellow) is constrained to lie in the cone between $\xi$ and $TS_t$. The barrier $L$ is shown in green.}
    \label{fig:barrier_3d}
    \end{figure}

    We now consider the general case. Two complications may occur:
    \begin{enumerate}
        \item The original closed characteristic might undergo bifurcations,
        \item After going around the $S^1$-direction and following the leaf tracking maps, the resulting leaf might differ from the original one.
    \end{enumerate}
    For the first item, we will modify the construction by considering annuli traced by attracting characteristics and transverse to the spheres $S_t$'s, and connecting them by annuli tangent to the spheres $S_t$ for $t = t_i$ when a bifurcation occurs. To deal with the second issue, we might perform the construction several times by going around the $S^1$ direction finitely many times. This amounts to considering the iterates of the map $\Sigma$.

\begin{figure}[ht]
    \centering
    \begin{subfigure}{\textwidth}
        \centering
        \def\svgwidth{\textwidth}
        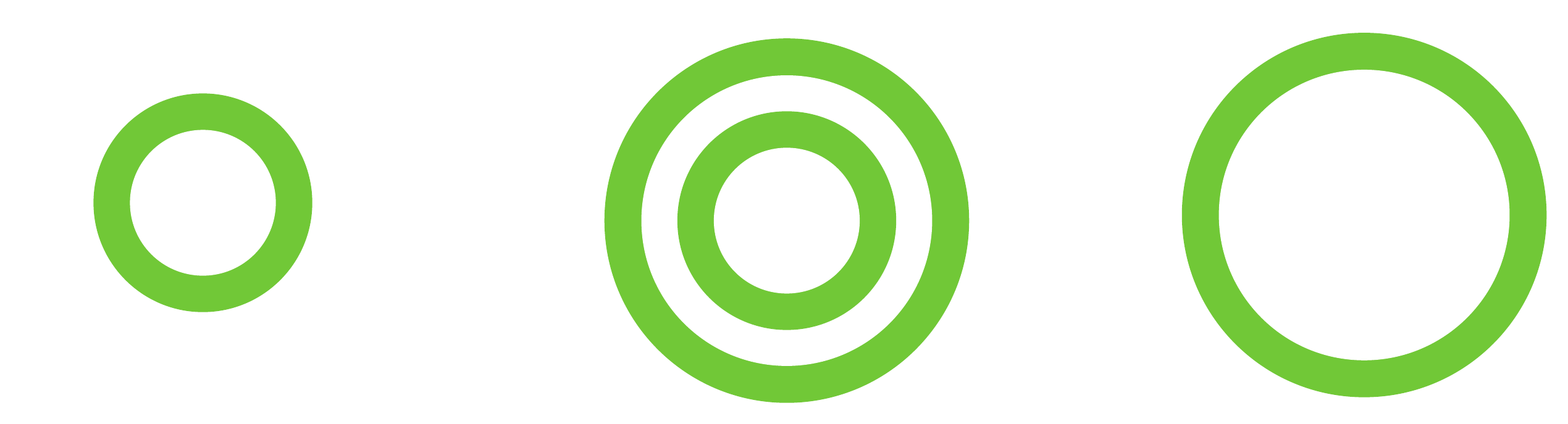
        \caption{A birth followed by a death near a disk component of $R_+$. Note that the birth must happen outside the disk.}
        \label{subfig:dividing_set_S2}
    \end{subfigure}
    
    \vspace{1em}
    \begin{subfigure}{0.6\linewidth}
        \centering
        \def\svgwidth{\textwidth}
        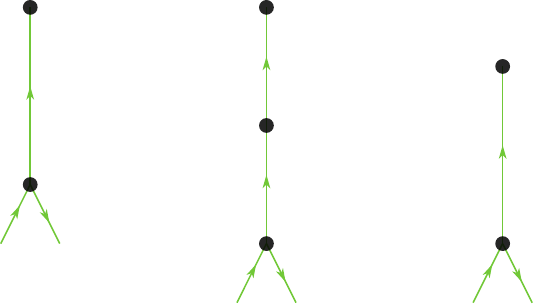
        \caption{The corresponding edge expansion and contraction.}
        \label{subfig:dividing_set_S2_graph}
    \end{subfigure}
    \caption{Birth, death, and edge operations.}
    \label{fig:dividing_set_S2}
\end{figure}

    Let $v=v_0$ be a sink vertex and $e=e_0$ the adjacent twig. The case of a source vertex follows from the symmetry of time reversal and reversal of the orientation of the $S^2$ factor. We consider the sequence of vertices $v_i \in G_i$ inductively defined by  $v_{i+1} = \sigma_i(v_i)$, where $v_N = \Sigma(v)$ is identified with a vertex of $G_0$. Let $e_i$ be the twig incident to $v_i$. The edges $e_0,\dots,e_N$ correspond to a sequence of attracting closed characteristics $\gamma_i \in S_{\tau_i}$ lying in a component of $R_-$ bounding a disk in $R_+$, and they are oriented counterclockwise around that disk. By \zcref{lem:disallowed_expansions}, $\gamma_i$ can either be identified with $\gamma_{i+1}$, or it undergoes a bifurcation and is canceled with a repelling characteristic to its right, and $\gamma_{i+1}$ is the nearest attracting characteristic parallel to $\gamma_i$. On each slice $[0, t_1] \times S^2, [t_1, t_2] \times S^2, \dots, [t_N, 1] \times S^2$, we consider the annulus $A_i$ obtained as the trace of the attracting characteristics corresponding to $e_i$, which is cooriented according to the orientation of $e_i$. By our previous analysis, we know that $\Phi$ is positively transverse to $A_i$ away from its boundary. If a boundary of $A_i$ corresponds to a birth or death of closed characteristics, we may further arrange that it is \emph{positively} tangent to the corresponding sphere ($S_{t_i}$ or $S_{t_{i+1}}$), after an arbitrarily $C^\infty$-small generic perturbation of $\xi$ near $\partial A_i$. That way, $\Phi$ is positively transverse to $A_i$ along its boundary as well. If $\gamma_i$ undergoes a bifurcation at $t = t_{i+1}$, we denote by  $C_i \subset S_{t_{i+1}}$ the annulus bounded by the degenerate closed characteristic coming from $\gamma_i$, and $\gamma_{i+1}$. The latter corresponds to the twig $e_{i+1}$ in $G_{i+1}$. Therefore, we may glue $A_i$, $C_i$, and $A_{i+1}$ along $\partial C_i$ and obtain a surface positively transverse to $\Phi$ after smoothing, see \zcref{fig:barrier}. We may further arrange that the resulting surface is transverse to the fibers $S_t$. Here, we crucially used that the bifurcation moves ``inside-to-out''. If there is no bifurcation, we simply glue $A_i$ to $A_{i+1}$ along their common boundary in $S_{t_{i+1}}$.

    After gluing all the $A_i$'s and $C_i$'s together, we obtain a cylinder $L(v) \subset S^1 \times S^2$ positively transverse to $\Phi$ and to the spheres $S_t$'s, and whose boundary components both lie in $S_0$. The boundary components correspond to the twigs $e=\twig(v)$ and $e_N=\twig(\Sigma(v))$. If those differ, then $L(v)$ and $L\big(\Sigma(v)\big)$ are disjoint and we may glue them together and iterate the construction, until we get an embedded torus $L$. By construction, $L$ is transverse to the $S_t$'s, and is positively transverse to $\Phi$. Since $L$ contains a curve which pairs positively with all these spheres, it is homologically trivial, and divides $S^1 \times S^2$ into two connected components.
\end{proof}

\begin{figure}[ht]
    \centering
    \def\svgwidth{0.9\textwidth}
    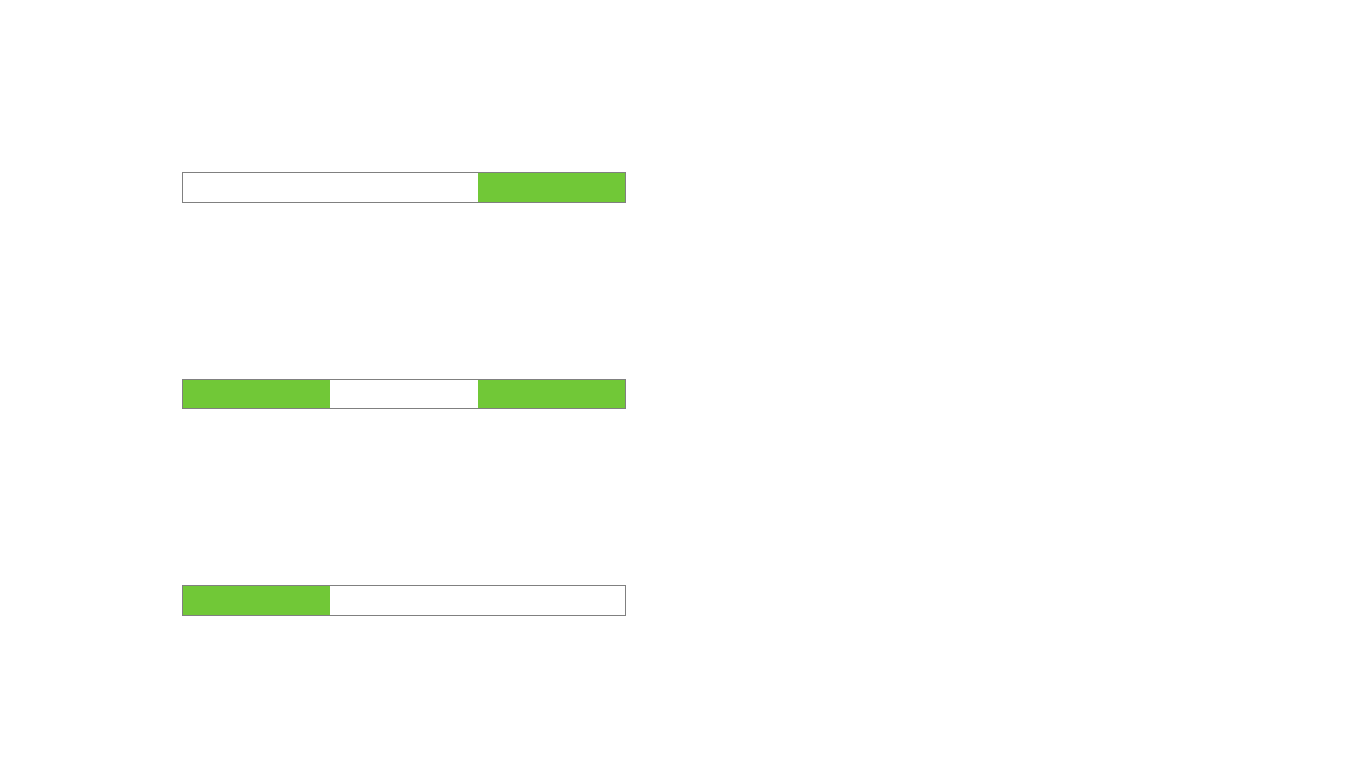
    \caption{A cross-section of the birth and death in \zcref{fig:dividing_set_S2}. The horizontal coordinate in these pictures is the radial coordinate in \zcref{fig:dividing_set_S2}. The cubic curve $S$ is the union of the closed characteristics in \zcref{fig:dividing_set_S2}.  On the left, the parts of $S$ corresponding to attracting characteristics of $\xi|_{S_t}$ are drawn in green. The dotted line denotes $C_{i+1}$, which cuts off the local minimum of $S$.
    On the right, we show the barrier we choose. It is the union of the vertical annuli $A_i, A_{i+1}$ and the annulus $C_{i+1}$. The contact structure $\xi$ (red) is everywhere transverse to $S$ by the contact condition. Our flow $\Phi$ (yellow) must be positively transverse to $\xi$ and positively transverse to the $S^2$ fibers, hence it must lie in the cone drawn in yellow. In particular, this means that flow may only cross the green parts of $S$ from right to left.
    }
    \label{fig:barrier}
\end{figure}

As an immediate corollary of \zcref{prop:contactS2suspension}, we recover that the standard foliation by spheres on $S^1 \times S^2$ cannot be approximated by contact structures (see~\cite{ET}).

\begin{rem}
    Our proof of \zcref{thmintro:rigidityspherical} passes through contact geometry, but it would be interesting to give a purely foliation theoretic argument. A natural approach would be to put the planar fibers in Roussarie--Thurston general position instead of convex position as in the proof above. We prefer our argument because the relevant version of Roussarie--Thurston theorem does not appear in the literature; we need to put the planar fibers in general position by an isotopy along a nonwandering flow $\Phi$.
\end{rem}

    \subsection{Consequences for ziggurats}

A horizontal compact genus $0$ leaf prevents the \emph{simultaneous} existence of strictly better slopes for positive and negative contact structures:

\begin{lem}\label{lem:rigidity_spherical}
Assume $\bm{s}$ is a spherical multislope. Then there do not exist multislopes $\bm{r}, \bm{t}\in \Zg(\Phi)$ with $\bm{r} < \bm{s} < \bm{t}$.
\end{lem}

\begin{proof}
    Suppose to the contrary that such multislopes exist. Then by \zcref{thmintro:contactzigg} and \zcref{lem:fillingcontact}, the Dehn filling $M(\bm{s})$ has $\Phi$-horizontal positive and negative contact structures $\xi_+$ and $\xi_-$. Since $\Phi$ is nonwandering, $(\xi_-, \xi_+)$ is taut. But $M(\bm{s})$ also has a horizontal $S^2$ since $\bm{s}$ is a spherical multislope. This contradicts \zcref{lem:transverseS2}.
\end{proof}

\begin{lem}\label{lem:rigidity_T2}
Assume $\bm{s}$ is a $\T^2$-type planar multislope. Then there do not exist a pair of multislopes $\bm{r}, \bm{t}\in \Zg(\Phi)$ with $\bm{r} < \bm{s} < \bm{t}$.
\end{lem}

\begin{proof}
    By \zcref{lem:spherical_limits2}, $\bm{s}$ is approximated by spherical multislopes. The result then follows from \zcref{lem:rigidity_spherical} applied to these approximating spherical multislopes.
\end{proof}

\begin{lem}\label{lem:rigidity_suspension}
Assume $\bm{s}$ is a spherical or annulus suspension multislope. Then there do not exist multislopes $\bm{r}, \bm{t}\in \Zg(\Phi)$ with either $\bm{r} < \bm{s}$ or $ \bm{s} < \bm{t}$.
\end{lem}

\begin{proof}
    If $\bm{s}$ is a spherical suspension multislope then $(\bm{s},\bm{\infty})$ is disjoint from $\Zg^+(\Phi)$; indeed, any contact structure with boundary slope in that quadrant would fill (by \zcref{lem:fillingcontact}) to a contact structure transverse to a blowdown of $\Phi$ along $\bm{s}$, contradicting \zcref{prop:contactS2suspension}. Similarly, $(\bm{-\infty},\bm{s})$ is disjoint from $\Zg^-(\Phi)$. By \zcref{thmintro:contactzigg}, both quadrants are also disjoint from $\Zg(\Phi)$.

    If $\bm{s}$ is an annulus suspension multislope, then just as in the proof of \zcref{lem:rigidity_T2}, the desired result follows from approximating $\bm{s}$ by spherical suspension multislopes using \zcref{lem:spherical_limits2}.
\end{proof}

\begin{proof}[Proof of \zcref{thmintro:rigidityspherical}]
Combine \zcref{lem:rigidity_spherical, lem:rigidity_T2, lem:rigidity_suspension}.
\end{proof}

    \section{Resolving \texorpdfstring{$\aleph$}{aleph} boundary components}

As before, let $\F$ be a $C^0$-foliation on $M$ which is $\Phi$-horizontal. If some boundary $\partial_i \F$ has $\aleph$ type, then under favorable conditions we can construct new foliations $\F_\pm$ which are $\Phi$-transverse and such that $\slope(\partial_i \F_\pm) = \pm \infty$; in other words, we can \textbf{resolve} the $\aleph$ boundary components. Combined with box-convexity, this will justify the treatment of $\aleph$ as a wildcard that can take on any real value.

Throughout this section, we will assume that $\Phi$ is not a disk nor annulus suspension flow, and that it is $\aleph$-unobstructed (\zcref{def:alephunobstructed}). Recall that fully punctured pseudo-Anosov flows without perfect fits (\zcref{lem:fptpafwpf_unobstructed}) and periodic flows are both examples of $\aleph$-unobstructed flows.

    \subsection{The fully unobstructed case}

We begin with a simplified situation without degeneracy compact planar leaves nor degeneracy $\T^2$-type planar minimal sets; in that case, the $\aleph$ boundary components can be resolved without modifying the slopes on the other boundary components.

\begin{prop}\label{prop:easyalephresolution}
    Let $\F$ be a $\Phi$-horizontal $C^0$-foliation with generalized boundary multislope $(s_1, \dots, s_k, \aleph, \dots, \aleph)$, where $s_i \in \R$. Moreover, assume that $\F$ has no degeneracy compact planar leaves nor degeneracy $\T^2$-type planar minimal sets. Then
    $$\{(s_1, \dots, s_k)\} \times \R^{n-k} \subset \mathcal{Z}(\Phi).$$
\end{prop}

\begin{proof}
    After applying \zcref{lem:isolatedcompactleaves} and performing suitable Denjoy blowups of the compact planar and $\T^2$-type planar minimal sets of $\F$, we may assume that those minimal sets are stable, in the sense that they admit boundaries with attracting holonomy; since those minimal sets have no degeneracy boundaries, the generalized boundary multislope remains unchanged. Then, the fine version of Eliashberg--Thurston (\zcref{thm:fineET}, see also \zcref{cor:fineET}) produces a contact pair $(\xi_-,\xi_+)$ with
    $$\partial \xi_- \prec \G \prec \partial \xi_+,$$
    where $\G$ has the same boundary slopes and the same $\aleph$ components as $\partial \F$.
    By \zcref{lem:resolution_of_aleph} (the instability of $+\infty^\flat$ and $-\infty^\sharp$ curves), we know
    \begin{align*}
        \bm{s}(\partial \xi_+) &\geq (s_1,\dots,s_k, +\infty,\dots, +\infty),\\
        \bm{s}(\partial \xi_-) &\leq (s_1,\dots,s_k,-\infty,\dots,-\infty).
    \end{align*}
    Adding twisting to $\xi_\pm$ (\zcref{lem:trim_twist}), we can replace $\pm\infty$ with $\pm N$ for arbitrarily large $N\in \R$. Now the result follows by box-convexity (\zcref{prop:boxconvex}) applied to this pair of contact structures.
\end{proof}

    \subsection{Obstructions: compact planar leaves}

In the presence of degeneracy compact planar leaves, we will now need to use the finest version of Eliashberg--Thurston (\zcref{thm:fineET}) to have precise control over the boundaries of the contact approximations. Throughout, we will assume that $\F$ has isolated compact planar leaves, and no degeneracy $\T^2$-type planar minimal sets.

The following ``mitosis'' trick allows us to maintain control over the slope of some boundary components while performing the sacrificial $\flat/\sharp$ modifications necessary in \zcref{thm:fineET}.

\begin{constr}[Mitosis\protect\footnote{\textit{Cell division by mitosis is an equational division which gives rise to genetically identical cells in which the total number of chromosomes is maintained}~\cite{wikipedia.Mitosis}.}]\label{constr:mitosis}
    Suppose $\F$ has a compact planar leaf $L$ with boundary components $\gamma_1, \dots, \gamma_\ell$, $\ell \geq 3$, with $\gamma_1 \subset \partial_i M$ and $\gamma_2 \subset \partial_j M$. A \textbf{mitosis} of $L$ is a $C^0$-small modification of $\F$ obtained as follows: we blow up $L$ and add canceling holonomies around $\gamma_1$ and $\gamma_2$, which creates a product region $L \times [0, \varepsilon]$ with parallel boundaries $L \times \{0\}$ and $\widetilde{L} = L \times \{\varepsilon\}$. We can further arrange that all the leaves between $L$ and $\widetilde{L}$ are noncompact and accumulate on both $L$ and $\widetilde{L}$. In particular, $\widetilde{L}$ is the only new minimal set produced by the blowup. Notice that we only create trivial Denjoy blowups along $\gamma_3, \dots, \gamma_\ell$, and that we have some freedom in choosing the holonomies around $\gamma_1$ and $\gamma_2$. For instance, we can make one $\sharp$ and the other $\flat$, or vice versa.

    This operation is useful for preparing for the application of \zcref{thm:fineET} while maintaining control over the coarse slopes of $\gamma_1$ and $\gamma_2$ and on the fine slopes of $\gamma_3,\dots,\gamma_\ell$. Indeed, the new leaf $\widetilde{L}$ has two special boundary components $\widetilde{\gamma}_1$ and $\widetilde{\gamma}_2$, parallel to $\gamma_1$ and $\gamma_2$, respectively. We declare $\gamma_1$ to be a distinguished boundary component of $L$, and $\widetilde{\gamma}_2$ to be a distinguished boundary component of $\widetilde{L}$. Then, performing $\flat/\sharp$ modifications on $\gamma_1$ and $\widetilde{\gamma}_2$ (independently of each other) do not change the slopes on $\partial_i M$ and $\partial_j M$, since each of these curves have a parallel, unmodified curve of the same slope, namely, $\widetilde{\gamma}_1$ and $\gamma_2$. See \zcref{fig:mitosis}.
    Of course, this might modify the boundary type of $\F$ along these boundary components, but its boundary type on the other boundary components of $M$ are not affected.
\end{constr}

\begin{figure}[ht]
    \centering
    \def\svgwidth{0.8\textwidth}
    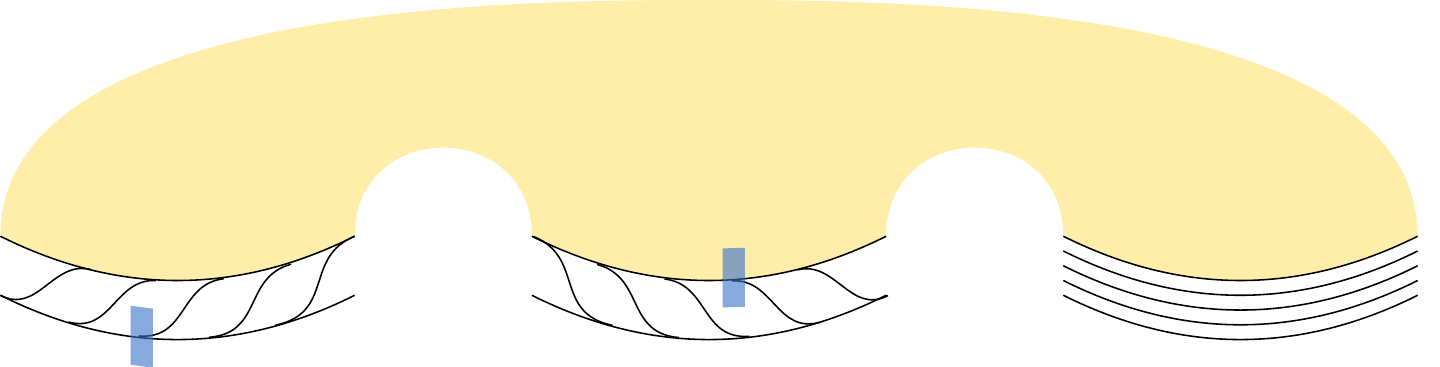
    \caption{Mitosis. The $\flat/\sharp$ twisting modifications are performed near the blue regions on $\partial M$.}
    \label{fig:mitosis}
\end{figure}

\paragraph{An example: single compact planar leaf.}

As a warm up to the general setting, we first explain our strategy in a simple scenario.

\begin{ex} \label{ex:compactplanarlatching}
Suppose $M$ has two boundary components and $\F$ is a $\Phi$-horizontal foliation with fine boundary multislope $(0^\heartsuit, \aleph)$, where $\heartsuit \in \{ \flat, \natural, \sharp, \dagger\}$ is unspecified, and that $\F$ has exactly \emph{one} compact planar leaf $L$. Since $M$ has only two boundary components and $\Phi$ is not a suspension, $\F$ does not have any $\T^2$-type planar minimal sets. Our goal is to resolve the $\aleph$ component to obtain foliations with well-defined boundary slopes, and get $\{0\} \times \R \subset \Zg(\Phi)$ as in the unobstructed case. However, in specific obstructed cases, we will get inclusions of the form $(-\varepsilon , 0) \times \R \subset \Zg(\Phi)$ or $(0, \varepsilon) \times \R \subset \Zg(\Phi)$ for some $\varepsilon > 0$: resolving the $\aleph$ component is possible at a(n infinitesimal) cost of modifying the slope on the first boundary component.

Concretely, we want to apply the fine version of Eliashberg--Thurston with boundary to obtain a positive contact structure $\xi$ approximating $\F$ with $\partial\F \prec \partial\xi$ (or rather $\G \prec \partial\xi$ for some $\G$ isotopic to $\partial \F)$. However, $L$ is an obstruction to applying \zcref{thmintro:ET}, and we need to perform some $\flat$ modifications along some boundary component of $L$ as in \zcref{thm:fineET}.

We distinguish several cases depending on the boundary signature of $L$; in each of those cases, we explain how to choose a distinguished boundary component of $L$ where we will apply a sacrificial $\flat$ modification. \emph{We will not perform a sacrificial $\flat$ modification along a $-\infty$ degeneracy component for the purpose of getting a positive contact approximation}, since we want to ``erase'' those; a $\flat$ modification would create a stable curve with slope $-\infty$. We also want to keep the same boundary slope on the nondegeneracy boundary, or modify it as little as possible. Here is how we proceed:

    \begin{itemize}
        \item If $L$ has boundary signature $(a,*,b)$ with $a+b\geq 2$, we apply mitosis (\zcref{constr:mitosis}) with $\gamma_1$ and $\gamma_2$ being nondegeneracy or $+\infty$ components of $\partial L$, and we perform the $\flat$ modifications on $\gamma_1$ and $\widetilde{\gamma}_2$.\footnote{We technically do not need to perform mitosis here since both $\gamma_1$ and $\gamma_2$ lie on $\partial_1M$, so we can perform the $\flat$ modification on one copy while the other one fixes the slope. However, mitosis will be necessary when $M$ has more than two boundary components.}
        \item If $L$ has degeneracy signature $(1,*,0)$, then we are forced to apply the $\flat$-modification along the non-degeneracy boundary component.
        \item The remaining possible degeneracy signatures $(0,*,0)$ and $(0,*,1)$ are ruled out. The first is ruled out because $L$ would have nondecreasing holonomy along all of its boundary components, implying that it has trivial holonomy along all its boundary components, contradicting the assumption that $L$ was isolated. The second is ruled out because we assumed that $\Phi$ is $\aleph$-unobstructed. 
    \end{itemize}

Applying the fine version of Eliashberg--Thurston for those choices of distinguished boundary components of $L$ (or a doubling thereof after mitosis), we have four main cases for the positive contact approximation $\xi_+$; in all of these cases, $s_2(\xi_+) = +\infty$ by design.

\begin{mdframed}
\begin{enumerate}[label=Case \arabic*$^+$:, leftmargin=*]
    \item \textit{$L$ has degeneracy signature $(1,*,0)$.} Then a sacrificial $\flat$ modification is performed on $\partial_1 M$.
    \begin{enumerate}
        \item \textit{$\bm\fineslope(\F)=(0^{\dagger}, \aleph)$.} Then $\slope_1(\xi_+)\geq 0$, and 
        $$(-\infty,0]\times \R\subset \Zg^+(\Phi).$$
        
        \item \textit{$\bm\fineslope(\F)=(0^{\flat}, \aleph)$.} Then $\slope_1(\xi_+) \leq 0$ can be made arbitrarily close to $0$, and 
        $$(-\infty,0)\times \R\subset \Zg^+(\Phi).$$
    \end{enumerate}    
    \item \textit{$L$ does not have degeneracy signature $(1,*,0)$.} Then no sacrificial $\flat$ modification is performed on $\partial_1 M$.
    \begin{enumerate}
        \item \textit{$\bm\fineslope(\F)=(0^{\natural/\sharp}, \aleph)$.} Then $\slope_1(\xi_+) = \delta > 0$ for some $\delta > 0$, and 
        $$(-\infty,\delta)\times \R\subset \Zg^+(\Phi).$$
        
        \item \textit{$\bm\fineslope(\F)=(0^{\flat/\dagger}, \aleph)$.} Then $\slope_1(\xi_+)\geq 0$, and 
        $$(-\infty,0]\times \R\subset \Zg^+(\Phi).$$
    \end{enumerate}
\end{enumerate}
\end{mdframed}

\medskip

Similarly, we have four main cases for a negative contact approximation $\xi_-$; in all of these cases, $s_2(\xi_-) = -\infty$ by design.

\begin{mdframed}
\begin{enumerate}[label=Case \arabic*$^-$:, leftmargin=*]
    \item \textit{$L$ has degeneracy signature $(1,0,*)$.} Then a sacrificial $\sharp$ modification is performed on $\partial_1 M$.
    \begin{enumerate}
        \item \textit{$\bm\fineslope(\F)=(0^{\dagger}, \aleph)$.} Then $\slope_1(\xi_-)\leq 0$, and 
        $$[0,+\infty) \times \R\subset \Zg^-(\Phi).$$
        
        \item \textit{$\bm\fineslope(\F)=(0^{\sharp}, \aleph)$.} Then $\slope_1(\xi_-) > 0$ can be made arbitrarily close to $0$, and 
        $$(0, +\infty)\times \R\subset \Zg^-(\Phi).$$
    \end{enumerate}
    Note that those are the only possible boundary types in that case, because of the type of  holonomy of $L$ along its degeneracy boundary components.
    
    \item \textit{$L$ does not have degeneracy signature $(1,0,*)$.} Then no sacrificial $\sharp$ modification is performed on $\partial_1 M$.
    \begin{enumerate}
        \item \textit{$\bm\fineslope(\F)=(0^{\flat/\natural}, \aleph)$.} Then $\slope_1(\xi_-) = -\delta < 0$ for some $\delta > 0$, and 
        $$(-\delta, \infty)\times \R\subset \Zg^-(\Phi).$$
        
        \item \textit{$\bm\fineslope(\F)=(0^{\dagger/\sharp}, \aleph)$.} Then $\slope_1(\xi_-)\leq 0$, and 
        $$[0, +\infty)\times \R\subset \Zg^-(\Phi).$$
    \end{enumerate}
\end{enumerate}
\end{mdframed}

\medskip

Notice that Case $1^+$(b) is a subset of Case $2^-$(a), and Case $1^-$(b) is a subset of Case $2^+$(a). Those cases deserve a particular attention.

Case $1^+$(b) is special since we are forced to perform the sacrificial $\flat$ modification along the nondegeneracy boundary component, which prevents $\{0\} \times \R \subset \Zg^+(\Phi)$
(this is referred to as the \emph{$\flat$-latching} case in the terminology below). However, we can produce a negative contact structure with \emph{negative} boundary slope along $\partial_1 M$ by choosing the $\sharp$-modification along a $-\infty$ degeneracy component (and turning it into a $\dagger$ type curve) as in Case $2^-$(a). Therefore, applying \zcref{thmintro:contactzigg}, we find
$$(-\delta,0)\times \R \subset \Zg(\Phi).$$
We emphasize that while $-\infty$ degeneracy curves are problematic for $\xi_+$, they are beneficial for $\xi_-$ since we can always increase the slope of $\xi_-$ arbitrarily by twisting (and these $-\infty$ curves can be ``fixed'' by mitosis).

Similarly, in the special case $1^-$(b) (the \emph{$\sharp$-latching} case), \zcref{thmintro:contactzigg} implies
$$(0,\delta)\times \R \subset \Zg(\Phi).$$

Apart from these two special cases, invoking \zcref{thmintro:contactzigg} gives 
$$\{0\}\times \R \subset \Zg(\Phi).$$

In all three cases, $\{0\} \times \R \subset \overline{\Zg(\Phi)}$.
\end{ex}

\paragraph{General case: multiple compact planar leaves.} We now move to the more general case where $\F$ has finitely many compact planar leaves, but no degeneracy $\T^2$-type planar minimal sets.

Let $\F$ be a horizontal foliation of boundary type $(s_1, \dots, s_k, \aleph, \dots, \aleph)$. A boundary component $\partial_i M$ with finite \emph{rational} slope $s_i$, $1 \leq i \leq k$, is \textbf{$\sharp$-latching (resp.~$\flat$-latching)} if it is of $\sharp$ (resp.~$\flat$) type and all of its closed curves are the nondegeneracy boundary of a compact genus $0$ leaf with degeneracy signature of the form $(1,0,*)$ (resp.~$(1,*,0)$). Note that a boundary component cannot be both $\sharp$-latching and $\flat$-latching. A boundary component that is neither $\sharp$-nor $\flat$-latching is \textbf{nonlatching}.\footnote{By ``latching'', we mean that it will not be possible to apply the fine version of Eliashberg--Thurston with boundary to remove the $\aleph$ components \emph{while keeping the slope on the latching boundary component fixed}. The latching compact planar leaves ``lock'' their boundary curves like a spring mechanism under tension in a latch.}

We now generalize the strategy of \zcref{ex:compactplanarlatching} to this scenario. 

\begin{prop} \label{prop:compactplanarlatching}
    Assume that $\F$ is a $\Phi$-horizontal foliation on $M$, where $\Phi$ is nonwandering and $\aleph$-unobstructed, and that $\F$ has finitely many compact planar leaves and no degeneracy $\T^2$-type planar minimal sets. Assume that the fine boundary multislope of $\F$ is of the form
    $$\big(\underset{k_\flat}{\underbrace{s_1^\flat, \dots, s_{k_\flat}^\flat}}, 
    \underset{k_\sharp}{\underbrace{t_1^\sharp, \dots, t^\sharp_{k_\sharp}}}, 
    \underset{k_\circ}{\underbrace{w^?_1, \dots, w^?_{k_\circ}}},
    \underset{k_\aleph}{\underbrace{\aleph, \dots, \aleph}}\big),$$
where
\begin{itemize}
    \item The first $k_\flat$ boundary components of $M$ are $\flat$-latching,
    \item The next $k_\sharp$ boundary components of $M$ are $\sharp$-latching,
    \item The next $k_\circ$ boundary components of $M$ are nonlatching and have unconstrained types,
    \item The last $k_\aleph$ boundary components of $M$ are of $\aleph$ type.
\end{itemize}
Then there exists $\varepsilon > 0$ such that
$$\prod_{i=1}^{k_\flat} \big(s_i - \varepsilon, s_i\big) \times \prod_{i=1}^{k_\sharp} (t_i, t_i+\varepsilon) \times \{(w_1, \dots, w_{k_\circ})\} \times \R^{k_\aleph} \subset \Zg(\Phi).$$
\end{prop}

\begin{proof}
    For notational simplicity, we may assume that the $s_i$'s and $t_i$'s are all $0$. We write $\bm w \coloneqq (w_1, \dots, w_{k_\circ})$.
    
    We show that there exists $\varepsilon>0$ such that 
    \begin{align}
        (-\varepsilon, +\infty)^{k_\flat} \times (0, +\infty)^{k_\sharp} \times \{ \bm w\} \times \R^{k_\aleph} &\subset \Zg^-(\Phi),\\
        (-\infty, 0)^{k_{\flat}} \times (-\infty, \varepsilon)^{k_{\sharp}} \times \{ \bm w\} \times \R^{k_\aleph} &\subset \Zg^+(\Phi). \label{eq:alephremoveZ+}
    \end{align}
    and use \zcref{thmintro:contactzigg} to conclude. 
    
    Let us focus on the inclusion~\eqref{eq:alephremoveZ+}, the other one being similar. We proceed in three steps. The first two steps modify the foliation, allowing us to apply the fine version of the Eliashberg--Thurston theorem with boundary in the third step.
    \begin{itemize}[leftmargin=*]
        \item \textit{Step 1: mitosis.} We apply the mitosis construction to all the compact planar leaves with boundary signature $(a, \ast, b)$ with $a+b \geq 2$, and choose boundary components of the resulting leaves to perform the $\flat$ modifications as explained in \zcref{ex:compactplanarlatching}. We also apply the mitosis construction to all the $\sharp$-latching compact planar leaves, so that the types of the $\sharp$-latching and $\aleph$ boundaries remain unchanged.

        \item \textit{Step 2: twisting.} For every $\sharp$-latching compact planar leaf, we choose an arc from the nondegeneracy boundary component to a $+\infty$ degeneracy boundary component, and we apply a twisting modification (see \zcref{constr:ribbontwist}) along a small ribbon along the arc, which is disjoint from all the other minimal sets of $\F$. We choose the twisting so that the nondegeneracy boundary component of the leaf is destroyed, and the slopes on all $\sharp$-latching boundary components are \emph{positive}. By the previous mitosis step, we can choose the arcs so that $+\infty$ boundary components survive on each $\aleph$ type boundaries of $M$ touched by such $\sharp$-latching leaves. The resulting foliation is still of $\aleph$ type on those boundaries. 

        After this step, we obtain a new $\Phi$-horizontal foliation $\F'$ satisfying the following properties:
        \begin{itemize}
            \item The first $k_\flat$ boundaries are $\flat$-latched, by the same compact planar leaves as before,
            \item The next $k_\sharp$ boundary components have positive slopes, bounded from below by some $\varepsilon > 0$,
            \item The next $k_\circ$ boundary components are still nonlatching and their slopes remain unchanged,
            \item The last $k_\aleph$ boundaries still have $\aleph$ type.
        \end{itemize}
        This procedure may produce new compact planar leaves and other minimal sets, all touching the previously $\sharp$-latching boundary components. We may arrange that those are in finite number after some blowdowns as in \zcref{lem:isolatedcompactleaves}. The $\flat$-latching and nonlatching compact planar leaves remain unaffected.
        
        \item \textit{Step 3: fine contact approximation.} We now apply the fine Eliashberg--Thurston theorem. For all the $\flat$-latching compact planar leaves, we perform the $\flat$ modification on their nondegeneracy boundary component. On all the (planar) minimal sets touching the formerly $\sharp$-latching boundary components, we apply the $\flat$-modification along those boundaries. On the compact planar leaves touching the nonlatching boundary components, we proceed as before and perform the $\flat$ modification on the appropriate $+\infty$ boundaries (the mitosis trick has already been applied). For every $\delta>0$, we obtain a positive contact approximation $\xi^\delta_+$ such that 
        \begin{itemize}
            \item For $1 \leq i \leq k_\flat$, $\slope_i(\xi^\delta_+) \geq -\delta$,
            \item For $1 \leq i \leq k_\sharp$, $\slope_i(\xi^\delta_+) \geq \varepsilon-\delta$,
            \item For $1 \leq i \leq k_\circ$, $\slope_i(\xi^\delta_+) \geq w_i$,
            \item For $1 \leq i \leq k_\aleph$, $\slope_i(\xi^\delta_+)=+\infty$.
        \end{itemize}

        We deduce that there exists $\varepsilon > 0$ such that for every $\delta >0$,
        $$(-\infty, -\delta)^{k_\flat} \times (-\infty, \varepsilon - \delta)^{k_\sharp} \times \{\bm w \} \times \R^{k_\aleph} \subset \Zg^+(\Phi),$$
        and \eqref{eq:alephremoveZ+} follows by letting $\delta \rightarrow 0$. \qedhere
    \end{itemize} 
\end{proof}

        \subsection{Obstructions: \texorpdfstring{$\T^2$}{T2}-type planar minimal sets}

In the presence of degeneracy $\T^2$-type planar minimal sets, we will also need to use the finest version of Eliashberg--Thurston (\zcref{thm:fineET}) to have precise control over the boundaries of the contact approximations. We will now assume that $\F$ has degeneracy $\T^2$-type planar minimal sets, and no degeneracy compact planar leaves.

\paragraph{An example: single $\T^2$-type planar minimal set.}

We start with the simplest situation of a single degeneracy $\T^2$-type planar minimal set on a manifold with three boundary components.

\begin{ex} \label{ex:T2latching}
Suppose that $M$ has three boundary components, and that $\F$ is a $\Phi$-horizontal foliation on $M$ with boundary type $(s, \overline{s}, \aleph)$ with $s, \overline{s} \in \R \setminus \Q$, and with a (necessarily unique) $\T^2$-type planar minimal set $\Lambda$. We also assume that  $\F$ has no degeneracy compact planar leaves, hence no compact planar leaves at all.

Because of the structure of $\T^2$-type planar minimal set, there is a distinguished branch of hyperbola $\mathcal{H} \subset \R^2$ corresponding to a locus of \emph{homological longitudes} on the first two boundary components of $M$, and $(s, \overline{s}) \in \mathcal{H}$. Concretely, the rational points in this locus is made of pairs of rational slopes which cobound an annulus in the $\T^2 \times I$ summand.

As before, we want to apply the fine version of the Eliashberg--Thurston theorem with boundary to resolve the $\aleph$ component and obtain $\{(s, \overline{s})\}\times \R \subset \Zg(\Phi)$. In particular, the instability of irrational slopes \zcref{lem:phaselocking} would then imply that there exists $\varepsilon > 0$ such that 
\begin{align} \label{eq:T2zig}
(s-\varepsilon, s+ \varepsilon) \times (\overline{s} - \varepsilon, \overline{s} + \varepsilon) \times \R \subset \Zg(\Phi).
\end{align}

However, as in \zcref{ex:compactplanarlatching}, we will need to apply a sacrificial $\flat/\sharp$ modification along a boundary component of $\Lambda$ to obtain adequate positive and negative contact approximations. Let us focus on the case of positive contact approximations. Once again, we distinguish several cases depending on the ``boundary signature'' of $\Lambda$:
\begin{itemize}
    \item If $\Lambda$ has some $+\infty$ degeneracy boundaries, then we perform a $\flat$ modification along such a boundary component to obtain an approximating positive contact structure. These boundary components locally form a Cantor set along $\partial_3M$, because of the structure of exceptional minimal sets, so we do not need a ``mitosis trick'' here.
    
    \item If $\Lambda$ only has $-\infty$ degeneracy boundary components, then we are forced to apply a $\flat$ modification along one of the two nondegeneracy boundaries.
\end{itemize}

Let us analyze these cases in more detail for the positive contact approximation $\xi_+$:

\begin{mdframed}
\begin{enumerate}[label=Case \arabic*$^+$:, leftmargin=*]
    \item \textit{$\Lambda$ has some $+\infty$ boundary components.} Then a sacrificial $\flat$ modification is performed on one of those, and
    $$s_1(\xi_+) > s, \qquad s_2(\xi_+) > \overline{s}, \qquad s_3(\xi_+) = +\infty.$$
    Therefore, 
    $$(-\infty, s+\varepsilon) \times (-\infty, \overline{s}+\varepsilon) \times \R \subset \Zg^+(\Phi).$$

    \item \textit{All the degeneracy boundaries of $\Lambda$ have slope $-\infty$.} Then no sacrificial $\flat$ modification is performed on $\partial_3 M$, and we can choose to perform it on $\partial_1 M$ or $\partial_2 M$. Performing it on $\partial_1 M$, we obtain $s_2(\xi_+) > \overline{s}$, and $s_1(\xi_+) < s$ can be made arbitrarily close to $s$. In particular
    $$(-\infty, s) \times (-\infty, \overline{s})  \times \R \subset \Zg^+(\Phi).$$
\end{enumerate}
\end{mdframed}

\medskip

As before, there are similar cases for $\xi_-$:

\begin{mdframed}
\begin{enumerate}[label=Case \arabic*$^-$:, leftmargin=*]
    \item \textit{$\Lambda$ has some $-\infty$ boundary components.} Then a sacrificial $\sharp$ modification is performed on one of those, and
    $$s_1(\xi_-) < s, \qquad s_2(\xi_-) < \overline{s}, \qquad s_3(\xi_-) = -\infty.$$
    Therefore, 
    $$(s-\varepsilon, +\infty) \times (\overline{s}-\varepsilon, +\infty) \times \R \subset \Zg^-(\Phi).$$

    \item \textit{All the degeneracy boundaries of $\Lambda$ have slope $+\infty$.} Then no sacrificial $\sharp$ modification is performed on $\partial_3 M$, and we can choose to perform it on $\partial_1 M$ or $\partial_2 M$. Performing it on $\partial_1 M$, we obtain $s_2(\xi_-) < \overline{s}$, and $s_1(\xi_-) > s$ can be made arbitrarily close to $s$. In particular
    $$(s,+\infty) \times (\overline{s}, +\infty)  \times \R \subset \Zg^-(\Phi).$$
\end{enumerate}
\end{mdframed}

Notice that Case $2^+$ is a subset of case $1^-$, and Case $2^-$ is a subset of Case $1^+$.

It follows from our analysis and \zcref{thmintro:contactzigg} that if $\Lambda$ has both $+\infty$ and $-\infty$ degeneracy boundaries, then the inclusion~\eqref{eq:T2zig} holds for some $\varepsilon > 0$.

If $\Lambda$ is as in Case $2^+$ (which we will refer to as the \emph{$\T^2_-$-latching} case below), then we instead get
$$(s-\varepsilon, s) \times (\overline{s}-\varepsilon, \overline{s}) \times \R \subset \Zg(\Phi).$$
However, we can improve this a bit. Indeed, we may first modify $\F$ by a ribbon twisting operation (see \zcref{constr:ribbontwist}) along a ribbon from $\partial_1M$ to $\partial_2M$. This ribbon can be constructed by choosing a $C^1$ curve tangent to $\Lambda$ from $\partial_1M$ to $\partial_2M$, smoothing $\F$ near it, and thickening it. By twisting along this ribbon, we obtain new foliations $\F_r$ on $M$ with multislope $(r, \overline{r}, \aleph)$, for any $(r, \overline{r}) \in \mathcal{H}$ sufficiently close to $(s, \overline{s})$. This crucially uses the structure of $\T^2$-type planar minimal sets; see \zcref{def:pems}. When $r \in \R \setminus \Q$, $\F_r$ has a $\T^2$-type planar minimal set $\Lambda_r$ with no $-\infty$ degeneracy boundary components. Therefore, the previous analysis implies 
$$\big(\mathcal{H}^- \cap [(-\infty, s+\varepsilon) \times (-\infty, \overline{s}+\varepsilon)] \big) \times \R \subset \Zg^+(\Phi),$$
where $\mathcal{H}^-$ denotes the connected component of $\R^2 \setminus \mathcal{H}$ below $\mathcal{H}$. This implies that there exists $\varepsilon > 0$ such that 
$$\big(\mathcal{H}^- \cap [(s-\varepsilon, s+\varepsilon) \times (\overline{s}-\varepsilon, \overline{s}+\varepsilon)] \big) \times \R \subset \Zg(\Phi).$$

Similarly, in the case where $\Lambda$ has no $-\infty$ degeneracy boundary components (the \emph{$\T^2_+$-latching} case), we obtain 
$$\big(\mathcal{H}^+ \cap [(s-\varepsilon, s+\varepsilon) \times (\overline{s}-\varepsilon, \overline{s}+\varepsilon)] \big) \times \R \subset \Zg(\Phi)$$
for some $\varepsilon > 0$. See \zcref{fig:h_pm} for a sketch of the set corresponding to the first two coordinates.
\end{ex}

\begin{figure}[ht]
    \centering
    \def\svgwidth{0.4\textwidth}
    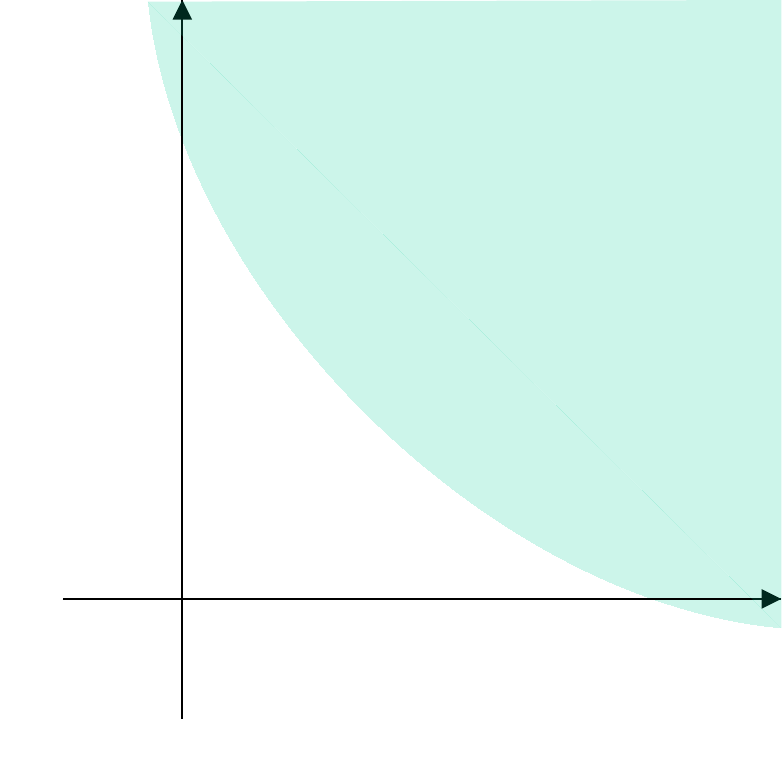
    \caption{Multislopes obtained by resolving a $\T^2_+$-latching planar minimal set (in gray).}
    \label{fig:h_pm}
\end{figure}

\paragraph{General case: multiple $\T^2$-type planar minimal sets.} We now consider the more general case where $\F$ has finitely many $\T^2$-type planar minimal sets, but no compact planar leaves.

A boundary component $\partial_i M$ with \emph{irrational} slope is \textbf{$\T^2_+$-latching (resp.~$\T^2_-$-latching)} if it intersects a $\T^2$-type minimal set all of whose closed boundary components have $+\infty$ (resp. $-\infty$) degeneracy slope. In that case, there is a companion irrational boundary component, which is the other irrational boundary component of that minimal set. Moreover, each irrational boundary component intersects at most one such a minimal set, so $\T^2_+$-latching (resp.~$\T^2_-$-latching) come in pairs, and cannot be both $\T^2_+$-~and $\T^2_-$-latching. As before, a boundary component that is neither $\T^2_+$- nor $\T^2_-$-latching is called nonlatching.\footnote{As a mnemonic for the reader: $\T^2_+$-latching boundary components are problematic for $\xi_-$, and force its boundary multislope to be \emph{bigger} than desired, while $\T^2_-$-latching boundary components are problematic for $\xi_+$, and force its boundary multislope to be \emph{smaller} than desired.}

We recall some useful notation. If $\partial_i M$ and $\partial_j M$ are the two irrational boundary components of a $\T^2$-type planar minimal set, then there is a natural identification of slopes on $\partial_i M$ and $\partial_j M$. If $u$ is a slope on $\partial_i M$, there is a corresponding slope $\overline{u}$ on $\partial_j M$.\footnote{Because the choices of framings coming from $\Phi$ for $\partial_1 M$ and $\partial_2 M$ might not agree, we do not necessarily have $\overline{r}=-r$.} The locus of points of the form $(r,\overline{r})\in \R^2$ is a branch of a hyperbola or a line which we denote by $\mathcal{H} \subset \R^2$. It divides the plane into two open regions $\mathcal{H}^\pm$, where $\mathcal{H}^-$ is unbounded in the $(-\infty, -\infty)$ direction and $\mathcal{H}^+$ is unbounded in the $(+\infty, +\infty)$ direction.

We now generalize the strategy of \zcref{ex:T2latching} to this scenario. Let us describe the setup carefully. As before, $\Phi$ is an $\aleph$-unobstructed and nonwandering flow on $M$. We consider a $\Phi$-horizontal foliation $\F$ which has no degeneracy compact planar leaves. We assume that the boundary multislope of $\F$ is of the form
$$\big(\underset{2k_-}{\underbrace{u_1,\overline{u}_1, \dots, u_{k_-}, \overline{u}_{k_-}}}, 
\underset{2k_+}{\underbrace{v_1,\overline{v}_1, \dots, v_{k_+}, \overline{v}_{k_+}}},
\underset{k_\circ}{\underbrace{w_1, \dots, w_{k_\circ}}},
\underset{k_\aleph}{\underbrace{\aleph, \dots, \aleph}}\big),$$
where
\begin{itemize}
    \item The first $2 k_-$ boundary components of $M$ are $\T^2_-$-latching, and $u_i$ and $\overline{u}_i$ are two irrational slopes for boundary components intersecting the same $\T^2_-$-latching minimal set $\Lambda^-_i$,
    \item The next $2 k_+$ boundary components of $M$ are $\T^2_+$-latching, and $v_i$ and $\overline{v}_i$ are two irrational slopes for boundary components intersecting the same $\T^2_+$-latching minimal set $\Lambda^+_i$,
    \item The next $k_\circ$ boundary components of $M$ are non-latching and have unconstrained types,
    \item The last $k_\aleph$ boundary components of $M$ are of $\aleph$ type.
\end{itemize}
For $\varepsilon > 0$, we set
\begin{align}
    U^\varepsilon_- &\coloneqq \prod_{i=1}^{k_-} \left(\mathcal{H}^-_{i,-} \cap \left[(u_i-\varepsilon, u_i +\varepsilon) \times (\overline{u}_i-\varepsilon, \overline{u}_i +\varepsilon)\right] \right) \subset \R^{2 k_-}, \label{eq:U-} \\
    U^\varepsilon_+ &\coloneqq 
    \prod_{i=1}^{k_+} \left(\mathcal{H}^+_{i,+} \cap \left[(v_i-\varepsilon, v_i +\varepsilon) \times (\overline{v}_i-\varepsilon, \overline{v}_i+\varepsilon)\right] \right) \subset \R^{2 k_+}, \label{eq:U+}
\end{align}
where $\mathcal{H}_{i, \pm} \subset \R^2$ denotes the locus of homological longitudes for the $i$-th pair of $\T^2_\pm$-latching boundaries of $M$.

\begin{prop} \label{prop:T2latching}
With the previous notation, there exists $\varepsilon > 0$ such that
\begin{align} \label{eq:removealephT2}
    U^\varepsilon_- \times U^\varepsilon_+ \times \{(w_1, \dots, w_{k_\circ})\} \times \R^{k_\aleph} \subset \Zg(\Phi).
\end{align}
\end{prop}

\begin{proof}
We write $\bm w = (w_1, \dots, w_{k_\circ})$, $\bm u_i = (u_i, \overline{u}_i)$, $\bm v_i = (v_i, \overline{v}_i)$. We show that there exists $\varepsilon > 0$ such that
\begin{align}
    \prod_{i=1}^{k_-} \big(\bm{u}_i- \boldsymbol{\varepsilon}, \bm{+\infty}\big) \times \prod_{i=1}^{k_+} \big(\bm{v}'_i, \bm{+\infty}\big) \times \{ \bm w\} \times \R^{k_\aleph} &\subset \Zg^-(\Phi), \label{eq:removealephT2Z-}\\
    \prod_{i=1}^{k_-} \big(\bm{-\!\infty}, \bm{u}'_i \big) \times \prod_{i=1}^{k_+}\big(\bm{-\!\infty}, \bm{v}_i + \boldsymbol{\varepsilon}\big) \times \{ \bm w\} \times \R^{k_\aleph} &\subset \Zg^+(\Phi), \label{eq:removealephT2Z+}
\end{align}
for all $\bm u'_i \in \mathcal{H}_{i,-}$ with $\vert \bm u'_i - \bm u_i \vert < \varepsilon$ and $\bm v'_i \in \mathcal{H}_{i,+}$ with $\vert \bm v'_i - \bm v_i \vert < \varepsilon$. By \zcref{thmintro:contactzigg}, this will readily imply
$$\prod_{i=1}^{k_-}(\bm{u}'_i- \boldsymbol{\varepsilon}, \bm{u}'_i) \times \prod_{i=1}^{k_+}(\bm{v}'_i, \bm{v}'_i + \boldsymbol{\varepsilon}) \times \{ \bm w\} \times \R^{k_\aleph} \subset \Zg(\Phi)$$
for all such $\bm u'_i, \bm v'_i$, which yields~\eqref{eq:removealephT2} (for a possibly smaller $\varepsilon$).

We now explain how to establish~\eqref{eq:removealephT2Z+}, \eqref{eq:removealephT2Z-} being similar. We proceed in two steps: we first modify the foliation along the latching minimal sets, and then apply the fine version of Eliashberg--Thurston with suitable choices of $\flat$ modifications.
\begin{itemize}[leftmargin=*]
    \item \textit{Step 1: twisting.} For each $\T^2_-$-latching minimal set, we choose an arc from one irrational boundary of $M$ to the other, and apply a twisting modification (see \zcref{constr:ribbontwist}) along a thickening of the arc disjoint from the other minimal sets. This way, we can modify the slopes on the irrational boundaries along the corresponding locus $\mathcal{H}$ as in \zcref{ex:T2latching}. We also choose, for each $\T^2_+$-latching minimal set, two arcs from a degeneracy boundary component to each of the irrational boundaries, and apply the twisting construction. The effect is to destroy some $+\infty$ degeneracy curves, but not all of them, so that the $\aleph$ boundaries still have $\aleph$ type. In exchange, this increases the slopes on both irrational boundary components of at least some $\varepsilon > 0$. Importantly, the twisting operations along the $\T^2_-$- and $\T^2_+$-latching minimal sets are \emph{independent of each other}.

    In summary, we obtain a new $\Phi$-horizontal foliation $\F'$ with boundary multislope
    $$\big(\bm u'_1, \dots, \bm u '_{k_-}, 
    \bm v_1 + \bm\varepsilon, \dots, \bm v_{k_+} + \bm \varepsilon, \bm w, \bm{\aleph}\big),$$
    where $\bm u'_i \in \mathcal{H}_{i,-}$ is an arbitrary (partial) multislope satisfying $\vert \bm u'_i - \bm u_i\vert < \varepsilon$. Note that the former nonlatching boundary components are still nonlatching and no other slopes are modified. The twisting operations might create new minimal sets, but they all intersect the formerly latching boundary components. We might apply some blowdowns (see \zcref{lem:isolatedcompactleaves}) to ensure that the compact ones (if any) are isolated.
    
    \item \textit{Step 2: fine contact approximation.} We now apply \zcref{thm:fineET} to the foliation $\F'$. For the minimal sets touching the previously latching boundary components, we may choose to perform the $\flat$ modifications along those nondegeneracy boundaries. For the nonlatching $\T^2$-type minimal sets, we perform the $\flat$ modifications along $+\infty$ degeneracy curves as in \zcref{ex:T2latching}. By choosing smaller and smaller $\flat$ modifications and applying \zcref{thm:fineET}, we obtain~\eqref{eq:removealephT2Z+}. \qedhere
\end{itemize}
\end{proof}

        \subsection{The general case}

We now treat the general case, allowing both degeneracy compact planar leaves and $\T^2$-type planar minimal sets. We explain how to combine the strategies in the proofs of \zcref{prop:compactplanarlatching} and \zcref{prop:T2latching} to treat both obstructions simultaneously.

Let $\F$ be a foliation transverse to $\Phi$, with finitely many compact planar leaves,\footnote{Recall that this can always be achieved after some blowdowns by \zcref{lem:isolatedcompactleaves}.} and with generalized fine boundary multislope
$$\big(\underset{k_\flat}{\underbrace{s_1^\flat, \dots, s_{k_\flat}^\flat}}, 
\underset{k_\sharp}{\underbrace{t_1^\sharp, \dots, t^\sharp_{k_\sharp}}}, 
\underset{2k_-}{\underbrace{u_1,\overline{u}_1, \dots, u_{k_-}, \overline{u}_{k_-}}}, 
\underset{2k_+}{\underbrace{v_1,\overline{v}_1, \dots, v_{k_+}, \overline{v}_{k_+}}},
\underset{k_\circ}{\underbrace{w_1, \dots, w_{k_\circ}}},
\underset{k_\aleph}{\underbrace{\aleph, \dots, \aleph}}\big)$$
where
\begin{itemize}
    \item The first $k_\flat$ boundary components of $M$ are $\flat$-latching,
    \item The next $k_\sharp$ boundary components of $M$ are $\sharp$-latching,
    \item The next $2 k_-$ boundary components of $M$ are $\T^2_-$-latching, and $u_i$ and $\overline{u}_i$ are two irrational boundary components intersecting the same $\T^2_-$-latching minimal set $\Lambda^-_i$
    \item The next $2 k_+$ boundary components of $M$ are $\T^2_+$-latching, and $v_i$ and $\overline{v}_i$ are two irrational boundary components intersecting the same $\T^2_+$-latching minimal set $\Lambda^+_i$ 
    \item The next $k_\circ$ boundary components of $M$ are non-latching and have unconstrained types
    \item The last $k_\aleph$ boundary components of $M$ are of $\aleph$ type.
\end{itemize}

For $\varepsilon > 0$, we define open set
\begin{align*}
    U^\varepsilon_\flat \coloneqq \prod_{i=1}^{k_\flat} \big(s_i - \varepsilon, s_i\big) \subset \R^{k_\flat}, \qquad U^\varepsilon_\sharp \coloneqq \prod_{i=1}^{k_\sharp} (t_i, t_i +\varepsilon) \subset \R^{k_\sharp},
\end{align*}

and we define $U^\varepsilon_\pm$ as in~\eqref{eq:U-} and~\eqref{eq:U+}.

\begin{thm}[Resolving $\aleph$-boundaries] \label{thm:resolvealeph}
Suppose $\Phi$ is $\aleph$-unobstructed. With the previous notation, there exists $\varepsilon > 0$ such that
\begin{align} \label{eq:resolvealeph}
    U^\varepsilon_\flat \times U^\varepsilon_\sharp \times U^\varepsilon_- \times U^\varepsilon_+ \times \{(w_1, \dots, w_{k_\circ})\} \times \R^{k_\aleph} \subset \Zg(\Phi).
\end{align}
\end{thm}

\begin{proof}
The result follows by combining the operations in the proofs of \zcref{prop:compactplanarlatching} and \zcref{prop:T2latching}.

The key observation is that the minimal sets of $\F$ are in finite number and admit pairwise disjoint neighborhoods. Moreover, the notions of $\flat/\sharp$- and $\T^2_\pm$-latching and nonlatching immediately extend to this more general setup. Importantly, the compact planar and $\T^2$-type planar minimal intersect disjoint sets of boundary components of $M$. Therefore, following the strategies of the aforementioned propositions, we may first prove inclusions for $\Zg^\pm(\Phi)$ and then apply \zcref{thmintro:contactzigg}. The inclusions~\eqref{eq:alephremoveZ+} and~\eqref{eq:removealephT2Z+} may be combined together by applying the first two steps in the proof of \zcref{prop:compactplanarlatching} and the first step in the one of \zcref{prop:T2latching} simultaneously. Then, one may apply the fine Eliashberg--Thurston theorem (\zcref{thm:fineET}) with the appropriate choices of $\flat$ modifications, and make the modifications arbitrarily small to get the desired contact approximations. Details should be clear and are left to the reader.
\end{proof}

In particular, we obtain:
\begin{cor}
    If $\Phi$ is an $\aleph$-unobstructed nonwandering flow on $M$, then
    $$\mathcal{Z}_\aleph(\Phi) \subset \overline{\mathcal{Z}(\Phi)}.$$
\end{cor}

\begin{proof}
    Let $\F$ be a $\Phi$-horizontal foliation realizing a generalized boundary multislope $\bm s \in \Zg_\aleph(\Phi)$. If $\F$ is a suspension foliation, no leaf has a degeneracy boundary and $\bm s \in \Zg(\Phi)$. Otherwise, we apply \zcref{lem:isolatedcompactleaves} to obtain a $\Phi$-horizontal foliation with the same boundary multislope $\bm s$ and finitely many (planar) compact leaves. Then, \zcref{thm:resolvealeph} applies and yields $\bm s \in \overline{\Zg(\Phi)}$.
\end{proof}

    \section{\texorpdfstring{$\aleph$}{aleph}-completion} \label{sec:alephclosure}

In this section, we explain that one may take limits of foliations carried by the same branched surface. The boundary slopes behave well under limits unless Reeb annuli appear. These are what we call $\aleph$ type boundary foliations.

Let $B$ be a branched surface in $M$ whose complementary regions are products. Any foliation $\F$ fully carried by $B$ specifies a \textbf{cofinal splitting sequence} of $B$, by which we mean a sequence of splittings $B=B_0\rightsquigarrow B_1 \rightsquigarrow B_2\rightsquigarrow\dots$ such that every branch locus of $B_i$ is split open in $B_{i+1}$. Conversely, any cofinal splitting sequence for $B$ specifies a lamination whose complementary regions can be filled with products to obtain a foliation. Note that the procedure 
$$\F \longrightarrow  \textrm{Cofinal splitting sequence} \longrightarrow \F'$$
is not necessarily the identity. However, for each $1\leq k \leq n$, $\partial_k \F'$ is monotone equivalent to $\partial_k \F$ and can only have $\sharp$ or $\flat$ annuli if $\F$ does as well. Let $\mathcal O_{B}$ be the space of cofinal splitting sequences for $B$. We give $\mathcal O_B$ the topology of convergence on finite subsequences. Note that $\mathcal O_{B}$ is compact with respect to this topology.

    In the next lemma, we think of slopes as elements of $(H_1(\T^2)\setminus \{0\})/\R_{>0}$ since we are not restricting to foliations with slope in $\R\cup \{\pm \infty\}$ at the moment. Also recall \zcref{def:monotonecurve}: a curve is positively mostly to a foliation $\G$ if it decomposes into finitely many subarcs, each of which is either contained in a leaf of $\G$ or transverse to $\G$.
    \begin{lem}\label{lem:openproperties}
    Let $s$ be a rational slope. The following condition is an open condition on a cofinal splitting sequence $c\in \mathcal O_B$:
         $c$ codes a foliation $\F$ such that $\partial_k \F$ admits a positively monotone closed curve of slope $s$.
    \end{lem}
    \begin{proof}
        We check that the condition can be certified with a finite prefix of $c$, from which it will follow that it is an open condition. The condition holds if and only if one of the following holds:
        \begin{enumerate}
        \item $\partial_k \F$ has a $\flat$-spiraling leaf $\lambda$ which limits onto a closed curve of slope $s$.
        \item $\partial_k \F$ is Reebless and $\slope(\partial_k \F)$ pairs positively with $s$.
        \end{enumerate}
        Let us first handle case 1. Since $\lambda$ limits onto a closed curve of slope $s$, it eventually recurs onto the same periodic sequence of tracks of $B$. Then as in \zcref{fig:transversalizing}, we may choose a monotone curve $\gamma=\gamma_\perp \cup \gamma_\parallel$ of slope $s$ whose tangent subarc $\gamma_\parallel$ is contained in $\lambda$ and whose transverse subarc $\gamma_\perp$ is contained in a single branch sector of $\partial_k B$. The two endpoints of $\gamma_\perp$ are eventually split apart in the cofinal sequence $c$ because $\gamma$ does not live in a product region. Therefore, the existence of $\gamma$ is certified in a finite prefix of $c$.

        Now let us handle case 2. Since $\partial_k \F$ is Reebless, there exists a splitting of $\partial_k B$ in the sequence $c$ which carries only foliations with slope in an $\varepsilon$-interval around $\slope(\partial_k \F)$. For $\varepsilon$ small enough, every foliation carried by this splitting pairs positively with $s$. Therefore, property 2 is witnessed by a finite prefix of $c$.
    \end{proof}
    
    \begin{cor}\label{cor:slopecontinuity}
        Coarse slope is a continuous function on the subspace of $\mathcal O_B$ consisting of Reebless foliations.
    \end{cor}
    
    \begin{rem}
        In the language of \zcref{sec:beyondaleph}, \zcref{cor:slopecontinuity} can be upgraded to the statement that the map $\mathcal O_B \to {\widehat{\mathscr{S}}}^n$ mapping a splitting sequence to its generalized fine boundary multislope is continuous.
    \end{rem}

\begin{lem}[Limits of foliations] \label{lem:limits}
    Let $B$ be a $\Phi$-horizontal branched surface and $(\mathcal F_i)_{i \geq 0}$ be a sequence of foliations carried by $B$. Let $\bm{s}^i$ be the generalized boundary multislope of $\mathcal F_i$. Suppose that the $\bm{s}^i$ converge to $\bm{s}^{\lim}\in \R^n$, by which we mean that for all $1\leq k \leq n$, we have $\lim_{i\to\infty} s^i_k = s^{\lim}_k$ or $s^i_k=\aleph$ for all large $i$. Then $B$ carries a limit foliation $\F_{\lim}$ with generalized boundary multislope $\bm{s}^{\lim}$, possibly with some entries in $\R$ replaced by $\aleph$.

    Moreover, we have the following control over the boundary type of $\F_{\lim}$:
    \begin{enumerate}
        \item If $s^i_k \to s^{\lim}_k$ strictly from above, then $\partial_k \F_{\lim}$ has $\sharp$, $\natural$, or $\aleph$ type.
        \item Similarly, if $s^i_k \to s_k^{\lim}$ strictly from below, then $\partial_k \F_{\lim}$ has $\flat$, $\natural$, or $\aleph$ type.
        \item If $\partial_k \F_i$ has generalized fine slope $(s^{\lim}_k)^\natural$ for all $i$, then $\partial_k \F_{\lim}$ has $\natural$ or $\aleph$ type.
        \item If $\partial_k \F_i$ has $\aleph$ type for all $i$, then $\partial_k \F_{\lim}$ also has $\aleph$ type.
    \end{enumerate}
\end{lem}

\begin{proof}
    After passing to a subsequence and a branched subsurface of $B$, we can assume that $B$ fully carries each $\F_i$. By compactness of $\mathcal{O}_B$, the cofinal splitting sequences for the $\F_i$ have a convergent subsequence yielding a limiting foliation $\F_{\lim}$.
    
    Split into two cases. First, suppose $\partial_k \F_{\lim}$ has a Reeb annulus. Since $\F_i$ does not admit a negatively monotone curve of slope $+\infty$ or positively monotone curve of slope $-\infty$, and this property is closed, the same is true for $\F_{\lim}$. Therefore, $\fineslope(\partial_k \F_{\lim})=\aleph$. 
    
    Second, suppose $\partial_k \F_{\lim}$ does not have a Reeb annulus. By \zcref{cor:slopecontinuity}, $\slope(\partial_k \F_{\lim})=s_k^{\lim}$. If $s^i_k \to s^{\lim}_k$ strictly from above as in (1), then none of the $\F_i$ admit a positively monotone curve of slope $s_k^{\lim}$. By \zcref{lem:openproperties}, this property must also hold for $\partial_k \F^{\lim}$, therefore $\partial_k \F_{\lim}$ does not have a $\flat$ annulus and must have type $\sharp$ or $\natural$. Items (2) and (3) follow similarly. Finally, since $\partial_k \F_{\lim}$ does not have a Reeb annulus, there exists $s\in \Q$ such that $\F_{\lim}$ admits a positively monotone curve of slope $s$. \zcref{lem:openproperties} implies that the same holds for $\F_i$, $i$ sufficiently large. This implies that $\F_i$ is not $\aleph$ type for all $i$, and (4) follows.
\end{proof}

\begin{proof}[Proof of \zcref{thmintro:alephclosure}]
 Let $\bm{s}\in \overline{\Zg_\aleph(\Phi)}$. Suppose $\Phi$ is LBSF (\zcref{def:condLBSF}). Recall that this means that there is a neighborhood $U$ of $\bm{s}$ and a finite collection of branched surfaces $B_1,\dots,B_N$ such that every $\Phi$-horizontal foliation with generalized boundary multislope in $U$ is carried, up to the action of a homeomorphism of $(M,\Phi)$ rel boundary, by one of the $B_i$'s. 

 Choose a sequence of foliations $\F_i$ with generalized boundary multislope $\bm{s}^i \to \bm{s}$. Passing to a subsequence and applying homeomorphisms of $(M,\Phi)$ rel boundary as necessary, we may assume that all the $\F_i$ are carried by a single branched surface $B$. By \zcref{lem:limits} applied to this sequence of foliations, $\bm{s}\in \Zg_\aleph(\Phi)$.
\end{proof}

    \section{Slippery points}

        \subsection{Generalities}

We first give a more precise description of the slippery points of closed subsets of $\R^n$.

\begin{defn}
For $\bm{x}\in \R^n$ and $1 \leq k \leq n$, we define the standard positive and negative quadrants at $\bm{x}$ in the $k$-first coordinates as
\begin{align*}
    C_k^+(\bm{x}) &\coloneqq [x_1,+\infty)\times\dots\times [x_k,+\infty) \times \{(x_{k+1}, \dots ,x_n)\} , \\
    C_k^-(\bm{x}) &\coloneqq (-\infty, x_1] \times \dots \times (-\infty,  x_k] \times \{(x_{k+1}, \dots ,x_n)\}.
\end{align*}

Let $A\subset \R^n$. A point $\bm{x}\in A$ is an \textbf{upper extremal point} in the $k$ first coordinates if $$C^+_k(\bm{x})\cap A = \{\bm{x}\},$$
and it is a \textbf{lower extremal point} in these coordinates if 
$$C^-_k(\bm{x})\cap A = \{\bm{x}\}.$$
\end{defn}

For ziggurats of nonwandering flows, \zcref{thmintro:irrationality} immediately implies

\begin{cor}[Rationality at extremal points] \label{cor:irrationality}
    Let $\bm{s}= (s_1, \dots, s_n) \in \R^n$. If $\bm{s}$ is an upper or lower extremal point of $\mathcal{Z}(\Phi)$ in the $k$ first coordinates, then $s_1,\dots,s_k \in \Q$.
\end{cor}

We now give an alternative definition for slipperiness which will be useful in practice, and captures the essence of those points:

\begin{defn}\label{defn:alternativeslippery}
    Let $A\subset \R^n$ be a closed set. A point $\bm{x}\in A$ is \textbf{upwards (resp.~downwards) slippery} in the $k\geq 1$ first coordinates if there is a sequence of points $\bm{x}^j \in A$, $j \geq 0$, converging to $\bm{x}$, such that each $\bm{x}^j$ is an upper (resp.~lower) extremal point in the first $k$ coordinates, the first $k$ coordinates are strictly increasing (resp.~decreasing) with $j$, and the remaining coordinates are nonincreasing (resp.~nondecreasing) with $j$.
\end{defn}

This definition extends to any set of $k$ coordinates, by permuting the coordinates. Intuitively, this definition means that there is a ``good exchange rate'' between the first $k$ and last $n-k$ coordinates near $\bm{x}$: one can increase the slope in some of these coordinates at the expense of potentially decreasing it in others.

\begin{figure}
\centering
    \begin{subfigure}{0.45\linewidth}
    \centering
    \includegraphics[width=\linewidth]{Figures/slippery.pdf}
    \caption{The origin is upwards slippery in the $x$-direction.}
    \label{fig:slippery}
    \end{subfigure}
    \hfill
    \begin{subfigure}{0.45\linewidth}
    \centering
    \includegraphics[width=\linewidth]{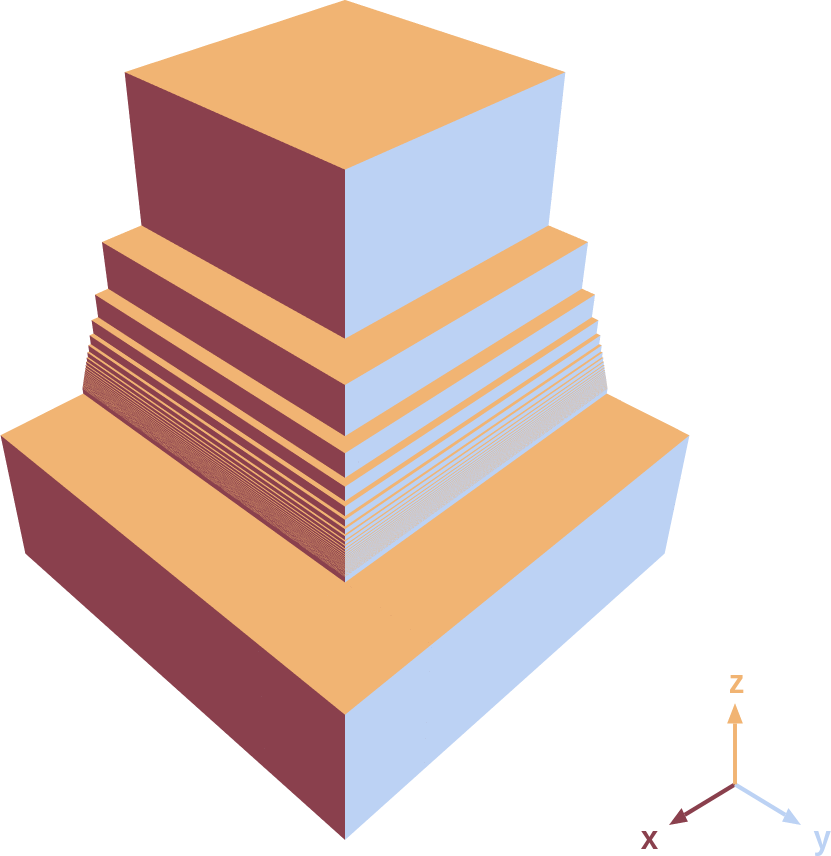}
    \caption{An $xy$-slippery point, a ray of upwards $x$-slippery points, and a ray of upwards $y$-slippery points.}
    \label{fig:slippery3d}
    \end{subfigure}
    \par\vspace{2em}
    \begin{subfigure}{0.45\linewidth}
    \centering
    \includegraphics[width=\linewidth]{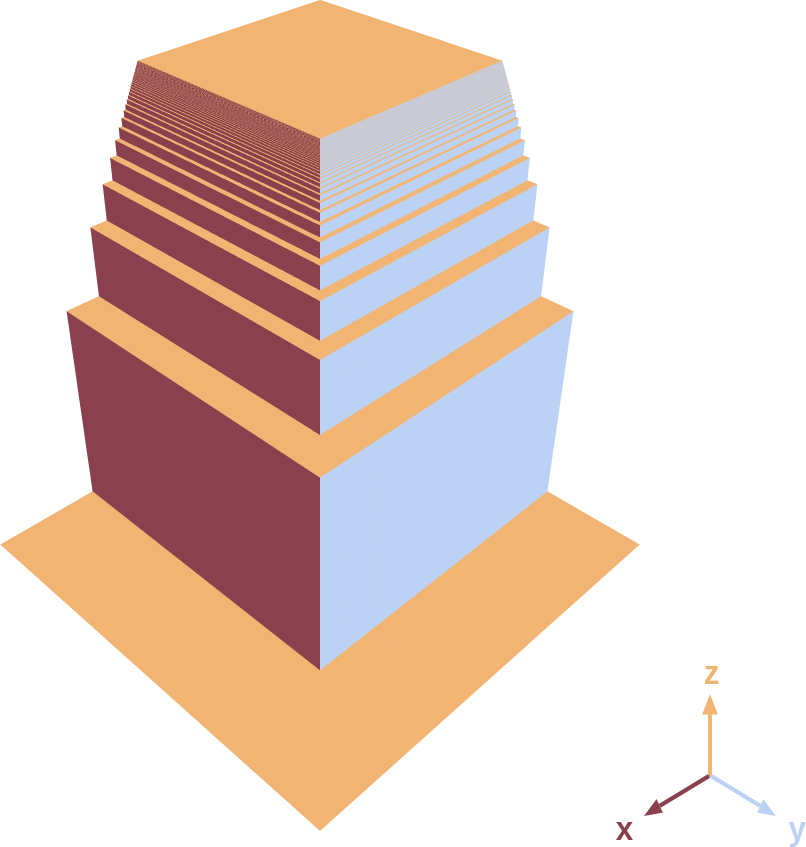}
    \caption{Two segments of upwards $z$-slippery points}
    \label{fig:slippery3d2}
    \end{subfigure}
    \hfill
    \begin{subfigure}{0.5\linewidth}
    \centering
    \includegraphics[width=\linewidth]{Figures/slippery3dline.png}
    \caption{A 1-parameter family of $xy$-slippery points}
    \label{fig:slippery3dline}
    \end{subfigure}
    \caption{Closed subsets of $\R^2$ and $\R^3$ with different kinds of slippery points.}\label{fig:slipperyexamples}
\end{figure}

\begin{longrem} \label{rem:slippery}
    Notice that if $\bm{x} \in A$ is upwards slippery in the $k$ first coordinates, then up to permuting the coordinates, there is an $\ell\geq 1$ and a sequence of points $\bm{x}^j$ in $A$ converging to $\bm{x}$ of the form
    $$\bm{x}^j = \big(\bm{x}^j_-, \bm{x}^j_+, \bm{x}^j_0\big) \in \R^k \times \R^\ell \times \R^{n-k-\ell},$$
    where $\bm{x}^j_-$ is (strictly) increasing in all coordinates, $\bm{x}^j_+$ is (strictly) decreasing in all coordinates, $\bm{x}^j_0$ is constant in all coordinates, and $\bm{x}^j$ is upper extremal in the $k$ first coordinates.

    Note that some points may be both upwards and downwards slippery, even in different sets of coordinates.
\end{longrem}

The next proposition justifies that \zcref{defn:alternativeslippery} agrees with the definition of slippery given in the introduction (\zcref{def:terraced}).
\begin{prop}\label{prop:slippery_corners}
    Let $A \subset \R^n$ be a closed box-convex set. A point in $A$ is slippery if and only if it is upwards or downwards slippery in some subset of coordinates.
\end{prop}

\begin{proof}
    Up to translation, we may consider the point $\bm{0} \in A$. For $\varepsilon > 0$, we write $B_\varepsilon \coloneqq (-\varepsilon, \varepsilon)^n$.
    
    By definition, $\bm{0}$ is a terraced (i.e., nonslippery) point of $A$ if and only if for every standard (closed) quadrant $C$ at $\bm{0}$, there is an $\varepsilon > 0$ such that either $A \cap \interior C \cap B_\varepsilon = \varnothing$, or $\interior C \cap B_\varepsilon \subset A$. In the first case, we say that $A$ is (locally) empty in this quadrant, and in the second case, we say that $A$ is (locally) full in this quadrant. From this, and using \zcref{rem:slippery}, it is clear that terraced points are not upwards nor downwards slippery in any coordinates.

    Let us now assume that $\bm{0}$ is neither upwards nor downwards slippery, in any subset of the coordinates. We will show by induction on the dimension of standard quadrants at $\bm{0}$ that $A$ is either empty or full in each of those quadrants. Note that this is obvious for quadrants of dimension $0$ and $1$. Let us assume that it is satisfied for all standard quadrants of dimension $\leq m-1$. Up to permuting coordinates, any $m$-dimensional quadrant $C$ may be written in the form
    $$C= C_- \times C_+ \times \{0\}^{n-m},$$
    where $$C_-=(-\infty,0]^{k_-}, \qquad C_+=[0,\infty)^{k_+},$$ and $m = k_- + k_+$. We may assume that $k_- \geq 1$, $k_+ \geq 1$; otherwise, box-convexity would imply that the quadrant is empty or full. Since $\bm{0}$ is not upwards slippery in the first $k_-$ coordinates and not downward slippery in the next $k_+$ coordinates, there exists $\varepsilon>0$ such that
    $$(\interior (C_-) \times C_+ \times \{0\}^{m-k}) \cap B_\varepsilon$$
    contains no point in $A$ which is upward extremal in the first $k_-$ coordinates, and
    $$( C_- \times \interior (C_+) \times \{0\}^{n-m}) \cap B_\varepsilon$$
    contains no point of $A$ which is downward extremal in the next $k_+$ coordinates.

    Let $\bm{x} = (\bm{x}_-, \bm{x}_+, \bm{0})$ be any point in $A \cap C \cap B_\varepsilon$. We claim that there exist $\bm{y}_- \in \partial C_-$, $\bm{y}_+\in \partial C_+$ such that 
    $$(\bm{x}_-, \bm{y}_+, \bm{0}), (\bm{y}_-, \bm{x}_+, \bm{0}) \in A \cap C \cap B_{\varepsilon},$$ and
    $$(\bm{x}_-, \bm{y}_+, \bm{0})\leq \bm{x}\leq (\bm{y}_-, \bm{x}_+, \bm{0}).$$ 
    
    We simply choose $\bm{y}_-$ to be a maximal element of $C_-$ satisfying $(\bm{y}_-, \bm{x}_+, \bm{0}) \in A \cap C \cap B_\varepsilon$ and $\bm{x}_- \leq \bm{y}_- \leq \bm{0}$. Such an element exists because $A\cap C \cap B_\varepsilon$ is closed. Furthermore, $\bm{y}_-\in \partial C_-$ since $A \cap (\interior (C_-) \times C_+ \times \{0\}^{n-m}) \cap B_\varepsilon$ contains no upwards extremal point of $A$ in the first $k_-$ coordinates. We define $\bm{y}_+$ in a similar way.

    We have shown that $A \cap C \cap B_\varepsilon$ is the box-closure of points in \emph{proper} subquadrants of $C$ near $\bm{0}$, which are all empty or full for $A$ by induction, so the same holds for $C$ as well. 
    \end{proof}

        \subsection{Proof of \texorpdfstring{\zcref{thmintro:slippery}}{thm:slippery}}

Suppose $\bm{s}$ is an upward slippery point of $\overline{\Zg(\Phi)}$ with respect to the first $k$ coordinates, and assume for the sake of contradiction that $\bm{s}$ is not a spherical or $\T^2$-type planar multislope. As in \zcref{rem:slippery}, $\overline{\Zg(\Phi)}$ contains a sequence of points $\bm{s}^j \to \bm{s}$, all upward extremal in the first $k$ coordinates, and such that sequence is strictly increasing in the first $k\geq 1$ coordinates, strictly decreasing in the next $\ell\geq 1$ and constant in the remaining $n-k-\ell$ coordinates.
         
By \zcref{lem:limits}, there exists a foliation $\F^\infty$ with boundary slope $\bm{s}$, possibly with some entries replaced by $\aleph$, and such that $\F^\infty$ has $\flat/\natural/\aleph$ type on the first $k$ boundary components and $\sharp/\natural/\aleph$ type on the next $\ell$ boundary components. First, let us warm up by assuming that $\F^\infty$ has no $\aleph$ type boundaries. Since we assumed that $\bm{s}$ is not a spherical multislope or a $\T^2$-type planar multislope, the fine version of Eliashberg--Thurston applies to $\mathcal F^\infty$. This gives a positive contact structure $\xi_+$ such that $\partial \xi_+ \succ \partial \F^\infty$. Therefore, 
$$\slope_i(\xi_+) \geq s_i \qquad \text{for } 1\leq i\leq n.$$ 
Moreover, since $\partial_i \F^\infty$ is $\sharp/\natural$ type for $i \in \{k+1, \dots, k+\ell\}$, \zcref{lem:phaselocking} implies that we have the strict inequality 
$$\slope_i(\xi_+) > s_i \qquad \text{for } i\in \{k+1,\dots,k+\ell\}.$$ 
Second, by \zcref{thmintro:contactzigg} there exists a contact structure $\xi_-$ with 
$$\bm{\slope}(\xi_-)=\bm{\slope}(\F^\infty)=\bm{s}.$$
Apply \zcref{prop:boxconvex} to the pair $(\xi_-,\xi_+)$. This results in a foliation with boundary multislope
$$\bm{t}:=(s_1,\dots,s_k,s_{k+1}+\varepsilon,\dots,s_{k+\ell}+\varepsilon, s_{k+\ell+1},\dots,s_n)\in \overline{\Zg(\Phi)}.$$
For $j$ sufficiently large, $\bm{t} \geq \bm{s}^j$ with strict inequality in the first $k+\ell$ coordinates. Therefore, $[\bm{s}^j, \bm{t}] \subset \overline{\Zg(\Phi)}$ by box-convexity, hence
$$\big(s_1, \dots, s_k, s^j_{k+1}, \dots, s^j_n\big) \in C^+_k(\bm{s}^j) \setminus \{ \bm{s}^j\},$$
contradicting our assumption that $\bm{s}^j$ was upward extremal in $\overline{\Zg(\Phi)}$ in the first $k$ coordinates.
         
Let us now examine the more general situation where $\F^\infty$ has some boundaries of type $\aleph$. Up to permuting the order of the boundary components of $M$, we have
$$\fineslope(\F^\infty) = \big(\underset{k_1}{\underbrace{s^{\flat/\natural}_1, \dots, s^{\flat/\natural}_{k_1}}}, 
\underset{k_2}{\underbrace{\aleph,\dots,\aleph}}, \ 
\underset{\ell_1}
{\underbrace{s^{\sharp/\natural}_{k+1}, \dots, s^{\sharp/\natural}_{k+\ell_1}}}, 
\underset{\ell_2}{\underbrace{\aleph,\dots,\aleph}}, \ 
\underset{m_1}
{\underbrace{s^{?}_{k+\ell+1}, \dots, s^{?}_{k+\ell+m_1}}}, 
\underset{m_2}{\underbrace{\aleph,\dots,\aleph}} 
\big)
$$
for some $k_1,k_2,\ell_1,\ell_2,m_1,m_2$ satisfying $k_1+k_2=k$, $\ell_1+\ell_2=\ell$, and $m_1+m_2=n-k-\ell$. Since $\bm{s}$ is strongly $\aleph$-unobstructed, $\F^\infty$ has no compact planar or $\T^2$-type planar minimal sets, so we may apply the fine Eliashberg--Thurston theorem to $\F^\infty$. By \zcref{lem:phaselocking} and \zcref{lem:resolution_of_aleph}, we obtain a contact structure $\xi_+$ with boundary slope $\slope(\xi_+)$ of the form

\begin{align*}
\begin{multlined}
\big(
\underset{k_1}{\underbrace{s_1,\dots,s_{k_1}}},
\underset{k_2}{\underbrace{+\infty,\dots,+\infty}},
\underset{\ell_1}{\underbrace{
  s_{k+1}+\varepsilon,\dots,s_{k+\ell_1}+\varepsilon}},
\underset{\ell_2}{\underbrace{+\infty,\dots,+\infty}}, \underset{m_1}{\underbrace{
  s_{k+\ell+1},\dots,s_{k+\ell+m_1}}},
\\
\underset{m_2}{\underbrace{+\infty,\dots,+\infty}}
\big).
\end{multlined}
\end{align*}
The argument now proceeds as in the previously discussed case. \qed

        \section{Connectedness}\label{sec:connectedness}

In this section, we study the case where $\Phi$ is a \emph{fully punctured pseudo-Anosov flow without perfect fits on $M$}, and we prove \zcref{thmintro:boundedconnected} from the Introduction. For flows of this type, every $\Phi$-horizontal foliation may be specified by an assignment of an element of $\Homeo(I)$ to each edge of the associated \emph{veering triangulation}; see~\cite{Z24}. One might try to use the contractibility of $\Homeo(I)$ to prove that $\Zg(\Phi)$ is contractible. However, some choices of such homeomorphisms yield foliations with Reeb annuli, which are possibly not even foliations of $\aleph$ type; such foliations are not cataloged by $\Zg(\Phi)$. We will partially circumvent this problem using \zcref{thmintro:alephremove}.

We begin with a lemma encapsulating the results of the veering triangulation machinery.

\begin{lem} \label{lem:continuousfol}
    Suppose $\Phi$ is a fully punctured pseudo-Anosov flow without perfect fits. Let $\F_0$ and $\F_1$ be two $\Phi$-horizontal foliations. Then there exists a path of $\Phi$-horizontal $C^0$-foliations $\F_t$, $0 \leq t \leq 1$, which varies continuously in the $C^0$ topology.
\end{lem}

\begin{proof}
    Let $B$ be the $2$-skeleton of the corresponding veering triangulation. By \zcref{thm:pABSF}, $B$ carries every $\Phi$-horizontal foliation. Recall that $B$ is naturally a branched surface without triple points whose branch loci correspond to the edges of triangulation. Therefore, every foliation carried by $B$ is specified by a choice of an element $h_e\in \Homeo(I)$ for each edge $e$ of the veering triangulation. The resulting foliations vary continuously as a function of $\{h_e\}_{e\in B^{(2)}}$. By connectedness of $\Homeo(I)$, any two foliations carried by $B$ may be connected by a path of foliations, which is continuous in the $C^0$ topology.
\end{proof}

\begin{proof}[Proof of \zcref{thmintro:boundedconnected}]
    Let $C_0$ be a compact connected component of $\Zg(\Phi)$, and assume by contradiction that $C_1$ is a different connected component. Let $\F_0$ and $\F_1$ be $\Phi$-horizontal foliations with boundary multislopes $\bm{s}_0 \in C_0$ and $\bm{s}_1 \in C_1$. By \zcref{lem:continuousfol}, there exists a family of $\Phi$-horizontal $C^0$-foliations $\F_t$, $0 \leq t \leq 1$, which is continuous in the $C^0$ topology, and which interpolates between $\F_0$ and $\F_1$. However, for $t \in (0,1)$, $\partial \F_t$ might not have a well-defined boundary multislope, nor be admissible.

    Let us assume that for all $t \in [0,1]$, all the components of $\partial \F_t$ are Reebless and have finite slopes. Then continuity of the slope with respect to the $C^0$ along $\partial M$ (see \zcref{lem:slconv}), the path $t \mapsto \bm{s}(\partial \F_t)$ is continuous and contained in $\Zg(\Phi)$, which contradicts that $C_0$ and $C_1$ are distinct connected components.

    We now assume by contradiction that there is a $\tau \in (0,1]$ such that $\F_\tau$ has at least one boundary component without finite slope. Note that the set of $t \in [0,1]$ such that $\partial \F_t$ has finite multislope is open, as one can realize such boundary foliations as suitable suspensions of circle homeomorphisms. Therefore, there is a well-defined first time $\tau_0 \in (0,1)$ for which $\F_{\tau_0}$ does not have finite boundary multislope. Then, the path $t \in [0, \tau_0) \mapsto \bm{s}(\partial \F_t)$ is well-defined, continuous, and contained in $C_0$; in particular, it is bounded. Therefore, by \zcref{lem:slconv}, no component of $\partial \F_{\tau_0}$ has slope $\pm \infty$, and the components without finite slopes are necessarily of $\aleph$ type. By \zcref{thmintro:alephremove}, there exists an unbounded connected set $R \subset \Zg(\Phi)$, obtained by resolving the $\aleph$ components of $\F_{\tau_0}$, such that $C_0 \cap \overline{R} \neq \varnothing$.\footnote{Here, we crucially use that $C_0$ is \emph{closed}! See \zcref{ex:nonconnected} below.} This readily implies $R \subset C_0$,\footnote{Indeed, if $p \in C_0 \cap \overline{R}$, then $R \cup \{p\}$ is connected as well and intersects $C_0$, so $R \cup \{p\} \subset C_0$ by the maximality of connected components.} contradicting that $C_0$ is bounded. 
\end{proof}

When $n\leq2$ (and perhaps $n=3$), it is enough to assume the existence of a \emph{bounded} (and not necessarily \emph{compact}) connected component of $\Zg(\Phi)$ to deduce that the ziggurat itself is connected, since box-convexity is restrictive enough. However, it might not be the case for $n \geq 4$, as the following example suggests.

\begin{ex} \label{ex:nonconnected}
    In $\R^4= \R^3 \times \R$, we consider a subset $C \subset \R^3 \times \{0\}$ of the form of the ``spiky'' ziggurat in \zcref{fig:weird_ziggurat}. As a subset of $\R^3$, it can be described as a connected, box-convex neighborhood of the interval
    $$\big\{(t, t, -t) \mid t \in (0,1)\big\}$$
    contained in the open octant $\{x >0, \, y> 0, \, z<0\}$. By construction, $C$ is bounded but not closed, since $\bm 0 \in \overline{C} \setminus C$.
    
    Consider $R \coloneqq (0,1) \times \{0\} \times \{0\} \times \R$. Then the following scenario could occur: $C$ is a connected component of a ziggurat for a nonwandering flow $\Phi$ with $n=4$, and a sequence of $\Phi$-horizontal foliations whose multislopes converge to $\bm 0 \in \R^4$ develops an $\aleph$ boundary foliation on the fourth component, while the component corresponding to $x$ is $\sharp$-latching and the components corresponding to $y$ and $z$ are nonlatching. This configuration could be resolved as $R \subset \Zg(\Phi)$. However, $R \not\subset C$.

    Nevertheless, we are not aware of a flow displaying this behavior.
\end{ex}

\begin{figure}[ht]
    \centering
    \includegraphics[width=0.5\linewidth]{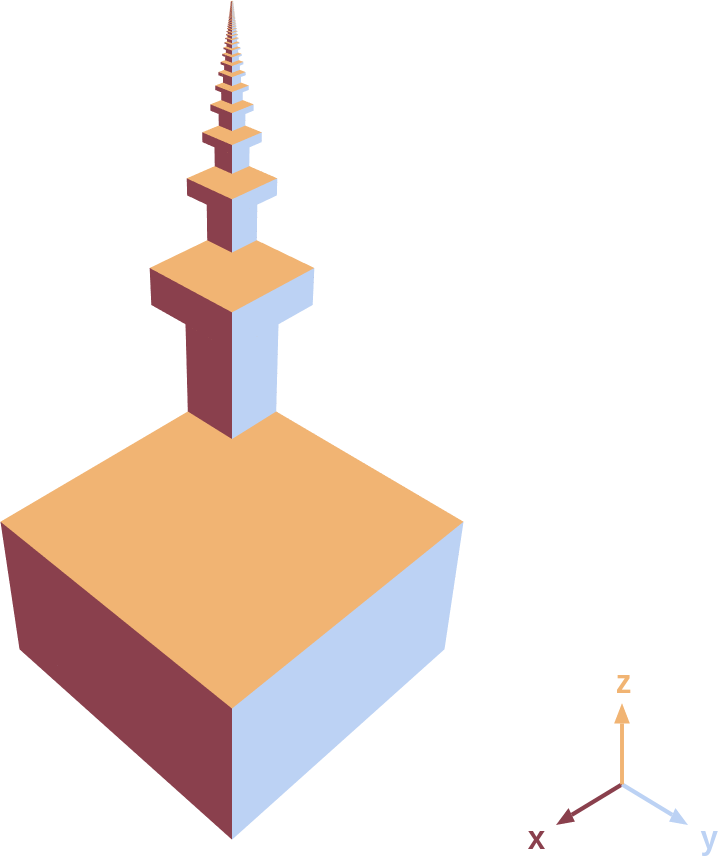}
    \caption{A spiky ziggurat.}
    \label{fig:weird_ziggurat}
\end{figure}

We remark that in the very special case $n=2$, box-convexity imposes strong restrictions on the number of \emph{unbounded} connected components:

\begin{lem}
    Suppose $A\subset \R^2$ is a box-convex set. Then $A$ has at most two unbounded connected components.
\end{lem}

\begin{proof}
    We argue by contradiction. Let $A_1, A_2, A_3$ be three distinct connected components of $A$. Box-convexity implies that the projections of the $A_i$'s onto the $x$-axis and the $y$-axis must be disjoint. Thus, we may sort the $A_i$'s in ascending order by their $x$-coordinates. The $A_i$'s must be sorted in descending order by their $y$-coordinates. Indeed, if there exists a pair $\bm{x}\in A_i$, $\bm{y}\in A_j$ such that $\bm{x} \leq \bm{y}$, then box convexity implies that $\bm{x}$ and $\bm{y}$ belong to the same component of $A$. Now the condition that the $A_i$'s are sorted in $x$ and $y$ implies that $A_2$ must be bounded, a contradiction.
\end{proof}

Unfortunately, we are not able to bound the number of unbounded components for $n \geq 3$; \zcref{lem:continuousfol} does not seem to constrain unbounded connected components, and box-convexity is not restrictive enough, as the following example shows.

\begin{ex}
    The set 
    $$\bigcup_{k \in \Z} \{(k, -k)\} \times \R \subset \R^3$$
    is box-convex and has infinitely many unbounded connected components.
\end{ex}

\clearpage
\appendix


    \section{Mean curvature flow for foliations on \texorpdfstring{$\T^2$}{T2}} \label{sec:mcf}

\textbf{Mean curvature flow} or \textbf{curve shortening flow} is a geometric flow for curves on Riemannian surfaces. Given an initial curve $\gamma_0: [0,1] \to S$, the curve evolves by $$\frac{d\gamma_t(x)}{dt} = \kappa \vec{n},$$ where $\kappa$ is the curvature and $\vec{n}$ is the unit normal pointing toward the convex side of $\gamma$. This partial differential equation can be considered as gradient flow for the length functional on $\gamma$; hence it tends to simplify curves. Indeed a series of work by Gage, Hamilton, Grayson, and others culminating in \cite{grayson.ShorteningEmbeddedCurves} proves that mean curvature flow on a closed curve in a Riemannian surface either shrinks to a point or converges to a closed geodesic \cite{gage.CurveShorteningSurfaces,gage.hamilton.HeatEquationShrinking, grayson.HeatEquationShrinks}.

A key property of mean curvature flow is the avoidance principle, which states that mean curvature flow applied to a pair of curves does not create new intersections. In particular, an embedded curve remains embedded during mean curvature flow.

We consider \textbf{leafwise mean curvature flow (LMCF)}: a geometric flow on foliations under which each leaf evolves by standard mean curvature flow. Thanks to the avoidance principle, this flow (as long as it is defined) preserves the property of being a foliation. The goal of this section is to prove the following convergence result for leafwise mean curvature flow:

\begin{prop}\label{prop:mean_curvature_flow}
    Let $\G$ be a Reebless $C^{\infty,0}$ foliation of $\T^2$. Equip $\T^2$ with a flat Euclidean metric. Then after a preparatory arbitrarily $C^0$-small isotopy to $\G$, leafwise mean curvature flow on $\G$ converges tangentially to a linear foliation of slope $\slope(\G)$. Moreover, if $\G$ is smooth, then no preparatory isotopy is needed, and the LMCF of $\G$ remains smooth.
\end{prop}

We prove \zcref{prop:mean_curvature_flow} by ``sandwiching'' $\G$ between two foliations by closed curves, whose convergence can be established by appealing to results in the literature. Before proving \zcref{prop:mean_curvature_flow}, we first explain that this sandwiching is possible at the level of homomorphisms of $S^1$. Recall that $\widetilde{\Homeo^+}(S^1)$ denotes the universal cover of $\Homeo^+(S^1)$, or equivalently the set of homeomorphisms $f : \R \to \R$ satisfying
$$f(x+1) = f(x)+1 \qquad \textrm{for all }x\in \R.$$
Let $R_s:\R \to \R$ be the translation by $+s \in \R$. The following lemma is standard; see for example~\cite[Proposition 1.1]{Militon2018}.

\begin{lem}\label{lem:approxrotation}
Let $h\in \widetilde{\Homeo^+}(S^1)$. Then for every $\varepsilon > 0$, $h$ is conjugate to another $h'\in \widetilde{\Homeo^+}(S^1)$ with $\big\vert h'-R_{\widetilde{\rot}(h)} \big\vert<\varepsilon$.
\end{lem}

\begin{cor}\label{cor:approxrotation2}
    Let $h \in \widetilde{\Homeo^+}(S^1)$ with translation number $\widetilde{\rot}(h) \in \R$. For every $s>\widetilde{\rot}(h)$ (resp.~$s < \widetilde{\rot}(h)$), there exists $f \in \widetilde{\Homeo^+}(S^1)$ such that $f > \mathrm{id}$ (resp.~$ f< \mathrm{id}$), and $fh$ is conjugate to a translation by $s$. 
\end{cor}

\begin{proof}
    Assume that $s>\widetilde{\rot}(h)$, and let $\varepsilon \coloneqq (s-\widetilde{\rot}(h))/2 > 0$. By \zcref{lem:approxrotation}, $h$ is conjugate by some homeomorphism $g$ to an $\varepsilon$-neighborhood of $R_{\widetilde{\rot}(h)}$. Then 
    \begin{align*}
        g h g^{-1} &< R_s,
    \end{align*}
    so $f \coloneqq (g^{-1} R_s g) h^{-1}$ satisfies $f > \mathrm{id}$, and $fh = g^{-1} R_s g$ is the desired homeomorphism.
\end{proof}

\begin{proof}[Proof of \zcref{prop:mean_curvature_flow}]
    The avoidance principle combined with continuous dependence of curve shortening flow on the initial conditions implies that $\G$ remains a foliation as long as leafwise mean curvature flow is defined.
    \begin{enumerate}[leftmargin=*]
        \item \textit{Step 1: Long time existence.} We mimic the proof in the closed case. First use Huisken's distance comparison principle~\cite[Theorem 2.1]{huisken.DistanceComparisonPrinciple}. For $p,q$ on the same leaf of $\widetilde{\G}$, define the extrinsic and intrinsic distances $d(p,q)$ and $l(p,q)$. Huisken's distance comparison principle says that 
        $$\inf_{p\neq q} d(p,q)/l(p,q)$$
        is nondecreasing. Since $\T^2$ is compact and $\G$ is Reebless, the leaves of $\widetilde \G$ are quasigeodesics with uniform quasigeodesic constant. Therefore, $\inf_{p\neq q} d(p,q)/l(p,q)$ is uniformly bounded below during curvature flow.

        If mean curvature flow develops a singularity in finite time, then the usual rescaling trick says that the rescaled flow is a contracting soliton (ruled out because $\G$ has no contractible curves) or a grim reaper (ruled out by Huisken's distance comparison principle). 
        
        \item \textit{Step 2: Convergence to a linear foliation.} First, we handle the result for any foliation $\G_s$ topologically isotopic to a linear foliation of slope $s\in \Q$. Grayson proves that essential closed curves on Riemannian surfaces converge to geodesics \cite{grayson.ShorteningEmbeddedCurves}. Therefore, every leaf of $\G_s$ converges to a geodesic in $\T^2$. By compactness of $\T^2$, $\G_s$ converges to a linear foliation in the tangential topology.

        Now we handle the result for an arbitrary Reebless foliation $\G$. Let $r,s$ be two rational slopes such that $(r,s)$ contains $\slope(\G)$. Our strategy is to sandwich $\G$ between two foliations $\G_r$ and $\G_s$ which are topologically isotopic to linear foliations of slope $r$ and $s$ respectively. 

        Since $\G$ is Reebless, $\G$ admits a positively monotone curve $\gamma$ which intersects every leaf by \zcref{lem:monotone_existence}. After $C^0$-small preparatory isotopies, we may make $\gamma$ smooth and transverse to $\G$, using \zcref{lem:transversalizing}. From now on, we work in coordinates in which $\gamma$ has slope $+\infty$. We further arrange that $\G$ is smooth near $\gamma$. The holonomy of $\G$ defines a first return map $h:\gamma \to \gamma$. By \zcref{cor:approxrotation2}, we may sandwich $h$ (or more precisely, $\widetilde h: \widetilde \gamma \to \widetilde \gamma$) between two maps $h_r, h_s:  \gamma \to  \gamma$ which are conjugate to rotations by $+r$ and $+s$ respectively.
        
        Now we can modify $\G$ in a neighborhood of $\gamma$ to obtain a new foliation $\G_s$ which is positively transverse to $\G$ near $\gamma$, whose leaves are equal to those of $\G$ away from $\gamma$, and which has first return map $h_s$. Similarly define $\G_r$ to be a new foliation which is negatively mostly-transverse to $\G$ and has first return map $h_r$. By construction, $\G_r$ has no positive topological crossings with $\G$ and $\G_s$ has no negative topological crossings with $\G$. Mean curvature flow preserves this property by~\cite{angenent.ParabolicEquationsCurvesII}. Since we have already shown that $\G_r$ and $\G_s$ converge to linear foliations, we deduce that $\G$ eventually lies in the cone between the linear foliations of slope $r-\varepsilon$, $s+\varepsilon$. Taking the interval $(r,s)$ arbitrarily small gives the desired convergence. \qedhere
    \end{enumerate}
\end{proof}

    \section{Ziggurats of annulus and disk suspensions} \label{sec:zigannulusdisk}

Our main theorems often exclude two special cases: $M\cong S^1\times D^2$ and $M\cong \T^2\times I$. In this section, we analyze $\Zg(\Phi)$ in these special cases.

\begin{defn}
    A flow $\Phi$ is a \textbf{planar suspension} if it is orbit equivalent to the suspension of a homeomorphism of a planar surface. We say that it is a \textbf{disk (resp.~annulus) suspension} if it is orbit equivalent to the suspension of a homeomorphism of the $2$-disk and $M \cong S^1 \times D^2$ (resp.~the annulus $S^1 \times I$ and $M \cong \T^2 \times I$).
\end{defn}

\begin{defn}
    We say that $\bm{s} \in \R^n$ is a \textbf{rational homological longitude} if there is a nontrivial element of $H_2(M,\partial M)$ represented by a surface $\Sigma$ meeting each $\partial_i M$ in a (possibly empty) collection of parallel curves of (oriented) slope $s_i$. The \textbf{locus of rational homological longitudes} is the set of all rational homological longitudes.

    More generally, a (real) \textbf{homological longitude} is a multislope $\bm{s} \in \R^n$ such that there exists a nonzero class in the image of
    $$\partial : H_1\big(M, \partial M ; \R\big) \longrightarrow H_1\big(\partial M; \R\big)$$
    which has slope $s_i$ (or vanishes) on $\partial_i M$.
\end{defn}

\begin{rem} \label{rem:homlong}
 The locus of homological longitudes is a union of quadric hypersurfaces. Two homological longitudes $\bm{r},\bm{s}$ cannot satisfy $\bm{r} < \bm{s}$, since the intersection pairing on the image of $\partial: H_2(M,\partial M) \to H_1(\partial M)$ vanishes.
\end{rem}

\begin{lem}\label{lem:special_classification} Taut foliations on $S^1\times D^2$ and $\T^2\times I$ transverse to the boundary are classified as follows:
    \begin{enumerate}
        \item If $\F$ is a taut $C^0$-foliation on $S^1 \times D^2$ transverse to the boundary, then $\F$ is isotopic to the standard foliation by disks.
    
        \item If $\F$ is a taut $C^0$-foliation on $\T^2 \times I$ transverse to the boundary, then $\F$ is isotopic to a product of a foliation on $\T^2$ with $I$.
    \end{enumerate}
\end{lem}

\begin{proof} We prove each item in turn.
    \begin{enumerate}
        \item First, suppose $\partial \F$ contains a curve $\gamma$ of slope $0$. Since leaves of taut foliations are $\pi_1$-injective, $\gamma$ bounds a disk in a leaf of $\F$. Now the Reeb stability theorem implies that $\F$ is a foliation by disks and hence is standard.

        Now suppose that $\partial \F$ does not contain a closed curve of slope $0$. We will find a null-homotopic transversal to $\F$, which will contradict tautness of $\F$. If $\partial \F$ is Reebless and has slope $\neq 0$, then we can find the null-homotopic transversal in $\partial \F$. If $\partial \F$ has a Reeb component, we may choose a leaf $\lambda$ of $\F$ containing a closed curve on $\partial M$. Since this closed curve does not have slope $0$ by assumption, $\pi_1(\lambda)$ is nontrivial. After lifting to a finite cover, we may assume that $\pi_1(\lambda)$ surjects onto $\pi_1(S^1\times D^2)$. Leaves of taut foliations are $\pi_1$-injective, so we know $\pi_1(\lambda)\cong \Z$. If $\lambda$ is compact, then it most be an annulus, and therefore must separate $M$, contradicting tautness of $\F$. If $\lambda$ is noncompact, then it accumulates on itself. Choose a short transversal from $\lambda$ to itself, and close it up to a positively monotone curve $\gamma$ using an arc in $\lambda$. Since $\pi_1(\lambda)$ surjects onto $\pi_1(M)$, we choose the closure so that $\gamma$ is null-homotopic. After a small isotopy, we get a closed null-homotopic (topological) transversal, contradicting tautness.
        
        \item Let $\widetilde \F$ be the lift of $\F$ to the universal cover and let $\mathcal{L}$ be the leaf space of $\widetilde \F$. Let $\mathcal{L}_i$, $1\leq i \leq 2$, be the leaf space of $\widetilde {\partial_i\F}$. Inclusions of leaves of $\partial \F$ into leaves of $\F$ induce immersions $f_i\colon \mathcal{L}_i \to \mathcal{L}$. 
        \begin{enumerate}
            \item \textit{Each $f_i$ is injective:} By the classification of foliations of $\T^2$, $\mathcal{L}_i$ is a possibly non-Hausdorff $1$-manifold all of whose non-separated points arise from Reeb annuli in $\partial_i \F$. Since $f_i$ is an immersion into a simply connected $1$-manifold, $f_i$ is injective in a neighborhood of any separated point. If $x,y$ is a pair of non-separated leaves in $\mathcal{L}_i$ corresponding with the two boundary components of a Reeb annulus, then we claim that $f_i$ is also a homeomorphism onto its image in a neighborhood of $\{x,y\}$. It suffices to check that $f_i(x)\neq f_i(y)$. Otherwise, there is a leaf $\lambda$ of $\F$ containing both $x$ and $y$. Since leaves of taut foliations are $\pi_1$-injective and both $x$ and $y$ are curves represent the same element of $\pi_1(\T^2\times I)$, $\lambda$ must be an annulus cobounding $x$ and $y$. But then $\lambda$ separates $M$ which contradicts tautness of $\F$. So in fact $f_i(x) \neq f_i(y)$ and $f_i$ is a homeomorphism in a neighborhood of $\{x,y\}$ as desired. Since the target $\mathcal{L}$ is a simply connected 1-manifold, $f_i$ must be injective.
            
            \item \textit{Each $f_i$ is surjective:} Since the boundary components of $M$ are $\pi_1$-surjective, $\pi_1(M)$ preserves $f(\mathcal{L}_1)$ and $f(\mathcal{L}_2)$. Since $\F$ is taut, every point in $\mathcal{L}$ is connected to $f(\mathcal{L}_1)$ by a positive and a negative transversal. Since $\mathcal{L}$ is a simply-connected 1-manifold, there can be no point outside $f(\mathcal{L}_1)$ with this property; therefore $f_1$ and $f_2$ are surjective.
        \end{enumerate}

        Thus, $\mathcal{L}_1=\mathcal{L}_2=\mathcal{L}$. In particular, every leaf $\lambda$ of $\F$ contains exactly one leaf in $\partial_1 \F$ and exactly one leaf of $\partial_2 \F$, and 
        $$\pi_1(\partial_1 \lambda) \cong \pi_1(\partial_2 \lambda) \cong \pi_1(\lambda) \cong \mathrm{Stab}_{\pi_1(M)}\widetilde \lambda.$$
        Choose a continuous leafwise metric on $\F$. By compactness of $M$, there exists $R>0$ large enough that for every point in $M$, there are leafwise paths of length $R$ to either boundary component of $M$. Then for every leaf $\lambda$ of $\F$ with boundary components $\partial_1 \lambda$ and $\partial_2 \lambda$, $\lambda$ is contained in a leafwise $R$-neighborhood of $\partial_1 \lambda$. Again using that the inclusion of $\partial_i \lambda$ into $\lambda$ induces an isomorphism of fundamental groups, we conclude that $\lambda$ is topologically the product of $\partial_1 \lambda$ with a compact interval $I$. It follows that $\F$ is a product foliation.\qedhere
        \end{enumerate}
\end{proof}

\begin{prop}[Ziggurats of disk and annulus suspension flows]\label{prop:special_ziggurats}
    Assume that $\Zg(\Phi) \neq \varnothing$.
    \begin{enumerate}
        \item If $M \cong S^1 \times D^2$, then $\Phi$ is orbit equivalent to the suspension of a homeomorphism of $D^2$, and $\Zg(\Phi)$ is a single point, the longitude determined by $\partial D$.
        \item If $M \cong \T^2 \times I$, then $\Phi$ is orbit equivalent to the suspension of a homeomorphism of an annulus, and $\Zg(\Phi)$ is an interval with nontrivial interior contained in the locus of homological longitudes (for a framing coming from the product structure).
    \end{enumerate}
\end{prop}

\begin{proof} We prove each item in turn.
    \begin{enumerate}
        \item This follows from the classification result in \zcref{lem:special_classification}.
        \item That $\Zg(\Phi)$ is contained in the locus of homological longitudes follows from the classification in \zcref{lem:special_classification}. We need to show that $\Zg(\Phi)$ is convex. Given any two $\Phi$-horizontal foliations $\F_0$ and $\F_1$, \zcref{lem:special_classification} says that $\F_0$ and $\F_1$ are products of foliations $\G_0$ and $\G_1$ on $\T^2$ with an interval. Since $\Phi$ is nonwandering, $\F_0$ and $\F_1$ are taut and $\G_0, \G_1$ cannot have Reeb annuli. Therefore, $\F_0$ and $\F_1$ have measured sublaminations $\Lambda_0$ and $\Lambda_1$ realizing the boundary slopes of $\F_0$ and $\F_1$ respectively. Convex combinations $\Lambda_t:=(1-t)\Lambda_0 + t\Lambda_1$ for $0 \leq t \leq 1$ are horizontal measured laminations realizing boundary slopes which are convex combinations of $\slope(\partial\F_0)$ and $\slope(\partial\F_1)$. We need to show that all of these complete to $\Phi$-horizontal foliations. Focus on one of the guts $G$ of $\Lambda_t$, which is naturally a sutured manifold with horizontal boundary $R_+\sqcup R_-$. We can choose the vertical boundary of $G$ to be tangent to $\Phi$, so $\Phi$ enters along $R_-$ and exits along $R_+$. Suppose for the sake of contradiction that there exists an orbit of $\Phi$ which do not exit $G$. Compactness yields a small ball to which $\Phi$ recurs infinitely many times. Closing up the recurring flow lines of $\Phi$, we obtain loops which homologically pair to zero with $\Lambda_t$ (since they are disjoint from $\Lambda_t$) and pair non-negatively with $\Lambda_0$ and $\Lambda_1$ (since $\Phi$ is transverse to $\F_0$ and $\F_1$). It follows that they pair to zero with both $\Lambda_0$ and $\Lambda_1$. Then either $\slope(\partial\F_0)=\slope(\partial \F_1)$ and we are done, or these recurring loops are contractible which contradicts tautness of $\F_0$ and $\F_1$. Thus, we can assume that every orbit of $\Phi$ passes from $R_-$ to $R_+$. This implies that $G$ is a product so that $\Lambda_t$ can be completed to a $\Phi$-horizontal foliation. This concludes our proof of convexity of $\Zg(\Phi)$. 
        
        To show that $\Zg(\Phi)$ has nontrivial interior, we may twist the transverse foliation by annuli along a ribbon connecting the two boundary components; the details are left to the reader. \qedhere
    \end{enumerate}
\end{proof}

\medskip

Finally, we briefly discuss the \emph{contact} ziggurats in the disk and annulus suspension case. We leave the details to the interested reader.

\begin{longrem} \label{rem:contactsuspension}
    Adapting methods from \zcref{sec:S2contact}, we expect the following to hold:
    \begin{itemize}
        \item If $\Phi$ is a nonwandering disk suspension on $S^1 \times D^2$, then 
        $$\Zg^- (\Phi) = (0, +\infty), \qquad \Zg^+(\Phi) = (-\infty, 0).$$
        \item If $\Phi$ is a nonwandering annulus suspension on $\T^2 \times I$ with locus of homological longitude $\mathcal{H} \subset \R^2$, then
        $$\Zg^- (\Phi) \subset \mathcal{H}^+, \qquad \Zg^+(\Phi) \subset \mathcal{H}^-,$$
        where $\mathcal{H}^\pm$ are the two connected components of $\R^2 \setminus \mathcal{H}$, and $\mathcal{H}^-$ lies below $\mathcal{H}^+$.
    \end{itemize}
    In particular, in both cases,
    $$\Zg^-(\Phi) \cap \Zg^+(\Phi) = \varnothing.$$
    
    In the first case, the inclusions $(0, +\infty) \subset \Zg^-(\Phi)$ and $(-\infty, 0) \subset \Zg^+(\Phi)$ are clear: start with a smooth foliation by disks transverse to $\Phi$ and twist radially. Moreover, $\Zg^-(\Phi) \cap (-\infty, 0) =\Zg^+(\Phi) \cap (0, +\infty) = \varnothing$, since otherwise we could blow down $\Phi$ along $\partial M$ and construct a positive/negative contact structure transverse to a nonwandering suspension flow on $S^1 \times S^2$, contradicting \zcref{prop:contactS2suspension}. Therefore, the main difficulty is to show $0 \notin \Zg^\pm(\Phi)$. Note that in view of \zcref{cor:universallytight}, we can deduce that $\Zg^-(\Phi) \cap \Zg^+(\Phi) = \varnothing$, since a positive/negative contact structure with boundary slope $0$ would have a Legendrian curve along $\partial M$ bounding an embedded disk in $M$, implying overtwistedness, while such a contact structure would necessarily be tight.

    In the second case, the inclusions $\Zg^\pm(\Phi) \subset \overline{\mathcal{H}^\mp}$, or equivalently $\Zg^\pm(\Phi) \cap \mathcal{H}^\pm = \varnothing$, follow from blowing down $\Phi$ and applying \zcref{prop:contactS2suspension}. Moreover, by instability of irrational slopes, we can also deduce that $\Zg^\pm(\Phi) \cap \mathcal{H} \subset \Q^2$. However, proving $\Zg^\pm(\Phi) \cap \mathcal{H} = \varnothing$ is more delicate; notice that the elements in $\Zg^\pm(\Phi) \cap \mathcal{H}$ are spherical multislopes. Showing $\Zg^-(\Phi) \cap \Zg^+(\Phi) = \varnothing$ using \zcref{cor:universallytight} seems delicate as well.
\end{longrem}

    \section{To \texorpdfstring{$\aleph$}{aleph} and beyond} \label{sec:beyondaleph}

It is natural to consider foliations with boundary slope $\pm\infty$, or non-$\aleph$ type foliations with Reeb annuli. For simplicity, we avoided addressing these foliations in the body of the article. In this section, we discuss some formalism for extending our results, and we suggest some conjectures about the ziggurat of fine boundary multislopes. 

        \subsection{Generalized multislopes}

First, we revise and extend our terminology to accommodate these new kinds of boundary foliations.

\begin{defn}
    A \textbf{generalized coarse slope} is an element of $\R \cup \{\pm \infty,\aleph\}$. A \textbf{generalized fine slope} is a coarse slope, plus a type in $\{\natural, \dagger, \sharp, \flat\}$. The meaning of $s^\heartsuit$ for $\heartsuit \in \{\natural, \dagger, \sharp, \flat\}$ remains the same as in \zcref{sec:T2foliations}, and we define the generalized fine slope of a $\Phi$-horizontal foliation $\G$ with a Reeb annulus as follows:

    \begin{enumerate}
        \item $\aleph^\natural$ if $\G$ has neither $+\infty^\sharp$ nor $-\infty^\flat$ annuli. This class was simply called $\aleph$ in the body of the article.
        \item $\aleph^\sharp$ if $\G$ has a $+\infty^\sharp$ annulus, but no $-\infty ^\flat$ annuli.
        
        \item $\aleph^\flat$ if $\G$ has a $-\infty^\flat$ annulus, but no $+\infty^\sharp$ annuli.
    
        \item $\aleph^\dagger$ if $\G$ has both $+\infty^\sharp$ and $-\infty^\flat$ annuli.
    \end{enumerate}
    We write $\mathscr{S}$ for the space of real fine slopes $\R \times \{\natural, \dagger, \sharp, \flat\}$, and we define the space of \textbf{(horizontal) generalized fine slopes} as 
    $$\widehat{\mathscr{S}} \coloneqq \mathscr{S} \cup \left( \{-\infty, +\infty, \aleph\} \times \{\natural, \dagger, \sharp, \flat\} \setminus \big\{ {+\infty}^{\sharp/\natural}, {-\infty}^{\flat/\natural}\big\}\right).$$
    Here, the symbols $+\infty^{\sharp/\natural}$ and $-\infty^{\flat/\natural}$ are excluded from $\widehat{\mathscr{S}}$ as they do not occur as fine slopes of $\Phi$-horizontal foliations, and would cause some unnecessary complications.
\end{defn}

\begin{defn}[Topology on fine slopes]\label{defn:genfineslopetop}
    For $s\in \R$, we declare the following sets to be open in $\widehat{\mathscr{S}}$:
    \begin{align}
        \big\{t^\heartsuit \mid t\in (s, +\infty], \heartsuit \in \{\natural, \dagger, \sharp, \flat\} \big\} &\cup \big\{s^\sharp, s^\dagger \big\},\\
        \left\{t^\heartsuit \mid t\in [-\infty, s), \heartsuit \in \{\natural, \dagger, \sharp, \flat\}\right\} &\cup \big\{s^\flat, s^\dagger\big\}
    \end{align}
    The following two sets are also open:
    \begin{align}
        &\big\{{+\infty}^\dagger, \aleph^\sharp, \aleph^\dagger\big\},\\
        &\big\{{-\infty}^\dagger, \aleph^\flat, \aleph^\dagger\big\}.
    \end{align}
    We define the topology on $\widehat{\mathscr{S}}$ to be the topology generated by these open sets.
\end{defn}
Note that $\widehat{\mathscr{S}}$ is non-Hausdorff; moreover, the sets of the form $\big\{s^\dagger\big\}$ for $s \in \R \cup \{- \infty, +\infty, \aleph\}$ are open, and the only neighborhood of $\{\aleph^\natural\}$ is the whole space.

\begin{lem}
    The map $$\fineslope : \{\textrm{$\Phi$-horizontal foliations}\} \longrightarrow \widehat{\mathscr{S}}$$ is continuous in the $C^0$ topology.
\end{lem}

\begin{proof}[Proof (sketch)]
    For each $s \in [-\infty, +\infty]$, the property of admitting a positive/negative monotone curve is open in the $C^0$ topology. These open sets correspond exactly to the preimages of the open sets specified in \zcref{defn:genfineslopetop}.
\end{proof}

\begin{defn}[Order on fine slopes]
    We define a partial relation $\prec$ on generalized fine slopes generated by the following inequalities. First, for $s,t\in \R \cup \{\pm\infty\}$ we impose
    \begin{equation}
        s^\heartsuit \prec t^\diamondsuit
    \end{equation}
   when $s < t$ and $\heartsuit, \diamondsuit \in \{\natural, \flat, \sharp, \dagger\}$, and
    \begin{equation}
        s^\flat \prec s^\dagger \prec s^\sharp, \quad \text{and} \quad s^\dagger \prec s^\dagger.
    \end{equation}
    Second, we have 
    \begin{equation}
        -\infty^\dagger \prec \aleph^\natural, \aleph^\sharp \quad \mathrm{and} \quad
        \aleph^\natural, \aleph^\flat \prec +\infty^\dagger.
    \end{equation}
\end{defn}

\begin{longrem}
    We can also define the topology on $\widehat{\mathscr{S}}$ in terms of closed sets. Let $[s^\heartsuit, t^\diamondsuit]$ be a closed interval for the fine slope order. Depending on the values of the endpoints, we define a subset $A \subset \{\flat, \natural, \sharp\}$\footnote{$A$ stands for \emph{accidental} here.} as in the following table:

\begin{table}[ht]
\centering
\begin{tblr}{
  hlines, vlines,
  columns = {4.8em, c, m},   
  rows = {2.7em},            
  row{1} = {bg=gray!15},
  column{1} = {bg=gray!15},
}
  \diagcell{$s^{\heartsuit}$}{$t^{\diamondsuit}$}
    & $t = -\infty$
    & {$t \in \mathbb{R},$ \\ $+\infty^{\flat}$}
    & $+\infty^{\dagger}$ \\
  $-\infty^{\dagger}$ & $\{\flat\}$ & $\{\flat,\, \natural\}$ & $\{\flat,\, \natural,\, \sharp\}$ \\
  {$-\infty^{\sharp},$ \\ $s \in \mathbb{R}$} & $\varnothing$ & $\{\natural\}$ & $\{\natural,\, \sharp\}$ \\
  $s = +\infty$ & $\varnothing$ & $\varnothing$ & $\{\sharp\}$ \\
\end{tblr}
\end{table}

We then declare the set $[s^\heartsuit, t^\diamondsuit] \cup \aleph^A$ to be closed in the generalized fine slope topology, and the latter topology is generated by those closed sets. Note that a map $f : X \rightarrow \widehat{\mathscr{S}}$ defined on a topological space $X$ is continuous if and only if the preimages of those above closed sets, which form a subbase of the topology on $\widehat{\mathscr{S}}$, are closed in $X$.
\end{longrem}

\begin{defn}
    We define $\Zg_{\mathrm{fine}}(\Phi) \subseteq \mathscr{S}^n$ as the set of fine boundary multislopes realized by $\Phi$-horizontal foliations. We define $\Zg_{\aleph, \mathrm{fine}} \subseteq \widehat{\mathscr{S}}^n$ as the set of \emph{all} generalized fine multislopes realized by $\Phi$-horizontal foliations, after applying the following wildcard substitutions to eliminate $\aleph$ components:
    \begin{align*}
    \aleph^\natural \longmapsto [-\infty^\sharp,+\infty^\flat],\qquad
    \aleph^\flat \longmapsto \{-\infty^\dagger\},\qquad
    \aleph^\sharp \longmapsto \{+\infty^\dagger\},\qquad
    \aleph^\dagger \longmapsto \varnothing.
    \end{align*}
    Similarly define $\Zg^\pm_{\mathrm{fine}}(\Phi)$ as the sets of fine boundary multislopes achieved by $\Phi$-horizontal contact structures.
\end{defn}

The substitutions correspond to the natural resolutions of the various flavors of $\aleph$.

        \subsection{Speculations and conjectures}

We now propose some conjectures for ziggurats of generalized fine multislopes.

\begin{conj}
    The subsets $\Zg^\pm_{\mathrm{fine}}(\Phi) \subset {\mathscr{S}}^n$ are open.
\end{conj}

\begin{conj}
    The posets $\Zg^\pm_{\mathrm{fine}}(\Phi)$ are downwards and upwards closed, respectively. 
\end{conj}

These two conjectures would follow from the twisting and trimming trick, as well as from definition of the order and the topology on $\mathscr{S}$.

\medskip

We also conjecture that the following generalization of \zcref{thmintro:contactzigg} holds, at least away from certain obstructed multislopes:

\begin{conj}
    If $\Phi$ admits no horizontal compact planar surfaces nor $\T^2$-type planar laminations, then
    $$\Zg_{\mathrm{fine}}(\Phi)=\Zg^-_{\mathrm{fine}}(\Phi) \cap \Zg_{\mathrm{fine}}^+(\Phi).$$
    More generally, the equality holds away from obstructed multislopes.
\end{conj}

\medskip

\begin{conj}[Generalized closedness]
    If $\Phi$ is LBSF, then $\Zg_{\aleph, \mathrm{fine}}(\Phi)$ is closed. Moreover, if $\Phi$ is $\aleph$-unobstructed, then 
    $$\overline{\Zg_{\mathrm{fine}}(\Phi)} = \Zg_{\aleph, \mathrm{fine}}(\Phi).$$
    If $\Phi$ admits no horizontal compact planar surfaces nor $\T^2$-type planar laminations, then $\Zg_{\mathrm{fine}}(\Phi)$ is closed.
\end{conj}

One could show that along the faces of $\Zg(\Phi)$, the multislopes are realized by foliations with $\flat/\sharp$ boundary types along the boundaries corresponding to extremal coordinates (up to $\aleph$-degeneration). Therefore, those correspond to special subsets of $\Zg_\mathrm{fine}(\Phi)$. It would be interesting to describe those more precisely.

\medskip

It remains to understand how to characterize slippery points in the context of generalized fine multislopes. As a first step, we propose a purely set-theoretic problem:

\begin{conj}
    If $A\subset \mathscr{S}^n$ is open and closed, then its projection to $\R^n$ is terraced.
\end{conj}

\begin{question}
    What is the correct extension of the notion of slipperiness to $\Zg_\mathrm{fine}(\Phi)$, in terms of extremality with respect to $\prec$? In particular, how to characterize the slippery points ``at infinity''?
\end{question}


\clearpage
\printbibliography[heading=bibintoc, title={References}]
\end{document}

%% file: Figures/generic_staircase.pdf_tex
\begingroup%
  \makeatletter%
  \providecommand\color[2][]{%
    \errmessage{(Inkscape) Color is used for the text in Inkscape, but the package 'color.sty' is not loaded}%
    \renewcommand\color[2][]{}%
  }%
  \providecommand\transparent[1]{%
    \errmessage{(Inkscape) Transparency is used (non-zero) for the text in Inkscape, but the package 'transparent.sty' is not loaded}%
    \renewcommand\transparent[1]{}%
  }%
  \providecommand\rotatebox[2]{#2}%
  \newcommand*\fsize{\dimexpr\f@size pt\relax}%
  \newcommand*\lineheight[1]{\fontsize{\fsize}{#1\fsize}\selectfont}%
  \ifx\svgwidth\undefined%
    \setlength{\unitlength}{362.20654345bp}%
    \ifx\svgscale\undefined%
      \relax%
    \else%
      \setlength{\unitlength}{\unitlength * \real{\svgscale}}%
    \fi%
  \else%
    \setlength{\unitlength}{\svgwidth}%
  \fi%
  \global\let\svgwidth\undefined%
  \global\let\svgscale\undefined%
  \makeatother%
  \begin{picture}(1,0.73463201)%
    \lineheight{1}%
    \setlength\tabcolsep{0pt}%
    \put(0,0){\includegraphics[width=\unitlength,page=1]{generic_staircase.pdf}}%
    \put(0.83898671,0.6715651){\makebox(0,0)[t]{\lineheight{1.25}\smash{\begin{tabular}[t]{c}$\widetilde{\gamma}_{i+1}$\end{tabular}}}}%
    \put(0.10961653,0.03918451){\makebox(0,0)[t]{\lineheight{1.25}\smash{\begin{tabular}[t]{c}$\widetilde{\gamma}_{i}$\end{tabular}}}}%
    \put(0,0){\includegraphics[width=\unitlength,page=2]{generic_staircase.pdf}}%
  \end{picture}%
\endgroup%

%% file: Figures/generic_staircase_perturbed.pdf_tex
\begingroup%
  \makeatletter%
  \providecommand\color[2][]{%
    \errmessage{(Inkscape) Color is used for the text in Inkscape, but the package 'color.sty' is not loaded}%
    \renewcommand\color[2][]{}%
  }%
  \providecommand\transparent[1]{%
    \errmessage{(Inkscape) Transparency is used (non-zero) for the text in Inkscape, but the package 'transparent.sty' is not loaded}%
    \renewcommand\transparent[1]{}%
  }%
  \providecommand\rotatebox[2]{#2}%
  \newcommand*\fsize{\dimexpr\f@size pt\relax}%
  \newcommand*\lineheight[1]{\fontsize{\fsize}{#1\fsize}\selectfont}%
  \ifx\svgwidth\undefined%
    \setlength{\unitlength}{362.20654345bp}%
    \ifx\svgscale\undefined%
      \relax%
    \else%
      \setlength{\unitlength}{\unitlength * \real{\svgscale}}%
    \fi%
  \else%
    \setlength{\unitlength}{\svgwidth}%
  \fi%
  \global\let\svgwidth\undefined%
  \global\let\svgscale\undefined%
  \makeatother%
  \begin{picture}(1,0.73463201)%
    \lineheight{1}%
    \setlength\tabcolsep{0pt}%
    \put(0,0){\includegraphics[width=\unitlength,page=1]{generic_staircase_perturbed.pdf}}%
    \put(0.83898671,0.6715651){\makebox(0,0)[t]{\lineheight{1.25}\smash{\begin{tabular}[t]{c}$\widetilde{\gamma}_{i+1}$\end{tabular}}}}%
    \put(0.10961653,0.03918451){\makebox(0,0)[t]{\lineheight{1.25}\smash{\begin{tabular}[t]{c}$\widetilde{\gamma}_{i}$\end{tabular}}}}%
    \put(0,0){\includegraphics[width=\unitlength,page=2]{generic_staircase_perturbed.pdf}}%
  \end{picture}%
\endgroup%

%% file: Figures/smoothing_monotone.pdf_tex
\begingroup%
  \makeatletter%
  \providecommand\color[2][]{%
    \errmessage{(Inkscape) Color is used for the text in Inkscape, but the package 'color.sty' is not loaded}%
    \renewcommand\color[2][]{}%
  }%
  \providecommand\transparent[1]{%
    \errmessage{(Inkscape) Transparency is used (non-zero) for the text in Inkscape, but the package 'transparent.sty' is not loaded}%
    \renewcommand\transparent[1]{}%
  }%
  \providecommand\rotatebox[2]{#2}%
  \newcommand*\fsize{\dimexpr\f@size pt\relax}%
  \newcommand*\lineheight[1]{\fontsize{\fsize}{#1\fsize}\selectfont}%
  \ifx\svgwidth\undefined%
    \setlength{\unitlength}{739.41984883bp}%
    \ifx\svgscale\undefined%
      \relax%
    \else%
      \setlength{\unitlength}{\unitlength * \real{\svgscale}}%
    \fi%
  \else%
    \setlength{\unitlength}{\svgwidth}%
  \fi%
  \global\let\svgwidth\undefined%
  \global\let\svgscale\undefined%
  \makeatother%
  \begin{picture}(1,0.33836932)%
    \lineheight{1}%
    \setlength\tabcolsep{0pt}%
    \put(0,0){\includegraphics[width=\unitlength,page=1]{smoothing_monotone.pdf}}%
    \put(0.04964447,0.13511497){\makebox(0,0)[rt]{\lineheight{1.25}\smash{\begin{tabular}[t]{r}$L_0$\end{tabular}}}}%
    \put(0.40608643,0.31588408){\makebox(0,0)[t]{\lineheight{1.25}\smash{\begin{tabular}[t]{c}$\gamma$\end{tabular}}}}%
    \put(0.92256039,0.3152483){\makebox(0,0)[t]{\lineheight{1.25}\smash{\begin{tabular}[t]{c}$\gamma'$\end{tabular}}}}%
  \end{picture}%
\endgroup%

%% file: Figures/transversalizing2.pdf_tex
\begingroup%
  \makeatletter%
  \providecommand\color[2][]{%
    \errmessage{(Inkscape) Color is used for the text in Inkscape, but the package 'color.sty' is not loaded}%
    \renewcommand\color[2][]{}%
  }%
  \providecommand\transparent[1]{%
    \errmessage{(Inkscape) Transparency is used (non-zero) for the text in Inkscape, but the package 'transparent.sty' is not loaded}%
    \renewcommand\transparent[1]{}%
  }%
  \providecommand\rotatebox[2]{#2}%
  \newcommand*\fsize{\dimexpr\f@size pt\relax}%
  \newcommand*\lineheight[1]{\fontsize{\fsize}{#1\fsize}\selectfont}%
  \ifx\svgwidth\undefined%
    \setlength{\unitlength}{342.87445549bp}%
    \ifx\svgscale\undefined%
      \relax%
    \else%
      \setlength{\unitlength}{\unitlength * \real{\svgscale}}%
    \fi%
  \else%
    \setlength{\unitlength}{\svgwidth}%
  \fi%
  \global\let\svgwidth\undefined%
  \global\let\svgscale\undefined%
  \makeatother%
  \begin{picture}(1,0.66357476)%
    \lineheight{1}%
    \setlength\tabcolsep{0pt}%
    \put(0,0){\includegraphics[width=\unitlength,page=1]{transversalizing2.pdf}}%
    \put(-0.0268706,0.09366271){\makebox(0,0)[rt]{\lineheight{1.25}\smash{\begin{tabular}[t]{r}$V$\end{tabular}}}}%
    \put(0,0){\includegraphics[width=\unitlength,page=2]{transversalizing2.pdf}}%
    \put(-0.02302188,0.46835299){\makebox(0,0)[rt]{\lineheight{1.25}\smash{\begin{tabular}[t]{r}$W$\end{tabular}}}}%
  \end{picture}%
\endgroup%

%% file: Figures/instability.pdf_tex
\begingroup%
  \makeatletter%
  \providecommand\color[2][]{%
    \errmessage{(Inkscape) Color is used for the text in Inkscape, but the package 'color.sty' is not loaded}%
    \renewcommand\color[2][]{}%
  }%
  \providecommand\transparent[1]{%
    \errmessage{(Inkscape) Transparency is used (non-zero) for the text in Inkscape, but the package 'transparent.sty' is not loaded}%
    \renewcommand\transparent[1]{}%
  }%
  \providecommand\rotatebox[2]{#2}%
  \newcommand*\fsize{\dimexpr\f@size pt\relax}%
  \newcommand*\lineheight[1]{\fontsize{\fsize}{#1\fsize}\selectfont}%
  \ifx\svgwidth\undefined%
    \setlength{\unitlength}{317.06101785bp}%
    \ifx\svgscale\undefined%
      \relax%
    \else%
      \setlength{\unitlength}{\unitlength * \real{\svgscale}}%
    \fi%
  \else%
    \setlength{\unitlength}{\svgwidth}%
  \fi%
  \global\let\svgwidth\undefined%
  \global\let\svgscale\undefined%
  \makeatother%
  \begin{picture}(1,0.71759576)%
    \lineheight{1}%
    \setlength\tabcolsep{0pt}%
    \put(0,0){\includegraphics[width=\unitlength,page=1]{instability.pdf}}%
  \end{picture}%
\endgroup%

%% file: Figures/aleph_resolution.pdf_tex
\begingroup%
  \makeatletter%
  \providecommand\color[2][]{%
    \errmessage{(Inkscape) Color is used for the text in Inkscape, but the package 'color.sty' is not loaded}%
    \renewcommand\color[2][]{}%
  }%
  \providecommand\transparent[1]{%
    \errmessage{(Inkscape) Transparency is used (non-zero) for the text in Inkscape, but the package 'transparent.sty' is not loaded}%
    \renewcommand\transparent[1]{}%
  }%
  \providecommand\rotatebox[2]{#2}%
  \newcommand*\fsize{\dimexpr\f@size pt\relax}%
  \newcommand*\lineheight[1]{\fontsize{\fsize}{#1\fsize}\selectfont}%
  \ifx\svgwidth\undefined%
    \setlength{\unitlength}{340.15748031bp}%
    \ifx\svgscale\undefined%
      \relax%
    \else%
      \setlength{\unitlength}{\unitlength * \real{\svgscale}}%
    \fi%
  \else%
    \setlength{\unitlength}{\svgwidth}%
  \fi%
  \global\let\svgwidth\undefined%
  \global\let\svgscale\undefined%
  \makeatother%
  \begin{picture}(1,0.4166667)%
    \lineheight{1}%
    \setlength\tabcolsep{0pt}%
    \put(0,0){\includegraphics[width=\unitlength,page=1]{aleph_resolution.pdf}}%
  \end{picture}%
\endgroup%

%% file: Figures/perturbing_monotone.pdf_tex
\begingroup%
  \makeatletter%
  \providecommand\color[2][]{%
    \errmessage{(Inkscape) Color is used for the text in Inkscape, but the package 'color.sty' is not loaded}%
    \renewcommand\color[2][]{}%
  }%
  \providecommand\transparent[1]{%
    \errmessage{(Inkscape) Transparency is used (non-zero) for the text in Inkscape, but the package 'transparent.sty' is not loaded}%
    \renewcommand\transparent[1]{}%
  }%
  \providecommand\rotatebox[2]{#2}%
  \newcommand*\fsize{\dimexpr\f@size pt\relax}%
  \newcommand*\lineheight[1]{\fontsize{\fsize}{#1\fsize}\selectfont}%
  \ifx\svgwidth\undefined%
    \setlength{\unitlength}{376.08156219bp}%
    \ifx\svgscale\undefined%
      \relax%
    \else%
      \setlength{\unitlength}{\unitlength * \real{\svgscale}}%
    \fi%
  \else%
    \setlength{\unitlength}{\svgwidth}%
  \fi%
  \global\let\svgwidth\undefined%
  \global\let\svgscale\undefined%
  \makeatother%
  \begin{picture}(1,0.77950101)%
    \lineheight{1}%
    \setlength\tabcolsep{0pt}%
    \put(0,0){\includegraphics[width=\unitlength,page=1]{perturbing_monotone.pdf}}%
    \put(0.79957497,0.75124643){\makebox(0,0)[t]{\lineheight{1.25}\smash{\begin{tabular}[t]{c}$\gamma^\perp_{i+1}$\end{tabular}}}}%
    \put(0.34608559,0.00893427){\makebox(0,0)[t]{\lineheight{1.25}\smash{\begin{tabular}[t]{c}$\gamma^\perp_{i}$\end{tabular}}}}%
    \put(0.20237939,0.74914643){\makebox(0,0)[t]{\lineheight{1.25}\smash{\begin{tabular}[t]{c}$\gamma^\parallel_{i}$\end{tabular}}}}%
    \put(0,0){\includegraphics[width=\unitlength,page=2]{perturbing_monotone.pdf}}%
  \end{picture}%
\endgroup%

%% file: Figures/mostlytransversecurve.pdf_tex
\begingroup%
  \makeatletter%
  \providecommand\color[2][]{%
    \errmessage{(Inkscape) Color is used for the text in Inkscape, but the package 'color.sty' is not loaded}%
    \renewcommand\color[2][]{}%
  }%
  \providecommand\transparent[1]{%
    \errmessage{(Inkscape) Transparency is used (non-zero) for the text in Inkscape, but the package 'transparent.sty' is not loaded}%
    \renewcommand\transparent[1]{}%
  }%
  \providecommand\rotatebox[2]{#2}%
  \newcommand*\fsize{\dimexpr\f@size pt\relax}%
  \newcommand*\lineheight[1]{\fontsize{\fsize}{#1\fsize}\selectfont}%
  \ifx\svgwidth\undefined%
    \setlength{\unitlength}{341.55371719bp}%
    \ifx\svgscale\undefined%
      \relax%
    \else%
      \setlength{\unitlength}{\unitlength * \real{\svgscale}}%
    \fi%
  \else%
    \setlength{\unitlength}{\svgwidth}%
  \fi%
  \global\let\svgwidth\undefined%
  \global\let\svgscale\undefined%
  \makeatother%
  \begin{picture}(1,0.1766785)%
    \lineheight{1}%
    \setlength\tabcolsep{0pt}%
    \put(0,0){\includegraphics[width=\unitlength,page=1]{mostlytransversecurve.pdf}}%
    \put(0.9436039,0.00167726){\makebox(0,0)[lt]{\lineheight{1.25}\smash{\begin{tabular}[t]{l}$\gamma$\end{tabular}}}}%
    \put(0.94383976,0.08846622){\color[rgb]{0.78431373,0.21568627,0.21568627}\makebox(0,0)[lt]{\lineheight{1.25}\smash{\begin{tabular}[t]{l}$\gamma'$\end{tabular}}}}%
    \put(0,0){\includegraphics[width=\unitlength,page=2]{mostlytransversecurve.pdf}}%
  \end{picture}%
\endgroup%

%% file: Figures/transversalizing.pdf_tex
\begingroup%
  \makeatletter%
  \providecommand\color[2][]{%
    \errmessage{(Inkscape) Color is used for the text in Inkscape, but the package 'color.sty' is not loaded}%
    \renewcommand\color[2][]{}%
  }%
  \providecommand\transparent[1]{%
    \errmessage{(Inkscape) Transparency is used (non-zero) for the text in Inkscape, but the package 'transparent.sty' is not loaded}%
    \renewcommand\transparent[1]{}%
  }%
  \providecommand\rotatebox[2]{#2}%
  \newcommand*\fsize{\dimexpr\f@size pt\relax}%
  \newcommand*\lineheight[1]{\fontsize{\fsize}{#1\fsize}\selectfont}%
  \ifx\svgwidth\undefined%
    \setlength{\unitlength}{341.55371719bp}%
    \ifx\svgscale\undefined%
      \relax%
    \else%
      \setlength{\unitlength}{\unitlength * \real{\svgscale}}%
    \fi%
  \else%
    \setlength{\unitlength}{\svgwidth}%
  \fi%
  \global\let\svgwidth\undefined%
  \global\let\svgscale\undefined%
  \makeatother%
  \begin{picture}(1,0.1766785)%
    \lineheight{1}%
    \setlength\tabcolsep{0pt}%
    \put(0,0){\includegraphics[width=\unitlength,page=1]{transversalizing.pdf}}%
    \put(0.94975433,0.08427916){\color[rgb]{0.78431373,0.21568627,0.21568627}\makebox(0,0)[lt]{\lineheight{1.25}\smash{\begin{tabular}[t]{l}$\gamma'$\end{tabular}}}}%
    \put(0.95051098,0.00041535){\makebox(0,0)[lt]{\lineheight{1.25}\smash{\begin{tabular}[t]{l}$\gamma$\end{tabular}}}}%
  \end{picture}%
\endgroup%

%% file: Figures/aleph_filling.pdf_tex
\begingroup%
  \makeatletter%
  \providecommand\color[2][]{%
    \errmessage{(Inkscape) Color is used for the text in Inkscape, but the package 'color.sty' is not loaded}%
    \renewcommand\color[2][]{}%
  }%
  \providecommand\transparent[1]{%
    \errmessage{(Inkscape) Transparency is used (non-zero) for the text in Inkscape, but the package 'transparent.sty' is not loaded}%
    \renewcommand\transparent[1]{}%
  }%
  \providecommand\rotatebox[2]{#2}%
  \newcommand*\fsize{\dimexpr\f@size pt\relax}%
  \newcommand*\lineheight[1]{\fontsize{\fsize}{#1\fsize}\selectfont}%
  \ifx\svgwidth\undefined%
    \setlength{\unitlength}{551.69935812bp}%
    \ifx\svgscale\undefined%
      \relax%
    \else%
      \setlength{\unitlength}{\unitlength * \real{\svgscale}}%
    \fi%
  \else%
    \setlength{\unitlength}{\svgwidth}%
  \fi%
  \global\let\svgwidth\undefined%
  \global\let\svgscale\undefined%
  \makeatother%
  \begin{picture}(1,0.61052271)%
    \lineheight{1}%
    \setlength\tabcolsep{0pt}%
    \put(0,0){\includegraphics[width=\unitlength,page=1]{aleph_filling.pdf}}%
    \put(0.19191006,0.00400945){\makebox(0,0)[t]{\lineheight{1.25}\smash{\begin{tabular}[t]{c}$A_{4i-1}$\end{tabular}}}}%
    \put(0.34592095,0.00400945){\makebox(0,0)[t]{\lineheight{1.25}\smash{\begin{tabular}[t]{c}$A_{4i}$\end{tabular}}}}%
    \put(0.50062268,0.00400945){\makebox(0,0)[t]{\lineheight{1.25}\smash{\begin{tabular}[t]{c}$A_{4i+1}$\end{tabular}}}}%
    \put(0.91067583,0.3532453){\makebox(0,0)[lt]{\lineheight{1.25}\smash{\begin{tabular}[t]{l}$\dots$\end{tabular}}}}%
    \put(0.08917196,0.35278745){\makebox(0,0)[rt]{\lineheight{1.25}\smash{\begin{tabular}[t]{r}$\dots$\end{tabular}}}}%
    \put(0.65419836,0.00400945){\makebox(0,0)[t]{\lineheight{1.25}\smash{\begin{tabular}[t]{c}$A_{4i+2}$\end{tabular}}}}%
    \put(0.80842985,0.00400945){\makebox(0,0)[t]{\lineheight{1.25}\smash{\begin{tabular}[t]{c}$A_{4i+3}$\end{tabular}}}}%
  \end{picture}%
\endgroup%

%% file: Figures/aleph_filling_origami.pdf_tex
\begingroup%
  \makeatletter%
  \providecommand\color[2][]{%
    \errmessage{(Inkscape) Color is used for the text in Inkscape, but the package 'color.sty' is not loaded}%
    \renewcommand\color[2][]{}%
  }%
  \providecommand\transparent[1]{%
    \errmessage{(Inkscape) Transparency is used (non-zero) for the text in Inkscape, but the package 'transparent.sty' is not loaded}%
    \renewcommand\transparent[1]{}%
  }%
  \providecommand\rotatebox[2]{#2}%
  \newcommand*\fsize{\dimexpr\f@size pt\relax}%
  \newcommand*\lineheight[1]{\fontsize{\fsize}{#1\fsize}\selectfont}%
  \ifx\svgwidth\undefined%
    \setlength{\unitlength}{551.69935812bp}%
    \ifx\svgscale\undefined%
      \relax%
    \else%
      \setlength{\unitlength}{\unitlength * \real{\svgscale}}%
    \fi%
  \else%
    \setlength{\unitlength}{\svgwidth}%
  \fi%
  \global\let\svgwidth\undefined%
  \global\let\svgscale\undefined%
  \makeatother%
  \begin{picture}(1,0.61052271)%
    \lineheight{1}%
    \setlength\tabcolsep{0pt}%
    \put(0,0){\includegraphics[width=\unitlength,page=1]{aleph_filling_origami.pdf}}%
    \put(0.50006173,0.00400945){\makebox(0,0)[t]{\lineheight{1.25}\smash{\begin{tabular}[t]{c}$B_{2i}$\end{tabular}}}}%
    \put(0,0){\includegraphics[width=\unitlength,page=2]{aleph_filling_origami.pdf}}%
    \put(0.80834329,0.00400945){\makebox(0,0)[t]{\lineheight{1.25}\smash{\begin{tabular}[t]{c}$B_{2i+1}$\end{tabular}}}}%
    \put(0.91067583,0.3532453){\makebox(0,0)[lt]{\lineheight{1.25}\smash{\begin{tabular}[t]{l}$\dots$\end{tabular}}}}%
    \put(0.08917196,0.35278745){\makebox(0,0)[rt]{\lineheight{1.25}\smash{\begin{tabular}[t]{r}$\dots$\end{tabular}}}}%
    \put(0,0){\includegraphics[width=\unitlength,page=3]{aleph_filling_origami.pdf}}%
    \put(0.19178017,0.00400945){\makebox(0,0)[t]{\lineheight{1.25}\smash{\begin{tabular}[t]{c}$B_{2i-1}$\end{tabular}}}}%
    \put(0,0){\includegraphics[width=\unitlength,page=4]{aleph_filling_origami.pdf}}%
  \end{picture}%
\endgroup%

%% file: Figures/fortune_teller.pdf_tex
\begingroup%
  \makeatletter%
  \providecommand\color[2][]{%
    \errmessage{(Inkscape) Color is used for the text in Inkscape, but the package 'color.sty' is not loaded}%
    \renewcommand\color[2][]{}%
  }%
  \providecommand\transparent[1]{%
    \errmessage{(Inkscape) Transparency is used (non-zero) for the text in Inkscape, but the package 'transparent.sty' is not loaded}%
    \renewcommand\transparent[1]{}%
  }%
  \providecommand\rotatebox[2]{#2}%
  \newcommand*\fsize{\dimexpr\f@size pt\relax}%
  \newcommand*\lineheight[1]{\fontsize{\fsize}{#1\fsize}\selectfont}%
  \ifx\svgwidth\undefined%
    \setlength{\unitlength}{224.76540756bp}%
    \ifx\svgscale\undefined%
      \relax%
    \else%
      \setlength{\unitlength}{\unitlength * \real{\svgscale}}%
    \fi%
  \else%
    \setlength{\unitlength}{\svgwidth}%
  \fi%
  \global\let\svgwidth\undefined%
  \global\let\svgscale\undefined%
  \makeatother%
  \begin{picture}(1,1.29343858)%
    \lineheight{1}%
    \setlength\tabcolsep{0pt}%
    \put(0,0){\includegraphics[width=\unitlength,page=1]{fortune_teller.pdf}}%
  \end{picture}%
\endgroup%

%% file: Figures/branching_foliation.pdf_tex
\begingroup%
  \makeatletter%
  \providecommand\color[2][]{%
    \errmessage{(Inkscape) Color is used for the text in Inkscape, but the package 'color.sty' is not loaded}%
    \renewcommand\color[2][]{}%
  }%
  \providecommand\transparent[1]{%
    \errmessage{(Inkscape) Transparency is used (non-zero) for the text in Inkscape, but the package 'transparent.sty' is not loaded}%
    \renewcommand\transparent[1]{}%
  }%
  \providecommand\rotatebox[2]{#2}%
  \newcommand*\fsize{\dimexpr\f@size pt\relax}%
  \newcommand*\lineheight[1]{\fontsize{\fsize}{#1\fsize}\selectfont}%
  \ifx\svgwidth\undefined%
    \setlength{\unitlength}{168.00001009bp}%
    \ifx\svgscale\undefined%
      \relax%
    \else%
      \setlength{\unitlength}{\unitlength * \real{\svgscale}}%
    \fi%
  \else%
    \setlength{\unitlength}{\svgwidth}%
  \fi%
  \global\let\svgwidth\undefined%
  \global\let\svgscale\undefined%
  \makeatother%
  \begin{picture}(1,0.86160708)%
    \lineheight{1}%
    \setlength\tabcolsep{0pt}%
    \put(0,0){\includegraphics[width=\unitlength,page=1]{branching_foliation.pdf}}%
  \end{picture}%
\endgroup%

%% file: Figures/branching_foliation_shaved.pdf_tex
\begingroup%
  \makeatletter%
  \providecommand\color[2][]{%
    \errmessage{(Inkscape) Color is used for the text in Inkscape, but the package 'color.sty' is not loaded}%
    \renewcommand\color[2][]{}%
  }%
  \providecommand\transparent[1]{%
    \errmessage{(Inkscape) Transparency is used (non-zero) for the text in Inkscape, but the package 'transparent.sty' is not loaded}%
    \renewcommand\transparent[1]{}%
  }%
  \providecommand\rotatebox[2]{#2}%
  \newcommand*\fsize{\dimexpr\f@size pt\relax}%
  \newcommand*\lineheight[1]{\fontsize{\fsize}{#1\fsize}\selectfont}%
  \ifx\svgwidth\undefined%
    \setlength{\unitlength}{372.00001153bp}%
    \ifx\svgscale\undefined%
      \relax%
    \else%
      \setlength{\unitlength}{\unitlength * \real{\svgscale}}%
    \fi%
  \else%
    \setlength{\unitlength}{\svgwidth}%
  \fi%
  \global\let\svgwidth\undefined%
  \global\let\svgscale\undefined%
  \makeatother%
  \begin{picture}(1,0.38911289)%
    \lineheight{1}%
    \setlength\tabcolsep{0pt}%
    \put(0,0){\includegraphics[width=\unitlength,page=1]{branching_foliation_shaved.pdf}}%
  \end{picture}%
\endgroup%

%% file: Figures/t2planarminimalset.pdf_tex
\begingroup%
  \makeatletter%
  \providecommand\color[2][]{%
    \errmessage{(Inkscape) Color is used for the text in Inkscape, but the package 'color.sty' is not loaded}%
    \renewcommand\color[2][]{}%
  }%
  \providecommand\transparent[1]{%
    \errmessage{(Inkscape) Transparency is used (non-zero) for the text in Inkscape, but the package 'transparent.sty' is not loaded}%
    \renewcommand\transparent[1]{}%
  }%
  \providecommand\rotatebox[2]{#2}%
  \newcommand*\fsize{\dimexpr\f@size pt\relax}%
  \newcommand*\lineheight[1]{\fontsize{\fsize}{#1\fsize}\selectfont}%
  \ifx\svgwidth\undefined%
    \setlength{\unitlength}{1133.85826772bp}%
    \ifx\svgscale\undefined%
      \relax%
    \else%
      \setlength{\unitlength}{\unitlength * \real{\svgscale}}%
    \fi%
  \else%
    \setlength{\unitlength}{\svgwidth}%
  \fi%
  \global\let\svgwidth\undefined%
  \global\let\svgscale\undefined%
  \makeatother%
  \begin{picture}(1,1.10625)%
    \lineheight{1}%
    \setlength\tabcolsep{0pt}%
    \put(0,0){\includegraphics[width=\unitlength,page=1]{t2planarminimalset.pdf}}%
  \end{picture}%
\endgroup%

%% file: Figures/fences_and_windows.pdf_tex
\begingroup%
  \makeatletter%
  \providecommand\color[2][]{%
    \errmessage{(Inkscape) Color is used for the text in Inkscape, but the package 'color.sty' is not loaded}%
    \renewcommand\color[2][]{}%
  }%
  \providecommand\transparent[1]{%
    \errmessage{(Inkscape) Transparency is used (non-zero) for the text in Inkscape, but the package 'transparent.sty' is not loaded}%
    \renewcommand\transparent[1]{}%
  }%
  \providecommand\rotatebox[2]{#2}%
  \newcommand*\fsize{\dimexpr\f@size pt\relax}%
  \newcommand*\lineheight[1]{\fontsize{\fsize}{#1\fsize}\selectfont}%
  \ifx\svgwidth\undefined%
    \setlength{\unitlength}{995.30492587bp}%
    \ifx\svgscale\undefined%
      \relax%
    \else%
      \setlength{\unitlength}{\unitlength * \real{\svgscale}}%
    \fi%
  \else%
    \setlength{\unitlength}{\svgwidth}%
  \fi%
  \global\let\svgwidth\undefined%
  \global\let\svgscale\undefined%
  \makeatother%
  \begin{picture}(1,0.65600974)%
    \lineheight{1}%
    \setlength\tabcolsep{0pt}%
    \put(0,0){\includegraphics[width=\unitlength,page=1]{fences_and_windows.pdf}}%
    \put(0.15192474,0.17800754){\makebox(0,0)[t]{\lineheight{1.25}\smash{\begin{tabular}[t]{c}$\mathcal A$\end{tabular}}}}%
    \put(0.80772004,0.48312462){\color[rgb]{0.21568627,0.44313725,0.78431373}\makebox(0,0)[lt]{\lineheight{1.25}\smash{\begin{tabular}[t]{l}$\mathcal W_1$\end{tabular}}}}%
    \put(0.89095571,0.17266655){\color[rgb]{0.21568627,0.44313725,0.78431373}\makebox(0,0)[lt]{\lineheight{1.25}\smash{\begin{tabular}[t]{l}$\mathcal W_3$\end{tabular}}}}%
    \put(0.89534838,0.06580854){\color[rgb]{0.21568627,0.44313725,0.78431373}\makebox(0,0)[lt]{\lineheight{1.25}\smash{\begin{tabular}[t]{l}$\mathcal W_4$\end{tabular}}}}%
    \put(0.57845035,0.33199077){\color[rgb]{0.21568627,0.44313725,0.78431373}\makebox(0,0)[lt]{\lineheight{1.25}\smash{\begin{tabular}[t]{l}$\mathcal W_2$\end{tabular}}}}%
    \put(0.32926013,0.38413818){\makebox(0,0)[t]{\lineheight{1.25}\smash{\begin{tabular}[t]{c}$R_1$\end{tabular}}}}%
    \put(0.61490519,0.43338529){\makebox(0,0)[t]{\lineheight{1.25}\smash{\begin{tabular}[t]{c}$R_2$\end{tabular}}}}%
    \put(0.69274299,0.27168165){\makebox(0,0)[t]{\lineheight{1.25}\smash{\begin{tabular}[t]{c}$R_3$\end{tabular}}}}%
    \put(0.65317699,0.15680121){\makebox(0,0)[t]{\lineheight{1.25}\smash{\begin{tabular}[t]{c}$R_4$\end{tabular}}}}%
    \put(0.23891769,0.05357162){\makebox(0,0)[t]{\lineheight{1.25}\smash{\begin{tabular}[t]{c}$R_5$\end{tabular}}}}%
    \put(0.69609941,0.01022479){\makebox(0,0)[t]{\lineheight{1.25}\smash{\begin{tabular}[t]{c}$R_6$\end{tabular}}}}%
    \put(0,0){\includegraphics[width=\unitlength,page=2]{fences_and_windows.pdf}}%
  \end{picture}%
\endgroup%

%% file: Figures/twisting_ribbon.pdf_tex
\begingroup%
  \makeatletter%
  \providecommand\color[2][]{%
    \errmessage{(Inkscape) Color is used for the text in Inkscape, but the package 'color.sty' is not loaded}%
    \renewcommand\color[2][]{}%
  }%
  \providecommand\transparent[1]{%
    \errmessage{(Inkscape) Transparency is used (non-zero) for the text in Inkscape, but the package 'transparent.sty' is not loaded}%
    \renewcommand\transparent[1]{}%
  }%
  \providecommand\rotatebox[2]{#2}%
  \newcommand*\fsize{\dimexpr\f@size pt\relax}%
  \newcommand*\lineheight[1]{\fontsize{\fsize}{#1\fsize}\selectfont}%
  \ifx\svgwidth\undefined%
    \setlength{\unitlength}{964.15513887bp}%
    \ifx\svgscale\undefined%
      \relax%
    \else%
      \setlength{\unitlength}{\unitlength * \real{\svgscale}}%
    \fi%
  \else%
    \setlength{\unitlength}{\svgwidth}%
  \fi%
  \global\let\svgwidth\undefined%
  \global\let\svgscale\undefined%
  \makeatother%
  \begin{picture}(1,0.62505502)%
    \lineheight{1}%
    \setlength\tabcolsep{0pt}%
    \put(0,0){\includegraphics[width=\unitlength,page=1]{twisting_ribbon.pdf}}%
    \put(0.44380103,0.27593045){\makebox(0,0)[t]{\lineheight{1.25}\smash{\begin{tabular}[t]{c}$N(R_m)$\end{tabular}}}}%
    \put(0,0){\includegraphics[width=\unitlength,page=2]{twisting_ribbon.pdf}}%
    \put(0.11574912,0.12000856){\color[rgb]{0.21568627,0.44313725,0.78431373}\makebox(0,0)[t]{\lineheight{1.25}\smash{\begin{tabular}[t]{c}$\mathcal W_i$\end{tabular}}}}%
    \put(0.83030542,0.00348494){\color[rgb]{0,0,0}\makebox(0,0)[t]{\lineheight{1.25}\smash{\begin{tabular}[t]{c}$\partial_*M$\end{tabular}}}}%
    \put(0,0){\includegraphics[width=\unitlength,page=3]{twisting_ribbon.pdf}}%
    \put(0.30084242,0.14354192){\makebox(0,0)[lt]{\lineheight{1.25}\smash{\begin{tabular}[t]{l}$\{1\}\times \mathsf{R}_\delta$\end{tabular}}}}%
    \put(0.58185876,0.07067993){\makebox(0,0)[rt]{\lineheight{1.25}\smash{\begin{tabular}[t]{r}$\{0\}\times \mathsf{R}_\delta$\end{tabular}}}}%
    \put(0,0){\includegraphics[width=\unitlength,page=4]{twisting_ribbon.pdf}}%
  \end{picture}%
\endgroup%

%% file: Figures/surface_parts.pdf_tex
\begingroup%
  \makeatletter%
  \providecommand\color[2][]{%
    \errmessage{(Inkscape) Color is used for the text in Inkscape, but the package 'color.sty' is not loaded}%
    \renewcommand\color[2][]{}%
  }%
  \providecommand\transparent[1]{%
    \errmessage{(Inkscape) Transparency is used (non-zero) for the text in Inkscape, but the package 'transparent.sty' is not loaded}%
    \renewcommand\transparent[1]{}%
  }%
  \providecommand\rotatebox[2]{#2}%
  \newcommand*\fsize{\dimexpr\f@size pt\relax}%
  \newcommand*\lineheight[1]{\fontsize{\fsize}{#1\fsize}\selectfont}%
  \ifx\svgwidth\undefined%
    \setlength{\unitlength}{547.21948819bp}%
    \ifx\svgscale\undefined%
      \relax%
    \else%
      \setlength{\unitlength}{\unitlength * \real{\svgscale}}%
    \fi%
  \else%
    \setlength{\unitlength}{\svgwidth}%
  \fi%
  \global\let\svgwidth\undefined%
  \global\let\svgscale\undefined%
  \makeatother%
  \begin{picture}(1,0.66862835)%
    \lineheight{1}%
    \setlength\tabcolsep{0pt}%
    \put(0,0){\includegraphics[width=\unitlength,page=1]{surface_parts.pdf}}%
    \put(0.49951715,0.36827365){\color[rgb]{0.78431373,0.21568627,0.21568627}\makebox(0,0)[t]{\lineheight{1.25}\smash{\begin{tabular}[t]{c}$\gamma'$\end{tabular}}}}%
    \put(0.66415069,0.28995049){\makebox(0,0)[t]{\lineheight{1.25}\smash{\begin{tabular}[t]{c}$P$\end{tabular}}}}%
    \put(0,0){\includegraphics[width=\unitlength,page=2]{surface_parts.pdf}}%
    \put(0.89760821,0.10429616){\makebox(0,0)[t]{\lineheight{1.25}\smash{\begin{tabular}[t]{c}$\Sigma'$\end{tabular}}}}%
    \put(0.67786363,0.42108746){\color[rgb]{0.78431373,0.21568627,0.21568627}\makebox(0,0)[rt]{\lineheight{1.25}\smash{\begin{tabular}[t]{r}$\beta$\end{tabular}}}}%
    \put(0.81913559,0.2722806){\color[rgb]{0.78431373,0.21568627,0.21568627}\makebox(0,0)[lt]{\lineheight{1.25}\smash{\begin{tabular}[t]{l}$\beta'$\end{tabular}}}}%
    \put(0.79521348,0.61373467){\makebox(0,0)[t]{\lineheight{1.25}\smash{\begin{tabular}[t]{c}$\Sigma$\end{tabular}}}}%
    \put(0.14053587,0.53269911){\makebox(0,0)[t]{\lineheight{1.25}\smash{\begin{tabular}[t]{c}$\beta_1$\end{tabular}}}}%
    \put(0.05767901,0.3714621){\makebox(0,0)[t]{\lineheight{1.25}\smash{\begin{tabular}[t]{c}$\beta_2$\end{tabular}}}}%
    \put(0.14893505,0.07492878){\makebox(0,0)[t]{\lineheight{1.25}\smash{\begin{tabular}[t]{c}$\beta_m$\end{tabular}}}}%
  \end{picture}%
\endgroup%

%% file: Figures/commutator.pdf_tex
\begingroup%
  \makeatletter%
  \providecommand\color[2][]{%
    \errmessage{(Inkscape) Color is used for the text in Inkscape, but the package 'color.sty' is not loaded}%
    \renewcommand\color[2][]{}%
  }%
  \providecommand\transparent[1]{%
    \errmessage{(Inkscape) Transparency is used (non-zero) for the text in Inkscape, but the package 'transparent.sty' is not loaded}%
    \renewcommand\transparent[1]{}%
  }%
  \providecommand\rotatebox[2]{#2}%
  \newcommand*\fsize{\dimexpr\f@size pt\relax}%
  \newcommand*\lineheight[1]{\fontsize{\fsize}{#1\fsize}\selectfont}%
  \ifx\svgwidth\undefined%
    \setlength{\unitlength}{485.78467662bp}%
    \ifx\svgscale\undefined%
      \relax%
    \else%
      \setlength{\unitlength}{\unitlength * \real{\svgscale}}%
    \fi%
  \else%
    \setlength{\unitlength}{\svgwidth}%
  \fi%
  \global\let\svgwidth\undefined%
  \global\let\svgscale\undefined%
  \makeatother%
  \begin{picture}(1,1.34800322)%
    \lineheight{1}%
    \setlength\tabcolsep{0pt}%
    \put(0,0){\includegraphics[width=\unitlength,page=1]{commutator.pdf}}%
    \put(0.16958092,0.00423565){\makebox(0,0)[lt]{\lineheight{1.25}\smash{\begin{tabular}[t]{l}$f$\end{tabular}}}}%
    \put(0.36852953,0.10135822){\makebox(0,0)[lt]{\lineheight{1.25}\smash{\begin{tabular}[t]{l}$g$\end{tabular}}}}%
    \put(0.59029385,0.0042563){\makebox(0,0)[t]{\lineheight{1.25}\smash{\begin{tabular}[t]{c}$\overline{f}$\end{tabular}}}}%
    \put(0.78555535,0.1031656){\makebox(0,0)[t]{\lineheight{1.25}\smash{\begin{tabular}[t]{c}$\overline{g}$\end{tabular}}}}%
    \put(-0.00226418,0.24942197){\makebox(0,0)[lt]{\lineheight{1.25}\smash{\begin{tabular}[t]{l}$x_1$\end{tabular}}}}%
    \put(-0.00163615,0.45032883){\makebox(0,0)[lt]{\lineheight{1.25}\smash{\begin{tabular}[t]{l}$x_2$\end{tabular}}}}%
    \put(-0.00239224,0.64609292){\makebox(0,0)[lt]{\lineheight{1.25}\smash{\begin{tabular}[t]{l}$x_3$\end{tabular}}}}%
    \put(0.00408199,0.84705998){\makebox(0,0)[lt]{\lineheight{1.25}\smash{\begin{tabular}[t]{l}$x_4$\end{tabular}}}}%
    \put(-0.00034202,1.04569228){\makebox(0,0)[lt]{\lineheight{1.25}\smash{\begin{tabular}[t]{l}$x_5$\end{tabular}}}}%
    \put(0.90878523,0.35263117){\makebox(0,0)[lt]{\lineheight{1.25}\smash{\begin{tabular}[t]{l}$y_1$\end{tabular}}}}%
    \put(0.90840227,0.54746966){\makebox(0,0)[lt]{\lineheight{1.25}\smash{\begin{tabular}[t]{l}$y_2$\end{tabular}}}}%
    \put(0.91258447,0.74768853){\makebox(0,0)[lt]{\lineheight{1.25}\smash{\begin{tabular}[t]{l}$y_3$\end{tabular}}}}%
    \put(0.91295122,0.94554848){\makebox(0,0)[lt]{\lineheight{1.25}\smash{\begin{tabular}[t]{l}$y_4$\end{tabular}}}}%
    \put(0.91591992,1.13969736){\makebox(0,0)[lt]{\lineheight{1.25}\smash{\begin{tabular}[t]{l}$y_5$\end{tabular}}}}%
    \put(0,0){\includegraphics[width=\unitlength,page=2]{commutator.pdf}}%
  \end{picture}%
\endgroup%

%% file: Figures/fence.pdf_tex
\begingroup%
  \makeatletter%
  \providecommand\color[2][]{%
    \errmessage{(Inkscape) Color is used for the text in Inkscape, but the package 'color.sty' is not loaded}%
    \renewcommand\color[2][]{}%
  }%
  \providecommand\transparent[1]{%
    \errmessage{(Inkscape) Transparency is used (non-zero) for the text in Inkscape, but the package 'transparent.sty' is not loaded}%
    \renewcommand\transparent[1]{}%
  }%
  \providecommand\rotatebox[2]{#2}%
  \newcommand*\fsize{\dimexpr\f@size pt\relax}%
  \newcommand*\lineheight[1]{\fontsize{\fsize}{#1\fsize}\selectfont}%
  \ifx\svgwidth\undefined%
    \setlength{\unitlength}{455.84965743bp}%
    \ifx\svgscale\undefined%
      \relax%
    \else%
      \setlength{\unitlength}{\unitlength * \real{\svgscale}}%
    \fi%
  \else%
    \setlength{\unitlength}{\svgwidth}%
  \fi%
  \global\let\svgwidth\undefined%
  \global\let\svgscale\undefined%
  \makeatother%
  \begin{picture}(1,0.79597711)%
    \lineheight{1}%
    \setlength\tabcolsep{0pt}%
    \put(0,0){\includegraphics[width=\unitlength,page=1]{fence.pdf}}%
    \put(0.87353825,0.11220289){\makebox(0,0)[lt]{\lineheight{1.25}\smash{\begin{tabular}[t]{l}$R_1$\end{tabular}}}}%
    \put(0.00949263,0.33006309){\makebox(0,0)[lt]{\lineheight{1.25}\smash{\begin{tabular}[t]{l}$R_1$\end{tabular}}}}%
    \put(0.05062461,0.04542986){\makebox(0,0)[lt]{\lineheight{1.25}\smash{\begin{tabular}[t]{l}$S_1$\end{tabular}}}}%
    \put(0.59685737,0.50939854){\makebox(0,0)[lt]{\lineheight{1.25}\smash{\begin{tabular}[t]{l}$S_1$\end{tabular}}}}%
    \put(0.33046979,0.77593378){\makebox(0,0)[lt]{\lineheight{1.25}\smash{\begin{tabular}[t]{l}$P$\end{tabular}}}}%
  \end{picture}%
\endgroup%

%% file: Figures/slat_operation.pdf_tex
\begingroup%
  \makeatletter%
  \providecommand\color[2][]{%
    \errmessage{(Inkscape) Color is used for the text in Inkscape, but the package 'color.sty' is not loaded}%
    \renewcommand\color[2][]{}%
  }%
  \providecommand\transparent[1]{%
    \errmessage{(Inkscape) Transparency is used (non-zero) for the text in Inkscape, but the package 'transparent.sty' is not loaded}%
    \renewcommand\transparent[1]{}%
  }%
  \providecommand\rotatebox[2]{#2}%
  \newcommand*\fsize{\dimexpr\f@size pt\relax}%
  \newcommand*\lineheight[1]{\fontsize{\fsize}{#1\fsize}\selectfont}%
  \ifx\svgwidth\undefined%
    \setlength{\unitlength}{363.33664301bp}%
    \ifx\svgscale\undefined%
      \relax%
    \else%
      \setlength{\unitlength}{\unitlength * \real{\svgscale}}%
    \fi%
  \else%
    \setlength{\unitlength}{\svgwidth}%
  \fi%
  \global\let\svgwidth\undefined%
  \global\let\svgscale\undefined%
  \makeatother%
  \begin{picture}(1,0.52837088)%
    \lineheight{1}%
    \setlength\tabcolsep{0pt}%
    \put(0,0){\includegraphics[width=\unitlength,page=1]{slat_operation.pdf}}%
    \put(1.01480216,0.21940008){\makebox(0,0)[lt]{\lineheight{1.25}\smash{\begin{tabular}[t]{l}$R_2\times[0,\varepsilon]$\end{tabular}}}}%
    \put(1.01480216,0.31968865){\makebox(0,0)[lt]{\lineheight{1.25}\smash{\begin{tabular}[t]{l}$R_3\times[0,\varepsilon]$\end{tabular}}}}%
    \put(1.01480216,0.4051782){\makebox(0,0)[lt]{\lineheight{1.25}\smash{\begin{tabular}[t]{l}$R_4\times[0,\varepsilon]$\end{tabular}}}}%
    \put(1.01480216,0.12651102){\makebox(0,0)[lt]{\lineheight{1.25}\smash{\begin{tabular}[t]{l}$R_1\times[0,\varepsilon]$\end{tabular}}}}%
    \put(0,0){\includegraphics[width=\unitlength,page=2]{slat_operation.pdf}}%
  \end{picture}%
\endgroup%

%% file: Figures/wiggled_fence.pdf_tex
\begingroup%
  \makeatletter%
  \providecommand\color[2][]{%
    \errmessage{(Inkscape) Color is used for the text in Inkscape, but the package 'color.sty' is not loaded}%
    \renewcommand\color[2][]{}%
  }%
  \providecommand\transparent[1]{%
    \errmessage{(Inkscape) Transparency is used (non-zero) for the text in Inkscape, but the package 'transparent.sty' is not loaded}%
    \renewcommand\transparent[1]{}%
  }%
  \providecommand\rotatebox[2]{#2}%
  \newcommand*\fsize{\dimexpr\f@size pt\relax}%
  \newcommand*\lineheight[1]{\fontsize{\fsize}{#1\fsize}\selectfont}%
  \ifx\svgwidth\undefined%
    \setlength{\unitlength}{615.74999712bp}%
    \ifx\svgscale\undefined%
      \relax%
    \else%
      \setlength{\unitlength}{\unitlength * \real{\svgscale}}%
    \fi%
  \else%
    \setlength{\unitlength}{\svgwidth}%
  \fi%
  \global\let\svgwidth\undefined%
  \global\let\svgscale\undefined%
  \makeatother%
  \begin{picture}(1,0.61926692)%
    \lineheight{1}%
    \setlength\tabcolsep{0pt}%
    \put(0,0){\includegraphics[width=\unitlength,page=1]{wiggled_fence.pdf}}%
  \end{picture}%
\endgroup%

%% file: Figures/maypoleweaving.pdf_tex
\begingroup%
  \makeatletter%
  \providecommand\color[2][]{%
    \errmessage{(Inkscape) Color is used for the text in Inkscape, but the package 'color.sty' is not loaded}%
    \renewcommand\color[2][]{}%
  }%
  \providecommand\transparent[1]{%
    \errmessage{(Inkscape) Transparency is used (non-zero) for the text in Inkscape, but the package 'transparent.sty' is not loaded}%
    \renewcommand\transparent[1]{}%
  }%
  \providecommand\rotatebox[2]{#2}%
  \newcommand*\fsize{\dimexpr\f@size pt\relax}%
  \newcommand*\lineheight[1]{\fontsize{\fsize}{#1\fsize}\selectfont}%
  \ifx\svgwidth\undefined%
    \setlength{\unitlength}{732.53228616bp}%
    \ifx\svgscale\undefined%
      \relax%
    \else%
      \setlength{\unitlength}{\unitlength * \real{\svgscale}}%
    \fi%
  \else%
    \setlength{\unitlength}{\svgwidth}%
  \fi%
  \global\let\svgwidth\undefined%
  \global\let\svgscale\undefined%
  \makeatother%
  \begin{picture}(1,0.92871669)%
    \lineheight{1}%
    \setlength\tabcolsep{0pt}%
    \put(0,0){\includegraphics[width=\unitlength,page=1]{maypoleweaving.pdf}}%
    \put(0.89926149,0.76295179){\makebox(0,0)[lt]{\lineheight{1.25}\smash{\begin{tabular}[t]{l}$P$\end{tabular}}}}%
    \put(0.89926149,0.64952834){\makebox(0,0)[lt]{\lineheight{1.25}\smash{\begin{tabular}[t]{l}$\Lambda$\end{tabular}}}}%
    \put(0.89931085,0.33996184){\makebox(0,0)[lt]{\lineheight{1.25}\smash{\begin{tabular}[t]{l}$N$\end{tabular}}}}%
  \end{picture}%
\endgroup%

%% file: Figures/slice.pdf_tex
\begingroup%
  \makeatletter%
  \providecommand\color[2][]{%
    \errmessage{(Inkscape) Color is used for the text in Inkscape, but the package 'color.sty' is not loaded}%
    \renewcommand\color[2][]{}%
  }%
  \providecommand\transparent[1]{%
    \errmessage{(Inkscape) Transparency is used (non-zero) for the text in Inkscape, but the package 'transparent.sty' is not loaded}%
    \renewcommand\transparent[1]{}%
  }%
  \providecommand\rotatebox[2]{#2}%
  \newcommand*\fsize{\dimexpr\f@size pt\relax}%
  \newcommand*\lineheight[1]{\fontsize{\fsize}{#1\fsize}\selectfont}%
  \ifx\svgwidth\undefined%
    \setlength{\unitlength}{375bp}%
    \ifx\svgscale\undefined%
      \relax%
    \else%
      \setlength{\unitlength}{\unitlength * \real{\svgscale}}%
    \fi%
  \else%
    \setlength{\unitlength}{\svgwidth}%
  \fi%
  \global\let\svgwidth\undefined%
  \global\let\svgscale\undefined%
  \makeatother%
  \begin{picture}(1,1)%
    \lineheight{1}%
    \setlength\tabcolsep{0pt}%
    \put(0,0){\includegraphics[width=\unitlength,page=1]{slice.pdf}}%
    \put(0.05035553,0.76387051){\makebox(0,0)[rt]{\lineheight{1.25}\smash{\begin{tabular}[t]{r}$y=t_2$\end{tabular}}}}%
    \put(0.05035553,0.5304826){\makebox(0,0)[rt]{\lineheight{1.25}\smash{\begin{tabular}[t]{r}$y=t_1$\end{tabular}}}}%
  \end{picture}%
\endgroup%

%% file: Figures/birthdeath.pdf_tex
\begingroup%
  \makeatletter%
  \providecommand\color[2][]{%
    \errmessage{(Inkscape) Color is used for the text in Inkscape, but the package 'color.sty' is not loaded}%
    \renewcommand\color[2][]{}%
  }%
  \providecommand\transparent[1]{%
    \errmessage{(Inkscape) Transparency is used (non-zero) for the text in Inkscape, but the package 'transparent.sty' is not loaded}%
    \renewcommand\transparent[1]{}%
  }%
  \providecommand\rotatebox[2]{#2}%
  \newcommand*\fsize{\dimexpr\f@size pt\relax}%
  \newcommand*\lineheight[1]{\fontsize{\fsize}{#1\fsize}\selectfont}%
  \ifx\svgwidth\undefined%
    \setlength{\unitlength}{850.83938238bp}%
    \ifx\svgscale\undefined%
      \relax%
    \else%
      \setlength{\unitlength}{\unitlength * \real{\svgscale}}%
    \fi%
  \else%
    \setlength{\unitlength}{\svgwidth}%
  \fi%
  \global\let\svgwidth\undefined%
  \global\let\svgscale\undefined%
  \makeatother%
  \begin{picture}(1,0.30435943)%
    \lineheight{1}%
    \setlength\tabcolsep{0pt}%
    \put(0,0){\includegraphics[width=\unitlength,page=1]{birthdeath.pdf}}%
    \put(0.23264505,0.14880748){\makebox(0,0)[t]{\lineheight{1.25}\smash{\begin{tabular}[t]{c}$\to$\end{tabular}}}}%
    \put(0.49913063,0.14940533){\makebox(0,0)[t]{\lineheight{1.25}\smash{\begin{tabular}[t]{c}$\to$\end{tabular}}}}%
    \put(0.7641397,0.14832099){\makebox(0,0)[t]{\lineheight{1.25}\smash{\begin{tabular}[t]{c}$\to$\end{tabular}}}}%
    \put(0,0){\includegraphics[width=\unitlength,page=2]{birthdeath.pdf}}%
  \end{picture}%
\endgroup%

%% file: Figures/edge_expansion.pdf_tex
\begingroup%
  \makeatletter%
  \providecommand\color[2][]{%
    \errmessage{(Inkscape) Color is used for the text in Inkscape, but the package 'color.sty' is not loaded}%
    \renewcommand\color[2][]{}%
  }%
  \providecommand\transparent[1]{%
    \errmessage{(Inkscape) Transparency is used (non-zero) for the text in Inkscape, but the package 'transparent.sty' is not loaded}%
    \renewcommand\transparent[1]{}%
  }%
  \providecommand\rotatebox[2]{#2}%
  \newcommand*\fsize{\dimexpr\f@size pt\relax}%
  \newcommand*\lineheight[1]{\fontsize{\fsize}{#1\fsize}\selectfont}%
  \ifx\svgwidth\undefined%
    \setlength{\unitlength}{709.373539bp}%
    \ifx\svgscale\undefined%
      \relax%
    \else%
      \setlength{\unitlength}{\unitlength * \real{\svgscale}}%
    \fi%
  \else%
    \setlength{\unitlength}{\svgwidth}%
  \fi%
  \global\let\svgwidth\undefined%
  \global\let\svgscale\undefined%
  \makeatother%
  \begin{picture}(1,0.40014087)%
    \lineheight{1}%
    \setlength\tabcolsep{0pt}%
    \put(0.49867382,0.19498591){\makebox(0,0)[t]{\lineheight{1.25}\smash{\begin{tabular}[t]{c}$\xsquigright{\textrm{birth}}$\end{tabular}}}}%
    \put(0,0){\includegraphics[width=\unitlength,page=1]{edge_expansion.pdf}}%
  \end{picture}%
\endgroup%

%% file: Figures/edge_expansion_graph.pdf_tex
\begingroup%
  \makeatletter%
  \providecommand\color[2][]{%
    \errmessage{(Inkscape) Color is used for the text in Inkscape, but the package 'color.sty' is not loaded}%
    \renewcommand\color[2][]{}%
  }%
  \providecommand\transparent[1]{%
    \errmessage{(Inkscape) Transparency is used (non-zero) for the text in Inkscape, but the package 'transparent.sty' is not loaded}%
    \renewcommand\transparent[1]{}%
  }%
  \providecommand\rotatebox[2]{#2}%
  \newcommand*\fsize{\dimexpr\f@size pt\relax}%
  \newcommand*\lineheight[1]{\fontsize{\fsize}{#1\fsize}\selectfont}%
  \ifx\svgwidth\undefined%
    \setlength{\unitlength}{284.13538174bp}%
    \ifx\svgscale\undefined%
      \relax%
    \else%
      \setlength{\unitlength}{\unitlength * \real{\svgscale}}%
    \fi%
  \else%
    \setlength{\unitlength}{\svgwidth}%
  \fi%
  \global\let\svgwidth\undefined%
  \global\let\svgscale\undefined%
  \makeatother%
  \begin{picture}(1,0.59976391)%
    \lineheight{1}%
    \setlength\tabcolsep{0pt}%
    \put(0.4974283,0.29071439){\makebox(0,0)[t]{\lineheight{1.25}\smash{\begin{tabular}[t]{c}$\xsquigright{\textrm{edge expansion}}$\end{tabular}}}}%
    \put(0,0){\includegraphics[width=\unitlength,page=1]{edge_expansion_graph.pdf}}%
  \end{picture}%
\endgroup%

%% file: Figures/forbidden_birth.pdf_tex
\begingroup%
  \makeatletter%
  \providecommand\color[2][]{%
    \errmessage{(Inkscape) Color is used for the text in Inkscape, but the package 'color.sty' is not loaded}%
    \renewcommand\color[2][]{}%
  }%
  \providecommand\transparent[1]{%
    \errmessage{(Inkscape) Transparency is used (non-zero) for the text in Inkscape, but the package 'transparent.sty' is not loaded}%
    \renewcommand\transparent[1]{}%
  }%
  \providecommand\rotatebox[2]{#2}%
  \newcommand*\fsize{\dimexpr\f@size pt\relax}%
  \newcommand*\lineheight[1]{\fontsize{\fsize}{#1\fsize}\selectfont}%
  \ifx\svgwidth\undefined%
    \setlength{\unitlength}{409.10183644bp}%
    \ifx\svgscale\undefined%
      \relax%
    \else%
      \setlength{\unitlength}{\unitlength * \real{\svgscale}}%
    \fi%
  \else%
    \setlength{\unitlength}{\svgwidth}%
  \fi%
  \global\let\svgwidth\undefined%
  \global\let\svgscale\undefined%
  \makeatother%
  \begin{picture}(1,0.85932128)%
    \lineheight{1}%
    \setlength\tabcolsep{0pt}%
    \put(0,0){\includegraphics[width=\unitlength,page=1]{forbidden_birth.pdf}}%
    \put(0.43172518,0.41016787){\makebox(0,0)[t]{\lineheight{1.25}\smash{\begin{tabular}[t]{c}$\textrm{\Huge ?}$\end{tabular}}}}%
    \put(0,0){\includegraphics[width=\unitlength,page=2]{forbidden_birth.pdf}}%
    \put(0.92735118,0.67584765){\makebox(0,0)[lt]{\lineheight{1.25}\smash{\begin{tabular}[t]{l}$\gamma_-$\end{tabular}}}}%
    \put(0.92526581,0.55666763){\makebox(0,0)[lt]{\lineheight{1.25}\smash{\begin{tabular}[t]{l}$\gamma_+$\end{tabular}}}}%
    \put(0,0){\includegraphics[width=\unitlength,page=3]{forbidden_birth.pdf}}%
  \end{picture}%
\endgroup%

%% file: Figures/dividing_set_S2.pdf_tex
\begingroup%
  \makeatletter%
  \providecommand\color[2][]{%
    \errmessage{(Inkscape) Color is used for the text in Inkscape, but the package 'color.sty' is not loaded}%
    \renewcommand\color[2][]{}%
  }%
  \providecommand\transparent[1]{%
    \errmessage{(Inkscape) Transparency is used (non-zero) for the text in Inkscape, but the package 'transparent.sty' is not loaded}%
    \renewcommand\transparent[1]{}%
  }%
  \providecommand\rotatebox[2]{#2}%
  \newcommand*\fsize{\dimexpr\f@size pt\relax}%
  \newcommand*\lineheight[1]{\fontsize{\fsize}{#1\fsize}\selectfont}%
  \ifx\svgwidth\undefined%
    \setlength{\unitlength}{1218.29338975bp}%
    \ifx\svgscale\undefined%
      \relax%
    \else%
      \setlength{\unitlength}{\unitlength * \real{\svgscale}}%
    \fi%
  \else%
    \setlength{\unitlength}{\svgwidth}%
  \fi%
  \global\let\svgwidth\undefined%
  \global\let\svgscale\undefined%
  \makeatother%
  \begin{picture}(1,0.27051236)%
    \lineheight{1}%
    \setlength\tabcolsep{0pt}%
    \put(0,0){\includegraphics[width=\unitlength,page=1]{dividing_set_S2.pdf}}%
    \put(0.30394444,0.12626908){\makebox(0,0)[t]{\lineheight{1.25}\smash{\begin{tabular}[t]{c}$\xrightarrow{\textrm{birth}}$\end{tabular}}}}%
    \put(0.68785569,0.12626908){\makebox(0,0)[t]{\lineheight{1.25}\smash{\begin{tabular}[t]{c}$\xrightarrow{\textrm{death}}$\end{tabular}}}}%
    \put(0,0){\includegraphics[width=\unitlength,page=2]{dividing_set_S2.pdf}}%
  \end{picture}%
\endgroup%

%% file: Figures/dividing_set_S2_graph.pdf_tex
\begingroup%
  \makeatletter%
  \providecommand\color[2][]{%
    \errmessage{(Inkscape) Color is used for the text in Inkscape, but the package 'color.sty' is not loaded}%
    \renewcommand\color[2][]{}%
  }%
  \providecommand\transparent[1]{%
    \errmessage{(Inkscape) Transparency is used (non-zero) for the text in Inkscape, but the package 'transparent.sty' is not loaded}%
    \renewcommand\transparent[1]{}%
  }%
  \providecommand\rotatebox[2]{#2}%
  \newcommand*\fsize{\dimexpr\f@size pt\relax}%
  \newcommand*\lineheight[1]{\fontsize{\fsize}{#1\fsize}\selectfont}%
  \ifx\svgwidth\undefined%
    \setlength{\unitlength}{255.78890342bp}%
    \ifx\svgscale\undefined%
      \relax%
    \else%
      \setlength{\unitlength}{\unitlength * \real{\svgscale}}%
    \fi%
  \else%
    \setlength{\unitlength}{\svgwidth}%
  \fi%
  \global\let\svgwidth\undefined%
  \global\let\svgscale\undefined%
  \makeatother%
  \begin{picture}(1,0.56860674)%
    \lineheight{1}%
    \setlength\tabcolsep{0pt}%
    \put(0.27723761,0.2080646){\makebox(0,0)[t]{\lineheight{1.25}\smash{\begin{tabular}[t]{c}$\xsquigright{\textrm{edge expansion}}$\end{tabular}}}}%
    \put(0.7215953,0.42970402){\makebox(0,0)[t]{\lineheight{1.25}\smash{\begin{tabular}[t]{c}$\xsquigright{\textrm{edge contraction}}$\end{tabular}}}}%
    \put(0,0){\includegraphics[width=\unitlength,page=1]{dividing_set_S2_graph.pdf}}%
  \end{picture}%
\endgroup%

%% file: Figures/barrier.pdf_tex
\begingroup%
  \makeatletter%
  \providecommand\color[2][]{%
    \errmessage{(Inkscape) Color is used for the text in Inkscape, but the package 'color.sty' is not loaded}%
    \renewcommand\color[2][]{}%
  }%
  \providecommand\transparent[1]{%
    \errmessage{(Inkscape) Transparency is used (non-zero) for the text in Inkscape, but the package 'transparent.sty' is not loaded}%
    \renewcommand\transparent[1]{}%
  }%
  \providecommand\rotatebox[2]{#2}%
  \newcommand*\fsize{\dimexpr\f@size pt\relax}%
  \newcommand*\lineheight[1]{\fontsize{\fsize}{#1\fsize}\selectfont}%
  \ifx\svgwidth\undefined%
    \setlength{\unitlength}{651.49557279bp}%
    \ifx\svgscale\undefined%
      \relax%
    \else%
      \setlength{\unitlength}{\unitlength * \real{\svgscale}}%
    \fi%
  \else%
    \setlength{\unitlength}{\svgwidth}%
  \fi%
  \global\let\svgwidth\undefined%
  \global\let\svgscale\undefined%
  \makeatother%
  \begin{picture}(1,0.56593342)%
    \lineheight{1}%
    \setlength\tabcolsep{0pt}%
    \put(0,0){\includegraphics[width=\unitlength,page=1]{barrier.pdf}}%
    \put(0.00257238,0.11796997){\makebox(0,0)[lt]{\lineheight{1.25}\smash{\begin{tabular}[t]{l}$t=\tau_{i-1}$\end{tabular}}}}%
    \put(0.00257238,0.42253873){\makebox(0,0)[lt]{\lineheight{1.25}\smash{\begin{tabular}[t]{l}$t=\tau_{i+1}$\end{tabular}}}}%
    \put(0.00257238,0.27025435){\makebox(0,0)[lt]{\lineheight{1.25}\smash{\begin{tabular}[t]{l}$t=\tau_i$\end{tabular}}}}%
    \put(0.00295985,0.34448156){\makebox(0,0)[lt]{\lineheight{1.25}\smash{\begin{tabular}[t]{l}$t=t_{i+1}$\end{tabular}}}}%
    \put(0.00273284,0.19280519){\makebox(0,0)[lt]{\lineheight{1.25}\smash{\begin{tabular}[t]{l}$t=t_i$\end{tabular}}}}%
    \put(0,0){\includegraphics[width=\unitlength,page=2]{barrier.pdf}}%
    \put(0.53473256,0.46057222){\makebox(0,0)[lt]{\lineheight{1.25}\smash{\begin{tabular}[t]{l}$A_{i+1}$\end{tabular}}}}%
    \put(0.53483381,0.31199456){\makebox(0,0)[lt]{\lineheight{1.25}\smash{\begin{tabular}[t]{l}$C_{i+1}$\end{tabular}}}}%
    \put(0.53335741,0.1289657){\makebox(0,0)[lt]{\lineheight{1.25}\smash{\begin{tabular}[t]{l}$A_i$\end{tabular}}}}%
  \end{picture}%
\endgroup%

%% file: Figures/mitosis.pdf_tex
\begingroup%
  \makeatletter%
  \providecommand\color[2][]{%
    \errmessage{(Inkscape) Color is used for the text in Inkscape, but the package 'color.sty' is not loaded}%
    \renewcommand\color[2][]{}%
  }%
  \providecommand\transparent[1]{%
    \errmessage{(Inkscape) Transparency is used (non-zero) for the text in Inkscape, but the package 'transparent.sty' is not loaded}%
    \renewcommand\transparent[1]{}%
  }%
  \providecommand\rotatebox[2]{#2}%
  \newcommand*\fsize{\dimexpr\f@size pt\relax}%
  \newcommand*\lineheight[1]{\fontsize{\fsize}{#1\fsize}\selectfont}%
  \ifx\svgwidth\undefined%
    \setlength{\unitlength}{691.16352712bp}%
    \ifx\svgscale\undefined%
      \relax%
    \else%
      \setlength{\unitlength}{\unitlength * \real{\svgscale}}%
    \fi%
  \else%
    \setlength{\unitlength}{\svgwidth}%
  \fi%
  \global\let\svgwidth\undefined%
  \global\let\svgscale\undefined%
  \makeatother%
  \begin{picture}(1,0.254316)%
    \lineheight{1}%
    \setlength\tabcolsep{0pt}%
    \put(0,0){\includegraphics[width=\unitlength,page=1]{mitosis.pdf}}%
    \put(0.49286415,0.17201869){\makebox(0,0)[t]{\lineheight{1.25}\smash{\begin{tabular}[t]{c}$\widetilde{L}$\end{tabular}}}}%
    \put(0,0){\includegraphics[width=\unitlength,page=2]{mitosis.pdf}}%
    \put(1.00827266,0.07228392){\makebox(0,0)[lt]{\lineheight{1.25}\smash{\begin{tabular}[t]{l}$\varepsilon$\end{tabular}}}}%
    \put(-0.0092036,0.04180664){\makebox(0,0)[rt]{\lineheight{1.25}\smash{\begin{tabular}[t]{r}$\gamma_1$\end{tabular}}}}%
    \put(-0.00896072,0.08373749){\makebox(0,0)[rt]{\lineheight{1.25}\smash{\begin{tabular}[t]{r}$\widetilde{\gamma}_1$\end{tabular}}}}%
    \put(0.35991037,0.04180664){\color[rgb]{0,0,0}\makebox(0,0)[rt]{\lineheight{1.25}\smash{\begin{tabular}[t]{r}$\gamma_2$\end{tabular}}}}%
    \put(0.36015325,0.08373749){\color[rgb]{0,0,0}\makebox(0,0)[rt]{\lineheight{1.25}\smash{\begin{tabular}[t]{r}$\widetilde{\gamma}_2$\end{tabular}}}}%
    \put(0.72902433,0.04180664){\color[rgb]{0,0,0}\makebox(0,0)[rt]{\lineheight{1.25}\smash{\begin{tabular}[t]{r}$\gamma_3$\end{tabular}}}}%
    \put(0.72926721,0.08373749){\color[rgb]{0,0,0}\makebox(0,0)[rt]{\lineheight{1.25}\smash{\begin{tabular}[t]{r}$\widetilde{\gamma}_3$\end{tabular}}}}%
  \end{picture}%
\endgroup%

%% file: Figures/h_pm.pdf_tex
\begingroup%
  \makeatletter%
  \providecommand\color[2][]{%
    \errmessage{(Inkscape) Color is used for the text in Inkscape, but the package 'color.sty' is not loaded}%
    \renewcommand\color[2][]{}%
  }%
  \providecommand\transparent[1]{%
    \errmessage{(Inkscape) Transparency is used (non-zero) for the text in Inkscape, but the package 'transparent.sty' is not loaded}%
    \renewcommand\transparent[1]{}%
  }%
  \providecommand\rotatebox[2]{#2}%
  \newcommand*\fsize{\dimexpr\f@size pt\relax}%
  \newcommand*\lineheight[1]{\fontsize{\fsize}{#1\fsize}\selectfont}%
  \ifx\svgwidth\undefined%
    \setlength{\unitlength}{375bp}%
    \ifx\svgscale\undefined%
      \relax%
    \else%
      \setlength{\unitlength}{\unitlength * \real{\svgscale}}%
    \fi%
  \else%
    \setlength{\unitlength}{\svgwidth}%
  \fi%
  \global\let\svgwidth\undefined%
  \global\let\svgscale\undefined%
  \makeatother%
  \begin{picture}(1,1)%
    \lineheight{1}%
    \setlength\tabcolsep{0pt}%
    \put(0,0){\includegraphics[width=\unitlength,page=1]{h_pm.pdf}}%
    \put(0.78133063,0.59598965){\makebox(0,0)[lt]{\lineheight{1.25}\smash{\begin{tabular}[t]{l}$(s,\overline{s})$\end{tabular}}}}%
    \put(0,0){\includegraphics[width=\unitlength,page=2]{h_pm.pdf}}%
    \put(0.77363696,0.74418517){\makebox(0,0)[t]{\lineheight{1.25}\smash{\begin{tabular}[t]{c}$\mathcal H^+$\end{tabular}}}}%
    \put(0,0){\includegraphics[width=\unitlength,page=3]{h_pm.pdf}}%
  \end{picture}%
\endgroup%